\documentclass[11pt]{article}
\pdfoutput=1 

\usepackage{enumerate}
\usepackage{authblk}
\usepackage{latexsym,amsmath,amssymb,stmaryrd}
\usepackage{color}
\usepackage[all]{xy}

\newtheorem{thm}{Theorem}[section]
\newtheorem{lem}[thm]{Lemma}
\newtheorem{fact}[thm]{Fact}
\newtheorem{prop}[thm]{Proposition}
\newtheorem{cor}[thm]{Corollary}

\newtheorem{defi}[thm]{Definition}
\newtheorem{rem}[thm]{Remark}
\newtheorem{exa}[thm]{Example}

\newcommand\proof{\emph{Proof.}}
\newcommand\qed{\hfill $\Box$}

\newcommand{\rat}{\mathbb{Q}}
\newcommand{\real}{\mathbb{R}}
\newcommand{\creal}{\overline{\mathbb{R}}_+}
\newcommand{\Rp}{\mathbb{R}_+}
\newcommand{\Rl}{\mathbb{R}_\ell}
\DeclareMathOperator{\upc}{\uparrow\!}
\DeclareMathOperator{\dc}{\downarrow\!}
\newcommand\nat{\mathbb{N}}
\newcommand\dreal{\text{\textsf{\upshape d}}_\real}
\newcommand\dRl{\text{\textsf{\upshape d}}_\ell}

\newcommand\limp{\mathrel{\Rightarrow}}

\newcommand\Lform{\mathcal L}

\newcommand\uuarrow{\rlap{$\uparrow$}\raise.5ex\hbox{$\uparrow$}}
\newcommand\ddarrow{\rlap{$\downarrow$}\raise.5ex\hbox{$\downarrow$}}
\newcommand\Dc{\mathop{\Downarrow}}
\newcommand\identity[1]{\mathrm{id}_{#1}}
\newcommand\Prev{\mathbf{P}}
\newcommand\Angel{{\mathtt{A}}}
\newcommand\Demon{{\mathtt{D}}}
\newcommand\Nature{{\mathtt{P}}}
\newcommand\AN{{\Angel\Nature}}
\newcommand\DN{{\Demon\Nature}}
\newcommand\ADN{{\Angel\Demon\Nature}}

\newcommand\Smyth{{\mathcal Q}}
\newcommand\Hoare{{\mathcal H}}
\newcommand\Hoarez{\Hoare_{0}}

\newcommand\Plotkin{\mathcal P\ell}
\newcommand\Plotkinn{\Plotkin_{\mathcal V}}
\newcommand\PV\Plotkinn 
\newcommand\Val{{\mathbf V}}

\newcommand\Open{{\mathcal O}}

\newcommand\dH{d_{\mathcal H}}
\newcommand\dQ{d_{\mathcal Q}}
\newcommand\mH[1]{{#1}_{\mathcal H}}
\newcommand\mQ[1]{{#1}_{\mathcal Q}}

\newcommand\mP[1]{{#1}_{\mathcal P}}
\newcommand\dP{\mP{d}}

\newcommand\KRH[1]{{#1}_{\text{KR}}}
\newcommand\dKRH{\KRH{d}}
\newcommand\dG{{\text{\textsf{d}}}}
\newcommand\patch{{\text{\textsf{patch}}}}

\newcommand\cl{cl_{\mathbf B}} 
\newcommand\upB{\upc_{\mathbf B}} 
\newcommand\sea[2]{{#1} \nearrow {#2}}
\newcommand\bnd{\mathsf{b}}

\newcommand\extF[1]{\widehat{#1}}

\newtheorem{innercustomgeneric}{\customgenericname}
\providecommand{\customgenericname}{}
\newcommand{\newcustomtheorem}[2]{%
  \newenvironment{#1}[1]
  {%
   \renewcommand\customgenericname{#2}%
   \renewcommand\theinnercustomgeneric{##1}%
   \innercustomgeneric
  }
  {\endinnercustomgeneric}
}

\newcustomtheorem{customlemma}{Lemma}
\newcustomtheorem{customprop}{Proposition}
\newcustomtheorem{customrem}{Remark}

\newcommand\ForAuthors[1]
 {\par\smallskip                     
  \begin{center}
   \fbox
   {\parbox{0.9\linewidth}
    {\raggedright--- #1}
   }
  \end{center}
  \par\smallskip                     %
 }        

\begin{document}

\title{Complete Quasi-Metrics for Hyperspaces, Continuous
  Valuations, and Previsions}

\author{Jean Goubault-Larrecq}
\affil{Universit\'e Paris-Saclay, CNRS, ENS Paris-Saclay, Laboratoire M\'ethodes Formelles, 91190, Gif-sur-Yvette, France}




\maketitle


\begin{abstract}
  \noindent The Kantorovich-Rubinshte\u\i n metric is an $L^1$-like
  metric on spaces of probability distributions that enjoys several
  serendipitous properties.  It is complete separable if the
  underlying metric space of points is complete separable, and in that
  case it metrizes the weak topology.  We introduce a variant of that
  construction in the realm of quasi-metric spaces, and prove that it
  is algebraic Yoneda-complete as soon as the underlying quasi-metric
  space of points is algebraic Yoneda-complete, and that the
  associated topology is the weak topology.  We do this not only for
  probability distributions, represented as normalized continuous
  valuations, but also for subprobability distributions, for various
  hyperspaces, and in general for different brands of functionals.
  Those functionals model probabilistic choice, angelic and demonic
  non-deterministic choice, and their combinations.  The mathematics
  needed for those results are more demanding than in the simpler case
  of metric spaces.  To obtain our results, we prove a few other
  results that have independent interest, notably: continuous
  Yoneda-complete spaces are consonant; on a continuous
  Yoneda-complete space, the Scott topology on the space of
  $\creal$-valued lower semicontinuous maps coincides with the
  compact-open and Isbell topologies, and the subspace topology on
  spaces of $\alpha$-Lipschitz continuous maps also coincides with the
  topology of pointwise convergence, and is stably compact; we
  introduce and study the so-called Lipschitz-regular quasi-metric
  spaces, and we show that the formal ball functor induces a
  Kock-Z\"oberlein monad, of which all algebras are Lipschitz-regular;
  and we prove a minimax theorem where one of the spaces is not
  compact Hausdorff, but merely compact.


  Subject classification: mathematics of computing - continuous mathematics -
  topology - point-set topology.
\end{abstract}

\tableofcontents

\section{Introduction}
\label{sec:intro}

A landmark result in the theory of topological measure theory states
that the space of probability distributions $\mathcal M_1 (X)$ on a
Polish space $X$, with the weak topology, is again Polish.  A Polish
space is a separable complete metric space $X, d$, and to prove that
theorem, one must build a suitable metric on $\mathcal M_1 (X)$.

There are several possibilities here.  One may take the
L\'evy-Prohorov metric for example, but we shall concentrate on the
$L^1$-like metric defined by:
\[
\dKRH (\mu, \nu) = \sup_h \left|\int_{x \in X} h (x) d\mu - \int_{x
    \in X} h (x) d\nu\right|,
\]
where $h$ ranges over the $1$-Lipschitz maps from $X$ to $\real$.
This is sometimes called the \emph{Hutchinson metric}, after
J. E. Hutchinson's work on the theory of fractals
\cite{Hutchinson:fractals}.  L. V. Kantorovich had already introduced
the same notion in 1942 in a famous paper \cite{Kantorovich:1942},
where he showed that, on a compact metric space, $\dKRH (\mu, \nu)$ is
also equal to the minimum of $\int_{(x, y) \in X^2} d (x, y) d\tau$,
where $\tau$ ranges over the probability distributions whose first and
second marginals coincide with $\mu$ and $\nu$, respectively.  (That
result also holds on general Polish spaces, assuming that the double
integrals of $d (x, y)$ with respect to $\mu \otimes \mu$ and
$\nu \otimes \nu$ are finite \cite{Fernique:KR}.)  G. P. Rubinshte\u\i
n made important contributions to the theory, and $\dKRH$ is variously
called the Hutchinson metric, the Kantorovich metric, the
Kantorovich-Rubinshte\u\i n metric, and in the latter cases the name
of Leonid Vaser\v ste\u\i n (who considered $L^p$-like variants on the
second form) is sometimes added.

However, the $\dKRH$ metric does not metrize the weak topology unless
the
metric
$d$ is bounded (see \cite[Lemma~3.7]{Kravchenko:complete:K}).  This is
probably the reason why some papers consider a variant of $\dKRH$
where $h$ ranges over the $1$-Lipschitz maps from $X$ to $[-1, 1]$
instead of $\real$, or over the maps with Lipschitz norm at most $1$,
where the Lipschitz norm of $h$ is the sum of the Lipschitz constant
of $h$ and of $\sup_{x \in X} |h (x)|$.

In any case, $\dKRH$ or a variant produces a complete metric on a
space of probability distributions, and its open ball topology
coincides with the weak topology.
The aim of this paper is to generalize those results, in the following
directions:
\begin{enumerate}
\item We consider non-Hausdorff topological spaces, and
  \emph{quasi-metrics} instead of metrics.  A quasi-metric is
  axiomatized just like a metric, except that the axiom of symmetry $d
  (x, y) = d (y, x)$ is dropped.
\item We consider not only measures, but also non-linear variants of
  measures, defined as specific (continuous, positively homogeneous)
  functionals that we have called previsions \cite{Gou-csl07}.
\end{enumerate}
Both generalizations stem from questions from computer science, and
from the semantics of probabilistic and non-deterministic languages,
as well as from the study of transition systems with probabilistic and
non-deterministic transition relations.  In the realm of semantics,
one often considers \emph{domains} \cite{GHKLMS:contlatt}, which are
certain directed complete partial orders.  Seen as topological spaces,
with their Scott topology, those spaces are never Hausdorff (unless
the order is equality).  The question of similarity (as opposed to
bisimilarity) of transition systems also requires quasi-metrics
\cite{Gou-fossacs08b}, leading to non-Hausdorff spaces.  The mixture
of probabilistic choice (where one chooses the next action according
to some known probability distribution) and of non-deterministic
choice (where the next action is chosen by a scheduler by some unknown
but deterministic strategy) calls for previsions, or for other models
such as the powercone model \cite{TKP:nondet:prob}, which are
isomorphic under weak assumptions \cite{JGL-mscs16,KP:mixed}.

Each of these generalizations is non-trivial in itself.  As far as
quasi-metrics are concerned already, there are at least two distinct
notions of completeness, and the one we shall consider here,
Yoneda-completeness, admits several refinements (algebraicity,
continuity).  By comparison, all those notions of completeness
coincide on metric spaces.  The open ball topology is probably not the
right topology on a quasi-metric space either.  The paper
\cite{Gou-fossacs08b} shows that one can do a lot with open ball
topologies, but sometimes under strange assumptions, such as total
boundedness, or symcompactness.  The right topology, or so we claim,
is the so-called $d$-Scott topology, which we shall (re)define below.
Again, in the metric case the $d$-Scott topology and the open ball
topology are the same.

Finally, let us mention that we shall retrieve the classical case of
metrics for probability distributions by restricting to metrics, and
to linear previsions.  This will be up to a small detail: in general,
linear previsions are integration functionals not of measures, but of
continuous valuations \cite{Jones:proba,JP:proba}, a closely related
concept that behaves better with topological notions.  Measures and
continuous valuations are isomorphic concepts in several important
settings that go well beyond Hausdorff spaces \cite{KL:measureext},
and we shall see that this is also true in our case, under natural
assumptions (Proposition~\ref{prop:mes=val}).

This paper is long, to a point that it may discourage the bravest
reader.  I plan to split it into more manageable subunits that I will
attempt to publish separately.  In that context, the role of the
present document will be to describe the big picture, at least, which
those subunits won't convey.

\subsection{Outline}

We recapitulate some of the basic notions we need in
Section~\ref{sec:basics-quasi-metric}, in particular on sober spaces,
quasi-metric spaces, standard quasi-metric spaces, and Lipschitz
continuous functions.

The space of formal balls $\mathbf B (X, d)$ of a quasi-metric space
$X, d$ is probably the single most important artifact that has to be
considered in the study of quasi-metric spaces.  In
Section~\ref{sec:formal-ball-monad} we show that the formal ball
functor $\mathbf B$ is part of a monad, even a Kock-Z\"oberlein monad,
on the category of standard quasi-metric spaces and $1$-Lipschitz
continuous maps.  Although this has independent interest, one may say
that this should be outside the scope of the present paper.  However,
the spaces that are algebras of that monad, in particular the free
algebras $\mathbf B (X, d)$, will be important in the sequel, since,
as we shall show in Section~\ref{sec:lipsch-regul-spac}, they are all
Lipschitz regular.  Lipschitz regularity is a new notion that arises
naturally from the consideration of formal balls, and which will be
instrumental in later sections.

In Section~\ref{sec:comp-subs-quasi}, we characterize the compact
saturated subsets of continuous Yoneda-complete quasi-metric space.
The classical proof that $\dKRH$ is complete, in the Polish case, goes
through the study of tight families of measures, and for that one
usually appeals to a result by Hausdorff stating that, in a complete
metric space, the compact subsets are the closed precompact sets.
I had initially meant the results of Section~\ref{sec:comp-subs-quasi}
to be a preparation for a similar argument, but tightness does not
seem to be a fruitful concept in the quasi-metric setting.

Instead, what worked is a long-winded route that explores the
properties of the various spaces of lower semicontinuous, resp.,
Lipschitz continuous maps from a quasi-metric space $X, d$ to
$\creal$.  We introduce the various spaces of functions that we need
in Section~\ref{sec:cont-lipsch-maps}, and examine several ways of
approximating functions by $\alpha$-Lipschitz continuous maps there.
In Section~\ref{sec:topol-spac-maps}, we show that every continuous
Yoneda-complete quasi-metric space is consonant, and in fact a bit
more.  This allows us to show that the Scott topology coincides with
the compact-open topology on the space of lower semicontinuous maps of
such a space.  We refine this in Section~\ref{sec:topol-lform_-x},
where we show that the induced topologies on the subspaces
$\Lform_\alpha (X, d)$ of $\alpha$-Lipschitz continuous maps coincide
with the topology of pointwise convergence.  This is a key step: it
follows that those spaces $\Lform_\alpha (X, d)$ are stably compact.
That will be crucial in establishing all of our algebraicity results
in the sequel, on spaces of previsions.

In Section~\ref{sec:topol-lform_-x-1}, we examine the question whether
the topology on the space $\Lform_\infty (X, d)$ of all Lipschitz
continuous maps is determined by the subspace topologies on
$\Lform_\alpha (X, d)$, $\alpha > 0$.  This is a nasty subtle issue,
which we explain there.  We show that the topology on
$\Lform_\infty (X, d)$ is indeed determined if $X, d$ is Lipschitz
regular, vindicating our earlier introduction of the notion.

If the topology of $\Lform_\infty (X, d)$ is indeed determined, then
we shall see that our quasi-metric analogue of $\dKRH$, which we again
write as $\dKRH$, is Yoneda-complete on spaces of previsions in
Section~\ref{sec:quasi-metr-prev}.  Since we only know that the
topology of $\Lform_\infty (X, d)$ is determined when $X, d$ is
Lipschitz regular, and that seems too restrictive in practice, we show
that we can obtain Yoneda-completeness under less demanding
conditions, provided that our spaces of previsions satisfy a certain
property on their supports.  That condition again relies crucially on
the space $\mathbf B (X, d)$ of formal balls, and the fact that
$\mathbf B (X, d)$ embeds the original space $X$ and is always
Lipschitz regular.

In the remaining section, we explore various kinds of previsions.  I
have said that previsions were introduced as models of probabilistic
choice, non-deterministic choice, or both.

In Section~\ref{sec:case-cont-valu}, we examine the case of linear
previsions.  Those are isomorphic to the so-called continuous
valuations, and provide models for pure probabilistic (or
subprobabilistic) choice.  Since measures are a more well-known notion
than continuous valuations, we start by showing that, in
$\omega$-continuous Yoneda-complete quasi-metric spaces (in
particular, on all Polish spaces), the (sub)normalized continuous
valuations are the same thing as the (sub)probability measures.  There
is another reason to that: this correspondence between continuous
valuations and measures is precisely the reason why the property on
supports mentioned above is satisfied on linear previsions.  The rest
of the section establishes that the spaces of (sub)normalized
continuous valuations are algebraic Yoneda-complete (resp., continuous
Yoneda-complete) under a bounded variant $\dKRH^a$ of the $\dKRH$
quasi-metric, provided we start from an algebraic Yoneda-complete
(resp., continuous Yoneda-complete) space of points, and that the
$\dKRH^a$-Scott topology coincides with the weak topology.  We shall
also establish a splitting lemma and prove an analogue of the
Kantorovich-Rubinshte\u\i n theorem
$\dKRH (\mu, \nu) = \min_\tau \int_{(x, y) \in X^2} d (x, y) d\tau$
for continuous quasi-metric spaces there.

We follow a similar plan with the so-called \emph{Hoare powerdomain}
of $X$, namely the hyperspace of closed subsets of $X, d$ in
Section~\ref{sec:hoare-powerdomain}, with a quasi-metric $\dH$ that is
close to one half of the definition of the Hausdorff metric.  The
Hoare powerdomain models non-deterministic choice in computer science,
and more precisely \emph{angelic} non-determinism, where one might
think of the scheduler trying to pick the next action that is most
favorable to you.  This choice of a Hausdorff-like metric $\dH$ works
because that hyperspace is isomorphic to the space of so-called
sublinear, discrete previsions, and the isomorphism transports $\dKRH$
over to become $\dH$ on closed sets.  Under the assumption that $X, d$
is algebraic Yoneda-complete (resp., continuous Yoneda-complete), we
show that this hyperspace is algebraic Yoneda-complete (resp.,
continuous Yoneda-complete), and that the $\dH$-Scott topology
coincides with the lower Vietoris topology, a standard topology that
is one half of the well-known Vietoris (a.k.a.\ hit-and-miss) topology.

We again follow a similar plan with the \emph{Smyth powerdomain} of
$X$, that is, the hyperspace of compact saturated subsets of $X, d$,
with a quasi-metric $\dQ$ that is exactly the other half of the
definition of the Hausdorff metric, in
Section~\ref{sec:smyth-powerdomain}.  The Smyth powerdomain models
\emph{demonic} non-determinism in computer science, where the
scheduler tries to pick the worst possible next action.  The Smyth
powerdomain is isomorphic to the space of superlinear, discrete
previsions, and transports $\dKRH$ over to $\dQ$.  Again we show that
the Smyth powerdomain of an algebraic Yoneda-complete (resp.,
continuous Yoneda-complete) space $X, d$ is algebraic Yoneda-complete
(resp., continuous Yoneda-complete), and that its $\dQ$-Scott topology
coincides with the upper Vietoris topology, the other half of the
Vietoris topology.

We turn to the spaces of sublinear and superlinear previsions in
Section~\ref{sec:other-previsions}.  The sublinear, not necessarily
discrete, previsions model all possible mixtures of (sub)probabilistic
and angelic non-deterministic choice.  The superlinear previsions
model all possible mixtures of (sub)probabilistic and demonic
non-deterministic choice.

We again obtain similar results: on an algebraic (resp., continuous)
Yoneda-complete space $X, d$, they form algebraic (resp., continuous)
Yoneda-complete spaces under the $\dKRH^a$ quasi-metrics.  We show
that they are retracts of the Hoare, resp.\ the Smyth powerdomain on
spaces of continuous valuations through $1$-Lipschitz continuous maps.
It follows that they are isomorphic to the upper and lower powercones
of Tix, Keimel, and Plotkin \cite{TKP:nondet:prob}, which are convex
variants of the above Hoare and Smyth powerdomains of valuations.
Those isomorphisms were known when all the spaces considered are dcpos
\cite{KP:mixed}, or topological spaces \cite{JGL-mscs16}, under some
assumptions.  We show that they are also isomorphisms of quasi-metric
spaces, i.e., (continuous) isometries.  Finally, we show that the
spaces of sublinear and superlinear previsions on a continuous
Yoneda-complete space have identical $\dKRH^a$-Scott and weak
topologies.

Section~\ref{sec:other-previsions} relies on essentially everything we
have done until then, and also depends on a minimax theorem in the
style of Ky Fan \cite{KyFan:minimax} or Frenk and Kassay
\cite{FK:minimax}, but which only requires one space to be compact,
\emph{without} the Hausdorff separation axiom.  We had announced that
generalized minimax theorem as Theorem~6 in \cite{Gou-fossacs08b},
without a proof.  We prove it here, in
Section~\ref{sec:minimax-theorem-1}.

Section~\ref{sec:forks-quasi-lenses} is devoted to a similar study on
the Plotkin powerdomain first, on spaces of forks second.  The Plotkin
powerdomain traditionally models erratic non-determinism, namely the
mixture of angelic and demonic non-determinism.  Forks model mixed
probabilistic and erratic non-deterministic choice.  The Plotkin
powerdomain, a.k.a.\ the convex powerdomain, denotes several slightly
different constructions in the literature.  The most natural one in
our case is Heckmann's $\mathbf A$-valuations \cite{Heckmann:absval},
which is directly related to what we call discrete forks.  If the
underlying space is sober, we can interpret elements of the Plotkin
powerdomain as certain compact subsets called \emph{quasi-lenses}, and
the $\dKRH$ metric translates to a slight variant $\dP$ of a two-sided
Hausdorff quasi-metric.  On metric spaces, we retrieve the usual
Hausdorff metric on non-empty compact subsets.  We show that the
Plotkin powerdomain of an algebraic (resp., continuous)
Yoneda-complete quasi-metric space is algebraic (resp., continuous)
Yoneda-complete, and that the $\dP$-Scott topology coincides with a
variant of the usual two-sided Vietoris topology.  We again obtain
similar results for spaces of forks: the space of forks on an
algebraic (resp., continuous) Yoneda-complete quasi-metric space is
algebraic (resp., continuous) Yoneda-complete with the $\dKRH^a$
quasi-metric, and the $\dKRH^a$-Scott topology is the weak topology.

We conclude with some open questions in Section~\ref{sec:open-questions}.

\section{Basics}
\label{sec:basics-quasi-metric}

\subsection{General Topology}
\label{sec:general-topology}

We refer the reader to \cite{JGL-topology} for basic notions and
theorems of topology, domain theory, and in theory of quasi-metric
spaces.  The book \cite{GHKLMS:contlatt} is the standard reference on
domain theory, and I will assume known the notions of directed
complete posets (dcpo), Scott-continuous functions, the way-below
relation $\ll$, and so on.  We write $\uuarrow x$ for the set of
points $y$ such that $x \ll y$.  The Scott topology on a poset
consists of the Scott-open subsets, the upwards-closed subsets $U$
such that every directed family that has a supremum in $U$ must
intersect $U$.  A Scott-continuous map between posets is one that is
monotonic and preserves existing directed suprema, and this is
equivalent to requiring that it is continuous for the underlying Scott
topologies.

The topic of the present paper is on quasi-metric spaces of
previsions.  Chapters~6 and~7 of \cite{JGL-topology} are a recommended
read on that subject.  The paper \cite{JGL:formalballs} gives
additional information on quasi-metric spaces, which we shall also
rely on.  We shall introduce some of the concepts in
Section~\ref{sec:quasi-metric-spaces}, and others as we require them.

As far as topology is concerned, compactness does not imply
separation.  In other words, we call a subset $K$ of a topological
space compact if and only if every open cover of $K$ contains a finite
subcover.  This property is sometimes called quasicompactness.

We shall always write $\leq$ for the specialization preordering of a
topological space: $x \leq y$ if and only if every open neighborhood
of $x$ is also an open neighborhood of $y$, if and only if $x$ is in
the closure of $y$.  As a result, the closure of a single point $y$ is
also its downward closure $\dc y$.  In general, we write $\dc A$ for
the downward closure of any set $A$, $\upc A$ for its upward closure,
and $\upc x = \upc \{x\}$.

A subset $A$ of a topological space is saturated if and only if it is
equal to the intersection of its open neighborhoods.  Equivalently,
$A$ is saturated if and only if it is upwards-closed, i.e., $x \in A$
and $x \leq y$ imply $y \in A$.  We write
$\upc A = \{y \mid \exists x \in A. x \leq y\}$ for the upward closure
of $A$.  This is also the saturation of $A$, defined as the
intersection of the open neighborhoods of $A$.  If $K$ is compact,
then $\upc K$ is compact, too, and saturated.

We shall regularly rely on sobriety.  The sober spaces are exactly the
Stone duals of locales, but here is a more elementary description.  A
closed subset $C$ of a topological space $X$ is \emph{irreducible} if
and only if it is non-empty, and for any two closed subsets $C_1$,
$C_2$ such that $C \subseteq C_1 \cup C_2$, $C$ is included in $C_1$
or in $C_2$.  Equivalently, $C$ is irreducible if and only if for any
two open subsets $U_1$ and $U_2$ that both intersect $C$,
$U_1 \cap U_2$ also intersects $C$.  The closures of single points are
always irreducible.  A $T_0$ space is \emph{sober} if and only if
those are its only irreducible closed subsets.  Every Hausdorff space
is sober, and every continuous dcpo is sober in its Scott topology
\cite[Proposition~8.2.12]{JGL-topology}.

Every sober space is \emph{well-filtered} (see Proposition~8.3.5,
loc.cit.).  A space is well-filtered if and only if, for every
filtered intersection $\bigcap_{i \in I} Q_i$ of compact saturated
subsets that is included in some open set $U$, it is already the case
that $Q_i \subseteq U$ for some $i \in I$.  In a well-filtered space,
filtered intersections of compact saturated sets are compact saturated
(Proposition~8.3.6, loc.cit.).

\subsection{Quasi-Metric Spaces}
\label{sec:quasi-metric-spaces}

Let $\creal$ be the set of extended non-negative reals.  A
\emph{quasi-metric} on a set $X$ is a map
$d \colon X \times X \to \creal$ satisfying: $d (x, x)=0$;
$d (x, z) \leq d (x, y) + d (y, z)$ (triangular inequality);
$d(x,y)=d(y,x)=0$ implies $x=y$.

Given a quasi-metric space $X, d$, the \emph{open ball} $B^d_{x, <r}$
with center $x \in X$ and radius $r \in \Rp$ is
$\{y \in X \mid d (x, y) < r\}$.  The open ball topology is the
coarsest containing all open balls, and is the standard topology on
metric spaces.

In the realm of quasi-metric spaces, we prefer to use the
\emph{$d$-Scott topology}.  This is defined as follows.  A
\emph{formal ball} is a pair $(x, r)$ of a point $x \in X$ (the
center) and a number $r \in \Rp$ (the radius).  Formal balls are
ordered by $(x, r) \leq^{d^+} (y, s)$ iff $d (x, y) \leq r-s$, and
form a poset $\mathbf B (X, d)$.  We equip the latter with its Scott
topology.  There is an injective map $x \mapsto (x, 0)$ from $X$ to
$\mathbf B (X, d)$, and the $d$-Scott topology is the coarsest that
makes it continuous.  This allows us to see $X$ as a topological
subspace of $\mathbf B (X, d)$.

The notation $\leq^{d^+}$ comes from the fact that it is the
specialization ordering of $\mathbf B (X, d), d^+$, where the
quasi-metric $d^+$ is defined by $d^+ ((x, r), (y, s)) = \max (d (x,
y) -r +s, 0)$.

The $d$-Scott topology coincides with the open ball topology when $d$
is a metric \cite[Proposition~7.4.46]{JGL-topology}, or when $X, d$ is
Smyth-complete \cite[Proposition~7.4.47]{JGL-topology}.  It coincides
with the generalized Scott topology of \cite{BvBR:gms} when $X, d$ is
an algebraic Yoneda-complete quasi-metric space
\cite[Exercise~7.4.69]{JGL-topology}.

We shall define all notions when they are required.  Algebraicity will
be defined later.  For now, let us make clear what we understand by a
Yoneda-complete quasi-metric space.  Recall from Section~7.2.1 of
\cite{JGL-topology} that a \emph{Cauchy-weighted net}
${(x_i, r_i)}_{i \in I, \sqsubseteq}$ is a monotone net of formal
balls on $X, d$ (i.e., $i \sqsubseteq j$ implies
$(x_i, r_i) \leq^{d^+} (x_j, r_j)$) such that
$\inf_{i \in I} r_i = 0$.  The underlying net
${(x_i)}_{i \in I, \sqsubseteq}$ is then called
\emph{Cauchy-weightable}.  A point $x \in X$ is a \emph{$d$-limit} of
the latter net if and only if, for every $y \in X$,
$d (x, y) = \limsup_{i \in I, \sqsubseteq} d (x_i, y)$.  This is
equivalent to: for every $y \in X$, $d (x, y)$ is the supremum of the
monotone net ${(d (x_i, y) - r_i)}_{i \in I, \sqsubseteq}$
\cite[Lemma~7.4.9]{JGL-topology}, a formula which we shall prefer for
its simplicity.  The $d$-limit is unique if it exists.

Then $X, d$ is \emph{Yoneda-complete} if and only if every
Cauchy-weightable net has a $d$-limit.  (Or: if and only if every
Cauchy net has a $d$-limit; but Cauchy-weighted nets will be easier to
work with.)  This is also equivalent to requiring that
$\mathbf B (X, d)$ is a dcpo, and in that case, the least upper bound
$(x, r)$ of ${(x_i, r_i)}_{i \in I, \sqsubseteq}$ is given by
$r = \inf_{i \in I} r_i$ and $x$ is the $d$-limit of
${(x_i)}_{i \in I, \sqsubseteq}$.  This is the Kostanek-Waszkiewicz
Theorem \cite[Theorem~7.4.27]{JGL-topology}.

\begin{exa}
  \label{exa:creal}
  $\creal$ comes with a natural quasi-metric $\dreal$, defined by
  $\dreal (x, y) = 0$ if $x \leq y$, $\dreal (+\infty, y) = +\infty$
  if $y \neq +\infty$, $\dreal (x, y) = x-y$ if $x > y$ and
  $x \neq +\infty$.  Then $\leq^{\dreal^+}$ is the usual ordering
  $\leq$.  We check that the Scott topology on $\creal$ coincides with
  the $\dreal$-Scott topology.  To this end, observe that
  $\mathbf B (\creal, \dreal)$ is order-isomorphic to
  $C = \{(a, b) \in (\real \cup \{+\infty\}) \times ]-\infty, 0] \mid
  a-b \geq 0\}$
  through the map $(x, r) \mapsto (x-r, -r)$.  Since $C$ is a
  Scott-closed subset of a continuous dcpo, it is itself a continuous
  dcpo.  A base of the Scott topology on $C$ is given by open subsets
  of the form $\uuarrow (a, b) = \{(c, d) \mid a<c, b<d\}$, hence a
  base of the Scott topology on $\mathbf B (\creal, \dreal)$ is given
  by sets of the form
  $\{(x, r) \in \mathbf B (\creal, \dreal) \mid a < x-r, b <-r\}$,
  $(a, b) \in C$.  The intersections of the latter with $\creal$ are
  intervals of the form $\creal \cap ]a, +\infty]$,
  $a \in \real \cup \{+\infty\}$.  Those are exactly the Scott open
  subsets of $\creal$.
\end{exa}

\begin{exa}
  \label{exa:poset}
  Any poset $X, \leq$ gives rise to a quasi-metric space in a
  canonical way, by defining $d_\leq (x, y)$ as $0$ if $x \leq y$,
  $+\infty$ otherwise.  The $d_\leq$-Scott topology is exactly the
  Scott topology on $X$ \cite[Example~1.8]{JGL:formalballs}.
\end{exa}

To avoid certain pathologies, we shall concentrate on \emph{standard}
quasi-metric spaces \cite[Section~2]{JGL:formalballs}.  $X, d$ is
standard if and only if, for every directed family of formal balls
${(x_i, r_i)}_{i \in I}$, for every $s \in \Rp$,
${(x_i, r_i)}_{i \in I}$ has a supremum in $\mathbf B (X, d)$ if and
only if ${(x_i, r_i+s)}_{i \in I}$ has a supremum in
$\mathbf B (X, d)$.  Writing the supremum of the former as $(x, r)$,
we then have that $r = \inf_{i \in I} r_i$, and that the supremum of
the latter is $(x, r+s)$---this holds not only for $s \in \Rp$, but
for every $s \geq -r$.  In particular, the radius map
$(x, r) \mapsto r$ is Scott-continuous from $\mathbf B (X, d)$ to
$\creal^{op}$ ($\creal$ with the opposite ordering $\geq$), and for
every $s \in \Rp$, the map $\_ + s : (x, r) \mapsto (x, r+s)$ is
Scott-continuous from $\mathbf B (X, d)$ to itself
\cite[Proposition~2.4]{JGL:formalballs}.

Most quasi-metric spaces---not all---are standard: all metric spaces,
all Yoneda-complete quasi-metric spaces, all posets are standard
\cite[Proposition~2.2]{JGL:formalballs}.  $\creal, \dreal$ is
standard, being Yoneda-complete.

Given a map $f$ from a quasi-metric space $X, d$ to a quasi-metric
space $Y, \partial$, $f$ is \emph{$\alpha$-Lipschitz} if and only if
$\partial (f (x), f (y)) \leq \alpha d (x, y)$ for all $x, y \in X$.
(When $\alpha=0$ and $d (x, y) = +\infty$, we take the convention that
$0.+\infty=0$.)

For every $\alpha \in \Rp$, and every map $f \colon X, d \to Y, \partial$,
let $\mathbf B^\alpha (f)$ map $(x, r) \in \mathbf B (X, d)$ to
$(f (x), \alpha r) \in \mathbf B (Y, \partial)$.
Then $f$ is $\alpha$-Lipschitz
if and only if $\mathbf B^\alpha (f)$ is monotonic.

Contrarily to the case of spaces with the open ball topology, a
Lipschitz map need not be continuous.  There is a notion of Lipschitz
\emph{Yoneda-continuous} map, characterized as preserving so-called
$d$-limits.  When both $X, d$ and $Y, \partial$ are standard, $f$ is
$\alpha$-Lipschitz Yoneda-continuous if and only if
$\mathbf B^\alpha (f)$ is Scott-continuous
\cite[Lemma~6.3]{JGL:formalballs}.  We take the latter as our
definition:
\begin{defi}[$\alpha$-Lipschitz continuous]
  \label{defn:Lipcont}
  A map $f \colon X, d \to Y, \partial$ between quasi-metric spaces is
  \emph{$\alpha$-Lipschitz continuous} if and only if $\mathbf
  B^\alpha (f)$ is Scott-continuous.
\end{defi}
The phrase ``$\alpha$-Lipschitz continuous'' should not be read as
``$\alpha$-Lipschitz \emph{and} continuous'', rather as another notion
of continuity.  The two notions are actually equivalent in the case of
standard quasi-metric spaces, as we show in Proposition~\ref{prop:cont}
below.  The proof is similar to Proposition~7.4.52 of
\cite{JGL-topology}, which states a similar result for Yoneda-complete
quasi-metric spaces, and relies on the following lemma, similar to
Lemma~7.4.48 of loc.cit.
\begin{lem}
  \label{lemma:hole}
  Let $X, d$ be a standard quasi-metric space.  Every \emph{open hole}
  $T^d_{x, > \epsilon}$, defined as
  $\{y \in X \mid d (y, x) > \epsilon\}$, where $\epsilon \in \Rp$, is
  open in the $d$-Scott topology: it is the intersection of the
  Scott-open set $T^{d^+}_{(x, 0), > \epsilon}$ with $X$.
\end{lem}
\proof Let $V$ be the open hole $T^{d^+}_{(x, 0), > \epsilon}$.  This
is the set of formal balls $(y, s)$ such that $d (y, x) > s+\epsilon$.
We claim that $V$ is upwards-closed: for every $(y, s) \in V$ and
every $(z, t)$ such that $(y, s) \leq^{d^+} (z, t)$, we have
$d (y, x) > s+\epsilon$ and $d (y, z) \leq s-t$; by the triangular
inequality $d (y, x) \leq d (y, z) + d (z, x) \leq s-t+d(y,x)$, so
$d (y, x) > t+\epsilon$, showing that $(z, t)$ is in $V$.

Next we claim that $V$ is Scott-open.  Let ${(y_i, s_i)}_{i \in I}$ be
a directed family of formal balls that has a supremum $(y, s)$ in $V$.
Since $X, d$ is standard, $(y, s+2\epsilon)$ is the supremum of the
directed family ${(y_i, s_i+2\epsilon)}_{i \in I}$.  If no
$(y_i, s_i)$ were in $V$, then we would have
$d (y_i, x) \leq s_i + \epsilon$, i.e.,
$(y_i, s_i+2\epsilon) \leq^{d^+} (x, \epsilon)$ for every $i \in I$.
Since $(y, s+2\epsilon)$ is the least upper bound of the family,
$(y, s+2\epsilon) \leq^{d^+} (x, \epsilon)$, so
$d (y, x) \leq s+\epsilon$, contradicting $(y, s) \in V$.  Therefore
$(y_i, s_i)$ is in $V$ for some $i \in V$, showing that $V$ is
Scott-open.

Finally, $V \cap X$ consists of those points $y$ such that $d (y, x) >
0 + \epsilon$, hence is equal to $T^d_{x, > \epsilon}$, whence the
claim.  \qed

\begin{prop}
  \label{prop:cont}
  Let $X, d$ and $Y, \partial$ be two quasi-metric spaces,
  $\alpha > 0$, and $f$ be a map from $X$ to $Y$.  Consider the
  following claims:
  \begin{enumerate}
  \item $f$ is $\alpha$-Lipschitz continuous in the sense of
    Definition~\ref{defn:Lipcont};
  \item $f$ is $\alpha$-Lipschitz and continuous, from $X$ with its
    $d$-Scott topology to $Y$ with its $\partial$-Scott topology.
  \end{enumerate}
  Then (1) implies (2), and (2) implies (1) provided that
  $X, d$ and $Y, \partial$ are standard.
\end{prop}
\proof (1) $\limp$ (2).  Assume $f$ is $\alpha$-Lipschitz continuous.
Let $V$ be a $\partial$-Scott open subset of $Y$.  By definition, and
equating $Y$ with a subspace of $\mathbf B (Y, \partial)$, $V$ is the
intersection of some Scott-open subset $\mathcal V$ of
$\mathbf B (Y, \partial)$ with $Y$.  Since $\mathbf B^\alpha (f)$ is
Scott-continuous, $\mathcal U = \mathbf B^1 (f)^{-1} (\mathcal V)$ is
Scott-open in $\mathbf B (X, d)$.  Look at $U = \mathcal U \cap X$, a
$d$-Scott open subset of $X$.  We note that $x \in U$ if and only if
$(x, 0) \in \mathcal U$ if and only if
$\mathbf B^1 (f) (x, 0) = (f (x), 0)$ is in $\mathcal V$, if and only
if $f (x) \in V$, so that $U = f^{-1} (V)$.  Hence $f$ is continuous.
(2) $\limp$ (1), assuming $X, d$ and $Y, \partial$ standard.  Assume
$f$ is $\alpha$-Lipschitz and continuous.  Since $f$ is
$\alpha$-Lipschitz, $\mathbf B^\alpha (f)$ is monotonic.  In order to
show that it is Scott-continuous, consider an arbitrary directed
family ${(x_i, r_i)}_{i \in I}$ in $\mathbf B (X, d)$, with a supremum
$(x, r)$.  We see that family as a monotone net, and let
$i \sqsubseteq j$ if and only if $(x_i, r_i) \leq^{d^+} (x_j, r_j)$.
Since $X, d$ is standard, $r = \inf_{i \in I} r_i$ and $(x, 0)$ is the
supremum of the directed family ${(x_i, r_i-r)}_{i \in I}$.

$\mathbf B^\alpha (f) (x, r) = (f (x), \alpha r)$ is an upper bound of
${(f (x_i), \alpha r_i)}_{i \in I}$ by monotonicity.  Assume that it
is not least.  Then there is a formal ball $(y, s)$ such that
$(f (x_i), \alpha r_i) \leq^{\partial^+} (y, s)$ for every $i \in I$,
i.e., $\partial (f (x_i), y) \leq \alpha r_i - s$ for every $i \in I$,
and such that $(f (x), \alpha r)$ is not below $(y, s)$, i.e.,
$\partial (f (x), y) > \alpha r - s$.  Pick a real number $\eta$ such
that $\partial (f (x), y) > \eta > \alpha r - s$.  In particular,
$f (x)$ is in the open hole $T^\partial_{y > \eta}$, which is
$\partial$-Scott open by Lemma~\ref{lemma:hole}.  Since $f$ is
continuous, $U = f^{-1} (T^\partial_{y > \eta})$ is $d$-Scott open,
and contains $x$ by definition.  Let $\mathcal U$ be a Scott-open
subset of $\mathbf B (X, d)$ whose intersection with $X$ is equal to
$U$.  Since $(x, 0) \in \mathcal U$, $(x_i, r_i-r)$ is in $\mathcal U$
for all $i$ large enough; in other words, there is an $i_0 \in I$ such
that $(x_i, r_i-r) \in \mathcal U$ for all $i \in I$ such that
$i_0 \sqsubseteq i$.  Since $\mathcal U$ is upwards-closed, $(x_i, 0)$
is in $\mathcal U$, so $x_i$ is in $U$, which implies that $f (x_i)$
is in $T^\partial_{y > \eta}$, for every $i \sqsupseteq i_0$.  The
latter expands to $\partial (f (x_i), y) > \eta$ for every
$i \sqsupseteq i_0$.  However,
$\partial (f (x_i), y) \leq \alpha r_i - s$ for every $i \in I$, and
since $r = \inf_{i \in I} r_i$ is also equal to
$\inf_{i \in I, i_0 \sqsubseteq i} r_i$ (by directedness of $I$ and
the fact that $i \sqsubseteq j$ implies $r_i \geq r_j$), we obtain
that $\alpha r - s \geq \eta$.  This is impossible since
$\eta > \alpha r - s$. \qed

The latter has the following nice consequence.
\begin{lem}
  \label{lemma:d(_,x)}
  Let $X, d$ be a standard quasi-metric space.  For every point $x'
  \in X$, the function $d (\_, x') \colon x \mapsto d (x, x')$ is
  $1$-Lipschitz continuous from $X, d$ to $\creal, \dreal$.
\end{lem}
\proof It is $1$-Lipschitz because of the triangular inequality.
Relying on Proposition~\ref{prop:cont}, and since $\creal, \dreal$ is
standard, we only need to check that $d (\_, x')$ is continuous.  By
Example~\ref{exa:creal}, the $\dreal$-Scott topology is the Scott
topology on $\creal$, hence it suffices to show that the inverse image
of the Scott open $]\epsilon, +\infty]$ by $d (\_, x')$ is $d$-Scott
open.  That inverse image is the open hole $T^d_{x', > \epsilon}$, and
we conclude by Lemma~\ref{lemma:hole}.  \qed

Of particular interest are the Lipschitz continuous functions from $X,
d$ to $\creal, \dreal$. Recall that $f \colon X, d \to \creal, \dreal$ is $\alpha$-Lipschitz continuous if and only if
$\mathbf B^\alpha (f)$ is Scott-continuous.
$\mathbf B (\creal, \dreal)$ is order-isomorphic with the Scott-closed
set
$C = \{(a, b) \in (\real \cup \{+\infty\}) \times (-\infty, 0] \mid
a-b \geq 0\}$, through the map $(x, r) \mapsto (x-r, -r)$: see Example~\ref{exa:creal}.
Every order isomorphism is Scott-continuous.  Therefore $f$ is
$\alpha$-Lipschitz continuous if and only if the composition
$\xymatrix{X \ar[r]^(0.4){\mathbf B^\alpha (f)} & \mathbf B (X, d)
  \ar[r]^(0.6){\cong} & C}$ is Scott-continuous.  That composition is
$(x, r) \mapsto (f' (x, r), -\alpha r)$, where $f'$ is defined by
$f' (x, r) = f (x) - \alpha r$.  The map $(x, r) \mapsto -\alpha r$ is
Scott-continuous when $X, d$ is standard.  Hence we obtain the second
part of the following result.  The first part is obvious.
\begin{lem}
  \label{lemma:f'}
  Let $X, d$ be a standard quasi-metric space, $\alpha > 0$, and let
  $f$ be a map from $X$ to $\creal$.  Let
  $f' \colon \mathbf B (X, d) \to \real \cup \{+\infty\}$ be defined
  by $f' (x, r) = f (x) - \alpha r$.  Then:
  \begin{enumerate}
  \item $f$ is $\alpha$-Lipschitz if and only if $f'$ is monotonic;
  \item $f$ is $\alpha$-Lipschitz continuous if and only if $f'$ is
    Scott-continuous.  \qed
  \end{enumerate}
\end{lem}
Lemma~\ref{lemma:f'} is
Lemma~6.4 of \cite{JGL:formalballs}, where Lipschitz Yoneda-continuous
maps are used instead of Lipschitz continuous maps.  The two notions
are equivalent on standard quasi-metric spaces, as we have noticed
before Definition~\ref{defn:Lipcont}.

The Lipschitz continuous functions to $\creal, \dreal$ are closed
under several constructions, which we recapitulate here.
\begin{prop}
  \label{prop:alphaLip:props}
  Let $X, d$ be a standard quasi-metric space,
  $\alpha, \beta \in \Rp$, and $f$, $g$ be maps from $X, d$ to
  $\creal, \dreal$.
  \begin{enumerate}
  \item If $f$ is $\beta$-Lipschitz continuous, then $\alpha f$ is
    $\alpha\beta$-Lipschitz continuous;
  \item If $f$ is $\alpha$-Lipschitz continuous and $g$ is
    $\beta$-Lipschitz continuous then $f+g$ is
    $(\alpha+\beta)$-Lipschitz continuous;
  \item If $f$, $g$ are $\alpha$-Lipschitz continuous, then so are
    $\min (f, g)$ and $\max (f, g)$;
  \item If ${(f_i)}_{i \in I}$ is any family of $\alpha$-Lipschitz
    continuous maps, then the pointwise supremum $\sup_{i \in I} f_i$
    is also $\alpha$-Lipschitz continuous.
  \item If $\alpha \leq \beta$ and $f$ is $\alpha$-Lipschitz
    continuous then $f$ is $\beta$-Lipschitz continuous.
  \item Every constant map is $\alpha$-Lipschitz continuous.
  \end{enumerate}
\end{prop}
\proof
(1--5) were proved in \cite[Proposition~6.7~(2)]{JGL:formalballs}, and
are easy consequences of Lemma~\ref{lemma:f'}~(2).  For (6), using the
same lemma, we observe that for each constant $a$, the map $(x, r)
\mapsto a - \alpha r$ is Scott-continuous, because in a standard
space, the radius map  is Scott-continuous
    from $\mathbf B (X, d)$ to $\Rp^{op}$.
\qed

We shall also need the following result, which is obvious considering
our definition of Lipschitz continuity.
\begin{lem}
  \label{lemma:comp:Lip}
  Let $X, d$ and $Y, \partial$ and $Z, \mathfrak d$ be three quasi-metric
  spaces.  For every $\alpha$-Lipschitz continuous map
  $f \colon X, d \to Y, \partial$ and every $\beta$-Lipschitz
  continuous map $g \colon Y, \partial \to Z, \mathfrak d$, $g \circ f$
  is $\alpha\beta$-Lipschitz continuous.
\end{lem}
\proof
$\mathbf B^{\alpha\beta} (g \circ f)$ maps $(x, r)$ to $(g (f (x)),
\alpha\beta r) = \mathbf B^\beta (g) (f (x), \alpha r) = \mathbf
B^\beta (g) (\mathbf B^\alpha (f) (x, r))$.  Since $\mathbf B^\alpha (f)$
and $\mathbf B^\beta (g)$ are Scott-continuous by assumption, so is their
composition $\mathbf B^{\alpha\beta} (g \circ f)$.  \qed

\section{The Formal Ball Monad}
\label{sec:formal-ball-monad}

We now examine the space $\mathbf B (X, d)$, with its quasi-metric
$d^+ ((x, r), (y, s)) = \max (d (x, y) -r +s, 0)$.  The following is
the first part of Exercise~7.4.54 of \cite{JGL-topology}.  It might
seem a mistake that this does not require $X, d$ to be standard: to
dispel any doubt, we give a complete proof.
\begin{lem}
  \label{lemma:mu:cont}
  Let $X, d$ be a quasi-metric space.  The map
  $\mu_X \colon ((x, r), s) \in \mathbf B (\mathbf B (X, d), d^+)
  \mapsto (x, r+s) \in \mathbf B (X, d)$ is Scott-continuous.
\end{lem}
\proof The map $\mu_X$ is monotonic: if
$((x, r), s) \leq^{d^{++}} ((x', r'), s')$, then
$d^+ ((x, r), (x', r')) \leq s-s'$, meaning that
$\max (d (x, x') - r + r', 0) \leq s-s'$, and this implies
$d (x, x') \leq r-r'+s-s'$, hence $(x, r+s) \leq^{d^+} (x', r'+s')$.

We claim that $\mu_X$ is Scott-continuous.  Consider a directed family
of formal balls ${((x_i, r_i), \allowbreak s_i)}_{i \in I}$ in
$\mathbf B (\mathbf B (X, d), d^+)$ with a supremum $((x, r), s)$.  We
must show that $(x, r+s)$ is the supremum of the directed family
${(x_i, r_i+s_i)}_{i \in I}$.  It is certainly an upper bound, since
$\mu_X$ is monotonic.  Let $(y, t)$ be another upper bound of
${(x_i, r_i+s_i)}_{i \in I}$.  Let $a = \max (t-s, 0)$,
$b = t-a = \min (s, t)$.

We claim that $((y, a), b)$ is an upper bound of
${((x_i, r_i), s_i)}_{i \in I}$.  For every $i \in I$, by assumption
$(x_i, r_i+s_i) \leq^{d^+} (y, t)$, so $d (x_i, y) \leq r_i+s_i - t$.
We must check that $d^+ ((x_i, r_i), (y, a)) \leq s_i-b$, namely that
$\max (d (x_i, y) - r_i + a, 0) \leq s_i-b$, and this decomposes into
$d (x_i, y) \leq r_i - a + s_i - b$ and $s_i \geq b$.  The latter is
proved as follows: since $b = \min (s, t)$, $b \leq s$, and since
$((x, r), s)$ is an upper bound of ${((x_i, r_i), s_i)}_{i \in I}$,
$s \leq s_i$ for every $i \in I$.  The former is equivalent to
$d (x_i, y) \leq r_i + s_i - t$, since $a+b=t$, and this is our
assumption.

Since $((x, r), s)$ is the least upper bound of ${((x_i, r_i),
  s_i)}_{i \in I}$, $((x, r), s) \leq^{d^{++}} ((y, a), b)$, so $\max
(d (x, y) - r + a, 0) \leq s-b$.  In particular, $d (x, y) \leq
r+s-a-b = r+s-t$, so $(x, r+s) \leq^{d^+} (y, t)$.  This shows that
$(x, r+s)$ is the least upper bound of ${(x_i, r_i+ s_i)}_{i \in
  I}$, hence that $\mu_X$ is Scott-continuous.  \qed

\begin{lem}
  \label{lemma:dScott=Scott}
  For every quasi-metric space $X, d$:
  \begin{enumerate}
  \item the map
    $\eta_{\mathbf B (X, d)} \colon (x, r) \mapsto ((x, r), 0)$ is
    Scott-continuous;
  \item the $d^+$-Scott topology on $\mathbf B (X, d)$ coincides with
    the Scott topology.
  \end{enumerate}
\end{lem}
\proof (1) This is \cite[Exercise~7.4.53]{JGL-topology}.  Monotonicity
is clear: if $(x, r) \leq^{d^+} (y, s)$, then $d (x, y) \leq r-s$, so
$d^+ ((x, r), (y, s)) = \max (d (x, y) - r + s, 0) = 0$.  For every
directed family ${(x_i, r_i)}_{i \in I}$ in $\mathbf B (X, d)$, with
supremum $(x, r)$, by monotonicity $((x, r), 0)$ is an upper bound of
${((x_i,r_i), 0)}_{i \in I}$.  Consider another upper bound
$((y, s), t)$.  For every $i \in I$,
$((x_i, r_i), 0) \leq^{d^{++}} ((y, s), t)$, namely
$d^+ ((x_i, r_i), (y, s)) \leq 0 - t$.  That implies $t=0$, and
$d (x_i, y) - r_i + s \leq 0$.  The latter means that
$(x_i, r_i) \leq^{d^+} (y, s)$, and as this holds for every $i \in I$,
$(x, r) \leq^{d^+} (y, s)$.  Therefore
$d^+ ((x, r), (y, s)) = \max (d (x, y) - r + s, 0) = 0$, showing that
$((x, r), 0) \leq ((y, s), 0) = ((y, s), t)$.

(2) This is Exercise~7.4.54 of \cite{JGL-topology}. Using (1), every
$d^+$-Scott open subset $V$ of $\mathbf B (X, d)$ is Scott-open: by
definition, $V = \eta_{\mathbf B (X, d)}^{-1} (\mathcal V)$ for some
Scott-open subset $\mathcal V$ of $\mathbf B (\mathbf B (X, d), d^+)$,
and since $\eta_{\mathbf B (X, d)}$ is Scott-continuous, $V$ is
Scott-open.  To show the converse implication, we observe that
$\mu_X \circ \eta_{\mathbf B (X, d)}$ is the identity map, by
Lemma~\ref{lemma:eta:mu}~$(ii)$.  Then for every Scott-open subset $V$
of $\mathbf B (X, d)$, $V$ is equal to
$(\mu_X \circ \eta_{\mathbf B (X, d)})^{-1} (V)$, hence to
$\eta_{\mathbf B (X, d)}^{-1} (\mathcal V)$ where $\mathcal V$ is the
Scott-open subset $\mu_X^{-1} (V)$.  This exhibits $V$ as a
$d^+$-Scott open subset of $\mathbf B (X, d)$.  \qed

We write $\eta_X \colon X \to \mathbf B (X, d)$ for the embedding $x
\mapsto (x, 0)$.
\begin{lem}
  \label{lemma:eta:lipcont}
  Let $X, d$ be a standard quasi-metric space.  The map $\eta_X \colon
  x \mapsto (x, 0)$ is $1$-Lipschitz continuous from $X, d$ to
  $\mathbf B (X, d), d^+$.
\end{lem}
\proof It is $1$-Lipschitz, because $d^+ ((x, 0), (y, 0)) = d (x, y)$.
It is continuous by definition.  Now apply
Proposition~\ref{prop:cont}.  \qed

\begin{lem}
  \label{lemma:eta:mu}
  Let $X, d$ be a quasi-metric space.  The following relations hold:
  $(i)$
  $\mu_X \circ \eta_{\mathbf B (X, d)} = \identity {\mathbf B (X,
    d)}$; $(ii)$
  $\mu_X \circ \mathbf B^1 (\eta_X) = \identity {\mathbf B (X, d)}$;
  $(iii)$
  $\mu_X \circ \mu_{\mathbf B (X, d)} = \mu_X \circ \mathbf B^1
  (\mu_X)$; $(iv)$
  $\eta_{\mathbf B (X, d)} \circ \mu_X \geq \identity {\mathbf B
    (\mathbf B (X, d), d^+)}$.
\end{lem}
\proof $(i)$
$\mu_X (\eta_{\mathbf B (X, d)} (x, r)) = \mu_X ((x, r), 0) = (x, r+0)
= (x, r)$.

$(ii)$
$\mu_X (\mathbf B^1 (\eta_X) (x, r)) = \mu_X (\eta_X (x), r) = \mu_X
((x, 0), r) = (x, 0+r) = (x, r)$.

$(iii)$
$\mu_X (\mu_{\mathbf B (X, d)} (((x, r), s), t)) = \mu_X ((x, r), s+t)
= (x, r+s+t)$, while
$\mu_X (\mathbf B^1 (\mu_X) (((x, r), \allowbreak s), t)) = \mu_X
(\mu_X ((x, r), s), t) = \mu_X ((x, r+s), t) = (x, r+s+t)$.

$(iv)$
$\eta_{\mathbf B (X, d)} (\mu_X ((x, r), s)) = \eta_{\mathbf B (X, d)}
(x, r+s) = ((x, r+s), 0)$.  We must check that this is larger than or
equal to $((x, r), s)$, namely that
$d^+ ((x, r), (x, r+s)) \leq s - 0$.  Since
$d^+ ((x, r), (x, r+s)) = \max (d (x, x) -r +r+s, 0) = s$, this is
clear.  \qed

A \emph{monad} on a category is the data of an endofunctor $T$, two
natural transformations $\eta \colon \identity\relax \to T$ and
$\mu \colon T^2 \to T$ satisfying:
$\mu_X \circ \eta_{TX} = \identity {TX}$,
$\mu_X \circ T \eta_X = \identity {TX}$, and
$\mu_X \circ \mu_{TX}= \mu_X \circ T (\mu_X)$.  The first three
statements of Lemma~\ref{lemma:eta:mu} seem to indicate that
$T = \mathbf B$ gives rise to a monad, where the functor $\mathbf B$
maps every quasi-metric space $X, d$ to the quasi-metric space
$\mathbf B (X, d), d^+$, and every $1$-Lipschitz continuous map
$f \colon X, d \to Y, \partial$ to
$\mathbf B^1 (f) \colon (x, r) \mapsto (f (x), r)$.

The devil hides in details, one says.  We must work in a category of
standard, not arbitrary quasi-metric spaces, for $\eta_X$ to be a
morphism (see Lemma~\ref{lemma:eta:lipcont}).  Then we must show that
$\mathbf B$ maps standard spaces to standard spaces.  This is done in
several steps.

\begin{lem}
  \label{lemma:sup=>dlim}
  Let $X, d$ be a standard quasi-metric space, and let
  ${(x_i, r_i)}_{i \in I, \sqsubseteq}$ be a monotone net of formal
  balls on $X, d$ with supremum $(x, r)$.  Then
  $r = \inf_{i \in I} r_i$ and $x$ is the $d$-limit of
  ${(x_i)}_{i \in I, \sqsubseteq}$.
\end{lem}
\proof This is similar to the proof of
\cite[Lemma~7.4.26]{JGL-topology}, which assumes that
$\mathbf B (X, d)$ is a dcpo, whereas we only assume that $X, d$ is
standard.  Since $X, d$ is standard, $r = \inf_{i \in I} r_i$.

Since $(x_i, r_i) \leq^{d^+} (x, r)$ for each $i \in I$, $d (x_i, x)
\leq r_i - r$.  For every $y \in X$, $d (x_i, y) \leq d (x_i, x) + d
(x, y) \leq r_i - r + d (x, y)$.  Taking suprema over $i \in I$,
$\sup_{i \in I} (d (x_i, y) - r_i + r) \leq d (x, y)$.

Assume that the inequality were strict.  Let
$s = \sup_{i \in I} (d (x_i, y) - r_i + r) < d (x, y)$.  In
particular, $s < +\infty$.  For every $i \in I$,
$d (x_i, y) - r_i + r \leq s$, so $d (x_i, y) \leq r_i-r+s$, i.e.,
$(x_i, r_i-r+s) \leq^{d^+} (y, 0)$.  Since $X, d$ is standard, and
$(x, r)$ is the supremum of ${(x_i, r_i)}_{i \in I}$, $(x, s)$ is the
supremum of ${(x_i, r_i-r+s)}_{i \in I}$.  It follows that
$(x, s) \leq^{d^+} (y, 0)$, that is, $d (x, y) \leq s$, which is
impossible.  \qed

This has the following interesting consequence (which we shall not
use, however).  Standardness says that if
${(x_i, r_i)}_{i \in I, \sqsubseteq}$ and
${(x_i, s_i)}_{i \in I, \sqsubseteq}$ are two monotone nets of formal
balls with the same underlying net ${(x_i)}_{i \in I, \sqsubseteq}$,
then one of them has a supremum if and only if the other one has,
\emph{provided} that $r_i$ and $s_i$ differ by a constant.  In that
case, those suprema are of the form $(x, r)$ and $(x, s)$ for the same
point $x$ (and where $r$ and $s$ differ by the same constant).  The
following lemma shows that this holds without any condition on $r_i$
and $s_i$. That might be used to (re)define the notion of $d$-limit
$x$ of a net ${(x_i)}_{i \in I, \sqsubseteq}$, as the center of the
supremum of ${(x_i, r_i)}_{i \in I, \sqsubseteq}$, for some family of
radii $r_i$ that make ${(x_i, r_i)}_{i \in I, \sqsubseteq}$ a monotone
net of formal balls.  The following lemma shows that that definition
is independent of the chosen radii $r_i$, assuming just standardness.
\begin{prop}
  \label{prop:std:dlim}
  Let $X, d$ be a standard quasi-metric space.  Let
  ${(x_i, r_i)}_{i \in I, \sqsubseteq}$ and
  ${(x_i, s_i)}_{i \in I, \sqsubseteq}$ be two monotone nets of formal
  balls with the same underlying net ${(x_i)}_{i \in I, \sqsubseteq}$.
  If ${(x_i, r_i)}_{i \in I, \sqsubseteq}$ has a supremum $(x, r)$,
  then $r = \inf_{i \in I} r_i$ and
  ${(x_i, s_i)}_{i \in I, \sqsubseteq}$ also has a supremum, which is
  equal to $(x, s)$, where $s = \inf_{i \in I} s_i$.
\end{prop}
\proof If ${(x_i, r_i)}_{i \in I, \sqsubseteq}$ has a supremum
$(x, r)$, then $r = \inf_{i \in I} r_i$ because $X, d$ is standard.
By Lemma~\ref{lemma:sup=>dlim}, $x$ is the $d$-limit of
${(x_i)}_{i \in I, \sqsubseteq}$.  Lemma~7.4.25 of \cite{JGL-topology}
states that if ${(x_i, s_i)}_{i \in I, \sqsubseteq}$ is a monotone net
of formal balls and if ${(x_i)}_{i \in I, \sqsubseteq}$ has a
$d$-limit $x$, then $(x, s)$ is the supremum of
${(x_i, s_i)}_{i \in I, \sqsubseteq}$, where $s = \inf_{i \in I} s_i$.
\qed

\begin{prop}
  \label{prop:B:std}
  For every standard quasi-metric space $X, d$,
  $\mathbf B (X, d), d^+$ is standard.
\end{prop}
\proof Let ${((x_i, r_i), s_i)}_{i \in I}$ be a directed family of
formal balls on $\mathbf B (X, d), d^+$.  This is a monotone net,
provided we define $\sqsubseteq$ by $i \sqsubseteq j$ if and only if
$((x_i, r_i), s_i) \leq^{d^{++}} ((x_j, r_j), s_j)$.

Assume that ${((x_i, r_i), s_i)}_{i \in I}$ has a supremum
$((x, r), s)$.  Since $\mu_X$ is Scott-continuous
(Lemma~\ref{lemma:mu:cont}), ${(x_i, r_i+s_i)}_{i \in I}$ is a
directed family with supremum $(x, r+s)$.  We use the fact that $X, d$
is standard and apply Lemma~\ref{lemma:sup=>dlim} to obtain that
$r+s = \inf_{i \in I} (r_i + s_i)$ and that $x$ is the $d$-limit of
${(x_i)}_{i \in I, \sqsubseteq}$.

Since ${((x_i, r_i), s_i)}_{i \in I}$ is directed, ${(s_i)}_{i \in I}$
is filtered.  Let $s_\infty = \inf_{i \in I} s_i$.  Since $\mu_X$ is
Scott-continuous hence monotonic, ${(x_i, r_i+s_i)}_{i \in I}$ is also
directed, so ${(r_i+s_i)}_{i \in I}$ is filtered.  Its infimum is
$r+s$.  Let
$r_\infty = \inf_{i \in I} (r_i + s_i) - s_\infty = r+s-s_\infty$.

Consider $((x, r_\infty), s_\infty)$.  For every $i \in I$, we claim
that $((x_i, r_i), s_i) \leq^{d^{++}} ((x, r_\infty), s_\infty)$.  For
that, we compute
$d^+ ((x_i, r_i), (x, r_\infty)) = \max (d (x_i, x) - r_i + r_\infty,
0)$, and we check that this is less than or equal to $s_i - s_\infty$.
Since $s_i - s_\infty \geq 0$ by definition of $s_\infty$, it remains
to verify that $d (x_i, x) - r_i + r_\infty \leq s_i - s_\infty$.
Using the equality $r_\infty + s_\infty = r+s$, obtained as a
consequence of the definition of $r_\infty$, we have to verify the
equivalent inequality $d (x_i, x) \leq r_i + s_i - r - s$.  That one
is obvious, since $(x_i, r_i+s_i)$ is below the supremum $(x, r+s)$.

Since $((x_i, r_i), s_i) \leq^{d^{++}} ((x, r_\infty), s_\infty)$ for
every $i \in I$,
$((x, r), s) \leq^{d^{++}} ((x, r_\infty), s_\infty)$.  However, we
claim that the reverse inequality also holds.  Indeed, we start by
observing that $s \leq s_i$ for every $i \in I$, since
$((x_i, r_i), s_i) \leq^{d^{++}} ((x, r), s)$.  Hence
$s \leq s_\infty$.  Since $r_\infty = r+s-s_\infty$,
$r_\infty \leq r$.  Therefore
$d^+ ((x, r_\infty), (x, r)) = \max (d (x, x) - r_\infty + r, 0) = r -
r_\infty$, and the latter is equal to, hence less than or equal to
$s_\infty-s$.  This means that
$((x, r_\infty), s_\infty) \leq^{d^{++}} ((x, r), s)$.

Having inequalities in both directions, we conclude that $((x, r), s)
= ((x, r_\infty), s_\infty)$.  This entails the important fact that $s
= s_\infty = \inf_{i \in I} s_i$.

We use that to show that for any $a \geq -s$, $((x, r), s+a)$
is the supremum of ${((x_i, r_i), s_i+a)}_{i \in I}$.  Since
$((x_i, r_i), s_i) \leq^{d^{++}} ((x, r), s)$, we have
$((x_i, r_i), s_i+a) \leq^{d^{++}} ((x, r), s+a)$.  Now consider any
other upper bound $((x', r'), s')$ of
${((x_i, r_i), s_i+a)}_{i \in I}$.  We have $s' \leq s_i + a$ for
every $i \in I$, whence using the equality $s = \inf_{i \in I} s_i$,
$s' \leq s+a$.  We wish to check that
$((x, r), s+a) \leq^{d^{++}} ((x', r'), s')$, equivalently
$d^+ ((x, r), (x', r')) \leq s+a-s'$, and that reduces to $s+a-s'$
(which we have just shown) and $d (x, x') \leq r+s+a-r'-s'$.  In order
to establish the latter, recall that $(x, r+s)$ is the supremum of
${(x_i, r_i+s_i)}_{i \in I}$.  Since $X, d$ is standard (and since
$a \geq -s \geq -r-s$), $(x, r+s+a)$ is also the supremum of
${(x_i, r_i+s_i+a)}_{i \in I}$.  Since $\mu_X$ is monotonic,
$(x', r'+s')$ is an upper bound of ${(x_i, r_i+s_i+a)}_{i \in I}$, so
$(x, r+s+a) \leq^{d^+} (x', r'+s')$, or equivalently
$d (x, x') \leq r+s+a-r'-s'$: that is exactly what we wanted to prove.

Let us recap: for every directed family
${((x_i, r_i), s_i)}_{i \in I}$ with supremum $((x, r), s)$, then
$s = \inf_{i \in I} s_i$ and for every $a \geq -s$, $((x, r), s+a)$ is
the supremum of ${((x_i, r_i), s_i+a)}_{i \in I}$.  This certainly
implies that if ${((x_i, r_i), s_i)}_{i \in I}$ has a supremum, then
${((x_i, r_i), s_i+a)}_{i \in I}$ also has one for every $a \in \Rp$.
Conversely, if ${((x_i, r_i), s_i+a)}_{i \in I}$ has a supremum for
some $a \in \Rp$, then it is of the form $((x, r), s+a)$ where
$s = \inf_{i \in I} s_i$, and for every $a' \geq -s-a$,
$((x, r), s+a+a')$ is the supremum of
${((x_i, r_i), s_i+a+a')}_{i \in I}$.  In particular, for $a' = -a$,
${((x_i, r_i), s_i)}_{i \in I}$ has a supremum.  \qed

\begin{lem}
  \label{lemma:mu:lipcont}
  Let $X, d$ be a standard quasi-metric space.  The map
  $\mu_X \colon ((x, r), s) \mapsto (x, r+s)$ is $1$-Lipschitz
  continuous from $\mathbf B (\mathbf B (X, d), d^+), d^{++}$ to
  $\mathbf B (X, d), d^+$.
\end{lem}
\proof
We first check that $\mu_X$ is $1$-Lipschitz:
\begin{eqnarray*}
  d^+ (\mu_X ((x, r), s), \mu_X ((x', r'), s'))
  & = & d^+ ((x, r+s), (x', r'+s')) \\
  & = & \max (d (x, x') - r - s + r' + r', 0),
\end{eqnarray*}
while
\begin{eqnarray*}
  d^{++} (((x, r), s), ((x', r'), s'))
  & = & \max (d^+ ((x, r), (x', r')) - s + s', 0) \\
  & = & \max (\max (d (x, x') - r + r', 0) - s + s', 0) \\
  & = & \max (d (x, x') - r + r' - s + s', -s+s', 0) \\
  & = & \max (d^+
        (\mu_X ((x, r), s), \mu_X ((x', r'), s')), -s+s'),
\end{eqnarray*}
which implies
$d^{++} (((x, r), s), ((x', r'), s')) \geq d^+ (\mu_X ((x, r), s),
\mu_X ((x', r'), s'))$.

Next, $\mu_X$ is Scott-continuous from
$\mathbf B (\mathbf B (X, d), d^+)$ to $\mathbf B (X, d)$ by
Lemma~\ref{lemma:mu:cont}.  The Scott topology on the former coincides
with its $d^{++}$-Scott topology and the Scott topology on the latter
coincides with its $d^+$-Scott topology, by
Lemma~\ref{lemma:dScott=Scott}~(2).  Hence $\mu_X$ is continuous from
$\mathbf B (\mathbf B (X, d), d^+)$ to $\mathbf B (X, d)$, with their
$d^{++}$-Scott, resp.\ $d^+$-Scott topologies.

Since $\mathbf B (X, d), d^+$ and
$\mathbf B (\mathbf B (X, d), d^+), d^{++}$ are standard, by
Proposition~\ref{prop:B:std}, we can apply the (2) $\limp$ (1)
direction of Proposition~\ref{prop:cont}, and we obtain that $\mu_X$
is $1$-Lipschitz continuous.  \qed

\begin{lem}
  \label{lemma:Balpha:f}
  Let $X, d$ and $Y, \partial$ be two standard quasi-metric spaces,
  and $f$ be an $\alpha$-Lipschitz continuous map from $X, d$ to
  $Y, \partial$, with $\alpha > 0$.  Then $\mathbf B^\alpha (f)$ is
  $\alpha$-Lipschitz continuous from $\mathbf B (X, d), d^+$ to
  $\mathbf B (Y, \partial), \partial^+$.
\end{lem}
\proof
We verify that $\mathbf B^\alpha (f)$ is $\alpha$-Lipschitz:
\begin{eqnarray*}
  \partial^+ (\mathbf B^\alpha (f) (x, r), \mathbf B^\alpha (y, s))
  & = & \partial^+ ((f (x), \alpha r), (f (y), \alpha s)) \\
  & = & \max (\partial (f (x), f (y)) - \alpha r + \alpha s, 0) \\
  & \leq & \max (\alpha d (x, y) - \alpha r + \alpha s, 0) \\
  & = & \alpha \max (d (x, y) - r +  s, 0) = \alpha d^+ ((x, r), (y, s)).
\end{eqnarray*}
By definition of $\alpha$-Lipschitz continuity, $\mathbf B^\alpha (f)$
is Scott-continuous.  Since the Scott topology on $\mathbf B (X, d)$
coincides with the $d^+$-Scott topology, and similarly for $Y$, thanks
to Lemma~\ref{lemma:dScott=Scott}~(2), $\mathbf B^\alpha (f)$ is
continuous with respect to the $d^+$-Scott and $\partial^+$-Scott
topologies.  Now use that $\mathbf B (X, d), d^+$ and
$\mathbf B (Y, \partial), \partial^+$ are standard, owing to
Proposition~\ref{prop:B:std}, and apply Proposition~\ref{prop:cont} to
conclude that $\mathbf B^\alpha (f)$ is $\alpha$-Lipschitz continuous.
\qed

\begin{prop}
  \label{prop:B:monad}
  The triple $(\mathbf B, \eta, \mu)$ is a monad on the category of
  standard quasi-metric spaces and $1$-Lipschitz continuous maps.
\end{prop}
\proof We shall show the equivalent claim that
$(\mathbf B, \eta, \_^\dagger)$ is a Kleisli triple, that is: $(i)$
$\mathbf B$ maps objects of the category (standard quasi-metric
spaces) to objects of the category; $(ii)$ $\eta_X$ is a morphism from
$X, d$ to $\mathbf B (X, d), d^+$ (a $1$-Lipschitz continuous map);
$(iii)$ for every morphism
$f \colon X, d \to \mathbf B (Y, \partial), \partial^+$, $f^\dagger$
is a morphism from $\mathbf B (X, d), d^+$ to
$\mathbf B (Y, \partial), \partial^+$ such that: $(a)$
$\eta_X^\dagger = \identity {\mathbf B (X, d)}$; $(b)$
$f^\dagger \circ \eta_X = f$; $(c)$
$f^\dagger \circ g^\dagger = (f^\dagger \circ g)^\dagger$.  For that,
we define $f^\dagger$ as mapping $(x, r)$ to $(y, r+s)$, where
$(y, s) = f (x)$.

Proposition~\ref{prop:B:std} gives us $(i)$, and
Lemma~\ref{lemma:eta:lipcont} gives us $(ii)$.  We devote the rest of
this proof to $(iii)$.

We must start by checking that $f^\dagger$ is a morphism for every
morphism $f \colon X, d \to \mathbf B (Y, \partial), \partial^+$.  We
have defined $f^\dagger (x, r)$ as $(y, r+s)$ where $(y, s) = f (x)$,
and we notice that $f^\dagger$ is equal to
$\mu_Y \circ \mathbf B^1 (f)$.  This is $1$-Lipschitz continuous
because $\mu_Y$ and $\mathbf B^1 (f)$ both are, by
Lemma~\ref{lemma:mu:lipcont} and Lemma~\ref{lemma:Balpha:f}
respectively.

The equalities $(a)$, $(b)$, $(c)$ are easily checked.

Any Kleisli triple $(T, \eta, \_^\dagger)$ gives rise to a monad
$(T, \eta, m)$ by letting $m_X = \identity X^\dagger$.  Here $m_X$
maps $((x, r), s)$ to $(x, r+s)$, hence coincides with $\mu_X$,
finishing the proof.  \qed

A \emph{left KZ-monad} \cite[Definition~4.1.2,
Lemma~4.1.1]{Escardo:properly:inj} (short for \emph{Kock-Z\"oberlein
  monad}) is a monad $(T, \eta, \mu)$ on a poset-enriched category
such that $T$ is monotonic on homsets, and either one of the following
equivalent conditions hold:
\begin{enumerate}
\item $T\eta_X \leq \eta_{TX}$ for every object $X$;
\item a morphism $\alpha \colon TX \to X$ is the structure map of a
  $T$-algebra if and only if $\alpha \circ \eta_X = \identity X$ and
  $\identity {TX} \leq \eta_X \circ \alpha$;
\item $\mu_X \dashv \eta_{TX}$ for every object $X$;
\item $T\eta_X \dashv \mu_X$ for every object $X$.
\end{enumerate}
The notion stems from work by A. Kock on doctrines in 2-categories
\cite{Kock:KZmonad}, and the above equivalence is due to Kock, in the
more general case of 2-categories.  The notation $f \dashv g$ means
that the two morphisms $f$ and $g$ are \emph{adjoint}, namely,
$f \circ g \leq \identity \relax$ and
$\identity \relax \leq g \circ f$.  A \emph{$T$-algebra} is an object
$X$ together with a morphism $\alpha \colon TX \to X$ called its
\emph{structure map}, such that $\alpha \circ \eta_X = \identity X$
and $\alpha \circ \mu_X = \alpha \circ T\alpha$.  $TX$ is always a
$T$-algebra with structure map $\mu_X$, called the \emph{free
  $T$-algebra} on $X$.

The category of standard quasi-metric spaces and $1$-Lipschitz
continuous maps is poset-enriched.  Each homset is ordered by: for
$f, g \colon X, d \to Y, \partial$, $f \leq g$ if and only if for
every $x \in X$, $f (x) \leq^\partial g (y)$.  If $f \leq g$, then
$\mathbf B^1 (f) \leq \mathbf B^1 (g)$, since for every
$(x, r) \in \mathbf B (X, d)$,
$\mathbf B^1 (f) (x, r) = (f (x), r) \leq^{\partial^+} (g (x), r) =
\mathbf B^1 (g) (x, r)$.

Condition (3) of a left KZ-monad reads:
$\mu_X \circ \eta_{TX} \leq \identity {TX}$ and
$\identity X \leq \eta_{TX} \circ \mu_X$.  For $T = \mathbf B$, those
follow from Lemma~\ref{lemma:eta:mu}~(1) and (4).  Hence:
\begin{prop}
  \label{prop:B:KZ}
  The triple $(\mathbf B, \eta, \mu)$ is a left KZ-monad on the
  category of standard quasi-metric spaces and $1$-Lipschitz
  continuous maps.  \qed
\end{prop}

Kock's theorem between the equivalence of the four conditions defining
KZ-monads yields the following immediately.
\begin{prop}
  \label{prop:B:alg}
  Let $X, d$ be a standard quasi-metric space, and
  $\alpha \colon \mathbf B (X, d), d^+ \to X, d$ be a $1$-Lipschitz
  continuous map.  The following are equivalent:
  \begin{enumerate}
  \item $\alpha$ is the structure map of a $\mathbf B$-algebra;
  \item for every $x \in X$, $\alpha (x, 0) = x$ and for all $r, s \in
    \Rp$, $\alpha (x, r+s) = \alpha (\alpha (x, r), s)$;
  \item for every $x \in X$, and every $r \in \Rp$, $\alpha (x, r)$ is
    a point in the closed ball $B^d_{x, \leq r}$, which is equal to
    $x$ if $r=0$;
  \item for every $x \in X$, $\alpha (x, 0) = x$.
  \end{enumerate}
\end{prop}
\proof The equivalence between (1) and (2) is the definition of an
algebra of a monad.  Look at the second equivalent condition defining
left KZ-monads, applied to the left KZ-monad $\mathbf B$
(Proposition~\ref{prop:B:KZ}).  This implies that (1) is equivalent to
$\alpha \circ \eta_X = \identity X$ (i.e., $\alpha (x, 0)=x$ for every
$x \in X$), and to $\identity X \leq \eta_X \circ \alpha$; the latter
means that for every $x \in X$ and every $r \in \Rp$,
$(x, r) \leq^{d^+} \eta_X (\alpha (x, r))$, equivalently,
$d (x, \alpha (x, r)) \leq r$, i.e.,
$\alpha (x, r) \in B^d_{x, \leq r}$.  Finally, clearly (3) implies
(4).  In the reverse direction, note that since $\alpha$ is
$1$-Lipschitz,
$d (\alpha (x, 0), \alpha (x, r)) \leq d^+ ((x, 0), (x, r)) = r$.
Since $\alpha (x, 0)=x$, this implies that $\alpha (x, r)$ is in
$B^d_{x, \leq r}$.  \qed

\section{Lipschitz Regular Spaces}
\label{sec:lipsch-regul-spac}

For every open subset $U$ of $X$ in its $d$-Scott topology, there is a
largest open subset $\widehat U$ of $\mathbf B (X, d)$ such that
$\widehat U \cap X \subseteq U$.  Then $\widehat U \cap X = U$.  This
was used in \cite[Definition~6.10]{JGL:formalballs} in order to define
the distance $d (x, \overline U)$ of any point $x$ to the complement
$\overline U$ of $U$ as
$\sup \{r \in \Rp \mid (x, r) \in \widehat U\}$.

We shall write $\Open Y$ for the lattice of open subsets of a
topological space $Y$.

The assignment $U \mapsto \widehat U$ is
monotonic.  Being a right adjoint to the frame homomorphism that maps
every open subset $V$ of $\mathbf B (X, d)$ to $V \cap X$, it also
preserves arbitrary meets, namely interiors of arbitrary
intersections; but it satisfies no other remarkable property in
general.
One property that we will need in a number of places is the following.
\begin{defi}[Lipschitz regular]
  \label{defn:lipreg}
  A quasi-metric space $X, d$ is \emph{Lipschitz regular} if and only
  if the map
  $U \in \Open X \mapsto \widehat U \in \Open \mathbf B (X, d)$ is
  Scott-continuous.
\end{defi}
The name stems from a result that we shall see later,
Proposition~\ref{prop:lipreg:Lip}.
In general, a subspace $X$ of a topology space $Y$ is \emph{finitarily
  embedded} in $Y$ if and only if the map $V \in \Open Y \mapsto V
\cap X$ is Scott-continuous, see \cite{Escardo:properly:inj}.

\begin{lem}
  \label{lemma:lipreg:dist}
  The following are equivalent for a standard quasi-metric space
  $X, d$:
  \begin{enumerate}
  \item $X, d$ is Lipschitz regular;
  \item for every point $x \in X$, the map $U \in \Open X \mapsto d
    (x, \overline U)$ is Scott-continuous.
  \end{enumerate}
\end{lem}
\proof
(1) $\limp$ (2).  The map $U \mapsto d (x, \overline U)$ is the
composition of $U \mapsto \widehat U$ and of the map $\mathcal U \in
\Open \mathbf B (X, d) \mapsto \sup \{r \in \Rp \mid (x, r) \in
\mathcal U\}$.  The latter is easily seen to be Scott-continuous, and
the former is Scott-continuous by (1).


(2) $\limp$ (1).  Let $U$ be the union of a directed family of open
subsets ${(U_i)}_{i \in I}$.  We only have to show that every
$(x, r) \in \widehat U$ is in some $\widehat U_i$.  By
\cite[Lemma~3.4]{JGL:formalballs}, $(x, r)$ is the supremum of the
chain of formal balls $(x, r+1/2^n)$, $n \in \nat$, so one of them is
in $\widehat U$.  This implies that
$d (x, \overline U) \geq r + 1/2^n > r$.  Using (2),
$d (x, \overline U_i) > r$ for some $i \in I$, and that implies the
existence of a real number $s > r$ such that
$(x, s) \in \widehat U_i$.  Since $(x, s) \leq^{d^+} (x, r)$, $(x, r)$
is also in $\widehat U_i$.  \qed

The following Proposition~\ref{prop:compactballs} gives a further
explanation of Lipschitz regularity, in the special case of algebraic
quasi-metric spaces.

A point $x$ in a standard quasi-metric space $X, d$ is a
center point if and only if, for every $\epsilon > 0$, the open ball
$B^{d^+}_{(x, 0), <\epsilon} = \{(y, s) \in \mathbf B (X, d) \mid d
(x, y) < \epsilon-s\}$ is Scott-open in $\mathbf B (X, d)$.  This is
equivalent to requiring that $x$ be a finite point in $X, d$, a notion
that has a more complicated definition
\cite[Lemma~5.7]{JGL:formalballs}.

$X, d$ itself is called \emph{algebraic} if and only if every point
$x$ is the $d$-limit of a Cauchy (or even Cauchy-weightable, see
loc.cit.) net of center points, or equivalently, for every $x \in X$,
there is a directed family of formal balls $(x_i, r_i)$, $i \in I$,
where every $x_i$ is a center point, such that
$\sup_{i \in I} (x_i, r_i) = (x, 0)$ (Lemma~5.15, loc.cit.).

Every metric space is (standard and) algebraic, since in a metric
space every point is a center point, as a consequence of results by
Edalat and Heckmann \cite{EH:comp:metric}.  Indeed, the poset of
formal balls of a metric space $X, d$ is continuous, and $(x, r) \ll
(y, s)$ if and only if $d (x, y) < r-s$ (Proposition~7 and
Corollary~10, loc.\ cit.): then $B^{d^+}_{(x, 0), <\epsilon}$ is equal
to $\uuarrow (x, \epsilon)$, hence is Scott-open.

Every standard algebraic quasi-metric space $X, d$ is
\emph{continuous} \cite[Proposition~5.18]{JGL:formalballs}, where a
continuous quasi-metric space is a standard quasi-metric space $X, d$
whose space of formal balls $\mathbf B (X, d)$ is a continuous poset
(Definition~3.10, loc.\ cit.)\@ Moreover, when $X, d$ is standard
algebraic, $\mathbf B (X, d)$ has a basis of formal balls whose
centers are center points, and for a center point $x$,
$(x, r) \ll (y, s)$ if and only if $d (x, y) < r-s$.  This is the same
relation as in metric spaces, but beware that we only require it when
$x$ is a center point.

In general, we shall call a \emph{strong basis} of a standard
quasi-metric space $X, d$ any set $\mathcal B$ of center points of $X$
such that, for every $x \in X$, $(x, 0)$ is the supremum of a directed
family of formal balls with center points in $\mathcal B$.  (Given
that $X, d$ is standard, this is equivalent to Definition~7.4.66 of
loc.\ cit.)  Hence $X, d$ is algebraic if and only if it has a strong
basis.

\begin{rem}
  \label{rem:strong:basis}
  In metric spaces, a strong basis is nothing else than the familiar
  concept of a dense subset (Exercise~7.4.67, loc.\ cit.)  Strong
  bases are the correct generalization of dense subsets in the realm
  of quasi-metric spaces.
\end{rem}

\begin{prop}
  \label{prop:compactballs}
  Let $X, d$ be a standard algebraic quasi-metric space.   The following are equivalent:
  \begin{enumerate}
  \item $X, d$ is Lipschitz regular;
  \item $X, d$ \emph{has relatively compact balls}, namely: for every
    center point $x$ of $X$, for all $r, s \in \Rp$ with $s < r$,
    every open cover of the open ball $B^d_{x, < r}$ contains a finite
    subcover of the closed ball $B^d_{x, \leq s}$;
  \item for every center point $x$ of $X$, for all $r, s \in \Rp$ with
    $s < r$, for every directed family of open subsets
    ${(U_i)}_{i \in I}$ of open subsets of $X$ such that
    $B^d_{x, < r} \subseteq \bigcup_{i \in I} U_i$, there is an
    $i \in I$ such that $B^d_{x, \leq s} \subseteq U_i$.
  \end{enumerate}
\end{prop}
\proof
The equivalence of (2) and (3) is a standard exercise.  In the
difficult direction, notice that any union of open sets can be written
as a directed union of finite unions.

(3) $\limp$ (1).  It is easy to see that $U \mapsto \widehat U$ is
monotonic.  Let ${(U_i)}_{i \in I}$ be a directed family of $d$-Scott
open subsets of $X$, and $U = \bigcup_{i \in I} U_i$.
Pick an arbitrary element $(y, s)$ in $\widehat U$.  Our task is to
show that $(y, s)$ lies in some $\widehat U_i$.

Since $X, d$ is algebraic, $(y, s)$ is the supremum of a directed
family of formal balls $(x, r)$ way-below $(y, s)$, where each $x$ is
a center point.  Since $\widehat U$ is Scott-open, one of them is in
$\widehat U$.  From $(x, r) \ll (y, s)$ we obtain $d (x, y) < r-s$.
Find a real number $\epsilon > 0$ so that $d (x, y) < r-s-\epsilon$.

The open ball $B^d_{x, < r}$ is the intersection of $\uuarrow (x, r)$
with $X$, and $\uuarrow (x, r)$ is included in $\widehat U$ because
$(x, r)$ is in $\widehat U$ and $\widehat U$ is upwards-closed.  Hence
$B^d_{x, < r}$ is included in
$\widehat U \cap X = U = \bigcup_{i \in I} U_i$.  By (3),
$B^d_{x, \leq r-\epsilon}$ is included in some $U_i$, $i \in I$.

Consider $\widehat U_i \cup \uuarrow (x, r-\epsilon)$.  This is an
open subset of $\mathbf B (X, d)$, and its intersection with $X$ is
$U_i \cup B^d_{x, < r-\epsilon} = U_i$.  
By the maximality of $\widehat U_i$,
$\widehat U_i = \widehat U_i \cup \uuarrow (x, r-\epsilon)$, meaning
that $\uuarrow (x, r-\epsilon)$ is included in $\widehat U_i$.  Since
$d (x, y) < r-s-\epsilon$, $(x, r-\epsilon) \ll (y, s)$.  It follows
that $(y, s)$ is in $\widehat U_i$.

(1) $\limp$ (3).  Fix a center point $x$, two real numbers $r$ and $s$
such that $0 < s < r$, and assume that $B^d_{x, < r}$ is included
in the union $U$ of some directed family of open subsets $U_i$ of $X$.

We claim that $(x, s)$ must be in $\widehat U$.  The argument is one
we have just seen.  Indeed, $\widehat U \cup \uuarrow (x, r)$ is an
open subset of $\mathbf B (X, d)$ whose intersection with $X$ equals
$U \cup B^d_{x, <r} = U$. 
By maximality $\widehat U \cup \uuarrow (x, r) = \widehat U$.
However, since $x$ is a center point and $d (x, x) < r - s$, we have
$(x, r) \ll (x, s)$, so $(x, s)$ is in $\widehat U$.

By (1), $(x, s)$ is in some $\widehat U_i$, so
$\upc (x, s) \subseteq \widehat U_i$, hence, taking intersections with
$X$, $B^d_{x, \leq s}$ is included in $U_i$.
\qed

\begin{rem}
  \label{rem:compactballs:metric}
  As a special case, every metric case in which closed balls are
  compact is Lipschitz regular.  Indeed, recall that every metric
  space is standard and algebraic, and compactness immediately implies
  the relatively compact ball property.
\end{rem}

Having relatively compact balls is a pretty strong requirement.  Any
standard algebraic quasi-metric space with that property must be
core-compact, for example: for every point $y$ and every open
neighborhood $U$ of $y$, $y$ is in some open ball $B^d_{x, <r}$
included in $U$, where $x$ is a center point and $r > 0$.  Hence
$d (x, y) < r$, so that $d (x, y) < r-\epsilon$ for some
$\epsilon > 0$.  Then $y$ is also in the open neighborhood
$B^d_{x, < r-\epsilon}$ of $x$, and
$B^d_{x, < r-\epsilon} \subseteq B^d_{x, \leq r-\epsilon}$ is
way-below $B^d_{x, <\epsilon}$, using property (3).

Another argument consists in using the definition of Lipschitz
regularity directly: then $\Open X$ is a retract of
$\Open \mathbf B (X, d)$, and when $\mathbf B (X, d)$ is a continuous
poset, $\Open \mathbf B (X, d)$ is a completely distributive lattice,
in particular a continuous lattice; any retract of a continuous
lattice is again continuous, so $\Open X$ is continuous, meaning that
$X$ is core-compact.

Not all standard algebraic quasi-metric spaces have relatively compact
balls.  For example, $\rat$ with its usual metric is not core-compact,
hence does not have relatively compact balls.

\begin{rem}
  \label{rem:lipreg}
  Lipschitz regularity is therefore a pretty strong requirement---in the
  case of standard algebraic quasi-metric spaces.  On the contrary, we
  shall see below that spaces of formal balls are always Lipschitz
  regular (Theorem~\ref{thm:B:lipreg}), even when not core-compact
  (Remark~\ref{rem:N2}).
\end{rem}

The following lemma shows that the construction $U \mapsto \widehat U$
admits a particularly simple form on $\mathbf B$-algebras.
\begin{lem}
  \label{lemma:Uhat:alpha}
  Let $X, d$ be a quasi-metric space, and assume that there is a
  continuous map $\alpha \colon \mathbf B (X, d) \to X$ (with respect
  to the $d^+$-Scott and $d$-Scott topologies) such that
  $\alpha (x, r) \in B^d_{x, \leq r}$ and $\alpha (x, 0) = x$ for all
  $x \in X$ and $r \in \Rp$.  Then:
  \begin{enumerate}
  \item For every $d$-Scott open subset $U$ of $X$, $\widehat U$ is
    equal to $\alpha^{-1} (U)$;
  \item $X, d$ is Lipschitz regular.
  \end{enumerate}
\end{lem}
\proof Since $\alpha$ is continuous, $\alpha^{-1} (U)$ is $d^+$-Scott
open in $\mathbf B (X, d)$.  Its intersection with $X$ is equal to
$U$, since $(x, 0) \in \alpha^{-1} (U)$ is equivalent to
$\alpha (x, 0) \in U$, and $\alpha (x, 0) = x$.  By the definition of
$\widehat U$ as largest, $\alpha^{-1} (U)$ is included in
$\widehat U$.  To show the converse implication, let $(x, r)$ be an
arbitrary element of $\widehat U$.  Since $\alpha (x, r)$ is an
element of $B^d_{x, \leq r}$, $d (x, \alpha (x, r)) \leq r$, so
$(x, r) \leq^{d^+} (\alpha (x, r), 0)$.  Since $\widehat U$ is
upwards-closed, $(\alpha (x, r), 0)$ is in $\mathcal U$.  It follows
that $\alpha (x, r)$ is in $U$, so that $(x, r)$ is in
$\alpha^{-1} (U)$.

(2) follows from (1), since $\alpha^{-1}$ commutes with unions.  \qed

\begin{rem}
  \label{rem:Uhat:alpha}
  Lemma~\ref{lemma:Uhat:alpha} in particularly applies when $X, d$ is
  a (standard) $\mathbf B$-algebra, with structure map $\alpha$.
  Indeed, by the (1) $\limp$ (2) direction of
  Proposition~\ref{prop:cont}, $\alpha$ is continuous, and the
  remaining assumptions are item~(3) of Proposition~\ref{prop:B:alg}.
\end{rem}

\begin{rem}
  \label{rem:Balg}
  By Lemma~\ref{lemma:Uhat:alpha}~(1), the standard quasi-metric
  spaces that are $\mathbf B$-algebras are much more than Lipschitz
  regular: the map $U \mapsto \widehat U$ must preserve \emph{all}
  unions, not just the directed unions, and all intersections.
\end{rem}

However rare as $\mathbf B$-algebras may appear to be, recall that
(when $X, d$ is standard) $\mathbf B (X, d), d^+$ is itself a
$\mathbf B$-algebra, with structure map $\mu_X$.  Hence the following
is clear under a standardness assumption.  However, this even holds
without standardness.
\begin{thm}
  \label{thm:B:lipreg}
  For every quasi-metric space, the quasi-metric space
  $\mathbf B (X, d), d^+$ is Lipschitz regular.  For every $d^+$-Scott
  open subset $U$ of $\mathbf B (X, d)$,
  $\widehat U = \mu_X^{-1} (U)$.
\end{thm}
\proof This is Lemma~\ref{lemma:Uhat:alpha} with $\alpha = \mu_X$.
This is a continuous map because it is Scott-continuous by
Lemma~\ref{lemma:mu:cont} and because the Scott topologies on
$\mathbf B (X, d)$ and on $\mathbf B (\mathbf B (X, d), d^+)$ coincide
with the $d^+$-Scott topology and with the $d^{++}$-Scott topology
respectively, by Lemma~\ref{lemma:dScott=Scott}~(2).  The other two
assumptions are Lemma~\ref{lemma:eta:mu}, items~$(i)$ and~$(iv)$.
\qed

\begin{rem}
  \label{rem:N2}
  We exhibit a Lipschitz regular, standard quasi-metric space that is
  not core-compact.  Necessarily, that quasi-metric space cannot be
  algebraic, by Proposition~\ref{prop:compactballs}.  We build that
  quasi-metric space as $\mathbf B (X, d)$ for some quasi-metric
  space $X, d$, so that Theorem~\ref{thm:B:lipreg} will give us
  Lipschitz regularity for free.

  Every poset $X$ can be turned into a quasi-metric space by letting
  $d (x, y) = 0$ if $x \leq y$, $+\infty$ otherwise.  Then
  $\mathbf B (X, d)$ is order-isomorphic with the poset
  $X \times ]-\infty, 0]$ \cite[Example~1.6]{JGL:formalballs}.

  Consider the dcpo $X = (\nat \times \nat) \cup \{\omega\}$, with the
  ordering defined by $(i, n) \leq (i', n')$ iff $i=i'$ and
  $n \leq n'$, and where $\omega$ is larger than any other element.
  The non-empty upwards-closed subsets of $X$ are the subsets of the
  form $\{\omega\} \cup \bigcup_{i \in S} \upc (i, n_i)$, where
  $S \subseteq \nat$ and for each $i$, $n_i \in \nat$.  Those that are
  compact are exactly those such that $S$ is finite, and those that
  are Scott-open are exactly those such that $S = \nat$.  In
  particular, note that all compact saturated subsets have empty
  interior.  The same happens in $X \times ]-\infty, 0]$.  Indeed,
  assume a compact saturated subset $Q$ of $X \times ]-\infty, 0]$
  with non-empty interior $U$.  Since $Q$ is compact, $\pi_1 [Q]$ is
  compact, too, and we see that $\pi_1 [Q]$ is also upwards-closed,
  hence of the form $\{\omega\} \cup \bigcup_{i \in S} \upc (i, n_i)$,
  with $S$ finite.  Pick some $j \in \nat$ outside of $S$.  Since $U$
  is non-empty, it must contain $(\omega, 0)$.  However, $(\omega, 0)$
  is the supremum of the chain of points $((j, n), 0)$, $n \in \nat$,
  so one of them is in $U$, hence in $Q$.  This is impossible since
  $j \not \in S$.  Since all compact saturated subsets of
  $X \times ]-\infty, 0]$ have empty interior, it follows that
  $X \times ]-\infty, 0]$ is not locally compact.

  Note that $X$ is sober.  Indeed, consider a non-empty closed subset
  $C$.  If its complement is empty, then $C = \dc \omega$.  Otherwise,
  $C$ is the complement of an open set
  $\{\omega\} \cup \bigcup_{i \in \nat} \upc (i, n_i)$, hence is equal
  to $\bigcup_{i \in S} \dc (i, n_i-1)$, where $S$ is the set of
  indices $i$ such that $n_i \geq 1$.  $S$ is non-empty since we have
  assumed $C$ non-empty.  Pick $i_0$ from $S$.  Then $C$ is included
  in the union of $\dc (i_0, n_{i_0}-1)$ and
  $C' = \bigcup_{i \in S \smallsetminus \{i_0\}} \dc (i, n_i-1)$.
  Note that $C$ is not included in $C'$, so if $C$ is irreducible,
  then $C \subseteq \dc (i_0, n_{i_0}-1)$, from which we obtain
  $C = \dc (i_0, n_{i_0}-1)$.  In any case, we have shown that every
  irreducible closed subset of $X$ is the downward closure of a unique
  point, hence $X$ is sober.

  Since $]-\infty, 0]$ is a continuous dcpo, it is sober in its Scott
  topology \cite[Proposition~8.2.12~$(b)$]{JGL-topology}.  Products of
  sober spaces are sober (Theorem~8.4.8, loc.\ cit.), so
  $X \times ]-\infty, 0]$ is sober.  Since every sober core-compact
  space is locally compact (Theorem~8.3.10, loc.\ cit.), we
  conclude that $X \times ]-\infty, 0]$ is not core-compact.

  We conclude that $\mathbf B (X, d) \cong X \times ]-\infty, 0]$ is
  Lipschitz regular but not core-compact.  \qed
\end{rem}

\section{Compact Subsets of Quasi-Metric Spaces}
\label{sec:comp-subs-quasi}

Let us characterize the compact saturated subsets of quasi-metric
spaces.  For that, we concentrate on the case where $X, d$ is a
continuous Yoneda-complete space, i.e., where $\mathbf B (X, d)$ is a
continuous dcpo \cite[Definition~7.4.72]{JGL-topology}; the notion was
initially introduced by Kostanek and Waszkiewicz
\cite{KW:formal:ball}.  In general, a \emph{continuous} quasi-metric
space is a standard quasi-metric space whose space of formal balls is
a continuous poset, see \cite[Section~3]{JGL:formalballs}.
(Standardness is automatically implied by Yoneda-completeness, but we
have to require it explicitly if we do not assume
Yoneda-completeness.)

We shall need the following useful (and certainly well-known) lemma.
\begin{lem}
  \label{lemma:compact:subspace}
  Let $Y$ be a subspace of a topological space $X$.  The compact
  subsets of $Y$ are exactly the compact subsets of $X$ that are
  included in $Y$.  If $Y$ is upwards-closed in $X$, then the compact
  saturated subsets of $Y$ are exactly the compact saturated subsets
  of $X$ that are included in $Y$.
\end{lem}
\proof Take a compact subset $K$ of $Y$.  $K$ is also compact in $X$,
as the image of $K$ by the inclusion map, which is continuous by
definition of subspaces.  Conversely, let $K$ be a compact subset of
$X$ that is included in $Y$.  Consider an open cover
${(U_i)}_{i \in I}$ of $K$ in $Y$.  For each $i \in I$, there is an
open subset $V_i$ of $X$ such that $U_i = V_i \cap Y$.  Then
${(V_i)}_{i \in I}$ is an open cover of $K$ in $X$.  Extract a finite
subcover ${(V_i)}_{i \in J}$ of $K$ in $X$.  Then ${(U_i)}_{i \in J}$
is a finite subcover of $K$ in $Y$.

For the second part, we first note that the specialization preordering
of $Y$ is the restriction of that of $X$ to $Y$
\cite[Proposition~4.9.5]{JGL-topology}.  If $Y$ is upwards-closed in
$X$, it follows that the saturated subsets of $Y$ are exactly the
saturated subsets of $X$ that are included in $Y$.  \qed

Let $X, d$ be a quasi-metric space.  Up to the embedding
$x \mapsto (x, 0)$, $X$ is a subspace of $\mathbf B (X, d)$.  $X$ also
happens to be upwards-closed in $\mathbf B (X, d)$.
Lemma~\ref{lemma:compact:subspace} will help us characterize the
compact saturated subsets of $X$ through this embedding.

For every $\epsilon > 0$, let $V_\epsilon$ be the set of formal balls
$(x, r)$ with $r < \epsilon$.
\begin{lem}
  \label{lemma:Veps}
  If $X, d$ is standard, then $V_\epsilon$ is Scott-open in $\mathbf B
  (X, d)$, and $X = \bigcap_{\epsilon > 0} V_\epsilon = \bigcap_{n \in
    \nat} V_{1/2^n}$; in particular, $X$ is a $G_\delta$-subset of
  $\mathbf B (X, d)$.
\end{lem}
\proof
This is a more explicit version of Proposition~2.6 of
\cite{JGL:formalballs}.  The key is that $V_\epsilon$ is the inverse
image of $[0, \epsilon[$ by the radius map $(x, r) \mapsto r$, which
is Scott-continuous from $\mathbf B (X, d)$ to $\Rp^{op}$ by
Proposition~2.4~(3) of loc.cit., owing to the fact that $X, d$ is
standard.  \qed

\begin{lem}
  \label{lemma:Q:radius}
  Let $X, d$ be a standard quasi-metric space.  For every non-empty
  compact saturated subset $Q$ of $\mathbf B (X, d)$, there is a
  largest $r \in \Rp$ such that $Q$ contains a point of the form
  $(x, r)$.
\end{lem}
We shall call the largest such $r \in \Rp$ the \emph{radius} $r (Q)$
of $Q$.

\proof The radius map $(x, r) \mapsto r$ is Scott-continuous from
$\mathbf B (X, d)$ to $\creal^{op}$, since $X, d$ is standard.  Since
continuous images of compact subsets are compact, this map reaches a
minimum $r$ in $\creal^{op}$, hence a maximum in $\creal$.  That
maximum cannot be equal to $+\infty$ since radii of formal balls are
all different from $+\infty$, so $r$ is in $\Rp$.  (Alternatively, $Q$
is contained in
$\mathbf B (X, d) = \bigcup_{\epsilon > 0} V_\epsilon$, so
$Q \subseteq V_\epsilon$ for some $\epsilon > 0$.)  \qed

The \emph{closed ball} $B^d_{x, \leq r}$ of center $x$ and radius $r$
is the set of points $y \in X$, such that $d (x, y) \leq r$.  Despite
the name, this is not a closed set in general, except when $X, d$ is
metric: closed balls are upwards-closed, whereas closed sets must be
downwards-closed.

Closed balls need not be compact either.  However, $B^d_{x, \leq r}$
is the intersection of $\upc (x, r)$ with $X$, considering the latter
as a subspace of $\mathbf B (X, d)$, and $\upc (x, r)$ is compact
saturated in $\mathbf B (X, d)$.  (The upward closure of any point
with respect to the specialization preordering of a topological space
is always compact saturated.)  Note that the radius of $\upc (x, r)$,
as introduced in Lemma~\ref{lemma:Q:radius}, is $r$.

In general, for every finite set of formal balls
$(x_1, r_1)$, \ldots, $(x_n, r_n)$, the upward closure
$Q = \upc \{(x_1, r_1), \cdots, (x_n, r_n)\}$ is compact saturated in
$\mathbf B (X, d)$.  The upward closures of finite sets, which are
compact saturated in any topological space, are called
\emph{finitary compact} in \cite[Proposition~4.4.21]{JGL-topology}.
When $n \neq 0$, the radius of $Q$ is $r (Q) = \max_{j=1}^n r_j$.

The intersection of $\upc \{(x_1, r_1), \cdots, (x_n, r_n)\}$ with $X$
is the union $\bigcup_{j=1}^n B^d_{x_j, \leq r_j}$ of finitely many
closed balls.  This may fail to be compact in $X$, but if we take the
intersection of a filtered family of such finite unions of formal
balls, in such a way that the radii go to $0$, we will obtain a
compact subset of $\mathbf B (X, d)$ that is included in $X$, hence a
compact subset of $X$.
\begin{lem}
  \label{lemma:Xd:compact:1}
  Let $X, d$ be a continuous Yoneda-complete quasi-metric space.  For
  every filtered family ${(Q_i)}_{i \in I}$ of non-empty compact
  saturated subsets of $\mathbf B (X, d)$ such that
  $\inf_{i \in I} r (Q_i) = 0$, $\bigcap_{i \in I} Q_i$ is a non-empty
  compact saturated subset of $X$ in its $d$-Scott topology.
\end{lem}
This is in particular the case when $Q_i$ is taken to consist of
non-empty finitary compact sets.

\proof Since $X, d$ is continuous, $\mathbf B (X, d)$ is a continuous
dcpo.  Recall that every continuous dcpo is sober in its Scott
topology, that every sober space is well-filtered, and that in a
well-filtered space every filtered intersection of compact saturated
subsets is compact saturated.  Hence $\bigcap_{i \in I} Q_i$ is
compact saturated in $\mathbf B (X, d)$.

In a well-filtered space, filtered intersections of non-empty compact
saturated subsets are non-empty as well.  Indeed, if
$\bigcap_{i \in I} Q_i$ is empty, then it is included in the open set
$\emptyset$, and that implies $Q_i \subseteq \emptyset$ for some
$i \in I$.

For every $\epsilon > 0$, there is an $i \in I$ such that
$r (Q_i) < \epsilon$.  That implies $Q_i \subseteq V_\epsilon$.
Hence $\bigcap_{i \in I} Q_i$ is included in $\bigcap_{\epsilon > 0}
V_\epsilon = X$, using Lemma~\ref{lemma:Veps}.

We now conclude that $\bigcap_{i \in I} Q_i$ is a compact saturated
subset of $X$ by Lemma~\ref{lemma:compact:subspace}.  \qed

Since $\bigcap_{i \in I} Q_i$ is included in $X$, it is also equal to
$\bigcap_{i \in I} (Q_i \cap X)$.  For a finitary compact
$Q = \upc \{(x_1, r_1), \cdots, (x_n, r_n)\}$, $Q \cap X$ is the
finite union of closed balls $\bigcup_{j=1}^n B^d_{x_j, \leq r_j}$.
Hence Lemma~\ref{lemma:Xd:compact:1} implies the following (for
non-empty finitary compacts; if one of them is empty, the claim is
obvious).  Condition~(1) means that
$Q_i = \upc \{(x_{ij}, r_{ij}) \mid 1\leq j\leq n_i\}$ contains
$Q_{i'}$ for all $i \sqsubseteq i'$ in $I$, so that
${(Q_i)}_{i \in I}$ is filtered.  Condition~(2) states that
$\inf_{i \in I} r (Q_i)=0$.
\begin{cor}
  \label{corl:Xd:compact:1}
  Let $X, d$ be a continuous Yoneda-complete quasi-metric space.  Let
  $I, \sqsubseteq$ be a directed preordered set.  Assume that:
  \begin{enumerate}
  \item for all $i \sqsubseteq i'$ in $I$, for every $j'$, $1\leq j'
    \leq n_{i'}$, there is a $j$, $1\leq j\leq n_i$, such that $d
    (x_{ij}, x_{i'j'}) \leq r_{ij} - r_{i'j'}$;
  \item for every $\epsilon > 0$, there is an $i \in I$ such that for
    every $j$, $1\leq j\leq n_i$, $r_{ij} \leq \epsilon$,
  \end{enumerate}
  then
  $\bigcap_{i \in I} \bigcup_{j=1}^{n_i} B^d_{x_{ij}, \leq r_{ij}}$ is
  compact saturated in $X$.  \qed
\end{cor}

\begin{lem}
  \label{lemma:cont:Q}
  Let $X, d$ be a continuous quasi-metric space.  For every compact
  saturated subset $Q$ of $X$, for every open neighborhood $U$ of $Q$,
  for every $\epsilon > 0$, one can find a finite union $A$ of closed
  balls $B^d_{x_i, \leq r_i}$ with $r_i < \epsilon$, and such that
  $Q \subseteq int (A) \subseteq A \subseteq U$.
\end{lem}

\proof The image of $Q$ by the embedding $x \mapsto (x, 0)$ of $X$
into $\mathbf B (X, d)$ is compact.  If we agree to equate $x$ with
$(x, 0)$, then $Q$ is compact not only in $X$, but also in
$\mathbf B (X, d)$.  Recall that $\widehat U$ is a Scott-open subset
of $\mathbf B (X, d)$ and that $\widehat U \cap X = U$.  By
intersecting it with
$V_\epsilon = \{(x, r) \mid x \in X, r < \epsilon\}$ (a Scott-open
subset, owing to the fact that $X, d$ is standard, see
Lemma~\ref{lemma:Veps}), we obtain an open neighborhood
$\widehat U \cap V_\epsilon$ of $Q$ in $\mathbf B (X, d)$.

Since $\mathbf B (X, d)$ is a continuous poset,
$\widehat U \cap V_\epsilon$ is the union of all open subsets
$\uuarrow (x, r)$, $(x, r) \in \widehat U \cap V_\epsilon$.  By
compactness, $Q$ is therefore included in some finite union
$\bigcup_{i=1}^n \uuarrow (x_i, r_i)$, where every $(x_i, r_i)$ is in
$\widehat U \cap V_\epsilon$.  In particular, $r_i < \epsilon$ for
each $i$.  Note also that $\upc (x_i, r_i)$ is included in
$\widehat U$, since $\widehat U$ is upwards-closed, so
$\upc (x_i, r_i) \cap X = B^d_{x_i, \leq r_i}$ is included in $U$.
Therefore $A = \bigcup_{i=1}^n B^d_{x_i, \leq r_i}$, whose interior
contains the open neighborhood
$\bigcup_{i=1}^n (\uuarrow (x_i, r_i) \cap X)$ of $Q$, fits.  \qed

\begin{prop}
  \label{prop:Xd:compact}
  Let $X, d$ be a continuous Yoneda-complete quasi-metric space.  The
  compact saturated subsets of $X$ with its $d$-Scott topology are
  exactly the sets
  $\bigcap_{i \in I} \bigcup_{j=1}^{n_i} B^d_{x_{ij}, \leq r_{ij}}$
  that satisfy the conditions of Corollary~\ref{corl:Xd:compact:1}, or
  equivalently, the filtered intersections $\bigcap_{i \in I} Q_i$ of
  finitary compacts $Q_i$ of $\mathbf B (X, d)$ with
  $\inf_{i \in I} r (Q_i)=0$.
\end{prop}
\proof Let $Q$ be compact saturated in $X$, and $U$ be an open
neighborhood of $Q$.  By Lemma~\ref{lemma:cont:Q}, we can build a
subset $A_{U,0}$ of $U$, which is a finite union of closed balls of
radius $< 1$ and whose interior $V_{U,0}$ contains $Q$.  Reusing that
lemma, we can find a subset $A_{U,1}$ of $V_{U,0}$, which is a finite
union of closed balls of radius $<1/2$ and whose interior $V_{U,1}$
contains $Q$.  We continue in this way to build $A_{U,n}$ for every
$n \in \nat$, together with its interior $V_{U,n}$ containing $Q$, in
such a way that $A_{U,n}$ is a finite union of closed balls of radius
$<1/2^n$.  The family of all sets $A_{U, n}$, where $U$ ranges over
the open neighborhoods of $Q$ and $n \in \nat$ is filtered, because
$A_{U, m}$ and $A_{U', n}$ both contain
$A_{V_{U,m} \cap V_{U',n}, 0}$.  Their intersection contains $Q$, and
is included in any open neighborhood $U$ of $Q$, hence equals $Q$,
since $Q$ is saturated.


In the reverse direction, we use Lemma~\ref{lemma:Xd:compact:1}.
\qed

\begin{rem}
  \label{rem:Hausdorff}
  Proposition~\ref{prop:Xd:compact} is a quasi-metric analogue of a
  result by Hausdorff stating that, in a complete metric space, the
  compact subsets are exactly the closed, precompact sets.  A set is
  precompact if and only if, for every $\epsilon > 0$, it can be
  covered by a finite union of open balls of radius $\epsilon$.  In
  particular, any set of the form
  $\bigcap_{i \in I} \bigcup_{j=1}^{n_i} B^d_{x_{ij}, \leq r_{ij}}$ as
  above, in a metric space, is both closed and precompact, hence
  compact.  In a quasi-metric space, beware that the closed balls
  $B^d_{x_{ij}, \leq r_{ij}}$ may fail to be closed.

  Another analogue of Hausdorff's result can be found in
  \cite[Proposition~7.2.22]{JGL-topology}: the symcompact quasi-metric
  spaces, namely the quasi-metric spaces $X, d$ that are compact in
  the open ball topology of $d^{sym}$ (where
  $d^{sym} (x, y) = \max (d (x, y), d (y, x))$) are exactly the
  Smyth-complete totally bounded quasi-metric spaces ($X, d$ is
  totally bounded if and only if $X, d^{sym}$ is precompact).  That is
  less relevant to our setting.
\end{rem}
It turns out I will not need Proposition~\ref{prop:Xd:compact} in the
following.  It has independent interest, and we shall definitely
require some of the auxiliary lemmas that led us to it.

\section{Spaces of Continuous and Lipschitz Maps}
\label{sec:cont-lipsch-maps}

We equip $\creal$ with the Scott topology of its ordering $\leq$, or
equivalently, with the $\dreal$-Scott topology.  For a topological
space $X$, the continuous maps from $X$ to $\creal$ are usually called
\emph{lower semicontinuous}.
\begin{defi}[$\Lform X$]
  \label{defn:LX}
  Let $\Lform X$ denote the set of lower semicontinuous maps from $X$
  to $\creal$.  We give it the Scott topology of the pointwise ordering.
\end{defi}

\begin{lem}
  \label{lemma:Lalpha:LX}
  Let $X, d$ be a standard quasi-metric space, and $\alpha > 0$.  Every
  $\alpha$-Lipschitz continuous map from $X, d$ to $\creal, \dreal$ is
  lower semicontinuous.
\end{lem}
\proof Recall that $f$ is $\alpha$-Lipschitz continuous if and only if
$f' \colon (x, r) \mapsto f (x) - \alpha r$ is Scott-continuous, by
Lemma~\ref{lemma:f'}.  Then
$f^{-1} (]t, +\infty]) = X \cap {f'}^{-1} (]t, +\infty])$ for every
$t \in \real$, showing that $f$ itself is lower semicontinuous.  \qed

\begin{defi}[$\Lform_\alpha (X, d)$, $\Lform_\infty (X, d)$]
  \label{defn:Lalpha}
  Let $\Lform_\alpha (X, d)$ be the set of $\alpha$-Lipschitz
  continuous maps from $X, d$ to $\creal, \dreal$, and let
  $\Lform_\infty (X, d) = \bigcup_{\alpha \in \Rp} \Lform_\alpha (X,
  d)$ be the set of all Lipschitz continuous maps from $X, d$ to
  $\creal, \dreal$.  We give those spaces the subspace topology from
  $\Lform X$ (which makes sense, by Lemma~\ref{lemma:Lalpha:LX}).
\end{defi}

We also write $\Lform_\infty X$ for $\Lform_\infty (X, d)$, and
$\Lform_\alpha X$ for $\Lform_\alpha (X, d)$.

Beware that there is no reason why the topologies on
$\Lform_\alpha (X, d)$ and $\Lform_\infty (X, d)$ would be the Scott
topology of the pointwise ordering.  We shall see a case where those
topologies coincide in Proposition~\ref{prop:Lalpha:retract}.

\begin{defi}[$\Lform_\alpha^a (X, d)$, $\Lform_\infty^a (X, d)$]
  \label{defn:Lalpha:bnd}
  Let $\Lform_\alpha^a X$ or $\Lform_\alpha^a (X, d)$ be the space of
  all $\alpha$-Lipschitz continuous maps from $X, d$ to
  $[0, \alpha a], \dreal$, for $\alpha \in \Rp$, where $a \in \Rp$,
  $a > 0$.  Give it the subspace topology from $\Lform_\alpha X$, or
  equivalently, from $\Lform X$.
\end{defi}

\begin{defi}[$\Lform_\infty^\bnd (X, d)$]
  \label{defn:Lbnd}
  Let $\Lform_\infty^\bnd (X, d)$ be the subspace of all
  \emph{bounded} maps in $\Lform_\infty (X, d)$, and
  $\Lform_\alpha^\bnd (X, d)$ be the corresponding subspace of all
  bounded maps in $\Lform_\alpha (X, d)$, with the subspace
  topologies.
\end{defi}
Since the specialization ordering on $\Lform X$ is the pointwise
ordering ($f \leq g$ if and only if for every $x \in X$, $f (x) \leq g
(x)$), the same holds for all the spaces defined above as well.
We shall always write $\leq$ for that ordering.

We note the following, although we shall only use it much later.
\begin{lem}
  \label{lemma:Linfa}
  Let $X, d$ be a quasi-metric space.  For every $a > 0$,
  $\Lform_\infty^\bnd (X, d) = \bigcup_{\alpha > 0} \Lform_\alpha^a
  (X, d)$.
\end{lem}
\proof Consider any bounded map $f$ from $\Lform_\infty (X, d)$.  By
definition, $f \leq b.\mathbf 1$ for some $b \in \Rp$, where
$\mathbf 1$ is the constant map equal to $1$, and
$f \in \Lform_\alpha (X, d)$ for some $\alpha > 0$.  Since
$\Lform_\alpha (X, d)$ grows as $\alpha$ increases, by
Proposition~\ref{prop:alphaLip:props}~(5), we may assume that
$\alpha \geq b/a$.  Then $f$ is in $\Lform_\alpha^a (X, d)$.  The
reverse inclusion is obvious. \qed

Assuming $X, d$ standard, for each $\alpha \in \Rp$, there is a
largest $\alpha$-Lipschitz continuous map $f^{(\alpha)}$ below any
lower semicontinuous map $f \in \Lform X$.  Moreover, the family
${(f^{(\alpha)})}_{\alpha \in \Rp}$ is a chain, and
$\sup_{\alpha \in \Rp} f^{(\alpha)} = f$, where suprema are taken
pointwise \cite[Theorem~6.17]{JGL:formalballs}.

We shall examine what $f^{(\alpha)}$ may be, and what properties it
may have, in two important cases: when $f$ is lower semicontinuous
and $X, d$ is Lipschitz regular; and when $f$ is already
$\alpha$-Lipschitz but not necessarily continuous, and $X, d$ is
algebraic.  Next, while $f^{(\alpha)}$ is largest, we shall introduce
functions that are smallest among the $1$-Lipschitz continuous maps
mapping some given center points to values above some specified
numbers.

\subsection{$f^{(\alpha)}$ for $f$ Lower Semicontinuous}
\label{sec:falpha-f-lower}

Let $X, d$ be a standard quasi-metric space.  We know that, for every
$d$-Scott open subset $U$ of $X$, for all $\alpha, r \in \Rp$,
$(r\chi_U)^{(\alpha)}$ is the map
$x \mapsto \min (r, \alpha d (x, \overline U))$ (Proposition~6.14,
loc.cit.).  We shall extend that below.  Also, for every
$\alpha$-Lipschitz continuous map $g \colon X \to \creal$, for every
$t \in \Rp$, $t g$ is $t \alpha$-Lipschitz continuous
(Proposition~\ref{prop:alphaLip:props}~(1)).

Let us call \emph{step function} any function from $X$ to $\creal$ of
the form $\sup_{i=1}^m a_i \chi_{U_i}$, where
$0 < a_1 < \cdots < a_m < +\infty$ and
$U_1 \supseteq U_2 \supseteq \cdots \supseteq U_m$ form a finite
antitone family of open subsets of $X$.
\begin{lem}
  \label{lemma:f(alpha):step}
  Let $X, d$ be a standard quasi-metric space.  For a step function
  $f = \sup_{i=1}^m a_i \chi_{U_i}$, and $\alpha > 0$, $f^{(\alpha)}$
  is the function that maps every $x \in X$ to
  $\min (\alpha d (x, \overline U_1), a_1 + \alpha d (x, \overline
  U_2), \cdots, \allowbreak a_{i-1} + \alpha d (x, \overline U_i),
  \cdots, a_{m-1} + \alpha d (x, \overline U_m), a_m)$.
\end{lem}
\proof Let
$g (x) = \min (\alpha d (x, \overline U_1), a_1 + \alpha d (x,
\overline U_2), \cdots, a_{i-1} + \alpha d (x, \overline U_i), \cdots,
\allowbreak a_{m-1} + \alpha d (x, \overline U_m), \allowbreak a_m)$.
Each of the maps $x \mapsto a_{i-1} + \alpha d (x, \overline U_i)$
(where, for convenience, we shall assume $a_0 = 0$, so as not to make
a special case for $i=1$) is $\alpha$-Lipschitz continuous, and
therefore $g$ is $\alpha$-Lipschitz continuous.  Indeed, the map
$d (\_, \overline U)$ is $1$-Lipschitz (Yoneda-)continuous, as shown
in Lemma~6.11~(3) of \cite{JGL:formalballs}; the rest of the argument
relies on Proposition~\ref{prop:alphaLip:props}.

We claim that $g (x) \leq f (x)$ for every $x \in X$.  Let $U_0=X$, so
that $U_i$ makes sense also when $i=0$, and let $U_{m+1}=\emptyset$.
The latter allows us to write $g (x)$ as
$\min_{i=0}^{m} (a_i + \alpha d (x, \overline U_{i+1}))$, noticing
that $d (x, \overline U_{m+1}) = 0$.  Indeed, by
\cite[Lemma~6.11~(1)]{JGL:formalballs}, for every open subset $U$,
$d (x, \overline U)=0$ if and only if $x \not\in U$.

There is a unique index $j$, $0\leq j \leq m$, such that $x \in U_j$
and $x \not\in U_{j+1}$.  Then
$g (x) \leq a_j + d (x, \overline U_{j+1}) = a_j$.  Noticing that
$f (x) = a_j$, it follows that $g (x) \leq f (x)$.

Now consider any $\alpha$-Lipschitz continuous map $h \leq f$, and let
us show that $h \leq g$.  We fix $x \in X$ and $i$ with
$0 \leq i \leq m$, and we claim that
$h (x) \leq a_i + \alpha d (x, \overline U_{i+1})$.  Since $h$ is
$\alpha$-Lipschitz continuous,
$h' \colon (x, r) \mapsto h (x) - \alpha r$ is Scott-continuous, so
$V = {h'}^{-1} (]a_i, +\infty])$ is open in $\mathbf B (X, d)$.

For every element of the form $(y, 0)$ in $V \cap X$,
$h' (y, 0) = h (y) > a_i$, hence $f (y) \geq h (y) > a_i$, which
implies that $y$ is in $U_{i+1}$.  We have just shown that
$V \cap X \subseteq U_{i+1}$, and that implies
$V \subseteq \widehat U_{i+1}$, by maximality of $\widehat U_{i+1}$.

Now, for every $s \in \Rp$ such that $s < (h (x) - a_i)/\alpha$, i.e.,
such that $h' (x, s) = h (x) - \alpha s$ is strictly larger than
$a_i$, by definition $(x, s)$ is in $V$, hence in $\widehat U_{i+1}$.
By definition, this means that $s \leq d (x, \overline U_{i+1})$.
Taking suprema over $s$, we obtain
$(h (x) - a_i) / \alpha \leq d (x, \overline U_{i+1})$, equivalently
$h (x) \leq a_i + \alpha d (x, \overline U_{i+1})$.  Since that holds
for every $i$, $0\leq i\leq m$, $h (x) \leq g (x)$.  Hence $g$ is the
largest $\alpha$-Lipschitz continuous map below $f$, in other words,
$g = f^{(\alpha)}$.  \qed

Given any topological space $X$, every lower semicontinuous function
$f \colon X \to \creal$ is the pointwise supremum of a chain of step
functions:
\begin{equation}
  \label{eq:fK}
  f_K (x) = \frac 1 {2^K} \sup\nolimits_{k=1}^{K2^K} k \chi_{f^{-1}
    (]k/2^K, +\infty])} (x)
\end{equation}
where $K \in \nat$.  If $X, d$ is a standard quasi-metric space,
$f_K^{(\alpha)}$ is given by Lemma~\ref{lemma:f(alpha):step}, namely:
\begin{equation}
  \label{eq:fK(alpha)}
  f_K^{(\alpha)} (x) = \min (\min\nolimits_{k=1}^{K2^K} (\frac {k-1} {2^K} + \alpha d (x,
  \overline {f^{-1} (]k/2^K, +\infty])}), K).
\end{equation}

\begin{prop}
  \label{prop:f(alpha)}
  Let $X, d$ be a standard quasi-metric space.  For every lower
  semicontinuous map $f \colon X \to \creal$, for every $\alpha \in
  \Rp$, for every $x \in X$,
  \[
    f^{(\alpha)} (x) = \sup_{K \in \nat} f_K^{(\alpha)} (x).
  \]
\end{prop}
\proof
  We first deal with the case where $f$ is already $\alpha$-Lipschitz
  continuous.  In that case, we claim the equivalent statement: $(*)$
  if $f$ is $\alpha$-Lipschitz continuous, then for every $x \in X$,
  $f (x) = \sup_{K \in \nat} f_K^{(\alpha)} (x)$.

  Fix $j \in \nat$ and $k$ such that $1 \leq k \leq K2^K$, and note
  that if $(x, j/(\alpha 2^K)) \in \widehat U_k$, where
  $U_k = {f^{-1} (]k/2^K, +\infty]}$, then
  $\alpha d (x, \overline U_k) \geq j/2^K$.  This is by definition of
  $d (x, \overline U_k)$.

  Recall that $f' (x, r) = f (x) - \alpha r$ defines a
  Scott-continuous map.  For every $(y, 0)$ in
  $X \cap {f'}^{-1} (]k/2^K, +\infty])$, $f' (y, 0) = f (y) > k/2^K$,
  so $X \cap {f'}^{-1} (]k/2^K, +\infty])$ is included in
  $f^{-1} (]k/2^K, +\infty]) = U_k$.  By maximality,
  ${f'}^{-1} (]k/2^K, +\infty])$ is included in $\widehat
  U_k$. 

  Hence if $(x, j/(\alpha 2^K))$ is in ${f'}^{-1} (]k/2^K,
  +\infty])$, then $\alpha d (x, \overline U_k) \geq
  j/2^K$.  That happens when $f (x) - j/2^K > k / 2^K$, i.e., when
  $x$ is in $f^{-1} (](k+j)/2^K, +\infty])$.  Therefore $\alpha d (x,
  \overline U_k) \geq j/2^K \chi_{f^{-1} (](k+j)/2^K, +\infty])}
  (x)$ for all $j \in \nat$ and $k$ such that $1\leq k\leq K2^K$.

  Now fix $k_0$ with $1\leq k_0\leq K2^K$.  For every $k$ with
  $1\leq k \leq k_0$, letting $j = k_0-k$, we obtain that
  $(k-1)/2^K + \alpha d (x, \overline U_k) \geq (k-1)/2^K +
  (k_0-k)/2^K \chi_{f^{-1} (]k_0/2^K, +\infty])} (x) \geq (k_0-1)/2^K
  \chi_{f^{-1} (]k_0/2^K, +\infty])} (x)$.  For every $k$ such that
  $k_0 < k \leq K2^K$,
  $(k-1)/2^K + \alpha d (x, \overline {U_k}) \geq k_0/2^K$ is larger
  than the same quantity already, and similarly for $K$, which is also
  larger than or equal to $k_0/2^K$.  Using (\ref{eq:fK(alpha)}), we
  obtain
  $f_K^{(\alpha)} (x) \geq (k_0-1)/2^K \chi_{f^{-1} (]k_0/2^K,
    +\infty])} (x)$, and therefore
  $f_K^{(\alpha)} (x) \geq k_0/2^K \chi_{f^{-1} (]k_0/2^K, +\infty])}
  (x) -1/2^K$.  Since that holds for every $k_0$ between $1$ and
  $K2^K$, it follows that $f_K^{(\alpha)} (x) \geq f_K (x) - 1/2^K$.
  Taking suprema over $K \in \nat$, we obtain
  $\sup_{K \in \nat} f_K^{(\alpha)} (x) \geq \sup_{K \in \nat} (f_K
  (x) - 1/2^K) = f (x)$, proving $(*)$.

  In the general case, where $f$ is only assumed to be lower
  semicontinuous, we note that $f \geq f^{(\alpha)}$ implies that
  $f_K \geq (f^{(\alpha)})_K$.  Indeed, that follows from formula
  (\ref{eq:fK}) and the fact that
  $(f^{(\alpha)})^{-1} (]k/2^K, +\infty])$ is included in
  $f^{-1} (]k/2^K, +\infty])$ for every $k$.  The mapping
  $g \mapsto g^{(\alpha)}$ is also monotonic, since $g^{(\alpha)}$ is
  defined as the largest $\alpha$-Lipschitz continuous map below $g$.
  Therefore $f_K^{(\alpha)} \geq (f^{(\alpha)})_K^{(\alpha)}$.  Taking
  suprema, we obtain that
  $\sup_{K \in \nat} f_K^{(\alpha)} (x) \geq \sup_{K \in \nat}
  (f^{(\alpha)})_K^{(\alpha)} (x) = f^{(\alpha)} (x)$, where the last
  equality follows from statement $(*)$ (first part of the proof),
  applied to the $\alpha$-Lipschitz continuous function
  $f^{(\alpha)}$.

  The reverse inequality
  $\sup_{K \in \nat} f_K^{(\alpha)} \leq f^{(\alpha)}$ is easy: for
  every $K \in \nat$, $f_K \leq f$, so
  $f_K^{(\alpha)} \leq f^{(\alpha)}$.  \qed

\begin{prop}
  \label{prop:lipreg:Lip}
  Let $X, d$ be a standard quasi-metric space.  The following are
  equivalent:
  \begin{enumerate}
  \item $X, d$ is Lipschitz regular;
  \item for every $\alpha \in \Rp$, the map $f \in \Lform X \mapsto
    f^{(\alpha)} \in \Lform_\alpha (X, d)$ is Scott-continuous;
  \item for some $\alpha > 0$, the map $f \in \Lform X \mapsto
    f^{(\alpha)} \in \Lform_\alpha (X, d)$ is Scott-continuous.
  \end{enumerate}
\end{prop}
\proof (1) $\limp$ (2).  Clearly $f \mapsto f^{(\alpha)}$ is
monotonic.  Let ${(f_i)}_{i \in I}$ be a directed family of lower
semicontinuous maps from $X$ to $\creal$, and $f$ be their
(pointwise) supremum.  Note that, for every $t \in \Rp$,
$f^{-1} (]t, +\infty])$ is the union of the directed family of open
sets $f_i^{-1} (]t, +\infty])$, $i \in I$.  Then, for every $x \in X$,
and every $K \in \nat$:
\begin{eqnarray*}
  f_K^{(\alpha)} (x)
  & = &
        \min (\min\nolimits_{k=1}^{K2^K} (\frac {k-1} {2^K} + \alpha d (x,
        \overline {f^{-1} (]k/2^K, +\infty])}), K) \\
  && \qquad \text{by Formula~(\ref{eq:fK(alpha)})} \\
  & = &
        \min (\min\nolimits_{k=1}^{K2^K} (\frac {k-1} {2^K} + \alpha d (x,
        \overline {\bigcup_{i \in I} f_i^{-1} (]k/2^K, +\infty])}), K)
        \\
  & = &
        \min (\min\nolimits_{k=1}^{K2^K} (\frac {k-1} {2^K} +
        \alpha \sup_{i \in I} d (x,
        \overline {f_i^{-1} (]k/2^K, +\infty])}), K)
        \\
  && \qquad\text{by Lipschitz-regularity
        (Lemma~\ref{lemma:lipreg:dist}~(2))} \\
  & = &
        \sup_{i \in I} \min (\min\nolimits_{k=1}^{K2^K} (\frac {k-1} {2^K} +
        \alpha d (x,
        \overline {f_i^{-1} (]k/2^K, +\infty])}), K)
        \\
  & = & \sup_{i \in I} {(f_i)}_K^{(\alpha)} (x)
\end{eqnarray*}
since multiplication by $\alpha$, addition, and $\min$ are
Scott-continuous.  Using Proposition~\ref{prop:f(alpha)}, it follows
that
$f^{(\alpha)} (x) = \sup_{K \in \nat} \sup_{i \in I}
{(f_i)}_K^{(\alpha)} (x) = \sup_{i \in I} \sup_{K \in \nat}
{(f_i)}_K^{(\alpha)} (x) = \sup_{i \in I} f_i^{(\alpha)} (x)$.

(2) $\limp$ (3): obvious.

(3) $\limp$ (1).  (3) applies notably to the family of maps
$r \chi_{U_i}$, where ${(U_i)}_{i \in I}$ is an arbitrary directed
family of open subsets of $X$, and $r \in \Rp$.  Let
$U = \bigcup_{i \in I} U_i$, so that
$\sup_{i \in I} r \chi_{U_i} = r \chi_U$.  Then (3) entails that
$(r\chi_U)^{(\alpha)} = \sup_{i \in I} (r\chi_{U_i})^{(\alpha)}$.
This means that for every $x \in X$,
$\min (r, \alpha d (x, \overline U)) = \sup_{i \in I} \min (r, \alpha
d (x, \overline U_i)) = \min (r, \alpha \sup_{i \in I} d (x, \overline
U_i))$.  Since $r$ is arbitrary, we make it tend to $+\infty$, leaving
$x$ fixed.  We obtain that
$\alpha d (x, \overline U)) = \alpha \sup_{i \in I} d (x, \overline
U_i)$, and since $\alpha > 0$, that
$d (x, \overline U) = \sup_{i \in I} d (x, \overline U_i)$.  Hence
$X, d$ is Lipschitz regular by Lemma~\ref{lemma:lipreg:dist}~(2).
\qed

\begin{cor}
  \label{corl:Lalpha:retract}
  Let $\alpha > 0$, and $X, d$ be a Lipschitz regular standard
  quasi-metric space.  Then the canonical injection
  $i_\alpha \colon \Lform_\alpha (X, d) \to \Lform X$ and the map
  $r_\alpha \colon f \in \Lform X \mapsto f^{(\alpha)} \in
  \Lform_\alpha (X, d)$ form an embedding-projection pair, viz.,
  $r_\alpha$ and $i_\alpha$ are continuous,
  $r_\alpha \circ i_\alpha = \identity {\Lform_\alpha X}$ and
  $i_\alpha \circ r_\alpha \leq \identity {\Lform X}$.
\end{cor}
\proof We know that $i_\alpha$ is continuous (by definition of the
subspace topology), the equalities
$r_\alpha \circ i_\alpha = \identity {\Lform_\alpha X}$ and
$i_\alpha \circ r_\alpha \leq \identity {\Lform X}$ are clear, and
$r_\alpha$ is Scott-continuous by Proposition~\ref{prop:lipreg:Lip}.
Recall however that the topology we have taken on
$\Lform_\alpha (X, d)$ is not the Scott topology.  In order to show
that $r_\alpha$ is continuous, we therefore proceed as follows.  Given
any open subset $V$ of $\Lform_\alpha (X, d)$, by definition of the
subspace topology there is a Scott-open subset $W$ of $\Lform X$ such
that $V = W \cap \Lform_\alpha (X, d)$.  Then
$r_\alpha^{-1} (V) = r_\alpha^{-1} (W)$ is Scott-open in $\Lform X$,
showing that $r_\alpha$ is continuous from $\Lform X$ to
$\Lform_\alpha (X, d)$.  \qed

In the proof of Corollary~\ref{corl:Lalpha:retract}, we have paid
attention to the fact that the subspace topology on
$\Lform_\alpha (X, d)$ might fail to coincide with the Scott topology.
However, when $X, d$ is Lipschitz regular and standard, this is
unnecessary:
\begin{prop}
  \label{prop:Lalpha:retract}
  Let $X, d$ be a Lipschitz regular standard quasi-metric space.  Then
  the subspace topology on $\Lform_\alpha (X, d)$ induced by the Scott
  topology on $\Lform X$ coincides with the Scott topology.
\end{prop}
\proof $r_\alpha$ is Scott-continuous by
Proposition~\ref{prop:lipreg:Lip}~(2), and $i_\alpha$ is also
Scott-continuous, since suprema are computed in the same way in
$\Lform_\alpha (X, d)$ and in $\Lform X$.  In a section-retraction
pair, the section is a topological embedding, so $i_\alpha$ is an
embedding of $\Lform_\alpha (X, d)$, with its Scott topology, into
$\Lform X$.  That implies that the Scott topology on
$\Lform_\alpha (X, d)$ coincides with the subspace topology.
\qed

A similar argument allows us to establish the following.  To show that
$\min (\alpha a . \mathbf 1, f^{(\alpha)}) \in \Lform_\alpha^a (X,
d)$, we use Proposition~\ref{prop:alphaLip:props}~(3) and (6), which
state that the pointwise min of two $\alpha$-Lipschitz continuous
maps is $\alpha$-Lipschitz continuous and that constant maps are
$\alpha$-Lipschitz continuous.  We write $\mathbf 1$ for the constant
map equal to $1$.
\begin{cor}
  \label{corl:Lalpha:retract:bnd}
  Let $\alpha > 0$, $a > 0$, and $X, d$ be a Lipschitz regular
  standard quasi-metric space.  Then the canonical injection
  $i_\alpha^a \colon \Lform_\alpha^a (X, d) \to \Lform X$ and the map
  $r_\alpha^a \colon f \in \Lform X \mapsto \min (a\alpha.\mathbf 1,
  f^{(\alpha)}) \in \Lform_\alpha^a (X, d)$ form an embedding-projection
  pair, viz., $r_\alpha^a$ and $i_\alpha^a$ are continuous,
  $r_\alpha^a \circ i_\alpha^a = \identity {\Lform_\alpha^a X}$ and
  $i_\alpha^a \circ r_\alpha^a \leq \identity {\Lform X}$.  \qed
\end{cor}

\begin{rem}
  \label{rem:Lalpha:retract:bnd}
  As in Proposition~\ref{prop:Lalpha:retract}, this also shows that, when
  $X, d$ is Lipschitz regular and standard, the subspace topology
  (induced by the inclusion into $\Lform X$) coincides with the Scott
  topology on $\Lform_\alpha^a (X, d)$.
\end{rem}


\subsection{$f^{(\alpha)}$ for $f$ $\alpha$-Lipschitz}
\label{sec:falpha-f-alpha}

We no longer assume $f$ lower semicontinuous.  Finding the largest
$\alpha$-Lipschitz (not necessarily continuous) map below $f$ is easy:
\begin{lem}
  \label{lemma:largestLip}
  Let $X, d$ be a quasi-metric space, and $\alpha \in \Rp$.  The
  largest $\alpha$-Lipschitz map below an arbitrary function
  $f \colon X \to \creal$ is given by:
  \begin{equation}
    \label{eq:falpha}
    f^\alpha (x) = \inf_{z \in X} (f (z) + \alpha d (x, z)).
  \end{equation}
\end{lem}
\proof For all $x, y \in X$,
$f^\alpha (y) + \alpha d (x, y) = \inf_{z \in X} (f (z) + \alpha d (x,
y) + \alpha d (y, z)) \geq \inf_{z \in X} (f (z) + \alpha d (x, z)) =
f^\alpha (x)$, so $f^\alpha$ is $\alpha$-Lipschitz.  It is clear that
$f^\alpha$ is below $f$: take $z=x$ in the infimum defining
$f^\alpha$.

Assume another $\alpha$-Lipschitz map $g$ below $f$.  For all
$x, z \in X$, $g (x) \leq g (z) + d (x, z) \leq f (z) + d (x, z)$,
hence by taking infima over all $z \in X$, $g (x) \leq f^\alpha (x)$.
\qed

Every monotonic map $g$ from a space $Y$ to $\creal$ has a lower
semicontinuous envelope $\overline g$, defined as the (pointwise)
largest map below $g$ that is lower semicontinuous.  When $Y$ is a
continuous poset, one can define $\overline g (y)$ as the directed
supremum $\sup_{y' \ll y} g (y')$ (see for example
\cite[Corollary~5.1.61]{JGL-topology}).  This is sometimes called
Scott's formula.

Now assume $f \colon X \to \creal$ is already $\alpha$-Lipschitz, but
not necessarily continuous.  There are two ways one can find an
$\alpha$-Lipschitz continuous map below $f$: either consider
$f^{(\alpha)}$, the largest possible such map, or, if $X, d$ is
continuous, extend $f$ to $f' \colon (x, r) \mapsto f (x) - \alpha r$,
apply Scott's formula to obtain $\overline {f'}$, then restrict the
latter to the subspace $X$ of $\mathbf B (X, d)$.  We show that the
two routes lead to the same function.
\begin{lem}
  \label{lemma:largestLipcont}
  Let $X, d$ be a continuous quasi-metric space, and $\alpha \in \Rp$.
  For any $\alpha$-Lipschitz map $f \colon X \to \creal$, the largest
  $\alpha$-Lipschitz continuous map below $f$, $f^{(\alpha)}$, is
  given by:
  \begin{equation}
    \label{eq:f(alpha)}
    f^{(\alpha)} (x)  = \sup_{(y, s) \ll (x, 0)} (f (y) - \alpha s).
  \end{equation}
  Moreover, $(f^{(\alpha)})'$, defined as mapping $(x, r)$ to
  $f^{(\alpha)} (x) - \alpha r$, is the largest lower semicontinuous
  map $\overline {f'}$ from $\mathbf B (X, d)$ to $\creal$ below
  $f' \colon (x, r) \mapsto f (x) - \alpha r$.
\end{lem}
\proof Take $Y=\mathbf B (X, d)$.  Since $f$ is $\alpha$-Lipschitz,
the map $f' \colon \mathbf B (X, d) \to \real \cup \{+\infty\}$
defined by $f' (x, r) = f (x) - \alpha r$ is monotonic, by
Lemma~\ref{lemma:f'}.  Then
$\overline {f'} (x, r) = \sup_{(y, s) \ll (x, r)} (f (y) - \alpha s)$.

By definition, $\overline {f'}$ is Scott-continuous.  Note that
$\overline {f'} (x, 0)$ is exactly the right-hand side of
(\ref{eq:f(alpha)}).  For clarity, let $g (x) = \overline {f'} (x, 0)
= \sup_{(y, s) \ll (x, 0)} (f (y) - \alpha s)$.

We check that for every $r \in \Rp$,
$\overline {f'} (x, r) = g (x) - \alpha r$.  For that, we use the fact
that, when $X, d$ is a continuous quasi-metric space, the way-below
relation $\ll$ on $\mathbf B (X, d)$ is \emph{standard}
\cite[Proposition~3.6]{JGL:formalballs}, meaning that, for every
$a \in \Rp$, for all formal balls $(x, r)$ and $(y, s)$,
$(y, s) \ll (x, r)$ if and only if $(y, s+a) \ll (x, r+a)$.  It
follows that
$\overline{f'} (x, r) = \sup_{s\geq r, (y, s-r) \ll (x, 0)} (f (y) -
\alpha s) = \sup_{(y, s') \ll (x, 0)} (f (y) - \alpha (s'+r)) =
\sup_{(y, s') \ll (x, 0)} (f (y) - \alpha s')- \alpha r = g (x) -
\alpha r$.

In other words, $\overline {f'} = g'$.  Since $\overline {f'}$ is
Scott-continuous, $g$ is $\alpha$-Lipschitz continuous.

We check that $g$ takes its values in $\creal$, not just
$\real \cup \{+\infty\}$.  For every $x \in X$, for every
$\epsilon > 0$, $(x, 0)$ is in the open set $V_\epsilon$
(Lemma~\ref{lemma:Veps}), and since $(x, 0)$ is the supremum of the
directed family of all formal balls $(y, s) \ll (x, 0)$, one of them
is in $V_\epsilon$; this implies that $g (x) \geq -\alpha\epsilon$,
and as $\epsilon$ is arbitrary, that $g (x) \geq 0$.

Also, $g \leq f$, since for every $x \in X$,
$g (x) = \overline {f'} (x, 0) \leq f' (x, 0) = f (x)$.  Hence $g$ is
an $\alpha$-Lipschitz continuous map from $X$ to $\creal$ below $f$,
from which we deduce that it must also be below the largest such map,
$f^{(\alpha)}$.

We show that $g$ is equal to $f^{(\alpha)}$.  To that end, we take any
$\alpha$-Lipschitz continuous map $h \colon X \to \creal$ below $f$,
and we show that $h \leq g$.  Since $h \leq f$, $h' \leq f'$.  Since
$h$ is $\alpha$-Lipschitz continuous, $h'$ is Scott-continuous.  We
use the fact that $\overline {f'}$ is the largest Scott-continuous map
below $f'$ to obtain $h' \leq \overline {f'}$, and apply both sides of
the inequality to $(x, 0)$ to obtain
$h (x) = h' (x) \leq \overline {f'} (x, 0) = g (x)$.

Since $g=f^{(\alpha)}$ and $g (x) = \sup_{(y, s) \ll (x, 0)} (f (y) -
\alpha s)$ by definition, (\ref{eq:f(alpha)}) follows.

Finally, we have seen that $\overline {f'} = g'$, namely that
$\overline {f'} = (f^{(\alpha)})'$, and that is the final part of the
lemma.  \qed

\begin{cor}
  \label{corl:ha=hb}
  Let $X, d$ be a continuous quasi-metric space, and
  $\alpha, \beta \in \Rp$.  For every
  $f \in L_\alpha (X, d) \cap L_\beta (X, d)$,
  $f^{(\alpha)} = f^{(\beta)}$.
\end{cor}
\proof Since $\mathbf B (X, d)$ is a continuous poset, for every
$x \in X$, $(x, 0)$ is the supremum of the directed family of formal
balls $(y, s) \ll (x, 0)$, and since $X, d$ is standard, such formal
balls have arbitrarily small radii $s$.

By (\ref{eq:f(alpha)}),
$f^{(\alpha)} (x) = \sup_{(y, s) \ll (x, 0)} (f (y) - \alpha s)$.  For
every $(y, s) \ll (x, 0)$, for every $\epsilon > 0$, we can find
another formal ball $(z, t) \ll (x, 0)$ such that $t < \epsilon$, by
the remark we have just made.  Using directedness, we can require
$(z, t)$ to be above $(y, s)$.  Then
$f (y) - \alpha s \leq f (z) - \alpha t$ since $f$ is
$\alpha$-Lipschitz, and
$f (z) - \alpha t = f (z) - \beta t + (\beta-\alpha) t$ is less than
or equal to $f^{(\beta)} (x) + |\beta-\alpha| \epsilon$.  Taking
suprema over $(y, s) \ll (x, 0)$, we obtain
$f^{(\alpha)} (x) \leq f^{(\beta)} (x) + |\beta-\alpha| \epsilon$.
Since $\epsilon$ can be made arbitrarily small,
$f^{(\alpha)} (x) \leq f^{(\beta)} (x)$.  We show
$f^{(\beta)} (x) \leq f^{(\alpha)} (x)$ symmetrically, using the fact
that $f$ is $\beta$-Lipschitz.  The equality follows.  \qed

\begin{lem}
  \label{lemma:largestLipcont:alg}
  Let $X, d$ be a standard algebraic quasi-metric space, and
  $\alpha \in \Rp$.  For any $\alpha$-Lipschitz map
  $f \colon X \to \creal$, the largest $\alpha$-Lipschitz continuous
  map below $f$, namely $f^{(\alpha)}$, is given by:
  \begin{equation}
    \label{eq:f(alpha):alg}
    f^{(\alpha)} (x)  = \sup_{\substack{z \text{ center point}\\
        t > d (z, x)}} (f (z) - \alpha t).
  \end{equation}
  Moreover, $(f^{(\alpha)})'$, defined as mapping $(x, r)$ to
  $f^{(\alpha)} (x) - \alpha r$, is the largest lower semicontinuous
  map $\overline {f'}$ from $\mathbf B (X, d)$ to $\creal$ below
  $f' \colon (x, r) \mapsto f (x) - \alpha r$.
\end{lem}

(\ref{eq:f(alpha):alg}) simplifies to $f^{(\alpha)} (x)  = \sup_{z
  \text{ center point}} (f (z) - \alpha d (z, x))$ when $d (z, x) \neq
+\infty$ for all center points $z$, or when $f (z) \neq +\infty$.

\proof Easy consequence of Lemma~\ref{lemma:largestLipcont}, using the
fact that, in a standard algebraic quasi-metric space,
$(y, s) \ll (x, r)$ if and only if there is a center point $z$ and
some $t \in \Rp$ such that $(y, s) \leq^{d^+} (z, t)$ and
$d (z, x) < t-r$ \cite[Proposition~5.18]{JGL:formalballs}.  \qed

\subsection{The Functions $\sea x b$}
\label{sec:functions-sea-x}

\begin{defi}[$\sea x b$]
  \label{defn:sea}
  Let $X, d$ be a quasi-metric space.  For each center point $x \in X$
  and each $b \in \creal$, let $\sea x b \colon X \to \creal$ map
  every $y \in X$ to the smallest element $t \in \creal$ such that
  $b \leq t + d (x, y)$.
\end{defi}
When $b \neq +\infty$, we might have said, more simply:
$(\sea x b) (y) = \max (b - d (x, y), 0)$.

The definition caters for the general situation.  When $b=+\infty$,
$(\sea x {+\infty}) (y)$ is equal to $0$ if $d (x, y) = +\infty$, and
to $+\infty$ if $d (x, y) < +\infty$.

\begin{lem}
  \label{lemma:sea:cont}
  Let $X, d$ be a standard quasi-metric space, $x$ be a center point
  of $X, d$ and $b \in \creal$.  The function $\sea x b$ is
  $1$-Lipschitz continuous.
\end{lem}
\proof It is enough to show that it is continuous, by
Proposition~\ref{prop:cont}.  Let $f = \sea x b$, and
$f' (y, r) = f (y) - r$.  For every $s \in \real$, we wish to show
that ${f'}^{-1} (]s, +\infty])$ is Scott-open.

If $b \neq +\infty$, a formal ball $(y, r)$ is in
${f'}^{-1} (]s, +\infty])$ if and only if
$\max (b - d (x, y), 0) - r > s$.  This is equivalent to
$d (x, y) + r < b - s$ or $r < -s$.  The set
$V_{-s} = \{(y, r) \mid r < -s\}$ is Scott-open since $X, d$ is
standard, by Lemma~\ref{lemma:Veps}.  The condition
$d (x, y) + r < b - s$ is vacuously false if $b \leq s$, and is
equivalent to $d^+ ((x, 0), (y, r)) < b-s$ otherwise.  It follows that
${f'}^{-1} (]s, +\infty])$ is equal to $V_{-s}$ if $b \leq s$, or to
$V_{-s} \cup B^{d^+}_{(x, 0), < b-s}$ otherwise.  This is Scott-open
in any case, because $x$ is a center point.

If $b = +\infty$, then $(y, r) \in {f'}^{-1} (]s, +\infty])$ entails
that $-r > s$ if $d (x, y) = +\infty$, or $d (x, y) < +\infty$.
Conversely, if $-r > s$, then $f' (y, r) > s$, whether
$d (x, y) = +\infty$ or $d (x, y) < +\infty$, and if
$d (x, y) < +\infty$, then $f' (y, r) = +\infty > s$.  Also,
$d (x, y) < +\infty$ if and only if $d^+ ((x, 0), (y, r)) < +\infty$,
if and only if $d^+ ((x, 0), (y, r)) < N$ for some natural number $N$.
Therefore ${f'}^{-1} (]s, +\infty])$ is equal to
$V_{-s} \cup \bigcup_N B^{d^+}_{(x, 0), < N}$, which is Scott-open.  \qed

\begin{lem}
  \label{lemma:sea:min}
  Let $X, d$ be a standard quasi-metric space, $x_i$ be center points
  of $X, d$ and $b_i \in \creal$, $1\leq i \leq n$.  The function
  $\bigvee_{i=1}^n \sea {x_i} {b_i}$, which maps every $y \in X$ to
  $\max \{(\sea {x_i} {b_i}) (y) \mid 1\leq i\leq n\}$, is the smallest
  $1$-Lipschitz map $f$ (hence also the smallest function in
  $\Lform_1 (X, d)$) such that $f (x_i) \geq b_i$ for every $i$,
  $1\leq i\leq n$.
\end{lem}
\proof First, $f = \bigvee_{i=1}^n \sea {x_i} {b_i}$ is in
$\Lform_1 (X, d)$, by Lemma~\ref{lemma:sea:cont} and
Proposition~\ref{prop:alphaLip:props}~(3), which states that the
pointwise max of two $\alpha$-Lipschitz continuous maps is
$\alpha$-Lipschitz continuous.  For any other $1$-Lipschitz map $h$
such that $h (x_i) \geq b_i$, $1\leq i\leq n$, for every $y \in X$,
$h (x_i) \leq h (y) + d (x_i, y)$.  In other words, $h (y)$ is a
number $t \in \creal$ such that $b_i \leq t + d (x_i, y)$.  The
smallest such $t$ is $\sea {x_i} {b_i} (y)$ by definition, so
$h (y) \geq (\sea {x_i} {b_i}) (y)$.  Since that holds for every $i$
and every $y$, $h \geq f$.  \qed

\section{Topologies on $\Lform X$}
\label{sec:topol-spac-maps}

Several results in the sequel will rely on a fine analysis of the
(Scott) topology of $\Lform X$ and of the (subspace) topology of
$\Lform_\infty X$, $\Lform_\alpha (X, d)$, and a few other subspaces.

We shall notably see that, in case $X, d$ is continuous
Yoneda-complete, the Scott topology on $\Lform X$, the compact-open
topology, and the Isbell topology all coincide.  We already know the
Scott topology, and recall that we order $\Lform X$ pointwise.  The
compact-open topology has a subbase of sets of the form
$[Q \in V] = \{f \mid f \text{ maps } Q \text{ to }V\}$, for $Q$ a
compact subset of the domain and $V$ an open subset of the codomain.
In the case of $\Lform X$, the codomain is $\creal$ with its Scott
topology, and we shall write $[Q > a]$ for the subbasic open set
$\{f \in \Lform X \mid \forall x \in Q . f (x) > a\}$.  The Isbell
topology has a subbase of sets of the form
$N (\mathcal U, V) = \{f \mid f^{-1} (V) \in \mathcal U\}$, where $V$
is open in the codomain and $\mathcal U$ is Scott-open in the lattice
of open subsets of the domain.

We will prove this result of coincidence of topologies in several
steps, concentrating first on the coincidence of the compact-open and
Isbell topologies.  The compact-open topology is always coarser than
the Isbell topology (which is coarser than the Scott topology), and
the converse is known to hold when $X$ is consonant
\cite{DGL:consonant}.

For a subset $Q$ of $X$, let $\blacksquare Q$ be the family of open
neighborhoods of $Q$.  A space is \emph{consonant} if and only if,
given any Scott-open family $\mathcal U$ of open sets, and given any
$U \in \mathcal U$, there is a compact saturated set $Q$ such that
$U \in \blacksquare Q \subseteq \mathcal U$.  Equivalently, if and
only if every Scott-open family of opens is a union of sets of the
form $\blacksquare Q$, $Q$ compact saturated.  The latter generate the
compact-open topology on the space $\Open X$ of open subsets of $X$,
if we equate $\Open X$ with the family of continuous maps from $X$ to
Sierpi\'nski space.  Using the same identification, the Isbell
topology on $\Open X$ coincides with the Scott topology.  The latter
Scott topology is always finer than the compact-open topology.  A
space is consonant if and only if the two topologies coincide.

In a locally compact space, every open subset $U$ is the union of the
interiors $int (Q)$ of compact saturated subsets of $U$, and that
family is directed.  It follows immediately that every locally compact
space is consonant, as we have already claimed.

In the following, a \emph{well-filtered} space is a topological space
in which the following property holds (see
\cite[Proposition~8.3.5]{JGL-topology}): if a filtered intersection of
compact saturated sets $Q_i$ is included in some open set $U$, then
some $Q_i$ is already included in $U$.  For a $T_0$ locally compact
space, well-filteredness is equivalent to the more well-known property
of \emph{sobriety} (Propositions~8.3.5 and 8.3.8, loc.cit.).  Hence,
under the additional assumption of being $T_0$, the following
theorem could also be stated thus: every $G_\delta$ subspace of a sober
locally compact space is consonant.

\begin{thm}
  \label{thm:Gdelta=>cons}
  Every $G_\delta$ subspace of a well-filtered locally compact space
  is consonant.
\end{thm}
Before we start the proof, let us recall that every locally compact
space is consonant.  But consonance is not preserved under the
formation of $G_\delta$ subsets \cite[Proposition~7.3]{DGL:consonant}.
We shall also use the following easy fact: in a well-filtered space,
the intersection of a filtered family of compact saturated subsets is
compact saturated (Proposition~8.3.6, loc.cit.).

\proof Let $Y$ be a $G_\delta$ subset of $X$, where $X$ is
well-filtered and locally compact.  We may write $Y$ as
$\bigcap_{n \in \nat} V_n$, where each $V_n$ is open in $X$ and
$V_0 \supseteq V_1 \supseteq \cdots \supseteq V_n \supseteq \cdots$.
Let $\mathcal U$ be a Scott-open family of open subsets of $Y$, and
$U \in \mathcal U$.

By the definition of the subspace topology, there is an open subset
$\widehat U$ of $X$ such that $\widehat U \cap Y = U$.  By local
compactness, $\widehat U \cap V_0$ is the union of the directed family
of the sets $int (Q)$, where $Q$ ranges over the family $\mathcal Q_0$
of compact saturated subsets of $\widehat U \cap V_0$.  We have
$\bigcup_{Q \in \mathcal Q_0} int (Q) \cap Y = \widehat U \cap V_0
\cap Y = \widehat U \cap Y = U$.
Since $U$ is in $\mathcal U$ and $\mathcal U$ is Scott-open,
$int (Q) \cap Y$ is in $\mathcal U$ for some $Q \in \mathcal Q_0$.
Let $Q_0$ be this compact saturated set $Q$,
$\widehat U_0 = int (Q_0)$, and $U_0 = \widehat U_0 \cap Y$.  Note
that $U_0 \in \mathcal U$,
$\widehat U_0 \subseteq Q_0 \subseteq \widehat U \cap V_0$.

Do the same thing with $\widehat U_0 \cap V_1$ instead of
$\widehat U \cap V_0$.  There is a compact saturated subset $Q_1$ of
$\widehat U_0 \cap V_1$ such that $int (Q_1) \cap Y$ is in
$\mathcal U$.  Then, letting $\widehat U_1 = int (Q_1)$,
$U_1 = \widehat U_1 \cap Y$, we obtain that $U_1 \in \mathcal U$,
$\widehat U_1 \subseteq Q_1 \subseteq \widehat U_0 \cap V_1$.

Iterating this construction, we obtain for each $n \in \nat$ a compact
saturated subset $Q_n$ and an open subset $\widehat U_n$ of $X$, and
an open subset $U_n$ of $Y$ such that $U_n \in \mathcal U$ for each
$n$, and $\widehat U_{n+1} \subseteq Q_{n+1} \subseteq \widehat U_n
\cap V_n$.

Let $Q = \bigcap_{n \in \nat} Q_n$.  Since $X$ is well-filtered, 
$Q$ is compact saturated.

Since $Q \subseteq \bigcap_{n \in \nat} V_n = Y$, $Q$ is a compact
saturated subset of $X$ that is included in $Y$, hence a compact
saturated subset of $Y$ by Lemma~\ref{lemma:compact:subspace}.  We
have
$Q \subseteq Q_0 \subseteq \widehat U \cap V_0 \subseteq \widehat U$,
and $Q \subseteq Y$, so $Q \subseteq \widehat U \cap Y = U$.
Therefore $U \in \blacksquare Q$.

For every $W \in \blacksquare Q$, write $W$ as the intersection of
some open subset $\widehat W$ of $X$ with $Y$.  Since $Q = \bigcap_{n
  \in \nat} Q_n \subseteq \widehat W$, by well-filteredness some $Q_n$
is included in $\widehat W$.   Hence $\widehat U_n \subseteq Q_n
\subseteq \widehat W$.  Taking intersections with $Y$, $U_n \subseteq
W$.  Since $U_n$ is in $\mathcal U$, so is $W$.  \qed

Every continuous dcpo is sober in its Scott topology
\cite[Proposition~8.2.12~$(b)$]{JGL-topology} ---hence
well-filtered---, and also locally compact (Corollary~5.1.36,
loc.cit.)  Hence when $X, d$ is a continuous Yoneda-complete
quasi-metric space, i.e., when $X, d$ is standard and
$\mathbf B (X, d)$ is a continuous dcpo, then $\mathbf B (X, d)$ is
well-filtered and locally compact.  Recall that every Yoneda-complete
space is standard, so that Lemma~\ref{lemma:Veps} applies, and we
obtain that $X$ is a $G_\delta$-subset of $\mathbf B (X, d)$.  Hence
Theorem~\ref{thm:Gdelta=>cons} immediately yields:

\begin{prop}
  \label{prop:cont=>cons}
  Every continuous Yoneda-complete quasi-metric space $X, d$ is
  consonant in its $d$-Scott topology.  \qed
\end{prop}

\begin{rem}
  The \emph{Sorgenfrey line} $\Rl$ is $\real, \dRl$, where
  $\dRl (x, y)$ is equal to $+\infty$ if $x > y$, and to $y-x$ if
  $x \leq y$ \cite{KW:formal:ball}.  Its specialization ordering
  $\leq^{\dRl}$ is equality.  Its open ball topology is the topology
  generated by the half-open intervals $[x, x+r)$, and is a well-known
  counterexample in topology.  In particular, $\Rl$ is not consonant
  in its open ball topology \cite{AC:consonant}.  However, $\Rl$ is a
  continuous Yoneda-complete quasi-metric space, since the map
  $(x, r) \mapsto (-x-r, x)$ is an order isomorphism from
  $\mathbf B (\real, \dRl)$ to the continuous dcpo
  $C_\ell = \{(a, b) \in \real^2 \mid a+b \leq 0\}$.
  Proposition~\ref{prop:cont=>cons} implies that $\Rl$ \emph{is}
  consonant in the $\dRl$-Scott topology.  You may confirm this by
  checking that the latter topology is just the usual metric topology
  on $\real$, which is locally compact hence consonant.
\end{rem}
The topological coproduct of two consonant spaces is not in general
consonant \cite[Example~6.12]{NS:consonant}.  I do not know whether,
given a consonant space $X$, its $n$th \emph{copower} $n \odot X$ (the
coproduct of $n$ copies of $X$) is consonant, but that seems unlikely.
\begin{defi}[$\odot$-consonant]
  \label{defn:realcons}
  A topological space $X$ is called \emph{$\odot$-consonant} if and only
  if, for every $n \in \nat$, $n \odot X$ is consonant.
\end{defi}
In particular, every $\odot$-consonant space is consonant.  Since
coproducts of locally compact spaces are locally compact, every
locally compact space is $\odot$-consonant.
\begin{lem}
  \label{lemma:coprod:cont}
  For every two continuous Yoneda-complete quasi-metric spaces $X, d$
  and $Y, \partial$, the quasi-metric space $X+Y$, with quasi-metric
  $d+\partial$ defined by $(d+\partial) (x, x') = d (x, x')$ if $x, x'
  \in X$, $(d+\partial) (y, y') = \partial (y, y')$ if $y, y' \in Y$,
  $(d + \partial) (z, z') = +\infty$ in all other cases, is continuous
  Yoneda-complete.  With the $(d+\partial)$-Scott topology, $X+Y$ is
  the topological coproduct of $X$ (with the $d$-Scott topology) and
  $Y$ (with the $\partial$-Scott topology).
\end{lem}
\proof It is easy to see that $\mathbf B (X+Y, d+\partial)$ is the
poset coproduct of $\mathbf B (X, d)$ and of
$\mathbf B (Y, \partial)$: $(x, r) \leq^{(d+\partial)^+} (x', r')$ iff
$(x, r) \leq^{d^+} (x', r')$ if $x, x' \in X$,
$(y, s) \leq^{(d+\partial)^+} (y', s')$ iff
$(y, s) \leq^{\partial^+} (y', s')$ if $y, y' \in Y$, and for
$x \in X$ and $y \in Y$, $(x, r)$ and $(y, s)$ are incomparable.  The
poset coproduct of two continuous dcpos is a continuous dcpo, so
$\mathbf B (X+Y, d+\partial)$ is a continuous dcpo, which shows the
first part of the lemma.  Also, the Scott topology on a poset
coproduct is the coproduct topology, which shows the second part of
the lemma.  \qed

Together with Proposition~\ref{prop:cont=>cons}, this yields the
following corollary.
\begin{prop}
  \label{prop:cont=>realcons}
  Every continuous Yoneda-complete quasi-metric space $X, d$ is
  $\odot$-consonant in its $d$-Scott topology.
\end{prop}
\proof For every $n \in \nat$, $n \odot X$ is continuous
Yoneda-complete, hence consonant.  \qed

\begin{prop}
  \label{prop:co=Scott}
  Let $X$ be a $\odot$-consonant space, for example, a continuous
  Yoneda-complete quasi-metric space.  The compact-open topology on
  $\Lform X$ is equal to the Scott topology on $\Lform X$.
\end{prop}
\proof Let $f [Q]$ denote the image of $Q$ by $f$.  Write
$[Q > a] = \{f \in \mathcal L X \mid f [Q] \subseteq ]a, +\infty]\} =
\{f \in \mathcal L X \mid Q \subseteq f^{-1} (]a, +\infty])$, where
$Q$ is compact saturated in $X$ and $a \in \real$, for the subbasic
open subsets of the compact-open topology.

It is easy to see that $[Q > a]$ is Scott-open, so that the Scott
topology is finer than the compact-open topology.  In the reverse
direction, let $\mathcal W$ be a Scott-open subset of $\Lform X$,
$f \in \mathcal W$.  Our task is to find an open neighborhood of $f$
in the compact-open topology that would be included in $\mathcal W$.

The function $f$ is the
pointwise supremum of a chain of so-called \emph{step functions}, of
the form $\frac 1 {2^N} \sum_{i=1}^{N2^N} \chi_{U_i}$, where $U_1$,
$U_2$, \ldots, $U_{N2^N}$ are open subsets, and $N \in \nat$.
Concretely, $U_i = f^{-1} (]i/2^N, +\infty])$.
Hence some $\frac 1 {2^N} \sum_{i=1}^{N2^N} \chi_{U_i}$ is in
$\mathcal W$, where we can even require $U_1 \supseteq U_2 \supseteq
\cdots \supseteq U_{N2^N}$.

Let $\mathcal V$ be the set of $N2^N$-tuples
$(V_1, V_2, \cdots, V_{N2^N})$ of open subsets of $X$ such that
$\frac 1 {2^N} \sum_{i=1}^{N2^N} \chi_{V_i}$ is in $\mathcal W$.
Ordering those tuples componentwise, where each open is compared
through inclusion, $\mathcal V$ is Scott-open.  Equating those tuples
with open subsets of $N2^N \odot X$ allows us to apply
$\odot$-consonance: there is a compact saturated subset of
$N2^N \odot X$, or equivalently, there are $N2^N$ compact saturated
subsets $Q_1$, $Q_2$, \ldots, $Q_{N2^N}$ of $X$ such that
$Q_i \subseteq U_i$ for each $i$, and such that for all open
neighborhoods $V_i$ of $Q_i$, $1\leq i\leq N2^N$,
$\frac 1 {2^N} \sum_{i=1}^{N2^N} \chi_{V_i}$ is in $\mathcal W$.

Clearly, $\frac 1 {2^N} \sum_{i=1}^{N2^N} \chi_{U_i}$ is in
$\bigcap_{i=1}^{N2^N} [Q_i > i/2^N]$, hence also the larger function
$f$.  Consider any $g \in \bigcap_{i=1}^{N2^N} [Q_i > i/2^N]$.  Let
$V_i = g^{-1} (]i/2^N, +\infty])$.  By definition, $Q_i$ is included
in $V_i$ for every $i$, so
$\frac 1 {2^N} \sum_{i=1}^{N2^N} \chi_{V_i}$ is in $\mathcal W$.  

For every $x \in X$, let $i$ be the smallest number between $1$ and
$N2^N$ such that $g (x) \leq (i+1)/2^N$, and $N2^N$ if there is none.
Then $x$ is in $V_1$, $V_2$, \ldots, $V_i$, and not in $V_{i+1}$,
\ldots, $V_{N2^N}$, hence is mapped to $i/2^N$ by
$\frac 1 {2^N} \sum_{i=1}^{N2^N} \chi_{V_i}$.  That is less than or
equal to $g (x)$.  It follows that $g$ is pointwise above
$\frac 1 {2^N} \sum_{i=1}^{N2^N} \chi_{V_i}$, and is therefore also in
$\mathcal W$, finishing the proof that
$\bigcap_{i=1}^{N2^N} [Q_i > i/2^N]$ is included in $\mathcal W$.
\qed

\section{Topologies on $\Lform_\alpha (X, d)$ and $\Lform_\alpha^a (X,
d)$}
\label{sec:topol-lform_-x}

The topology we have taken on $\Lform_\alpha (X, d)$ and on
$\Lform_\alpha^a (X, d)$ is the subspace topology from $\Lform X$
(with its Scott topology).  Proposition~\ref{prop:co=Scott}
immediately implies that, when $X, d$ is continuous Yoneda-complete,
the topologies of $\Lform_\alpha (X, d)$ and $\Lform_\alpha^a (X, d)$
are the compact-open topology, i.e., the topology generated by the
subsets of the form
$\{f \in \Lform_\alpha X \mid f [Q] \subseteq ]a, +\infty]\}$ (resp.,
$\{f \in \Lform_\alpha^a X \mid f [Q] \subseteq ]a, +\infty]\}$), $Q$
compact saturated in $X$, and which we again write as $[Q > a]$.

Let us refine our understanding of the shape of compact saturated
subsets $Q$ of $X$.  Recall that we write $B^d_{x, \leq r}$ for the closed ball
$\{y \in X \mid d (x, y) \leq r\}$, and beware that, despite their
names, closed balls are not closed.

The topology of pointwise convergence, on any set of functions from
$X$ to $Y$, is the subspace topology from the ordinary product
topology on $Y^X$.  It is always coarser than the compact-open
topology.  When $Y = \creal$ or $Y = [0, a]$, with its Scott-topology,
a subbase for the pointwise topology is given by the subsets
$[x > r] = \{f \mid f (x) > r\}$, where $x \in X$ and $r \in \Rp$.
\begin{prop}
  \label{prop:cont:Lalpha:pw}
  Let $X, d$ be a continuous Yoneda-complete quasi-metric space.  For
  every $\alpha \in \Rp$, the topology of $\Lform_\alpha (X, d)$
  (resp., $\Lform_\alpha^a (X, d)$, for any $a > 0$) coincides with
  the topology of pointwise convergence.
\end{prop}
\proof We already know that the topology of $\Lform_\alpha (X, d)$,
and of $\Lform_\alpha^a (X, d)$, is the compact-open topology, as a
consequence of Proposition~\ref{prop:co=Scott}.  It therefore suffices
to show that every subset $[Q > r]$ of $\Lform_\alpha (X, d)$ or
$\Lform_\alpha^a (X, d)$, $Q$ compact saturated in $X$ and
$r \in \Rp$, is open in the topology of pointwise convergence.  We
assume without loss of generality that $Q$ is non-empty.  Let
$f \in \Lform_\alpha (X, d)$, resp.\ $f \in \Lform_\alpha^a (X, d)$,
be an arbitrary element of $[Q > r]$.  Since $f$ is lower
semicontinuous, i.e., continuous from $X$ to $\creal$ (resp., to
$[0, \alpha a]$), the image of $Q$ by $f$ is compact, and its upward
closure is compact saturated.  The non-empty compact saturated subsets
of $\creal$ (resp., $[0, \alpha a]$), in its Scott topology, are the
intervals $[b, +\infty]$, $b \in \creal$ (resp., $[b, \alpha a]$,
$b \in [0, \alpha a]$); hence the image of $Q$ by $f$ is of this form,
and $b$ is necessarily the minimum value attained by $f$ on $Q$,
$\min_{x \in X} f (x)$.  Since $f \in [Q > r]$, we must have $b > r$,
and therefore there is an $\eta > 0$ such that $b > r+\eta$.

Let $\epsilon > 0$ be such that $\alpha \epsilon \leq \eta$.  Put in a
simpler form, let $\epsilon = \eta/\alpha$ if $\alpha\neq 0$,
otherwise let $\epsilon > 0$ be arbitrary.

Apply Lemma~\ref{lemma:cont:Q}, and find finitely many closed balls
$B_{x_i, \leq r_i}$, $1\leq i\leq n$, included in
$U = f^{-1} (]r+\eta, +\infty])$, with $r_i < \epsilon$, and whose
union contains $Q$.  Now consider
$W = \bigcap_{i=1}^n [x_i > r + \eta]$.  Since each $x_i$ is in $U$,
$f$ is in $W$.  For every $g \in \Lform_\alpha (X, d)$ (resp.,
$g \in \Lform_\alpha^a (X, d)$) that is in $W$, we claim that $g$ is in
$[Q > r]$.  For every $x \in Q$, there is an index $i$ such that
$x \in B_{x_i, \leq r_i}$, and since $g$ is $\alpha$-Lipschitz,
$g (x_i) \leq g (x) + \alpha r_i$.  By assumption $g (x_i) > r+\eta$,
and $r_i < \epsilon$, so $g (x) > r$, proving the claim.  Since $W$ is
open in the topology of pointwise convergence, we have proved that
every element of $[Q > r]$ is an open neighborhood, for the topology
of pointwise convergence, of each of its elements, which proves the
result.  \qed

We shall now use some results in the theory of stably compact spaces
\cite[Chapter~9]{JGL-topology}.  A \emph{stably compact} space is a sober,
locally compact, compact and coherent space, where coherence means
that the intersection of any two compact saturated subsets is compact.
For a compact Hausdorff space $Z$ equipped with an ordering $\preceq$
whose graph is closed in $Z^2$ (a so-called \emph{compact pospace}),
the space $Z$ with the upward topology, whose opens are by definition
the open subsets of $Z$ that are upwards-closed with respect to
$\preceq$, is stably compact.  More: all stably compact spaces can be
obtained this way.  If $X$ is stably compact, then form a second
topology, the \emph{cocompact} topology, whose closed subsets are the
compact saturated subsets of the original space.  With the cocompact
topology, we obtain a space called the \emph{de Groot dual} of $X$,
$X^\dG$, which is also stably compact.  We have $X^{\dG\dG} = X$.
Also, with the join of the original topology on $X$ and of the
cocompact topology, one obtains a compact Hausdorff space $X^\patch$,
the \emph{patch space} of $X$.  Together with the specialization
ordering $\leq$ of $X$, $X^\patch$ is then a compact pospace.
Moreover, passing from $X$ to its compact pospace, and conversely, are
mutually inverse operations.

We shall need to go through the following auxiliary notion.
\begin{defi}[$L_\alpha (X, d)$, $L_\alpha^a (X, d)$]
  \label{defn:Lalpha:notcont}
  Let $L_\alpha (X, d)$ denote the space of all $\alpha$-Lipschitz
  (not necessarily $\alpha$-Lipschitz continuous) maps from $X, d$ to
  $\creal$.  Let also $L_\alpha^a (X, d)$ be the subspace of all
  $h \in L_\alpha (X, d)$ such that $h \leq \alpha a$.
\end{defi}
Please pay attention to the change of font compared with $\Lform_\alpha (X, d)$.

$\creal$ (resp., $[0, \alpha a]$), with its Scott topology, is stably
compact, and its patch space $\creal^\patch$ (resp.,
$[0, \alpha a]^\patch$) has the usual Hausdorff topology.  Since the
product of stably compact spaces is stably compact (see
\cite[Proposition~9.3.1]{JGL-topology}), $\creal^X$ (resp.,
$[0, \alpha a]^X$) is stably compact.  Moreover, the patch operation
commutes with products, so $(\creal^X)^\patch = (\creal^\patch)^X$
(resp., $([0, \alpha a]^X)^\patch = ([0, \alpha a]^\patch)^X$), and
the specialization ordering is pointwise.

The subset $L_\alpha (X, d)$ (resp., $L_\alpha^a (X, d)$) is then
patch-closed in $\creal^X$ (resp., $[0, \alpha a]^X$), meaning that it
is closed in $(\creal^X)^\patch = (\creal^\patch)^X$ (resp., in
$([0, \alpha a]^X)^\patch = ([0, \alpha a]^\patch)^X$).  Indeed,
$L_\alpha (X, d) = \{f \in \creal^X \mid \forall x, y \in X . f (x)
\leq f (y) + \alpha d (x, y)\}$ is the intersection of the sets
$\{f \in \creal^X \mid f (x) \leq f (y) + \alpha d (x, y)\}$,
$x, y \in X$, and each is patch-closed, because the graph of $\leq$ is
closed in $(\creal^\patch)^2$ and the maps $f \mapsto f (x)$ are
continuous from $(\creal^\patch)^X$ to $\creal^\patch$---similarly
with $[0, \alpha a]$ in lieu of $\creal$ and $L_\alpha^a (X, d)$
instead of $L_\alpha (X, d)$.  It is well-known (Proposition~9.3.4,
loc.cit.) that the patch-closed subsets $C$ of stably compact spaces
$Y$ are themselves stably compact, and that the topology of $C^\patch$
is the subspace topology of $Y^\patch$.  It follows that:
\begin{lem}
  \label{lemma:cont:Lalpha:patchclos}
  Let $X, d$ be a continuous quasi-metric space.  $L_\alpha (X, d)$
  and $L_\alpha^a (X, d)$, for every $a \in \Rp$, $a > 0$, are stably
  compact in the topology of pointwise convergence.  \qed
\end{lem}

\begin{lem}
  \label{lemma:cont:Lalpha:retr}
  Let $X, d$ be a continuous quasi-metric space, $\alpha \in \Rp$, and
  $a \in \Rp$, $a > 0$.  Let $\rho_\alpha$ be the map
  $f \in L_\alpha (X, d) \mapsto f^{(\alpha)} \in \Lform_\alpha (X,
  d)$.  Then:
  \begin{enumerate}
  \item $\rho_\alpha$ is continuous from $L_\alpha (X, d)$ to
    $\Lform_\alpha (X, d)$ (resp., from $L_\alpha^a (X, d)$ to
    $\Lform_\alpha^a (X, d)$), both spaces being equipped with the
    topology of pointwise convergence;
  \item $\Lform_\alpha (X, d)$ is a retract of $L_\alpha (X, d)$
    (resp., $\Lform_\alpha^a (X, d)$ is a retract of
    $L_\alpha^a (X, d)$), with $\rho_\alpha$ the retraction and where
    inclusion serves as section; again both spaces are assumed to have
    the topology of pointwise convergence here;
  \item $\Lform_\alpha (X, d)$ and $\Lform_\alpha^a (X, d)$, with the
    topology of pointwise convergence, are stably compact;
  \item If $X, d$ is also Yoneda-complete, then $\Lform_\alpha (X, d)$
    and $\Lform_\alpha^a (X, d)$ are stably compact (in the subspace
    topology of the Scott topology on $\Lform X$).
  \end{enumerate}
\end{lem}
\proof Recall from Lemma~\ref{lemma:largestLipcont} that
$f^{(\alpha)} (x) = \sup_{(y, s) \ll (x, 0)} (f (y) - \alpha s)$.

(1) The inverse image of $[x > r]$ by $\rho_\alpha$ is the set of
$\alpha$-Lipschitz maps $f$ such that, for some $(y, s) \ll (x, 0)$,
$f (y) - \alpha s > r$, hence is the union of the open sets $[y >
r+\alpha s]$ over the elements $(y, s)$ way-below $(x, 0)$.

(2) If $\iota$ denotes the inclusion of $\Lform_\alpha (X, d)$ into
$L_\alpha (X, d)$ (resp., of $\Lform_\alpha^a (X, d)$ into
$L_\alpha^a (X, d)$), we must show that $\rho_\alpha \circ \iota$ is
the identity map.  For every $f \in \Lform_\alpha (X, d)$,
$\rho_\alpha (\iota (f)) = f^{(\alpha)}$.  Since $f$ is already
$\alpha$-Lipschitz continuous, $f^{(\alpha)} = f$, and we conclude.

(3) follows from Lemma~\ref{lemma:cont:Lalpha:patchclos}, from (2),
and the fact that retracts of stably compact spaces are stably compact
(see Proposition~9.2.3 in loc.cit.; the result is due to Jimmie Lawson
\cite[Proposition, bottom of p.153, and subsequent
discussion]{Lawson:versatile}).

(4) follows from (3) and Proposition~\ref{prop:cont:Lalpha:pw}.  \qed

The proof of (3) above is inspired by a similar argument due to Achim
Jung \cite{Jung:scs:prob}.

\begin{lem}
  \label{lemma:cont:Lalpha:compact}
  Let $X, d$ be a continuous Yoneda-complete quasi-metric space, and
  $\alpha \in \Rp$.  For every center point $x$ in $X$ and every
  $b \in \Rp$, the set $[x \geq b]$ of all
  $f \in \Lform_\alpha (X, d)$ (resp., $f \in \Lform_\alpha^a (X, d)$)
  such that $f (x) \geq b$ is compact saturated in
  $\Lform_\alpha (X, d)$ (resp., in $\Lform_\alpha^a (X, d)$, for
  every $a \in \Rp$, $a > 0$).
\end{lem}
\proof We deal with the case of $\Lform_\alpha (X, d)$.  The case of
$\Lform_\alpha^a (X, d)$ is entirely similar.

Let us start with the following observation.  The projection map
$\pi_x \colon L_\alpha (X, d) \to \creal$ that maps $f$ to $f (x)$ is
patch-continuous, that is, continuous from $L_\alpha^\patch (X, d)$ to
$\creal^\patch$, where $L_\alpha (X, d)$ has the topology of pointwise
convergence.  (Beware: $L_\alpha$, not $\Lform_\alpha$.)  The reason
was given before Lemma~\ref{lemma:cont:Lalpha:patchclos}: the map
$f \mapsto f (x)$ is continuous from $(\creal^\patch)^X$ to
$\creal^\patch$, and therefore restricts to a continuous map on
$L_\alpha^\patch (X, d)$, which is a subspace of $(\creal^\patch)^X$,
since $L_\alpha (X, d)$ is patch-closed in the product space
$\creal^X$.

Since $[b, +\infty]$ is closed in $\creal^\patch$,
$\pi_x^{-1} ([b, +\infty])$ is also closed in
$L_\alpha (X, d)^\patch$.  It is clearly upwards-closed, and in any
stably compact space $Y$, the closed upwards-closed subsets of
$Y^\patch$ are the compact saturated subsets of $Y$ (see
\cite[Proposition~9.1.20]{JGL-topology}): so
$\pi_x^{-1} ([b, +\infty])$ is compact saturated in $L_\alpha (X, d)$.

Note that $\pi_x^{-1} ([b, +\infty])$ is the set of all
$\alpha$-Lipschitz (not necessarily $\alpha$-Lipschitz continuous)
maps $f$ such that $f (x) \geq b$.  We claim that its image by
$\rho_\alpha$ is exactly $[x \geq b]$.  This will prove that the
latter is compact as well.  The fact that it is upwards-closed
(saturated) is obvious.

For every $f \in \pi_x^{-1} ([b, +\infty])$, $\rho_\alpha (f) (x)$ is
clearly less than or equal to $f (x)$.  Because $x$ is a center point,
for all $y$, $r$, and $s$, $(x, r) \ll (y, s)$ is equivalent to
$d (x, y) < r-s$; in particular, $(x, \epsilon) \ll (x, 0)$.
Therefore $\rho_\alpha (f) (x) \geq f (x) - \epsilon$.  Since
$\epsilon$ is arbitrary, $\rho_\alpha (f) (x) = f (x)$, and since
$f (x) \geq b$, $\rho_\alpha (f) (x) \geq b$.  We have proved that
$\rho_\alpha (f)$ is in $[x \geq b]$.

Conversely, every $f \in [x \geq b]$ is of the form $\rho_\alpha (g)$ for
some $g \in \pi_x^{-1} ([b, +\infty])$, namely $g=f$, because $\rho_\alpha$
is a retraction.  Hence $[x \geq b]$ is the image of the compact set
$\pi_x^{-1} ([b, +\infty])$, as claimed.

Our whole argument works provided all considered spaces of functions
have the topology of pointwise convergence.  We use
Proposition~\ref{prop:cont:Lalpha:pw} to conclude.  \qed

\begin{rem}
  \label{rem:cont:Lalpha:compact}
  Lemma~\ref{lemma:cont:Lalpha:compact} does not hold for non-center
  points $x$, as we now illustrate.  Let $X = \nat_\omega$ be the dcpo
  obtained by adding a top element $\omega$ to $\nat$, with its usual
  ordering $\leq$.  We consider it as a quasi-metric space with
  the quasi-metric $d_\leq$.  On formal balls,
  $(x, r) \leq^{d_\leq} (y, s)$ if and only if $x \leq y$ and
  $r \geq s$, hence its space of formal balls is isomorphic to the
  continuous dcpo $\nat_\omega \times ]-\infty, 0]$, through the map
  $(x, r) \mapsto (x, -r)$.  We consider the point $x = \omega$, and
  $b=1$, say.  The elements of $\Lform_\alpha (X, d)$ are exactly the
  Scott-continuous maps, so $[x \geq b]$ is the set of
  Scott-continuous maps $f \colon \nat_\omega \to \creal$ such that
  $f (\omega) \geq 1$.  Let us pick a fixed element $f$ of
  $[x \geq b]$, for example, the constant function with value
  $+\infty$.  For every $n \in \nat$, let $f_n \colon X \to \creal$ be
  defined by $f_n (m) = 0$ if $m \leq n$, $f_n (m) = f (m)$ otherwise.
  The family ${(f_n)}_{n \in \nat}$ is a decreasing sequence of
  elements of $[x \geq b]$.  If $[x \geq b]$ were compact, then the
  intersection of the decreasing family of closed sets $\dc f_n$ would
  intersect $[x \geq b]$; but the only element in that intersection is
  the constant zero map.
\end{rem}

\begin{cor}
  \label{corl:cont:Lalpha:open}
  Let $X, d$ be a continuous Yoneda-complete quasi-metric space,
  $\alpha \in \Rp$, and $a \in \Rp$, $a > 0$.
  \begin{enumerate}
  \item For every center point $x$ in $X$ and every $b \in \Rp$, the
    set $[x < b]$ of all $f \in \Lform_\alpha (X, d)$ (resp.,
    $f \in \Lform_\alpha^a (X, d)$) such that $f (x) < b$ is open and
    downwards-closed in $\Lform_\alpha (X, d)^\patch$ (resp., in
    $\Lform_\alpha^a (X, d)^\patch$).
  \item For every center point $x \in X$, the map $f \mapsto f (x)$ is
    continuous from $\Lform_\alpha (X, d)^\dG$ to $(\creal)^\dG$, and
    from $\Lform_\alpha^a (X, d)^\dG$ to $[0, \alpha a]^\dG$.
  \end{enumerate}
\end{cor}
\proof (1) $[x < b]$ is the complement of $[x \geq b]$.  Since
$[x \geq b]$ is compact saturated
(Lemma~\ref{lemma:cont:Lalpha:compact}), it is closed and
upwards-closed in $\Lform_\alpha (X, d)^\patch$ (resp., in
$\Lform_\alpha^a (X, d)^\patch$).  (2) is just a reformulation of
(1). \qed

\begin{cor}
  \label{corl:cont:Lalpha:simple}
  Let $X, d$ be a continuous Yoneda-complete quasi-metric space,
  $\alpha \in \Rp$, and $a \in \Rp$, $a > 0$.  For every $n \in \nat$,
  for all $a_1, \ldots, a_n \in \Rp$ and every $n$-tuple of center
  points $x_1$, \ldots, $x_n$ in $X, d$, the maps:
  \begin{enumerate}
  \item $f \mapsto \sum_{i=1}^n a_i f (x_i)$,
  \item $f \mapsto \max_{i=1}^n f (x_i)$,
  \item $f \mapsto \min_{i=1}^n f (x_i)$
  \end{enumerate}
  are continuous from $\Lform_\alpha (X, d)^\dG$ (resp.,
  $\Lform_\alpha^a (X, d)^\dG$) to $(\creal)^\dG$.
\end{cor}
\proof
By Corollary~\ref{corl:cont:Lalpha:open}, using the fact that scalar
multiplication and addition are continuous on $(\creal)^\dG$.  The
latter means that those operations are monotonic and preserve filtered
infima.  \qed

The same argument also shows the following, and an endless variety of
similar results.
\begin{cor}
  \label{corl:cont:Lalpha:simple:2}
  Let $X, d$ be a continuous Yoneda-complete quasi-metric space,
  $\alpha \in \Rp$, and $a \in \Rp$, $a > 0$.  For every family of
  non-negative reals $a_{ij}$ and of center points $x_{ij}$,
  $1\leq i \leq m$, $1\leq j\leq n_i$, the maps:
  \begin{enumerate}
  \item $f \mapsto \min_{i=1}^m \sum_{j=1}^{n_i} a_{ij} f (x_{ij})$,
  \item $f \mapsto \max_{i=1}^m \sum_{j=1}^{n_i} a_{ij} f (x_{ij})$,
  \end{enumerate}
  are continuous from $\Lform_\alpha (X, d)^\dG$ (resp.,
  $\Lform_\alpha^a (X, d)^\dG$) to $(\creal)^\dG$.
\end{cor}



\section{Topologies on $\Lform_\infty (X, d)$ Determined by
  $\Lform_\alpha (X, d)$, $\alpha > 0$}
\label{sec:topol-lform_-x-1}

Recall that
$\Lform_\infty (X, d) = \bigcup_{\alpha \in \Rp} \Lform_\alpha (X,
d)$.  As usual, we equip $\Lform_\infty (X, d)$ with the subspace
topology from $\Lform X$, the latter having the Scott topology.

A nasty subtle issue we will have to face is the following.  Imagine
we have a function $F \colon \Lform_\infty (X, d) \to \creal$, and we
wish to show that $F$ is continuous.  One might think of showing that
the restriction of $F$ to $\Lform_\alpha (X, d)$ is continuous to that
end.  This is certainly necessary, but by no means sufficient.

For sufficiency, we would need the topology of $\Lform_\infty (X, d)$
to be \emph{coherent with}, or \emph{determined by}, the topologies of
its subspaces $\Lform_\alpha (X, d)$, $\alpha \in \Rp$.  The name
``coherent with'' is unfortunate, for a coherent space also means a
space in which the intersection of any two compact saturated subsets
is compact, as we have already mentioned.  We shall therefore prefer
``determined by''.

Let us recall some basic facts about topologies determined by families
of subspaces.  Fix a topological space $Y$, and a chain of subspaces
$Y_i$, $i \in I$, where $I$ is equipped with some ordering $\leq$, in
such a way that $i \leq j$ implies $Y_i \subseteq Y_j$.  Let also
$e_{ij} \colon Y_i \to Y_j$ and $e_i \colon Y_i \to Y$ be the
canonical embeddings, and assume that $Y = \bigcup_{i \in I} Y_i$.
Then there is a unique topology $\Open_c$ on $Y$ that is determined by
the topologies of the subspaces $Y_i$, $i \in I$.

We call it the \emph{determined topology} on $Y$.
Standard names for such a topology include the \emph{weak topology} or
the \emph{inductive topology}.
Its open subsets are the subsets $A$ of $Y$ such that $A \cap Y_i$ is
open in $Y_i$ for every $i \in I$.  In particular, every open subset
of $Y$, with its original topology, is open in $\Open_c$.
Categorically, $Y$ with the $\Open_c$ topology, together with the map
$e_i \colon Y_i \to Y$, is the inductive limit, a.k.a., the colimit,
of the diagram given by the arrows
$\xymatrix{Y_i \ar[r]^{e_{ij}} & Y_j}$, $i \leq j$.

It is known that $Y$ is automatically determined by the topologies of
$Y_i$, $i \in I$, in the following cases: when every $Y_i$ is open;
when every $Y_i$ is closed and the cover ${(Y_i)}_{i \in I}$ is
locally finite; when $Y$ has the discrete topology.  The case of
$\Lform_\infty (X, d)$ and its subspaces $\Lform_\alpha (X, d)$ falls
into none of those subcases.

Here is another case of a determined topology.
\begin{prop}
  \label{prop:Y:det:ep}
  Let ${(Y_i)}_{i \in I, \sqsubseteq}$ be a monotone net of subspaces of a
  topological space $Y$, forming a cover of $Y$.  Assume there are
  continuous projections of $Y$ onto $Y_i$, namely continuous maps
  $p_i \colon Y \to Y_i$ such that $p_i (x) \leq x$ for every
  $x \in Y$ and $p_i (x)=x$ for every $x \in Y_i$, for every
  $i \in I$.  Then the topology of $Y$ is determined by the topologies
  of $Y_i$, $i \in I$.
\end{prop}
What the assumption means is that we do not just have arrows
$\xymatrix{Y_i \ar[r]^{e_i} & Y}$, $i \in I$, but embedding-projection
pairs $\xymatrix{Y_i \ar@<1ex>[r]^{e_i} & Y \ar@<1ex>[l]^{p_i}}$,
$i \in I$.  This is a standard situation in domain theory.  The
inequality $e_i \circ p_i \leq \identity Y$ should be read by keeping
in mind the specialization preorder of $Y$.

\proof
Let $A$ be an open subset in the topology on $Y$ determined by the
subspaces $Y_i$.  For every $i \in I$, $A \cap Y_i$ is open in $Y_i$,
and since $p_i$ is continuous, $U_i = p_i^{-1} (A \cap Y_i)$ is open
in $Y$.

We claim that $A = \bigcup_{i \in I} U_i$, which will show that $A$ is
open, allowing us to conclude.

For every $x \in A$, we find some $i \in I$ such that $x$ is in $Y_i$,
using the fact that ${(Y_j)}_{j \in I, \sqsubseteq}$ is a cover.
Then $x$ is in $A \cap Y_i$.  Since $p_i$ is a projection, $p_i (x) =
x$, so $x$ is in $p_i^{-1} (A \cap Y_i) = U_i$.

Conversely, for every element $x$ of any $U_i$, $i \in I$, $p_i (x)$
is in the open subset $A \cap Y_i$ of $Y_i$.  Since
${(Y_j)}_{j \in I, \sqsubseteq}$ is a cover, there is a $j \in I$ such
that $x \in Y_j$, and since that is also a monotone net, we may assume
that $i \sqsubseteq j$.  Hence $Y_i \subseteq Y_j$, so that $p_i (x)$
is also in $A \cap Y_j$.  Moreover, $p_i (x) \leq x$.  Since the
specialization preordering on a subspace (here, $Y_j$) coincides with
the restriction of the specialization preordering on the superspace
(here, $Y$), and $A \cap Y_j$ is open hence upwards-closed in $Y_j$,
$x$ is in $A \cap Y_j$.  In particular, $x$ is in $A$.  \qed

Using the fact that $\Lform_\infty (X, d)$ has the subspace topology
from $\Lform X$, Corollary~\ref{corl:Lalpha:retract} implies that
there is an embedding-projection pair $i_\alpha, r_\alpha$ between
$\Lform_\alpha (X, d)$ and $\Lform_\infty (X, d)$.  When $\alpha > 0$
varies, the spaces $\Lform_\alpha (X, d)$ form a cover of
$\Lform_\infty (X, d)$.  Hence we can apply
Proposition~\ref{prop:Y:det:ep}, and we obtain:
\begin{prop}
  \label{prop:Linfty:determined}
  Let $X, d$ be a Lipschitz regular standard quasi-metric space.  The
  topology of $\Lform_\infty (X, d)$ is determined by those of the
  subspaces $\Lform_\alpha (X, d)$, $\alpha > 0$.  \qed
\end{prop}

We have a similar result for bounded maps, using
Corollary~\ref{corl:Lalpha:retract:bnd} instead of
Corollary~\ref{corl:Lalpha:retract}, and relying on
Lemma~\ref{lemma:Linfa} to ensure that the spaces
$\Lform_\alpha^a (X, d)$, $\alpha > 0$, form a cover of
$\Lform_\infty^\bnd (X, d)$.
\begin{prop}
  \label{prop:Linfty:determined:bnd}
  Let $X, d$ be a Lipschitz regular standard quasi-metric space.  Fix
  $a > 0$.  Then the topology of $\Lform_\infty^\bnd (X, d)$ is
  determined by those of the subspaces $\Lform_\alpha^a (X, d)$,
  $\alpha > 0$.  \qed
\end{prop}

\section{The Kantorovich-Rubinshte\u\i n-Hutchinson Quasi-Metric on Previsions}
\label{sec:quasi-metr-prev}

At last we come to our main object of study.  This was introduced in
\cite{Gou-csl07}, and is a common generalization of continuous
valuations (a concept close to that of measure), spaces of closed,
resp., compact saturated sets, resp.\ of lenses, and more.
\begin{defi}[Prevision]
  \label{defn:prev}
  A \emph{prevision} on a topological space $X$ is a Scott-continuous
  map $F$ from $\Lform X$ to $\creal$ that is \emph{positively
    homogeneous}, namely: $F (\alpha h) = \alpha F (h)$ for all
  $\alpha \in \Rp$, $h \in \Lform X$.

  It is:
  \begin{itemize}
  \item \emph{sublinear} if $F (h+h') \leq F (h) + F (h')$ holds for
    all $h, h' \in \Lform X$,
  \item \emph{superlinear} if $F (h+h') \geq F (h) + F (h')$ holds,
  \item \emph{linear} if it is both sublinear and superlinear,
  \item \emph{subnormalized} if
    $F (\alpha.\mathbf 1 + h) \leq \alpha + F (h)$ holds, where
    $\mathbf 1$ is the constant $1$ map,
  \item \emph{normalized} if
    $F (\alpha.\mathbf 1 + h) = \alpha + F (h)$ holds,
  \item \emph{discrete} if $F (f \circ h) = f (F (h))$ for every
    $h \in \Lform X$ and every strict $f \in \Lform \creal$---a map
    $f$ is strict if and only if $f (0) = 0$.
  \end{itemize}
\end{defi}
We write $\Prev X$ for the set of previsions on $X$.

Let now $X, d$ be a quasi-metric space, with its $d$-Scott topology.

\subsection{The Unbounded Kantorovich-Rubinshte\u\i n-Hutchinson
  Quasi-Metric}
\label{sec:unbo-kant-rubinsht}

\begin{defi}[$\dKRH$]
  \label{defn:KRH}
  Let $X, d$ be a quasi-metric space.  The
  \emph{Kantorovich-Rubinshte\u\i n-Hutchinson quasi-metric} on the
  space of previsions on $X$ is defined by:
  \begin{equation}
    \label{eq:dKRH}
    \dKRH (F, F') = \sup_{h \in \Lform_1 X} \dreal (F (h), F' (h)).
  \end{equation}
\end{defi}
This is a quasi-metric on any standard quasi-metric space, as remarked
at the end of Section~6 of \cite{JGL:formalballs}.  We give a slightly
expanded argument in Lemma~\ref{lemma:KRH:qmet} below.

\begin{rem}
  \label{rem:dKRH:classical}
  The name of that quasi-metric stems from analogous definitions of
  metrics on spaces of measures.  The \emph{classical}
  Kantorovich-Rubinshte\u\i n-Hutchinson metric between two bounded
  measures $\mu$ and $\mu'$ is given as
  $\sup_{h \in L_1 X, h \leq \mathbf 1} |\int_{x \in X} h (x) d\mu -
  \int_{x \in X} h (x) d\mu'|$.  Writing $F$ for the functional
  $h \mapsto \int_{x \in X} h (x) d\mu$ and $F'$ for
  $h \mapsto \int_{x \in X} h (x) d\mu'$, this can be rewritten as
  $\sup_{h \in L_1 X, h \leq \mathbf 1} \dreal^{sym} (F (h), F' (h))$.
  In the case of metric (as opposed to quasi-metric) spaces, we shall
  see in Lemma~\ref{lemma:dKRH:metric} that the use of the symmetrized
  metric $\dreal^{sym}$ is irrelevant, and the distance between $\mu$
  and $\mu'$ is equal to
  $\sup_{h \in L_1 X, h \leq \mathbf 1} \dreal (F (h), F' (h))$.
  Additionally, on a metric space, all $1$-Lipschitz maps $h$ are
  automatically continuous.  In the end, on metric spaces, the only
  difference between our definition and the latter is that $h$ ranges
  over functions bounded above by $1$ in the latter.  We shall also
  explore this variant in Section~\ref{sec:bound-kant-rubinsht}.
\end{rem}

The $\dKRH$ quasi-metric is interesting mostly on subspaces of
subnormalized, resp.\ normalized previsions, as we shall illustrate in
Remark~\ref{rem:KRH:useless} in the case of linear previsions.

\begin{lem}
  \label{lemma:KRH:qmet}
  Let $X, d$ be a standard quasi-metric space.  For all previsions
  $F$, $F'$ on $X$, the following are equivalent: $(a)$ $F \leq F'$;
  $(b)$ $\dKRH (F, F') = 0$.  In particular, $\dKRH$ is a
  quasi-metric.
\end{lem}
\proof $(a) \limp (b)$.  For every $h \in \Lform_1 (X, d)$,
$F (h) \leq F' (h)$, so $\dreal (F (h), F' (h)) = 0$.
$(b) \limp (a)$.  Let $f$ be an arbitrary element of $\Lform X$.  For
every $\alpha > 0$, $1/\alpha f^{(\alpha)}$ is in $\Lform_1 (X, d)$,
so $F (1/\alpha f^{(\alpha)}) \leq F' (1/\alpha f^{(\alpha)})$.
Multiply by $\alpha$: $F (f^{(\alpha)}) \leq F' (f^{(\alpha)})$.  The
family ${(f^{(\alpha)})}_{\alpha > 0}$ is a chain whose supremum is
$f$.  Using the fact that $F$ and $F'$ are Scott-continuous,
$F (f) \leq F' (f)$, and as $f$ is arbitrary, $(a)$ follows.  \qed

The following shows that we can restrict to bounded maps $h \in
\Lform_1 (X, d)$.  A map $h$ is \emph{bounded} if and only if there is
a constant $a \in \Rp$ such that for every $x \in X$, $h (x) \leq a$.
\begin{lem}
  \label{lemma:KRH:bounded}
  Let $X, d$ be a quasi-metric space.  For all previsions $F$, $F'$ on
  $X$,
  \[
    \dKRH (F, F') = \sup_{h \text{ bounded }\in \Lform_1 X} \dreal (F (h), F' (h)).
  \]
\end{lem}
\proof For every $h \in \Lform_1 (X, d)$, $h$ is the pointwise
supremum of the chain ${(\min (h, a))}_{a \in \Rp}$, where
$\min (h, a) \colon x \mapsto \min (h (x), a)$.  For the final claim,
it suffices to show that for every $r \in \Rp$ such that
$r < \dKRH (F, F')$, there is a bounded map $h' \in \Lform_1 (X, d)$
such that $r < \dreal (F (h'), F' (h'))$.  Since $r < \dKRH (F, F')$,
there is an $h \in \Lform_1 (X, d)$ such that
$r < \dreal (F (h), F' (h))$.  If $F (h) < +\infty$, this implies that
$F (h) > F' (h) + r$.  Since $F$ is Scott-continuous, there is an
$a \in \Rp$ such that
$F (\min (h, a)) > F' (h) + r \geq F' (\min (h, a)) + r$, so we can
take $h' = \min (h, a)$.  If $F (h) = +\infty$,
then note that since
$0 \leq r < \dreal (F (h), F' (h))$, we must have
$F' (h) < +\infty$.  By Scott-continuity again,
there is an $a \in \Rp$ such that $F (\min (h, a)) > r + F' (h)$, and
then we can again take $h' = \min (h, a)$.  \qed


\begin{lem}
  \label{lemma:LPrev:ord}
  Let $X, d$ be a standard quasi-metric space, let $F$, $F'$ be two
  previsions on $X$, and $r$, $r'$ be two elements of $\Rp$.  The
  following are equivalent:
  \begin{enumerate}
  \item $(F, r) \leq^{\dKRH^+} (F', r')$;
  \item $r \geq r'$ and, for every $h \in \Lform_1 (X, d)$,
    $F (h) - r \leq F' (h) - r'$;
  \item $r \geq r'$ and, for every $h \in \Lform_\alpha (X, d)$,
    $\alpha > 0$, $F (h) - \alpha r \leq F' (h) - \alpha r'$.
  \end{enumerate}
\end{lem}
\proof (1) $\limp$ (2).  If $(F, r) \leq^{\dKRH^+} (F', r')$, then for
every $h \in \Lform_1 (X, d)$, $\dreal (F (h), F' (h)) \leq r-r'$.
This implies that $r \geq r'$, on the one hand, and on the other hand
that either $F (h) = F' (h) = +\infty$ or $F (h) \neq +\infty$ and
$F (h) - F' (h) \leq r-r'$; in both cases,
$F (h) - r \leq F' (h) - r'$.

(2) $\limp$ (3).  For every $h \in \Lform_\alpha (X, d)$, $\alpha >
0$, $1/\alpha h$ is in $\Lform_1 (X, d)$, so (2) implies $F (1/\alpha
h) - r \leq F' (1/\alpha h) - r'$.  Multiplying by $\alpha$, and using
positive homogeneity, we obtain (3).

(3) $\limp$ (1).  For $\alpha=1$, we obtain that for every
$h \in \Lform_1 (X, d)$, $F (h) - r \leq F' (h) - r'$.  If
$F' (h) = +\infty$, then $\dreal (F (h), F' (h)) = 0 \leq r-r'$.  If
$F' (h) \neq +\infty$, then $F (h)$ cannot be equal to $+\infty$, so
$\dreal (F (h), F' (h)) = \max (F (h) - F' (h), 0) \leq \max (r-r', 0)
= r-r'$.  \qed

Recall that $\Lform_\infty X$ has the subspace topology from
$\Lform X$.  Let us introduce the following variant on the notion of
prevision.  We do this, because our completeness theorem will
naturally produce $\Lform$-previsions, not previsions, as
$\dKRH$-limits.  Showing that every $\Lform$-prevision defines a
unique prevision will be the subject of Proposition~\ref{prop:Gbar}.
\begin{defi}[$\Lform$-prevision]
  \label{defn:Lprev}
  Let $X, d$ be a quasi-metric space.  An \emph{$\Lform$-prevision} on
  $X$
  is
  any continuous map $G$ from $\Lform_\infty (X, d)$ to $\creal$ such
  that $G (\alpha h) = \alpha G (h)$ for all $\alpha \in \Rp$ and
  $h \in \Lform_\infty (X, d)$.
\end{defi}
The notions of sublinearity, superlinearity, linearity, normalization,
subnormalization, discreteness, carry over to $\Lform$-previsions, taking
care to quantify over $h, h' \in \Lform_\infty X$ and over $f$ strict
in $\Lform_\infty \creal$.  We write $\Lform\Prev X$ for the set of all
$\Lform$-previsions on $X$, and equip it with a quasi-metric defined by the
same formula as Definition~\ref{defn:KRH}, and which we denote by
$\dKRH$ again.

Every $F \in \Prev X$ defines an element $F_{|\Lform_\infty X}$ of
$\Lform\Prev X$ by restriction.  Conversely, for every $G \in \Lform\Prev X$,
let $\overline G (h) = \sup_{\alpha \in \Rp} G (h^{(\alpha)})$.
\begin{prop}
  \label{prop:Gbar}
  Let $X, d$ be a standard quasi-metric space.
  \begin{enumerate}
  \item For every $G \in \Lform\Prev X$, $\overline G$ is a prevision.
  \item The maps $G \in \Lform\Prev X \mapsto \overline G \in \Prev X$  and
    $F \in \Prev X \mapsto F_{|\Lform_\infty X} \in \Lform\Prev X$ are
    inverse of each other.
  \item If $G$ is sublinear, resp.\ superlinear, resp.\ linear, resp.\
    subnormalized, resp.\ normalized, resp.\ discrete, then so is
    $\overline G$.
  \item Conversely, if $\overline G$ is sublinear, resp.\ superlinear,
    resp.\ linear, resp.\ subnormalized, resp.\ normalized, resp.\
    discrete, then so is $G$.
  \end{enumerate}
\end{prop}
\proof
%
(1) We must first show that $\overline G$ is Scott-continuous, i.e.,
continuous from $\Lform X$ to $\creal$, both with their Scott
topologies.  Let $h \in \Lform X$ such that $\overline G (h) > a$.
For some $\alpha \in \Rp$, $G (h^{(\alpha)}) > a$, so $h^{(\alpha)}$
is in the open subset $G^{-1} (]a, +\infty])$ of $\Lform_\infty X$.
By definition of a subspace topology, there is a (Scott-)open subset
$W$ of $\Lform X$ such that
$G^{-1} (]a, +\infty]) = W \cap \Lform_\infty X$.  Then $h^{(\alpha)}$
is in $W$, and since $h \geq h^{(\alpha)}$, $h$ is also in $W$.
Moreover, for every $g \in W$, $g^{(\beta)}$ is in $W$ for some
$\beta \in \Rp$, hence in $G^{-1} (]a, +\infty])$.  It follows that
$G (g^{(\beta)}) > a$, and therefore $\overline G (g) > a$: hence $W$
is an open neighborhood of $h$ contained in
${\overline G}^{-1} (]a, +\infty])$.  We conclude that the latter is
open in $\Lform X$, which implies that $\overline G$ is continuous.

Before we proceed, we note the following.
\begin{lem}
  \label{lemma:ah:(alpha)}
  Let $X, d$ be a standard quasi-metric space.  For every map
  $h \colon X \to \creal$, for all $a, \alpha \in \Rp$,
  $a h^{(\alpha)} = {(ah)}^{(a\alpha)}$.
\end{lem}
\proof We have $a h^{(\alpha)} \leq {(ah)}^{(a\alpha)}$, because
$a h^{(\alpha)}$ is $a\alpha$-Lipschitz continuous
by Proposition~\ref{prop:alphaLip:props}~(1), below $ah$, and because
${(ah)}^{(a\alpha)}$ is the largest $a\alpha$-Lipschitz continuous
function below $ah$.  When $a > 0$, by the same argument,
$1/a {(ah)}^{(a\alpha)} \leq h^{(\alpha)}$, so
$a h^{(\alpha)} = {(ah)}^{(a\alpha)}$.  The same equality holds,
trivially, when $a=0$.  \qed

We return to the proof of the theorem.  When $a > 0$,
$\overline G (ah) = \sup_{\beta \in \Rp} G ({(ah)}^{(\beta)}) =
\sup_{\alpha \in \Rp} G ({(ah)}^{(a\alpha)}) = \sup_{\alpha \in \Rp} G
(a h^{(\alpha)})$ (by Lemma~\ref{lemma:ah:(alpha)})
$= a \overline G (h)$.  When $a=0$, $\overline G (0) = 0$.  Therefore
$\overline G$ is a prevision.

(2) For every $G \in \Lform\Prev X$, the restriction of $\overline G$ to
$\Lform_\infty X$ maps every Lipschitz (say, $\alpha$-Lipschitz)
continuous map $g \colon X \to \creal$ to
$\sup_{\beta \in \Rp} G (g^{(\beta)})$.  For every
$\beta \geq \alpha$, $g$ is $\beta$-Lipschitz continuous, and
$g^{(\beta)}$ is the largest $\beta$-Lipschitz continuous map below
$g$, so $g^{(\beta)} = g$.  It follows that
$\sup_{\beta \in \Rp} G (g^{(\beta)}) = G (g)$, showing that
$\overline G_{|\Lform_\infty X} = G$.

Note that this says that $G$ and $\overline G$ coincide on Lipschitz
continuous maps, a fact that we will use several times below.

In the reverse direction, for every $F \in \Prev X$,
$\overline {(F_{|\Lform_\infty X})}$ maps every function
$h \in \Lform X$ to $\sup_{\alpha \in \Rp} F (h^{(\alpha)}) = F$,
because $F$ is Scott-continuous and
$\sup_{\alpha \in \Rp} h^{(\alpha)} = h$ (see
\cite[Theorem~6.17]{JGL:formalballs}).

(3) For all $g, h \in \Lform X$, for all $\alpha, \beta \in \Rp$,
$g^{(\alpha)} + h^{(\beta)} \leq (g+h)^{(\alpha+\beta)}$, because the
left-hand side is an $(\alpha+\beta)$-Lipschitz continuous map below
$g+h$, using Proposition~\ref{prop:alphaLip:props}~(2), and the
right-hand side is the largest.  If $G$ is superlinear, it follows
that
$\overline G (g) + \overline G (h) = \sup_{\alpha, \beta \in \Rp} G
(g^{(\alpha)}) + G (h^{(\beta)}) \leq \sup_{\alpha, \beta \in \Rp} G
(g^{(\alpha)} + h^{(\beta)}) \leq \sup_{\alpha, \beta \in \Rp} G
((g+h)^{(\alpha+\beta)}) = \overline G (g+h)$, hence that
$\overline G$ is superlinear, too.

If $G$ is sublinear, then we need another argument.  We wish to show
that $\overline G (g+h) \leq \overline G (g) + \overline G (h)$.  To
this end, let $a$ be an arbitrary element of $\Rp$ such that
$a < \overline G (g+h)$.  Since $\overline G$ is Scott-continuous
(item~(1) above), $g+h$ is in the Scott-open set
${\overline G}^{-1} (]a, +\infty])$.  Now
$g+h = \sup_{\alpha \in \Rp} g^{(\alpha)} + \sup_{\alpha \in \Rp}
h^{(\alpha)} = \sup_{\alpha \in \Rp} (g^{(\alpha)} + h^{(\alpha)})$,
since addition is Scott-continuous on $\creal$, so
$g^{(\alpha)} + h^{(\alpha)}$ is in
${\overline G}^{-1} (]a, +\infty])$ for some $\alpha \in \Rp$.  The
function $g^{(\alpha)} + h^{(\alpha)}$ is $2\alpha$-Lipschitz
continuous (this is again Proposition~\ref{prop:alphaLip:props}~(2)).
Item~(2)
above shows that $\overline G$ coincides with $G$ on Lipschitz
continuous maps, so
$\overline G (g^{(\alpha)} + h^{(\alpha)}) = G (g^{(\alpha)} +
h^{(\alpha)})$.  Since $G$ is sublinear, the latter is less than or
equal to $G (g^{(\alpha)}) + G (h^{(\alpha)})$, so
$G (g^{(\alpha)}) + G (h^{(\alpha)}) > a$.  Since
$\overline G (g) \geq G (g^{(\alpha)})$, and similarly with $h$, we
obtain that $\overline G (g) + \overline G (h) > a$.  Since $a$ is
arbitrary, $\overline G (g) + \overline G (h) \geq \overline G (g+h)$.

If $G$ is subnormalized, then we show that $\overline G$ is
subnormalized, too, by a similar argument.  Let $\alpha \in \Rp$,
$h \in \Lform X$.  Fix $a \in \Rp$ such that
$a < \overline G (\alpha.\mathbf 1 + h)$.  Then
$\alpha . \mathbf 1 + h$ is in the Scott-open set
$\overline G^{-1} (]a, +\infty])$.  We observe that
$\alpha . \mathbf 1 + h$ is the pointwise supremum of the chain of
maps $\alpha . \mathbf 1 + h^{(\beta)}$, $\beta \in \Rp$, and that
those maps are $\beta$-Lipschitz continuous, by
Proposition~\ref{prop:alphaLip:props}~(2) and~(6).  Therefore
$\alpha . \mathbf 1 + h^{(\beta)}$ is also in
$\overline G^{-1} (]a, +\infty])$, for some $\beta \in \Rp$.  Since
that function is $\beta$-Lipschitz continuous, $\overline G$ maps it
to $G (\alpha . \mathbf 1 + h^{(\beta)})$, which is less than or equal
to $\alpha + G (h^{(\beta)})$ since $G$ is subnormalized.  In
particular,
$a < \overline G (\alpha . \mathbf 1 + h^{(\beta)}) \leq \alpha + G
(h)$.  Taking suprema over $a$ proves the claim.

If $G$ is normalized, it remains to show that
$\overline G (\alpha . \mathbf 1 + h) \geq \alpha + G (h)$.  This is
similar to the argument for the preservation of superlinearity.  We
use the fact that
$\alpha . \mathbf 1 + h^{(\beta)} \leq {(\alpha . \mathbf 1 +
  h)}^{(\beta)}$, which follows from the fact that the left-hand side
is $\beta$-Lipschitz continuous below $\alpha . \mathbf 1 + h$.  Then
$\overline G (\alpha . \mathbf 1 + h) = \sup_{\beta \in \Rp} G
({(\alpha . \mathbf 1 + h)}^{(\beta)}) \geq \sup_{\beta \in \Rp} G
(\alpha . \mathbf 1 + h^{(\beta)}) = \sup_{\beta \in \Rp} (\alpha + G
(h^{(\beta)})) = \alpha + \overline G (h)$.

If $G$ is discrete, then we need the following auxiliary lemma.
\begin{lem}
  \label{lemma:comp:alphaLip}
  Let $X, d$ be a standard quasi-metric space, $\alpha, \beta \in \Rp$.
  \begin{enumerate}
  \item for all $h \in \Lform_\alpha (X, d)$,
    $f \in \Lform_\beta (\creal, \dreal)$, $f \circ h$ is in
    $\Lform_{\alpha\beta} (X, d)$;
  \item for all $h \in \Lform X$, $f \in \Lform \creal$, $f^{(\alpha)}
    \circ h^{(\beta)} \leq (f \circ h)^{(\alpha\beta)}$.
  \end{enumerate}
\end{lem}
\proof (1) Direct consequence of Lemma~\ref{lemma:comp:Lip}.

(2) By (1), $f^{(\alpha)} \circ h^{(\beta)}$ is
$\alpha\beta$-Lipschitz continuous.  For every $x \in X$,
$f^{(\alpha)} (h^{(\beta)} (x)) \leq f (h^{(\beta)} (x)) \leq f (h
(x))$, so $f^{(\alpha)} \circ h^{(\beta)}$ is less than or equal to
$f \circ h$, hence also to the largest $\alpha\beta$-Lipschitz
continuous map below $f \circ h$, $(f \circ h)^{(\alpha\beta)}$.  \qed

Let us resume our proof of Proposition~\ref{prop:Gbar}~(3), where $G$
is discrete.  Fix $h \in \Lform X$, and let $f \in \Lform \creal$ be
strict.  Note that for every $\beta \in \Rp$, $f^{(\beta)}$ is strict,
too, since $0 \leq f^{(\beta)} \leq f$.

We have
$\overline G (f \circ h) = \sup_{\alpha \in \Rp} G ((f \circ
h)^{(\alpha)})$.  The map $f^{(\sqrt\alpha)} \circ h^{(\sqrt \alpha)}$
is $\alpha$-Lipschitz continuous and below $(f \circ h)^{(\alpha)}$ by
Lemma~\ref{lemma:comp:alphaLip}~(1, 2).  Therefore
$\overline G (f \circ h) \geq \sup_{\alpha \in \Rp} G
(f^{(\sqrt\alpha)} \circ h^{(\sqrt\alpha)})$.  For all $\beta$,
$\gamma$ in $\Rp$, $\beta$ and $\gamma$ are less than or equal to
$\sqrt \alpha$ for $\alpha$ sufficiently large, so
$\overline G (f \circ h) \geq \sup_{\beta, \gamma \in \Rp} G
(f^{(\beta)} \circ h^{(\gamma)}) = \sup_{\gamma\in \Rp} \sup_{\beta
  \in \Rp} f^{(\beta)} ( G (h^{(\gamma)}))$ (using the fact that $G$
is discrete)
$= \sup_{\gamma\in \Rp} f (G (h^{(\gamma)})) = f (\sup_{\gamma\in \Rp}
G (h^{(\gamma)}))$ (since $f$ is lower semicontinuous from $\creal$ to
itself, hence Scott-continuous) $= f (\overline G (h))$.

Conversely, $f \circ h$ is the supremum of the directed family
${(f^{(\beta)} \circ h^{(\gamma)})}_{\beta, \gamma \in \Rp}$.  This is
directed because both ${(f^{(\beta)})}_{\beta \in \Rp}$ and
${(h^{(\gamma)})}_{\gamma \in \Rp}$ are chains, and every
$f^{(\beta)}$ is monotonic (remember that
$\Lform_\beta (X, d) \subseteq \Lform X$ by
Lemma~\ref{lemma:Lalpha:LX}, and that every lower semicontinuous map
is monotonic).  Moreover, for every $x \in X$,
$\sup_{\beta, \gamma \in \Rp} f^{(\beta)} (h^{(\gamma)} (x)) =
\sup_{\gamma \in \Rp}\sup_{\beta \in \Rp} f^{(\beta)} (h^{(\gamma)}
(x)) = \sup_{\gamma \in \Rp} f (h^{(\gamma)} (x)) = f (\sup_{\gamma
  \in \Rp} h^{(\gamma)} (x)) = f (h (x))$.

Using that, we show that
$\overline G (f \circ h) \leq f (\overline G (h))$, from which the
equality will follow.  For every $a < \overline G (f \circ h)$, since
$\overline G$ is Scott-continuous by (1), and using the previous
observation, there are $\beta, \gamma \in \Rp$ such that
$a < \overline G (f^{(\beta)} \circ h^{(\gamma)})$.  Since
$f^{(\beta)} \circ h^{(\gamma)}$ is in $\Lform_{\beta\gamma} (X, d)$
(Lemma~\ref{lemma:comp:alphaLip}~(1)),
$\overline G (f^{(\beta)} \circ h^{(\gamma)}) = G (f^{(\beta)} \circ
h^{(\gamma)})$.  Since $G$ is discrete, the latter is equal to
$f^{(\beta)} (G (h^{(\gamma)}))$, which is less than or equal to
$f (G (h^{(\gamma)}))$, hence to $f (\overline G (h))$ (recall that
$f$ is lower semicontinuous hence monotonic).  Since $a$ is arbitrary,
$\overline G (f \circ h) \leq f (\overline G (h))$.

(4) is obvious, since by (2) $G$ is the restriction of $\overline G$
to $\Lform_\infty (X, d)$.  \qed

\begin{cor}
  \label{corl:=:L1}
  Let $X, d$ be a standard quasi-metric space, and $F$, $F'$ be two
  previsions on $X$.  Then $F=F'$ if and only if $F (h) = F' (h)$ for
  every $h \in \Lform_1 (X, d)$.
\end{cor}
\proof The only if direction is trivial.  In the if direction, for
every $\alpha > 0$, for every $h \in \Lform_\alpha (X, d)$,
$1/\alpha h$ is in $\Lform_1 (X, d)$ by
Proposition~\ref{prop:alphaLip:props}~(1).  We have
$F (1/\alpha h) = F' (1/\alpha h)$ by assumption.  Multiplying by
$\alpha$ and relying on positive homogeneity, $F (h) = F (h')$.  This
shows that $F_{|\Lform_\infty X} = F'_{|\Lform_\infty X}$, whence $F=F'$
by Proposition~\ref{prop:Gbar}~(2).  \qed

Recall that $\Lform_\infty^\bnd (X, d)$ is the space of all bounded
Lipschitz continuous maps from $X$ to $\creal$.

\begin{defi}[$\Lform^\bnd$-prevision]
  \label{defn:Lbnd:prev}
  Let $X, d$ be a quasi-metric space.  An \emph{$\Lform^\bnd$-prevision} on
  $X$ is any continuous map $H$ from $\Lform_\infty^\bnd X$ to $\creal$
  such that $H (\beta h) = \beta H (h)$ for all $\beta \in \Rp$ and
  $h \in \Lform_\infty^\bnd X$.
\end{defi}
Let $\Lform^\bnd\Prev X$ be the set of all $\Lform^\bnd$-previsions on
$X$.  We define linear, superlinear, sublinear, subnormalized, and
normalized $\Lform^\bnd$-previsions in the usual way.  Finally, we say
that $H$ is \emph{discrete} iff it satisfies
$H (f \circ h) = f (H (h))$ for every strict map
$f \in \Lform_\infty^\bnd \creal$ (not $\Lform_\infty \creal$) and
every $h \in \Lform_\infty^\bnd (X, d)$.

Every $G \in \Lform\Prev X$ defines an element $G_{|\Lform_\infty^\bnd X}$
of $\Lform^\bnd\Prev X$ by restriction.  Conversely, for every
$H \in \Lform^\bnd\Prev X$, let
$\overline{\overline H} (h) = \sup_{\beta > 0} H (\min (h, \beta))$
for every $h \in \Lform_\infty (X, d)$.

\begin{lem}
  \label{lemma:Hbar}
  Let $X, d$ be a standard quasi-metric space, and $a > 0$.
  \begin{enumerate}
  \item For every $H \in \Lform^\bnd\Prev X$, $\overline{\overline H}$ is an
    $\Lform$-prevision.
  \item The maps
    $G \in \Lform\Prev X \mapsto G_{|\Lform_\infty^\bnd X} \in \Lform^\bnd\Prev
    X$ and
    $H \in \Lform^\bnd\Prev X \mapsto \overline{\overline H} \in \Lform\Prev X$
    are inverse of each other.
  \item For every $G \in \Lform\Prev X$, $G$ is linear, resp.\ superlinear,
    resp.\ sublinear, resp.\ subnormalized, resp.\ normalized, resp.\
    discrete, if and only if $G_{|\Lform_\infty^\bnd X}$ is.
  \end{enumerate}
\end{lem}
\proof We first show that the map
$t_{\beta} \colon \Lform_\infty (X, d) \to \Lform_\infty^\bnd (X, d)$
defined by $t_{\beta} (h) = \allowbreak \min (h, \beta)$ is continuous
for every $\beta > 0$.  This follows from the definition of subspace
topologies, and the fact that $t_{\beta}$ is the restriction of the
Scott-continuous map
$h \in \Lform X \mapsto \min (h, \beta) \in \Lform X$.

(1) $\overline{\overline H}$ is continuous, since
$(\overline{\overline H})^{-1} (]r, +\infty]) = \{h \in \Lform_\infty (X, d)
\mid \exists \beta > 0 , H (\min (h, \beta)) > r\} = \bigcup_{\beta >
  0} \{h \in \Lform_\infty (X, d) \mid \min (h, \beta) \in H^{-1} (]r,
+\infty])\} = \bigcup_{\beta > 0} t_{\beta}^{-1} (H^{-1} (]r,
+\infty]))$, which is open since $H$ and $t_{\beta}$ are continuous.

We must show that
$\overline{\overline H} (\alpha h) = \alpha \overline{\overline H} (h)$ for every
$\alpha \in \Rp$.  When $\alpha > 0$,
$\overline{\overline H} (\alpha h) = \sup_{\beta > 0} H (\min (\alpha h,
\beta)) = \sup_{\beta' > 0} H (\min (\alpha h, \alpha \beta')) =
\sup_{\beta' > 0} \alpha H (\min (h, \beta')) = \alpha \overline{\overline H}
(h)$.  When $\alpha = 0$,
$\overline{\overline H} (0) = \sup_{\beta > 0} H (0) = 0$.

(2) For every $h \in \Lform_\infty^\bnd (X, d)$,
$(\overline{\overline H})_{|\Lform_\infty^\bnd X} (h) =
\overline{\overline H} (h) = \sup_{\beta > 0} H (\min (h, \beta)) = H
(h)$, since $\min (h, \beta) = h$ for $\beta$ large enough.  In the
other direction, for every $h \in \Lform_\infty (X, d)$,
$\overline {\overline {G_{|\Lform_\infty^\bnd X}}} (h) = \sup_{\beta >
  0} G_{|\Lform_\infty^\bnd X} (\min (h, \beta)) = \sup_{\beta > 0} G
(\min (h, \beta))$.  Recall now from Proposition~\ref{prop:Gbar} that
$G$ is the restriction to $\Lform_\infty X$ of a prevision $F$ on $X$,
which is Scott-continuous.  So
$\overline {\overline {G_{|\Lform_\infty^\bnd X}}} (h) = \sup_{\beta >
  0} F (\min (h, \beta)) = F (h) = G (h)$.

(3)
All the claims except the one on discreteness follow from the fact
that
$\min (g, \beta) + \min (h, \gamma) \leq \min (g+h, \beta + \gamma)$
and $\min (g+h, \beta) \leq \min (g, \beta) + \min (h, \beta)$ for all
maps $g$, $h$ and all $\beta, \gamma > 0$.  We turn to discreteness.
If
$G \in \Lform\Prev X$ is discrete, then clearly $G_{|\Lform_\infty^\bnd X}$
is discrete, too.  Conversely, let $H \in \Lform^\bnd\Prev X$ be discrete.
To show that $\overline{\overline H}$ is discrete, let
$f \in \Lform_\infty \creal$ be strict, $h \in \Lform_\infty (X, d)$,
and let us show that
$\overline{\overline H} (f \circ h) = f (\overline {\overline H}
(h))$.  The left-hand side, $\overline{\overline H} (f \circ h)$, is
equal to $\sup_{\beta > 0} H (\min (f \circ h, \beta))$, hence to
$\sup_{\beta > 0} H (f_\beta \circ h)$, where
$f_\beta = \min (f, \beta)$.  Since $f_\beta$ is Scott-continuous and
$H$ is the restriction of a Scott-continuous map
$F \colon \Lform (X, d) \to \creal$,
$\overline{\overline H} (f \circ h)$ is also equal to
$\sup_{\beta, \gamma > 0} H (f_\beta \circ \min (h, \gamma))$.  We
know that $\min (h, \gamma)$ is in $\Lform_\infty^\bnd (X, d)$, and
that$f_\beta$ is in $\Lform_\infty^\bnd \creal$.  Since $H$ is
discrete, $\overline{\overline H} (f \circ h)$ is therefore equal to
$\sup_{\beta, \gamma> 0} f_\beta (H (\min (h, \gamma))) = \sup_{\gamma
  > 0} f (H (\min (h, \gamma)))$.  Now $f$ is Scott-continuous, so
this is equal to
$f (\sup_{\gamma > 0} H (\min (h, \gamma))) = f (\overline {\overline
  H} (h))$.  \qed

\begin{cor}
  \label{corl:=:Lbnd1}
  Let $X, d$ be a standard quasi-metric space, and $F$, $F'$ be two
  previsions on $X$, and $a > 0$.  Then $F=F'$ if and only if
  $F (h) = F' (h)$ for every $h \in \Lform_1^a (X, d)$.
\end{cor}
\proof The only if direction is trivial.  In the if direction, for
every $\alpha > 0$, for every $h \in \Lform_\alpha^a (X, d)$,
$1/\alpha h$ is in $\Lform_1^a (X, d)$ by
Proposition~\ref{prop:alphaLip:props}~(1).  We have
$F (1/\alpha h) = F' (1/\alpha h)$ by assumption.  Multiplying by
$\alpha$ and relying on positive homogeneity, $F (h) = F (h')$.  Using
Lemma~\ref{lemma:Linfa}, this shows that
$F_{|\Lform_\infty^\bnd X} = F'_{|\Lform_\infty^\bnd X}$, whence
$F=F'$ by Lemma~\ref{lemma:Hbar}~(2) and
Proposition~\ref{prop:Gbar}~(2).  \qed


We can equip $\Lform\Prev X$ and $\Lform^\bnd\Prev X$ with a quasi-metric
defined by formula (\ref{eq:dKRH}), which we shall again denote by
$\dKRH$.  Recall that an \emph{isometry} is a map that preserves
distances on the nose.  That the maps
$F \in \Prev X \mapsto F_{|\Lform_\infty X} \in \Lform\Prev X$ and
$G \in \Lform\Prev X \mapsto G_{|\Lform_\infty^\bnd X} \in \Lform^\bnd\Prev X$
are isometries follows by definition.
\begin{lem}
  \label{lemma:dKRH:iso}
  Let $X, d$ be a quasi-metric space.  The maps:
  \[
    \xymatrix@R=0pt{
      F \ar@{|->}[r] & F_{|\Lform_\infty X}&& G \ar@{|->}[r] &
      G_{\Lform_\infty^a X} \\
      \Prev X & \Lform\Prev X & & \Lform\Prev X & \Lform^\bnd\Prev X \\
      \overline G & G \ar@{|->}[l] && \overline{\overline H} & H \ar@{|->}[l]
    }
  \]
  are isometries between $\Prev X, \dKRH$ and $\Lform\Prev X, \dKRH$, and
  between $\Lform\Prev X, \dKRH$ and $\Lform^\bnd\Prev X, \dKRH$ for every
  $a > 0$.  \qed
\end{lem}

If two maps $f \colon X, d \mapsto Y, \partial$ and
$g \colon Y, \partial \to X, d$ are mutually inverse isometries, then
they are not only $1$-Lipschitz but also $1$-Lipschitz continuous.
Indeed, $\mathbf B^1 (f)$ and $\mathbf B^1 (g)$ are
order-isomorphisms, and therefore are both Scott-continuous.  (This is
a standard exercise.)

It follows that, for the purpose of quasi-metrics, of the underlying
$d$-Scott topologies, and for the underlying specialization orderings,
$\Lform\Prev X$, $\Lform^\bnd\Prev X$, and $\Prev X$ can be regarded as the
same space under $\dKRH$.

The following proposition \emph{almost} shows that our spaces of
previsions are Yoneda-complete, and gives a simple formula for the
supremum $(G, r)$ of a directed family of formal balls
${(G_i, r_i)}_{i \in I}$.  I say ``almost'' because $G$ is not
guaranteed to be continuous.  We address that problem in the
subsequent Theorem~\ref{thm:LPrev:sup}.
\begin{prop}
  \label{prop:LPrev:simplesup}
  Let $X, d$ be a standard quasi-metric space, and
  ${(G_i, r_i)}_{i \in I}$ be a directed family in
  $\mathbf B (\Lform\Prev X, \dKRH)$.  Let $r = \inf_{i \in I} r_i$
  and, for each $h \in \Lform_\alpha (X, d)$ with $\alpha > 0$, define
  $G (h)$ as the directed supremum
  $\sup_{i \in I} (G_i (h) + \alpha r - \alpha r_i)$.  Then:
  \begin{enumerate}
  \item $\overline G$ is a well-defined, positively homogeneous function from
    $\Lform X$ to $\creal$;
  \item for any upper bound $(G', r')$ of ${(G_i, r_i)}_{i \in I}$,
    $r' \leq r$ and, for every $h \in \Lform_\alpha (X, d)$, $\alpha >
    0$, $G (h) \leq G' (h) + \alpha r - \alpha r'$;
  \item if $G$ is continuous, then $(G, r)$ is the supremum of
    ${(G_i, r_i)}_{i \in I}$;
  \item if every $G_i$ is sublinear, resp.\ superlinear, resp.\
    linear, resp.\ subnormalized, resp.\ normalized, resp.\ discrete,
    then so is $G$.
  \end{enumerate}
\end{prop}
We call $(G, r)$ the \emph{naive supremum} of
${(G_i, r_i)}_{i \in I}$.

\proof We check that
$\sup_{i \in I} (G_i (h) + \alpha r - \alpha r_i)$ is a directed
supremum, for every $h \in \Lform_\alpha (X, d)$.  Define
$\sqsubseteq$ by $i \sqsubseteq j$ if and only if
$(G_i, r_i) \leq^{\dKRH^+} (G_j, r_j)$.
Then $\sqsubseteq$ turns $I$ into
a directed preordered set, and the family ${(G_i, r_i)}_{i \in I}$
into a monotone net ${(G_i, r_i)}_{i \in I, \sqsubseteq}$.  It remains
to show that $i \sqsubseteq j$ implies
$G_i (h) + \alpha r - \alpha r_i \leq G_j (h) + \alpha r - \alpha
r_j$.  Since
$(G_i, r_i) \leq^{\dKRH^+} (G_j, r_j)$,
and $1/\alpha h$ is in $\Lform_1 (X, d)$,
$G_i (1/\alpha h) \leq G_j (1/\alpha h) + r_i - r_j$.  We obtain the
result by multiplying both sides by $\alpha$ and adding $\alpha r$.

(1) We must first show that $G$ is well-defined, in the following
sense.  When $h \in \Lform_\alpha X$, $h$ is also in $\Lform_\beta X$
for every $\beta \geq \alpha$ by
Proposition~\ref{prop:alphaLip:props}~(5), and our tentative
definition is not unique, apparently: we have defined $G (h)$ both as
$\sup_{i \in I} (G_i (h) + \alpha r - \alpha r_i)$ and as
$\sup_{i \in I} (G_i (h) + \beta r - \beta r_i)$.  This is not a
problem: the two suprema coincide, since
$G_i (h) + \alpha r - \alpha r_i$ and $G_i (h) + \beta r - \beta r_i$
only differ by $(\beta- \alpha) (r_i - r)$, which can be made
arbitrarily small as $i$ varies in $I$.  In fact, both ``definitions''
are equal to $\lim_{i \in I, \sqsubseteq} G_i (h)$, where the limit is
taken in $\creal$ with its usual Hausdorff topology, a base of which
is given by the intervals $[0, b[$, $]a, b[$ and $]a, +\infty]$,
$0 < a < b < +\infty$.  This remark will be helpful in the sequel.  In
particular, taking the definition
$G (h) = \lim_{i \in I, \sqsubseteq} G_i (h)$, it is easy to show that
$G$ commutes with products with non-negative constants, which finishes
to prove (1).

We proceed with (4), and we will return to (2) and (3) later.

(4) Using again the formula
$G (h) = \lim_{i \in I, \sqsubseteq} G_i (h)$, if every $G_i$ is
sublinear, resp.\ superlinear, resp.\ linear, resp.\ subnormalized,
resp.\ normalized, then so is $G$, because because $+$ and products by
scalars are continuous on $\creal$ with its usual Hausdorff topology.
The same argument works to show that $G$ is discrete when every $G_i$
is, because of the following lemma.

\begin{lem}
  \label{lemma:Lbeta:creal}
  Every function $f \in \Lform_\infty \creal$ is continuous from
  $\creal$ to $\creal$, with its usual Hausdorff topology.
\end{lem}
\proof Let $f \in \Lform_\beta \creal$, $\beta > 0$.  The Hausdorff
topology on $\creal$ is generated by the Scott-open subsets
$]a, +\infty]$ and the subsets of the form $[0, b[$.  Since
$\Lform_\beta \creal \subseteq \Lform \creal$ (owing to the fact that
$\creal, \dreal$ is standard), $f^{-1} (]a, +\infty])$ is Scott-open
hence open in $\creal$.  And $f^{-1} ([0, b[)$ is open because $f$ is
$\beta$-Lipschitz: for every $x \in f^{-1} ([0, b[)$, let
$\epsilon > 0$ be such that $f (x) + \epsilon < b$; for every
$x' \in [0, x + \epsilon/\beta[$,
$f (x') \leq f (x) + \beta \dreal (x', x) < f (x) + \epsilon < b$, so
$[0, x+\epsilon/\beta[$ is an open neighborhood of $x$ contained in
$f^{-1} ([0, b[)$.  \qed

We return to the proof of (2) and of (3).

(2) Let $(G', r')$ be any upper bound of ${(G_i, r_i)}_{i \in I}$.
For every $i \in I$, for every $h \in \Lform_1 (X, d)$ the inequality
$(G_i, r_i) \leq^{\dKRH^+} (G', r')$, equivalently
$\dKRH (G_i, G') \leq r_i-r'$, implies that
$G_i (h) \leq G' (h) + r_i - r'$ (and $r_i \geq r'$).  Then
$G_i (h) - r_i \leq G' (h) - r'$.  Taking suprema over $i \in I$,
$G (h) - r \leq G' (h) - r'$, hence $G (h) \leq G' (h) + r - r'$ (and,
by taking infima, $r = \inf_{i \in I} r_i \geq r'$).  Now take any
$h \in \Lform_\alpha (X, d)$, $\alpha > 0$.  Then $1/\alpha h$ is in
$\Lform_1 (X, d)$, and by positive homogeneity
$G (h) \leq G' (h) + \alpha r - \alpha r'$.

(3) Assume $G$ continuous.  The definition of $G$ ensures that
$(G, r)$ is an element of $\mathbf B (\Lform\Prev X, \dKRH)$, and is an
upper bound of ${(G_i, r_i)}_{i \in I}$.  For every upper bound
$(G', r')$ of ${(G_i, r_i)}_{i \in I}$, (2) entails that $r \geq r'$
and, for every $h \in \Lform_1 (X, d)$, $G (h) \leq G' (h) + r - r'$,
whence $\dKRH (G, G') \leq r-r'$.  It follows that
$(G, r) \leq^{\dKRH^+} (G', r')$, showing that $(G, r)$ is the least
upper bound.  \qed

\begin{thm}[Yoneda-completeness]
  \label{thm:LPrev:sup}
  Let $X, d$ be a standard quasi-metric space, and assume that the
  topology of $\Lform_\infty (X, d)$ is determined by those of
  $\Lform_\alpha (X, d)$, $\alpha > 0$ (e.g., if $X, d$ is standard
  and Lipschitz regular, see
  Proposition~\ref{prop:Linfty:determined}).

  Then $\Prev X, \dKRH$ is Yoneda-complete, and all suprema of
  directed families of formal balls of previsions are naive suprema,
  as described in Proposition~\ref{prop:LPrev:simplesup}.

  The same result holds for the subspace of previsions satisfying any
  given set of properties among: sublinearity, superlinearity,
  linearity, subnormalization, normalization, and discreteness.
\end{thm}
\proof Use the assumptions, notations and results of
Proposition~\ref{prop:LPrev:simplesup}.  It remains to show that the
naive supremum $G$, as defined there, \emph{is} continuous from
$\Lform_\infty (X, d)$ to $\creal$.  For that, by the assumption that
the topology of $\Lform_\infty (X, d)$ is determined, it suffices to
show that the restriction $G_{|\Lform_\alpha X}$ of $G$ is continuous
from $\Lform_\alpha (X, d)$ to $\creal$, for any $\alpha > 0$.  Fix
$a \in \real$.  We must show that
$\mathcal U = \{h \in \Lform_\alpha (X, d) \mid G (h) > a\}$ is open
in $\Lform_\alpha (X, d)$.  Using the definition of $G$, we write
$\mathcal U$ as the set of maps $h \in \Lform_\alpha (X, d)$ such that
$G_i (h) + \alpha r - \alpha r_i > a$ for some $i \in I$.  Therefore
$\mathcal U = \bigcup_{i \in I} G_{i|\Lform_\alpha X}^{-1} (]a +
\alpha r_i - \alpha r, +\infty])$, which is open.  \qed

\subsection{The Bounded Kantorovich-Rubinshte\u\i n-Hutchinson
  Quasi-Metrics}
\label{sec:bound-kant-rubinsht}

We shall sometimes consider the following bounded variant $\dKRH^a$ of
the $\dKRH$ quasi-metric.  Typically, we shall be able to show that,
in certain contexts, the weak topology coincides with the
$\dKRH^a$-Scott topology, but not with the $\dKRH$-Scott topology.

\begin{defi}[$\dKRH^a$]
  \label{defn:KRH:bounded}
  Let $X, d$ be a quasi-metric space, and $a \in \Rp$, $a > 0$.  The
  \emph{$a$-bounded Kantorovich-Rubinshte\u\i n-Hutchinson quasi-metric}
  on the space of previsions on $X$ is defined by:
  \begin{equation}
    \label{eq:dKRHa}
    \dKRH^a (F, F') = \sup_{h \in \Lform_1^a X} \dreal (F (h), F' (h)).
  \end{equation}
\end{defi}
The single point to pay attention to is that $h$ ranges over
$\Lform_1^a (X, d)$, not $\Lform_1 (X, d)$, in the definition of $\dKRH^a$.


\begin{lem}
  \label{lemma:KRHa:qmet}
  Let $X, d$ be a standard quasi-metric space, and $a \in \Rp$,
  $a > 0$.  For all previsions $F$, $F'$ on $X$, the following are
  equivalent: $(a)$ $F \leq F'$; $(b)$ $\dKRH^a (F, F') = 0$.  In
  particular, $\dKRH^a$ is a quasi-metric.
\end{lem}
\proof $(a) \limp (b)$.  For every $h \in \Lform_1^a (X, d)$,
$F (h) \leq F' (h)$, so $\dreal (F (h), F' (h)) = 0$.
$(b) \limp (a)$.  Let $f$ be an arbitrary element of $\Lform X$.  For
every $\alpha > 0$, $1/\alpha f^{(\alpha)}$ is in $\Lform_1 (X, d)$,
so $\min (1/\alpha f^{(\alpha)}, a)$ is also in $\Lform_1 (X, d)$, and
is below $a.\mathbf 1$.  By $(b)$,
$F (\min (1/\alpha f^{(\alpha)}, a)) \leq F' (\min (1/\alpha
f^{(\alpha)}, a))$.  Multiply by $\alpha$:
$F (\min (f^{(\alpha)}, \alpha a)) \leq F' (\min (f^{(\alpha)}, \alpha
a))$.  The family ${(f^{(\alpha)})}_{\alpha > 0}$ is a chain whose
supremum is $f$, so ${(\min (f^{(\alpha)}, \alpha a))}$ is also a
chain whose supremum is $f$.  Using the fact that $F$ and $F'$ are
Scott-continuous, $F (f) \leq F' (f)$, and as $f$ is arbitrary, $(a)$
follows.  \qed

The relation with the $\dKRH$ quasi-metric is as follows.
\begin{lem}
  \label{lemma:KRH:KRHa}
  Let $X, d$ be a quasi-metric space.  Order quasi-metrics on any
  space of previsions pointwise.  Then ${(\dKRH^a)}_{a \in \Rp, a >
    0}$ is a chain, and for all previsions $F$, $F'$, $\dKRH (F, F') =
  \sup_{a \in \Rp, a > 0} \dKRH^a (F, F')$.
\end{lem}
\proof Clearly, for $0 < a \leq a'$, $\Lform_1^a (X, d)$ is included
in $\Lform_1^{a'} (X, d)$, so
$\dKRH^a (F, F') \leq \dKRH^{a'} (F, F')$.  Therefore
${(\dKRH^a)}_{a \in \Rp, a > 0}$ is a chain.  Similarly,
$\dKRH^a (F, F') \leq \dKRH (F, F')$.  Finally,
$\sup_{a \in \Rp, a > 0} \dKRH^a (F, F')$ is equal to
$\sup_{a \in \Rp, a > 0, h \in \Lform_1^a (X, d)} \dreal (F (h),F'
(h))$, in other words to
$\sup_{h \text{ bounded }\in\Lform_1 (X, d)} \dreal (F (h),F' (h))$,
and that is equal to $\dKRH (F, F')$ by Lemma~\ref{lemma:KRH:bounded}.
\qed

\begin{lem}
  \label{lemma:LaPrev:ord}
  Let $X, d$ be a standard quasi-metric space, $a > 0$, let $F$, $F'$
  be two previsions on $X$, and $r$, $r'$ be two elements of $\Rp$.
  The following are equivalent:
  \begin{enumerate}
  \item $(F, r) \leq^{\dKRH^{a+}} (F', r')$;
  \item $r \geq r'$ and, for every $h \in \Lform_1^a (X, d)$,
    $F (h) - r \leq F' (h) - r'$;
  \item $r \geq r'$ and, for every $h \in \Lform_\alpha^a (X, d)$,
    $\alpha > 0$, $F (h) - \alpha r \leq F' (h) - \alpha r'$.
  \end{enumerate}
\end{lem}
\proof The proof is exactly as for Lemma~\ref{lemma:LPrev:ord}.  \qed



Equip $\Lform\Prev X$ and $\Lform^\bnd\Prev X$ with the quasi-metric defined by
formula (\ref{eq:dKRHa}), which we shall again denote by $\dKRH^a$.
The maps $F \in \Prev X \mapsto F_{|\Lform_\infty X} \in \Lform\Prev X$ and
$G \in \Lform\Prev X \mapsto G_{|\Lform_\infty^\bnd X} \in \Lform^\bnd\Prev X$
are isometries, by definition.
\begin{lem}
  \label{lemma:dKRHa:iso}
  Let $X, d$ be a quasi-metric space.  The maps:
  \[
    \xymatrix@R=0pt{
      F \ar@{|->}[r] & F_{|\Lform_\infty X}&& G \ar@{|->}[r] &
      G_{\Lform_\infty^\bnd X} \\
      \Prev X & \Lform\Prev X & & \Lform\Prev X & \Lform^\bnd\Prev X \\
      \overline G & G \ar@{|->}[l] && \overline{\overline H} & H \ar@{|->}[l]
    }
  \]
  are isometries between $\Prev X, \dKRH^a$ and $\Lform\Prev X, \dKRH^a$,
  and between $\Lform\Prev X, \dKRH^a$ and $\Lform^\bnd\Prev X, \dKRH^a$ for every
  $a > 0$.  \qed
\end{lem}
Hence we can regard $\Prev X$, $\Lform\Prev X$, and $\Lform^\bnd\Prev X$ as the
same quasi-metric spaces, when equipped with $\dKRH^a$.

\begin{prop}
  \label{prop:LaPrev:simplesup}
  Fix $a > 0$.  Let $X, d$ be a standard quasi-metric space, and
  ${(H_i, r_i)}_{i \in I}$ be a directed family in
  $\mathbf B (\Lform^\bnd\Prev X, \dKRH^a)$.  Let $r = \inf_{i \in I} r_i$ and,
  for each $h \in \Lform_\alpha^a (X, d)$, $\alpha > 0$, define $H (h)$
  as $\sup_{i \in I} (H_i (h) + \alpha r - \alpha r_i)$.  Then:
  \begin{enumerate}
  \item $H$ is a well-defined, positively homogeneous functional from
    $\Lform^\bnd\Prev X$ to $\creal$;
  \item for any upper bound $(H', r')$ of ${(H_i, r_i)}_{i \in I}$,
    $r' \leq r$ and, for every $h \in \Lform_\alpha^a (X, d)$,
    $\alpha > 0$, $H (h) \leq H' (h) + \alpha r - \alpha r'$;
  \item if $H$ is continuous, then $(H, r)$ is the supremum of
    ${(H_i, r_i)}_{i \in I}$;
  \item if every $H_i$ is sublinear, resp.\ superlinear, resp.\
    linear, resp.\ subnormalized, resp.\ normalized, resp.\ discrete,
    then so is $H$.
  \end{enumerate}
\end{prop}
Note that $H$ is indeed defined on the whole of
$\Lform_\infty^\bnd (X, d)$, by Lemma~\ref{lemma:Linfa}.  We again
call $(H, r)$ the \emph{naive supremum} of ${(H_i, r_i)}_{i \in I}$.

\proof We mimic Proposition~\ref{prop:LPrev:simplesup}.  Again, this
definition does not depend on $\alpha$, really, because $H (h)$ is the
limit of ${(H_i (h))}_{i \in I, \sqsubseteq}$ taken in $\creal$ with
its usual topology (not the Scott topology).  Using the formula
$H (h) = \lim_{i \in I, \sqsubseteq} H_i (h)$, where the limit is
taken in $\creal$ with its usual, Hausdorff topology, it is easy to
show that $H$ is sublinear, resp.\ superlinear, resp.\ linear, resp.\
subnormalized, resp.\ normalized, when every $H_i$ is, because $+$ and
products by scalars are continuous on $\creal$ with its usual
topology.  Similarly, $H$ is discrete whenever every $H_i$ is, using
Lemma~\ref{lemma:Lbeta:creal}.  This shows (1) and (4).  (2) and (3)
are proved as in Proposition~\ref{prop:LPrev:simplesup}, taking care to
quantify over $h$ in $\Lform_1^a (X, d)$, or in $\Lform_\alpha^a (X,
d)$.  \qed

We deduce the following theorem with exactly the same proof as for
Theorem~\ref{thm:LPrev:sup}, using
Proposition~\ref{prop:LaPrev:simplesup} instead of
Proposition~\ref{prop:LPrev:simplesup}.
\begin{thm}[Yoneda-completeness]
  \label{thm:LaPrev:sup}
  Let $X, d$ be a standard quasi-metric space, and fix $a > 0$.
  Assume that the topology of $\Lform_\infty^\bnd (X, d)$ is
  determined by those of $\Lform_\alpha^a (X, d)$, $\alpha > 0$ (e.g.,
  if $X, d$ is standard and Lipschitz regular, see
  Proposition~\ref{prop:Linfty:determined:bnd}).

  Then $\Prev X, \dKRH^a$ is Yoneda-complete, and all suprema of
  directed families of formal balls of previsions are naive suprema,
  as described in Proposition~\ref{prop:LaPrev:simplesup}.

  The same result holds for the subspace of previsions satisfying any
  given set of properties among: sublinearity, superlinearity,
  linearity, subnormalization, normalization, and discreteness.  \qed
\end{thm}

\subsection{Supports}
\label{sec:supports}

Our previous Yoneda-completeness theorems for spaces of previsions
make a strong assumption, Lipschitz regularity.  We show that we also
obtain Yoneda-completeness under different assumptions, which have to
do with supports.

\begin{defi}[Support]
  \label{defn:support}
  Let $X$ be a topological space.  A subset $A$ of $X$ is called a
  \emph{support} of a prevision $F$ if and only if, for all
  $g, h \in \Lform X$ with the same restriction to $A$ (i.e., such
  that $g_{|A} = h_{|A}$), then $F (g) = F (h)$.
\end{defi}
We also say that $F$ is \emph{supported on $A$} in that case.

We do not know whether $\Prev$ defines a functor on the category of
standard quasi-metric spaces and $1$-Lipschitz continuous maps,
however, we have the following partial result.
\begin{lem}
  \label{lemma:Pf:lip}
  Let $X, d$ and $Y, \partial$ be two quasi-metric spaces, and
  $f \colon X, d \mapsto Y, \partial$ be an $\alpha$-Lipschitz
  continuous map, where $\alpha > 0$.  Then
  $\Prev f \colon \Prev X, \dKRH \to \Prev Y, \KRH\partial$, defined
  by:
  \begin{equation}
    \label{eq:Prevf}
    \Prev f (F) (k) = F (k \circ f)
  \end{equation}
  for every $k \in \Lform Y$, is a well-defined, $\alpha$-Lipschitz
  map, which maps previsions to previsions, preserving sublinearity,
  superlinearity, linearity, subnormalization, normalization, and
  discreteness.

  A similar result holds for $\Prev f$, seen as a map from $\Prev X,
  \dKRH^a$ to $\Prev Y, \KRH\partial^a$, for every $a \in \Rp$, $a > 0$.
\end{lem}
\proof It is easy to check that $\Prev f (F)$ is a prevision, which is
sublinear, resp.\ superlinear, resp.\ linear, resp.\ subnormalized,
resp.\ normalized, resp.\ discrete, whenever $F$ is.

To show that $\Prev f$ is $\alpha$-Lipschitz from $\Prev X,
  \dKRH$ to $\Prev Y, \KRH\partial$, we assume that
$\dKRH (F, F') \leq r$, and we wish to show that
$\KRH{\partial} (\Prev f (F), \Prev f (F')) \leq \alpha r$.  To this end,
we let $k \in \Lform_1 (\mathbf B (Y, \partial), \partial^+)$ be arbitrary, and we
attempt to show that
$\Prev f (F) (k) \leq \Prev f (F') (k) + \alpha r$.  By
Lemma~\ref{lemma:comp:Lip}, $k \circ f$ is in $\Lform_\alpha (X, d)$,
so $1/\alpha (k \circ f)$ is in $\Lform_1 (X, d)$.  We expand the
definition of $\dKRH$ in $\dKRH (F, F') \leq r$, and we obtain that
$F (1/\alpha (k \circ f)) \leq F' (1/\alpha (k \circ f)) + r$.  Using
positive homogeneity, we deduce
$F (k \circ f) \leq F' (k \circ f) + \alpha r$, which is what we
wanted to prove.

The proof that $\Prev f$ is $\alpha$-Lipschitz from $\Prev X, \dKRH^a$
to $\Prev Y, \KRH\partial^a$ is the same, noticing that if $k$ is
bounded by $a$, then $k \circ f$ is also bounded by $a$.
%
\qed

The one thing that is missing in order to show that $\Prev$ is a
functor, in a category whose morphisms are $1$-Lipschitz continuous
maps, is continuity.  We have the following approximation, which shows
that naive suprema are preserved.
\begin{lem}
  \label{lemma:Pf:lipcont}
  Let $X, d$ and $Y, \partial$ be two quasi-metric spaces, and
  $f \colon X, d \mapsto Y, \partial$ be a $1$-Lipschitz continuous
  map.  For every directed family of formal balls $(G_i, r_i)$,
  $i \in I$, in $\mathbf B (\Prev X, \dKRH)$ (resp., $\dKRH^a$), with
  naive supremum $(G, r)$, $(\Prev f (G), r)$ is the naive supremum of
  ${(\Prev f (G_i), r_i)}_{i \in I}$ in
  $\mathbf B (\Prev Y, \KRH{\partial})$ (resp., $\KRH{\partial}^a$).
\end{lem}
\proof By definition, $r = \inf_{i \in I} r_i$, and
$G (h) = \sup_{i \in I} (G_i (h) + \alpha r - \alpha r_i)$ for every
$h \in \Lform_\alpha (X, d)$ (resp., in $\Lform_\alpha^a (X, d)$), for
every $\alpha \in \Rp$, $\alpha > 0$.  It follows that, for every
$k \in \Lform_\alpha (Y, \partial)$ (resp., in
$\Lform_\alpha^a (Y, \partial)$), and using the fact that $k \circ f$
is in $\Lform_\alpha (X, d)$ by Lemma~\ref{lemma:comp:Lip} (resp., in
$\Lform_\alpha^a (X, d)$, since $k \circ f$ is bounded by $\alpha a$,
just like $k$),
$\Prev f (G) (k) = G (k \circ f) = \sup_{i \in I} (G_i (k \circ f) +
\alpha r - \alpha r_i) = \sup_{i \in I} (\Prev f (G_i) (k) + \alpha r
- \alpha r_i)$.  \qed


If $X, d$ is standard, then Lemma~\ref{lemma:eta:lipcont} states that
$\eta_X$ is $1$-Lipschitz continuous.  Hence Lemma~\ref{lemma:Pf:lip}
implies that $\Prev \eta_X (F)$ is a prevision on
$\mathbf B (X, d), d^+$, which one can see as an extension of $F$ from
$X$ to the superspace $\mathbf B (X, d)$.  Moreover, $\Prev \eta_X$ is
$1$-Lipschitz, so $\mathbf B^1 (\Prev \eta_X)$ is monotonic.  This
implies that for every directed family ${(F_i, r_i)}_{i \in I}$ in
$\mathbf B (\Prev X, \dKRH)$, ${(\Prev \eta_X (F_i), r_i)}_{i \in I}$
is also a directed family, this time in
$\mathbf B (\Prev (\mathbf B (X, d)), \KRH{d^+})$.  Because
$\mathbf B (X, d), d^+$ is Lipschitz regular
(Theorem~\ref{thm:B:lipreg}) and standard
(Proposition~\ref{prop:B:std}),
${(\Prev \eta_X (F_i), r_i)}_{i \in I}$ has a supremum $(G, r)$, as an
application of Theorem~\ref{thm:LPrev:sup}.  That supremum $(G, r)$ is
the naive supremum, which is to say that $r = \inf_{i \in I} r_i$,
and:
\begin{equation}
  \label{eq:G}
  G (k) = \sup_{i \in I} (F_i (k \circ \eta_X) + \alpha r - \alpha
  r_i)
\end{equation}
for every $k \in \Lform_\alpha (\mathbf B (X, d), d^+)$, $\alpha > 0$.

Recall that $V_\epsilon$ is the set of formal balls of radius strictly
less than $\epsilon$.
\begin{lem}
  \label{lemma:G:supp:Veps}
  Let $X, d$ be a standard quasi-metric space, ${(F_i, r_i)}_{i \in I}$
  be a directed family in $\mathbf B (\Prev X, \dKRH)$, let
  $r = \inf_{i \in I} r_i$ and $G$ be the prevision on
  $\mathbf B (X, d)$ defined in (\ref{eq:G}).  For every
  $\epsilon > 0$, $G$ is supported on $V_\epsilon$.
\end{lem}
\proof


Formally, $G$, as defined in (\ref{eq:G}), is merely an
$\Lform$-prevision, and the prevision we are considering is
$\overline G$, not $G$.  We wish to show that $\overline G$ is
supported on $V_\epsilon$.  In order to do so, we will prove that
$\overline G (f) = \overline G (f \chi_{V_\epsilon})$ for every
$f \in \Lform (\mathbf B (X, d))$.
Using (\ref{eq:G}) and the definition of $\overline G$, then
Proposition~\ref{prop:f(alpha)}, we have:
\begin{align*}
  \overline G (f)
  & =  \sup_{\alpha \in \Rp, i \in I} \left(
    F_i (f^{(\alpha)} \circ \eta_X) + \alpha r - \alpha r_i
    \right) \\
  & =  \sup_{\alpha \in \Rp, i \in I, K \in \nat} \left(
    F_i (f_K^{(\alpha)} \circ \eta_X) + \alpha r - \alpha r_i
    \right) \\
\end{align*}
which is a directed supremum.  We claim that, for $\alpha$ large
enough,
$f_K^{(\alpha)} \circ \eta_X = {(f \chi_{V_\epsilon})}_K^{(\alpha)}
\circ \eta_X$, equivalently that
$f_K^{(\alpha)} (x, 0) = {(f \chi_{V_\epsilon})}_K^{(\alpha)} (x, 0)$
for every $x \in X$.  This will prove the desired result.

Let us fix $K$, and let us write $f_K$ as
$\frac 1 {2^K} \sup_{k=1}^{K2^K} k \chi_{U_k}$, where
$U_k = f^{-1} (]k/2^K, +\infty])$.  Since
${(f \chi_{V_\epsilon})}^{-1} (]k/2^K, +\infty]) = U_k \cap
V_\epsilon$, we have
$({f \chi_{V_\epsilon})}_K = \frac 1 {2^K} \sup_{k=1}^{K2^K} k
\chi_{U_k \cap V_\epsilon}$.  Using Lemma~\ref{lemma:f(alpha):step},
we have:
\begin{equation}
  \label{eq:G:supp:Veps}
  f_K^{(\alpha)} (x, 0)
  = \min \left(\min\nolimits_{k=1}^{K2^K} \left(\frac {k-1} {2^K} + \alpha d^+ ((x, 0),
    \overline {U_k})\right), K \right),
\end{equation}
and similarly for $({f \chi_{V_\epsilon})}_K^{(\alpha)}$, replacing
each $U_k$ by $U_k \cap V_\epsilon$.

We pick $\alpha > K/\epsilon$.  For every $x \in X$, for every $k$
($1\leq k \leq K2^K$),
$d^+ ((x, 0), \overline {U_k}) = \sup \{s \in \Rp \mid ((x, 0), s) \in \widehat
{U_k})$, where, for each $d^+$-Scott open subset $V$ of
$\mathbf B (X, d)$, $\widehat V$ denotes the largest Scott-open subset
of $\mathbf B (\mathbf B (X, d), d^+)$ whose intersection with
$\mathbf B (X, d)$ is equal to $V$.  We apply
Lemma~\ref{lemma:Uhat:alpha} (see also Remark~\ref{rem:Uhat:alpha}) to
the case of the free $\mathbf B$-algebra on $X$, whose structure map
is $\mu_X$, and we obtain that $\widehat V = \mu_X^{-1} (V)$.  Hence
$d^+ ((x, 0), U_k) = \sup \{s \in \Rp \mid \mu_X ((x, 0), s) \in U_k\}
= \sup \{s \in \Rp \mid (x, s) \in U_k\}$.  Hence we can rewrite
equation~\ref{eq:G:supp:Veps} as:
\[
  f_K^{(\alpha)} (x, 0)
  = \min \left(\min\nolimits_{k=1}^{K2^K}
    \left(\sup \left\{\frac {k-1} {2^K} + \alpha s \in \Rp \mid s \in \Rp, (x, s) \in
      U_k\right\}\right),
    K \right),
\]
In the supremum over $s$, we can ignore all the values $s > \epsilon$,
since they will contribute $\frac {k-1} {2^K} + \alpha s > K$ to the
min, and the latter formula defines $f_K^{(\alpha)} (x, r)$ as a min
with $K$ anyway.  Hence:
\begin{align*}
  f_K^{(\alpha)} (x, 0)
  & = \min \left(\min\nolimits_{k=1}^{K2^K}
    \left(\sup \left\{\frac {k-1} {2^K} + \alpha s \in \Rp \mid s \in [0, \epsilon[, (x, s) \in
    U_k\right\}\right),
    K \right) \\
  & = \min \left(\min\nolimits_{k=1}^{K2^K}
    \left(\sup \left\{\frac {k-1} {2^K} + \alpha s \in \Rp \mid s \in \Rp, (x, s) \in
    U_k \cap V_\epsilon\right\}\right),
    K \right) \\
  & = ({f \chi_{V_\epsilon})}_K^{(\alpha)} (x, 0),
\end{align*}
which finishes to prove the claim.  \qed

So $G$ is supported on every $V_\epsilon$, $\epsilon > 0$.  If we
could prove the slightly stronger statement that $G$ is supported on
$X$, seen as a subset of $\mathbf B (X, d)$ (formally, on the set of
open balls of radius $0$), then we claim that $G$ would be a prevision
on $X$; i.e., that $G = \Prev \eta_X (F)$ for some prevision $F$ on
$X$.  Moreover, $(F, r)$ would be the supremum of
${(F_i, r_i)}_{i \in I}$ in $\mathbf B (\Prev X, \dKRH)$.
\begin{prop}
  \label{prop:supp:X}
  Let $X, d$ be a standard quasi-metric space, ${(F_i, r_i)}_{i \in I}$ be a
  directed family in $\mathbf B (\Prev X, \dKRH)$, let
  $r = \inf_{i \in I} r_i$ and $G$ be the prevision on
  $\mathbf B (X, d)$ defined in (\ref{eq:G}).  If $G$ is supported on
  $X$, then:
  \begin{enumerate}
  \item $G = \Prev \eta_X (F)$ for some unique prevision $F$ on $X$;
  \item if $G$ is superlinear, resp.\ sublinear, resp.\ linear, resp.\
    subnormalized, resp.\ normalized, resp.\ discrete, then so is $F$;
  \item $(F, r)$ is the naive supremum of ${(F_i, r_i)}_{i \in I}$ in
    $\mathbf B (\Prev X, \dKRH)$;
  \item $(F, r)$ is the supremum of ${(F_i, r_i)}_{i \in I}$ in
    $\mathbf B (\Prev X, \dKRH)$.
  \end{enumerate}
\end{prop}
\proof (1) For every $h \in \Lform X$, let
$\widehat h \colon \mathbf B (X, d) \to \creal$ be defined by
$\widehat h (x, r) = \sup \{a \in \Rp \mid (x, r) \in \widehat {h^{-1}
  (]a, +\infty])}\}$.  For every $b \in \real$,
${\widehat h}^{-1} (]b, +\infty]) = \bigcup_{a > b} \widehat {h^{-1}
  (]a, +\infty])}$ is open, so $\widehat h$ is (Scott-)continuous.
For $r=0$,
$\widehat h (x, 0) = \sup \{a \in \Rp \mid (x, 0) \in \widehat {h^{-1}
  (]a, +\infty])}\} = \sup \{a \in \Rp \mid x \in h^{-1} (]a,
+\infty])\} = \sup \{a \in \Rp \mid h (x) > a\} = h (x)$.  Therefore
$\widehat h$ is a Scott-continuous extension of $h$ to
$\mathbf B (X, d)$.

If $G = \Prev \eta_X (F)$ for some $F$, then for every
$h \in \Lform X$,
$G (\widehat h) = F (\widehat h \circ \eta_X) = F (h)$, and this shows
that $F$ is unique if it exists.

Now define $F (h)$ as $G (\widehat h)$.  One can show that
$\widehat {ah} = a \widehat h$ for every $h \in \Lform X$ and
$a \in \Rp$, and this will imply that $F$ is positively homogeneous.
But $h \mapsto \widehat h$ does not commute with directed suprema
(unless $X, d$ is Lipschitz regular, an assumption we do not make), so
that a similar argument cannot show that $F$ is Scott-continuous.

Instead, we observe that $\widehat {ah}$ and $a \widehat h$ have the
same restriction to $X$:
$\widehat {ah} (x, 0) = (ah) (x) = a h (x) = a \widehat h (x, 0)$.
Since $G$ is supported on $X$, $G (\widehat {ah}) = G (a \widehat h)$,
and since the latter is equal to $a G (\widehat h)$, we obtain that
$F (ah) = a F (h)$.

We can show that $F$ is Scott-continuous by a similar argument.
First, $h \mapsto \widehat h$ is monotonic, owing to the fact that
$U \mapsto \widehat U$ is itself monotonic.  It follows that $F$ is
monotonic.  Let now ${(h_i)}_{i \in I}$ be a directed family in
$\Lform X$, and $h$ be its supremum.  Then
$\widehat h (x, 0) = h (x) = \sup_{i \in I} h_i (x) = \sup_{i \in I}
\widehat {h_i} (x, 0)$, so $\widehat h$ and
$\sup_{i \in I} \widehat {h_i}$ have the same restriction to $X$.
Hence
$F (h) = G (\widehat h) = G (\sup_{i \in I} \widehat {h_i}) = \sup_{i
  \in I} G (\widehat {h_i}) = \sup_{i \in I} F (h_i)$.

(2) One may observe that (2) is a consequence of (3), using
Proposition~\ref{prop:LPrev:simplesup}~(4) to show that naive suprema
of directed families of formal balls of previsions on
$\mathbf B (X, d)$ preserve superlinearity, sublinearity, etc.
However, a direct proof is equally simple.  Using the fact that
$\widehat g + \widehat h$ and $\widehat {g+h}$ have the same
restriction $g+h$ to $X$, we obtain that $F$ is superlinear, resp.\
sublinear, resp.\ linear, resp.\ subnormalized, resp.\ normalized if
$G$ is.  It remains to show that If $G$ is discrete, then $F$ is
discrete as well.  Let $f$ be any strict map in $\Lform \creal$.  Then
$F (f \circ h) = G (\widehat {f \circ h})$.  We have
$\widehat {f \circ h} (x, 0) = f (h (x)) = f (\widehat h (x, 0))$, so
$\widehat {f \circ h}$ and $f \circ \widehat h$ have the same
restriction to $X$.  It follows that
$F (f \circ h) = G (f \circ \widehat h) = f (G (\widehat h)) = f (F
(h))$, showing that $F$ is discrete.

(3) For every $h \in \Lform_\alpha (X, d)$, $\alpha > 0$,
$F (h) = G (\widehat h)$ by definition.


However, we can build another map $h'$ which coincides with $h$ on $X$
as follows.  Since $\creal, \dreal$ is Yoneda-complete hence standard,
$\mathbf B (\creal, \dreal), \dreal^+$ is also standard by
Proposition~\ref{prop:B:std}.  We can therefore apply
Lemma~\ref{lemma:d(_,x)}: $\dreal^+ (\_, (y, s))$ is $1$-Lipschitz
continuous for every $(y, s) \in \mathbf B (\creal, \dreal)$.  Taking
$y=s=0$, the map
$m = \dreal^+ (\_, (0, 0)) \colon (x, r) \mapsto \max (x-r, 0)$ is
$1$-Lipschitz continuous.  We define $h'$ as
$m \circ \mathbf B^\alpha (h)$.  Note that $\mathbf B^\alpha (h)$ is
$\alpha$-Lipschitz continuous by Lemma~\ref{lemma:Balpha:f}, hence
$h'$ is also $\alpha$-Lipschitz continuous, by
Lemma~\ref{lemma:comp:Lip}.

Explicitly, $h' (x, r) = \max (h (x) - \alpha r, 0)$.  Taking $r=0$,
we see that $h'$ coincides with $h$ on $X$, or more precisely, $h'
\circ \eta_X = h$.

Since $h'$ is $\alpha$-Lipschitz continuous,
$G (h') = \sup_{i \in I} (F_i (h' \circ \eta_X) + \alpha r - \alpha
r_i)$.  We rewrite $h' \circ \eta_X$ as $h$ and recall that
$F (h) = G (h')$, whence
$F (h) = \sup_{i \in I} (F_i (h) + \alpha r - \alpha r_i)$.  That
shows that $F$ is the naive supremum of ${(F_i)}_{i \in I}$.

(4) By Proposition~\ref{prop:LPrev:simplesup}~(3) (and the isomorphism
between $\Lform\Prev X$ and $\Prev X$ that ensues from
Proposition~\ref{prop:Gbar}),
any naive supremum of a directed family that happens to be continuous,
such as $F$, is the supremum.  \qed

To recap, we would be able to show that $\Prev X$ (or any of its
subspaces $Y$ of interest, defined by any combination of
superlinearity, sublinearity, linearity, subnormalization,
normalization, discreteness) is Yoneda-complete, assuming $X, d$
standard, as follows.  Consider a directed family
${(F_i, r_i)}_{i \in I}$ in $\mathbf B (Y, \dKRH)$.  We make a detour
through $\mathbf B (X, d)$, and realize that
${(\Prev \eta_X (F_i), r_i)}_{i \in I}$ is directed, and has a (naive)
supremum $(G, r)$ given in (\ref{eq:G}).  $G$ is supported on
$V_\epsilon$ for $\epsilon > 0$ as small as we wish
(Lemma~\ref{lemma:G:supp:Veps}).  If we could show that $G$ is
supported on $X$, which happens to be the intersection of countably
many such sets $V_\epsilon$ (for $\epsilon=1/2^n$, $n \in \nat$;
recall Lemma~\ref{lemma:Veps}), then Proposition~\ref{prop:supp:X}
would imply that ${(F_i, r_i)}_{i \in I}$ would itself have a supremum
in $\mathbf B (Y, \dKRH)$, and that it would be computed as the naive
supremum.  Then $Y$ would be Yoneda-complete.

While it might seem intuitive that any prevision that is supported on
$V_\epsilon$ for arbitrarily small values of $\epsilon > 0$ should in
fact be supported on $X$, this is exactly what is missing in our
argument.  We shall fix that in several special cases.  In the
meantime, we state what we have as the following proposition.
\begin{prop}[Conditional Yoneda-completeness]
  \label{prop:supp:complete}
  Let $X, d$ be a standard quasi-metric space.  Let $S$ be any subset
  of the following properties on previsions: sublinearity,
  superlinearity, linearity, subnormalization, normalization,
  discreteness; let $Y$ denote the set of previsions on $X$
  satisfying $S$, and $Z$ be the set of previsions on $\mathbf B (X,
  d)$ that satisfy $S$.

  If every element of $Z$ that is supported on $V_{1/2^n}$ for every
  $n \in \nat$ is in fact supported on $X$, then $Y, \dKRH$ is
  Yoneda-complete, and directed suprema in $\mathbf B (Y, \dKRH)$ are
  computed as naive suprema.

  The same result holds for $\dKRH^a$ in lieu of $\dKRH$, for any $a >
  0$. \qed
\end{prop}



\subsection{Preparing for Algebraicity}
\label{sec:continuity-results}

Having dealt with Yoneda-completeness, under some provisos, we now
prepare the grounds for later results, which will state that certain
spaces of previsions are algebraic.  The following will be
instrumental in characterizing the center points of those spaces.
\begin{prop}
  \label{prop:dKRH:cont}
  Let $X, d$ be a continuous
  Yoneda-complete
  quasi-metric space, $\alpha \in \Rp$, and $a \in \Rp$, $a > 0$.  For
  every $n \in \nat$, for all $a_1, \ldots, a_n \in \Rp$ and every
  $n$-tuple of center points $x_1$, \ldots, $x_n$ in $X, d$, for every
  prevision $F$ on $X$, the following maps are continuous from
  $\Lform_\alpha (X, d)^\patch$ (resp.,
  $\Lform_\alpha^a (X, d)^\patch$) to $(\creal)^\dG$:
  \begin{enumerate}
  \item $f \mapsto \dreal (\sum_{i=1}^n a_i f (x_i), F (f))$,
  \item $f \mapsto \dreal (\max_{i=1}^n f (x_i), F (f))$,
  \item $f \mapsto \dreal (\min_{i=1}^n f (x_i), F (f))$,
  \end{enumerate}
\end{prop}
\proof We only deal with $\Lform_\alpha (X, d)$, as the case of
$\Lform_\alpha^a (X, d)$ is similar.

As a consequence of Corollary~\ref{corl:cont:Lalpha:simple} and the
fact that the patch topology is finer than the cocompact topology, the
map $f \mapsto \sum_{i=1}^n a_i f (x_i)$ is continuous from
$\Lform_\alpha (X, d)^\patch$ to $(\creal)^\dG$.  Similarly for
$f \mapsto \max_{i=1}^n f (x_i)$ and for
$f \mapsto \min_{i=1}^n f (x_i)$.  $F$ is continuous from $\Lform X$
to $\creal$, hence also from $\Lform_\alpha (X, d)$ to $\creal$
(because $\Lform_\alpha (X, d)$ has the subspace topology from
$\Lform X$), hence also from $\Lform_\alpha (X, d)^\patch$ to
$\creal$.  The claim will then follow from the fact that
$(s, t) \mapsto \dreal (s, t)$ is continuous from
$(\creal)^\dG \times \creal$ to $(\creal)^\dG$, which we now prove.
The non-trivial open subsets of $(\creal)^\dG$ are the open intervals
$[0, a[$.  The inverse image of the latter by $\dreal$ is
$\{(s, t) \in \creal \times \creal \mid s < t + a\} = \{(s, t) \in
\creal \times \creal \mid \exists b \in \Rp . s < b+a, b < t\} =
\bigcup_{b \in \Rp} [0, b+a[ \times ]b, +\infty]$.  \qed

The sum $\sum_{i=1}^n a_i f (x_i)$ is just the integral of $f$ with
respect to the so-called simple valuation
$\sum_{i=1}^n a_i \delta_{x_i}$, all notions we shall introduce later.
Since all continuous maps from a compact space to $(\creal)^\dG$ reach
their maximum, we obtain immediately:
\begin{lem}
  \label{lemma:dKRH:max}
  Let $X, d$ be a continuous
  Yoneda-complete
  quasi-metric space, let
  $\nu = \sum_{i=1}^n a_i \delta_{x_i}$ be a simple valuation on $X$
  such that every $x_i$ is a center point, and let $F$ be a prevision
  on $X$.  There is an $h \in \Lform_1 (X, d)$ such that
  $\dKRH (\nu, F) = \dreal (\sum_{i=1}^n a_i h (x_i), F (h))$; in
  other words, the supremum in (\ref{eq:dKRH}) is reached.

  Similarly, for every $a \in \Rp$, $a > 0$, the supremum in
  (\ref{eq:dKRHa}) is reached: there is an $h \in \Lform_1^a (X, d)$
  such that $\dKRH^a (\nu, F) = \dreal (\sum_{i=1}^n a_i h (x_i), F
  (h))$.  \qed
\end{lem}

One can make a function $h$ explicit where the supremum is reached.
Recall the functions $\sea x b$ from
Section~\ref{sec:functions-sea-x}.

\begin{prop}
  \label{prop:dKRH:max:d}
  Let $X, d$ be a continuous
  Yoneda-complete
  quasi-metric space, let
  $\nu = \sum_{i=1}^n a_i \delta_{x_i}$ be a simple valuation on $X$
  such that every $x_i$ is a center point, and let $F$ be a prevision
  on $X$.  There are numbers $b_i \in \creal$, $1\leq i\leq n$, such
  that $\dKRH (\nu, F) = \dreal (\sum_{i=1}^n a_i h (x_i), F (h))$,
  where $h = \bigvee_{i=1}^n \sea {x_i} {b_i}$.

  Similarly, for every $a \in \Rp$, $a > 0$, there are numbers
  $b_i \in [0, \alpha a]$, $1\leq i \leq n$, such that
  $\dKRH^a (\nu, F) = \dreal (\sum_{i=1}^n a_i h (x_i), F (h))$, where
  $h = \bigvee_{i=1}^n \sea {x_i} {b_i}$.
\end{prop}
\proof We deal with the $\dKRH$ case, the $\dKRH^a$ case is similar.

Using Lemma~\ref{lemma:dKRH:max}, there is an
$h_0 \in \Lform_1 (X, d)$ such that
$\dKRH (\nu, F) = \dreal (\sum_{i=1}^n a_i h_0 (x_i), \allowbreak F
(h_0))$.  Let $b_i = h_0 (x_i)$, $1\leq i\leq n$, and define $h$ as
$\bigvee_{i=1}^n \sea {x_i} {b_i}$.  By Lemma~\ref{lemma:sea:min}, $h$
is in $\Lform_1 (X, d)$ and $h \leq h_0$, so $F (h) \leq F (h_0)$.
Since $\dreal$ is antitone in its second argument,
$\dKRH (\nu, F) \leq \dreal (\sum_{i=1}^n a_i h_0 (x_i), F (h))$, and
since $h_0 (x_i) = h (x_i)$ for each $i$,
$\dKRH (\nu, \nu') \leq \dreal (\sum_{i=1}^n a_i h (x_i), F (h))$.
The reverse inequality is by definition of $\dKRH$.  \qed

\subsection{Topologies on Spaces of Previsions}
\label{sec:topol-spac-prev}


There is a natural topology on $\Prev X$, called the \emph{weak}
topology, or the \emph{Scott weak$^*$} topology:
\begin{defi}[Weak topology]
  \label{defn:weak}
  For every topological space $X$, the \emph{weak topology} on
  $\Prev X$ has subbasic open sets
  $[h > b] = \{F \in \Prev X \mid F (h) > b\}$, $h \in \Lform X$,
  $b \in \Rp$.
\end{defi}
In other words, the weak topology is the coarsest topology such that
$F \mapsto F (h)$ is continuous from $\Prev X$ to $\creal$ (with its
Scott topology), for every $h \in \Lform X$.

When $X, d$ is a standard quasi-metric space, another subbase for the
weak topology is given by the open sets $[h > b]$, where now $h$ is
restricted to be 1-Lipschitz continuous.  This is because, for
$h \in \Lform X$, $h = \sup_{\alpha > 0} h^{(\alpha)}$, so
$[h > b] = \bigcup_{\alpha > 0} [1/\alpha h^{(\alpha)} > b/\alpha]$.

By the same argument, this time using the fact that
$h = \sup_{\beta > 0} \min (h, \beta)$, hence that
$[h > b] = \bigcup_{\beta > 0} [\min (h, \beta) > b]$, the open sets
$[h > b]$ also form a subbase of the weak topology, where $h$ is now
\emph{bounded} and in $\Lform X$.

Now fix $a \in \Rp$, $a > 0$.  For $h$ bounded in $\Lform X$, $h$ is
the directed supremum of bounded Lipschitz continuous maps
$h^{(\alpha)}$, $\alpha > 0$, so the same argument shows that an even
smaller subbase is given by the sets $[h > b]$ where
$h \in \Lform_\infty^\bnd (X, d)$.  By Lemma~\ref{lemma:Linfa}, $h$ is
in $\Lform_\alpha^a (X, d)$ for some $\alpha > 0$.  Then
$[h > b] = [1/\alpha h > b/\alpha]$, so that we can take the sets
$[h > b]$ as a subbase of the weak topology, where $h$ is now in
$\Lform_1^a (X, d)$.  This will be used in the next proposition.

\begin{prop}
  \label{prop:weak:dScott:a}
  Let $X, d$ be a standard quasi-metric space, $a, a' > 0$ with
  $a \leq a'$.  Let $S$ be a subset of the following properties on
  previsions: sublinearity, superlinearity, linearity,
  subnormalization, normalization, discreteness; and let $Y$ denote
  the set of previsions on $X$ satisfying $S$.

  Assume finally that the spaces $Y, \dKRH^a$, $Y, \dKRH^{a'}$ and $Y,
  \dKRH$ are Yoneda-complete and that directed suprema in their spaces
  of formal balls are computed as naive suprema.

  Then we have the following inclusions of topologies on $Y$:
  \begin{quote}
    weak $\subseteq$ $\dKRH^a$-Scott $\subseteq$ $\dKRH^{a'}$-Scott
    $\subseteq$ $\dKRH$-Scott.
  \end{quote}
%
\end{prop}
\proof \emph{First inclusion.}  We use yet another topology, this time
on $\mathbf B (Y, \dKRH^a)$.  Let $[h > b]^+$ be defined as the set of
those formal balls $(F, r)$ with $F \in Y$ such that $F (h) > r+b$.
The \emph{weak$^+$ topology} on $\mathbf B (Y, \dKRH^a)$ has a subbase
of open sets given by the sets $[h > b]^+$, for every
$h \in \Lform_1^a (X, d)$ and $b \in \Rp$.

$[h > b]^+$ is upwards-closed in $\mathbf B (Y, \dKRH^a)$.  Indeed,
assume that $(G, r) \in [h > b]^+$, where $h \in \Lform_1^a (X, d)$,
and that $(G, r) \leq^{\dKRH^{a+}} (G', r')$.  By
Lemma~\ref{lemma:LaPrev:ord}~(2), $G (h) - r \leq G' (h) - r'$.  Since
$G (h) > r+b$, $G' (h) > r'+b$, namely $G' \in [h > b]^+$.

To show that $[h > b]^+$ is Scott-open in $\mathbf B (Y, \dKRH^a)$,
let ${(G_i, r_i)}_{i \in I}$ be a directed family with (naive)
supremum $(G, r)$ in $\mathbf B (Y, \dKRH^a)$ and assume that
$(G, r) \in [h > b]^+$.  Since
$G (h) = \sup_{i \in I} (G_i (h) + r -r_i) > r+b$, $G_i (h) > r_i + b$
for some $i \in I$, i.e., $(G_i, r_i)$ is in $[h > b]^+$.

We can now proceed to show that $[h > b]$ is $\dKRH^a$-Scott open, for
every $h \in \Lform_1^a (X, d)$.  Equating $Y$ with the subset of all
formal balls $(G, 0)$, $G \in Y$, $[h > b]$ is equal to
$Y \cap [h > b]^+$.  Since $[h > b]^+$ is Scott-open, $[h > b]$ is
$\dKRH^a$-Scott open.

\emph{Second and third inclusions.}
The proofs of the second and third inclusions are similar.  We rely on
the easily checked inequalities $\dKRH^a \leq \dKRH^{a'} \leq \dKRH$,
for $a \leq a'$.  This implies that: $(*)$ $(F, r) \leq^{\dKRH^+} (F',
s)$ implies $(F, r) \leq^{\dKRH^{a'+}} (F', s)$, and that $(F, r)
\leq^{\dKRH^{a'+}} (F', s)$ implies $(F, r) \leq^{\dKRH^{a+}} (F',
s)$.

Let $\mathcal U$ be a Scott-open subset of $\mathbf B (Y,
\dKRH^a)$.  Since $\mathcal
U$ is upwards-closed in
$\leq^{\dKRH^{a+}}$, it is also upwards-closed in
$\leq^{\dKRH^{a'+}}$ by $(*)$.  Assume now that $(F, r) \in \mathcal
U$ is the supremum in $\mathbf B (Y, \dKRH^{a'})$ of a family ${(F_i,
  r_i)}_{i \in I}$ that is directed with respect to
$\leq\dKRH^{a'+}$.  By $(*)$ again, ${(F_i, r_i)}_{i \in
  I}$ is directed with respect to $\leq\dKRH^{a+}$.  Therefore ${(F_i,
  r_i)}_{i \in I}$ has a supremum $(F', r')$ in $\mathbf B (Y,
\dKRH^a)$ that is uniquely characterized by $r'=\inf_{i \in I}
r_i$, $F' (h) = \sup_{i \in I} (F_i (h) + r - r_i)$ for every $h \in
\Lform_1^a (X,
d)$, since suprema are assumed to be naive.  However, by naivety
again, $r = \inf_{i \in I} r_i$ and $F (h) = \sup_{i \in I} (F_i (h) +
r - r_i)$ for every $h \in \Lform_1^{a'} (X,
d)$.  This holds in particular for every $h \in \Lform_1^a (X,
d)$, so $F=F'$.  We have therefore obtained that $(F,
r)$ is also the supremum of the directed family ${(F_i, r_i)}_{i \in
  I}$ in $\mathbf B (Y, \dKRH^a)$.  Since $(F, r) \in \mathcal
U$, some $(F_i, r_i)$ is also in $\mathcal U$ since $\mathcal
U$ is Scott-open in $\mathbf B (Y,
\dKRH^a)$.  This shows that $\mathcal U$ is Scott-open in $\mathbf B
(Y,
\dKRH^{a'})$.  Similarly, we show that every Scott-open subset of
$\mathbf B (Y, \dKRH^{a'})$ is Scott-open in $\mathbf B (Y,
\dKRH)$.

By taking intersections $\mathcal U \cap Y$, it follows that every
$\dKRH^a$-Scott open is $\dKRH^{a'}$-Scott open, and that every
$\dKRH^{a'}$-Scott open is $\dKRH$-Scott open.  \qed

We shall see that all the mentioned topologies are in fact equal in a
number of interesting cases.  Proposition~\ref{prop:weak:dScott:a}
is the common core of those results.



\subsection{Extending Previsions to Lipschitz Maps}
\label{sec:extend-prev-lipsch}

A prevision $F$ can be thought as a generalized form of integral: for
every $h \in \Lform X$, $F (h)$ is the integral of $h$.  This will be
an actual integral in case $F$ is a linear prevision, and we shall
study that case in Section~\ref{sec:case-cont-valu}.

In particular, $F (h)$ makes sense when $h$ is $\alpha$-Lipschitz
continuous.  As we now observe, the results of
Section~\ref{sec:falpha-f-alpha} imply that $F (h)$ also makes sense
when $h$ is $\alpha$-Lipschitz, not necessarily continuous.
This will be an important gadget in Section~\ref{sec:minimal-transport}.

Recall that $L_\alpha (X, d)$ is the space of $\alpha$-Lipschitz maps
from $X, d$ to $\creal$, whereas $\Lform_\alpha (X, d)$ is the space
of $\alpha$-Lipschitz continuous maps.  We write $L_\infty (X, d)$ for
$\bigcup_{\alpha > 0} L_\alpha (X, d)$.
\begin{defi}[$\extF F$]
  \label{defn:extF}
  Let $X, d$ be a continuous quasi-metric space.  For every prevision
  $F$ on $X$, define $\extF F \colon L_\infty (X, d) \to \creal$ by
  $\extF F (h) = F (h^{(\alpha)})$, for every $h \in L_\alpha (X, d)$,
  $\alpha > 0$.
\end{defi}
This is well-defined since $F (h^{(\alpha)})$ does not depend on
$\alpha$, as long as $h$ is $\alpha$-Lipschitz.  In other words, if
$h$ is also $\beta$-Lipschitz, then computing $\extF F (h)$ as
$F (h^{(\beta)})$ necessarily yields the same value, by
Corollary~\ref{corl:ha=hb}.

\begin{lem}
  \label{lemma:extF:prev}
  Let $X, d$ be a continuous quasi-metric space.  For every prevision
  $F$ on $X$,
  \begin{enumerate}
  \item $\extF F$ is a monotonic, positively homogeneous map from
    $L_\infty (X, d)$ to $\creal$;
  \item $\extF F$ is sublinear, resp.\ superlinear, resp.\ linear,
    resp.\ subnormalized, resp.\ normalized if $F$ is;
  \item for every $h \in \Lform_\infty (X, d)$, $\extF F (h) = F (h)$;
  \item the restriction of $\extF F$ to $\Lform_\infty (X, d)$ is
    lower semicontinuous.
  \end{enumerate}
\end{lem}
\proof (1) If $f \leq g$ in $L_\infty (X, d)$, then let $\alpha > 0$
be such that $f$ and $g$ are $\alpha$-Lipschitz.  Since $f^{(\alpha)}$
is $\alpha$-Lipschitz continuous and is below $f$, it is below $g$,
hence below the largest $\alpha$-Lipschitz continuous map below $g$,
that is $g^{(\alpha)}$.  Therefore
$\extF F (f) = F (f^{(\alpha)}) \leq F (g^{(\alpha)}) = \extF F (g)$.

For every $h \in L_\alpha (X, d)$, and for every $a \in \Rp$, $ah$ is
in $L_{a\alpha} (X, d)$, so $\extF F (ah) = F ((ah)^{(a\alpha)})$.
That is equal to $F (a h^{(\alpha)})$ by Lemma~\ref{lemma:ah:(alpha)},
hence to $a F (h^{(\alpha)}) = a \extF F (h)$.

(2) Let $f$, $g$ be in $L_\infty (X, d)$, say in $L_\alpha (X, d)$.
Then $f+g$ is in $L_{2\alpha} (X, d)$, and by
Lemma~\ref{lemma:largestLipcont},
$(f+g)^{(2\alpha)} (x) = \sup_{(y, s) \ll (x, 0)} (f (y) + g (y)
-2\alpha s)$.  The maps $(y, s) \mapsto f (y)-\alpha s$ and
$(y, s) \mapsto g (y)-\alpha s$ are monotonic since $f$ and $g$ are
$\alpha$-Lipschitz, so the families
${(f (y) - \alpha s)}_{(y, s) \ll (x, 0)}$ and
${(g (y) - \alpha s)}_{(y, s) \ll (x, 0)}$ are directed.  Using the
Scott-continuity of $+$ on $\creal$,
$(f+g)^{(2\alpha)} (x) = \sup_{(y, s) \ll (x, 0)} (f (y) - \alpha s) +
\sup_{(y, s) \ll (x, 0)} (g (y) - \alpha s) = f^{(\alpha)} (x) +
g^{(\alpha)} (x)$.  It follows immediately that $\extF F$ is
sublinear, resp.\ superlinear, resp.\ linear when $F$ is.  Similarly
for subnormalized and normalized previsions, since
$\mathbf 1^{(\alpha)} = \mathbf 1$.

(3) For every $h \in \Lform_\infty (X, d)$, say $h \in \Lform_\alpha
(X, d)$, $h^{(\alpha)} = h$, so $\extF F (h) = F (h^{(\alpha)}) = F
(h)$.

(4) By (3), the restriction of $\extF F$ to $\Lform_\infty (X, d)$
coincides with $F$, and we recall that the topology on $\Lform_\infty
(X, d)$ is the subspace topology, induced from the Scott topology on
$\Lform X$.  \qed


\section{The Case of Continuous Valuations}
\label{sec:case-cont-valu}

A \emph{valuation} $\nu$ on a topological space $X$ is a map $\nu
\colon \Open X \to \creal$, where $\Open X$ is the complete lattice of
open subsets of $X$, such that $\nu (\emptyset) = 0$ (strictness), for
all $U, V$ such that $U \subseteq V$, $\nu (U) \leq \nu (V)$
(monotonicity), for all $U$, $V$, $\nu (U \cup V) + \nu (U \cap V) =
\nu (U) + \nu (V)$ (modularity).  It is \emph{continuous} if and only
if it is Scott-continuous, \emph{subnormalized} if and only if $\nu
(X) \leq 1$, and \emph{normalized} if and only if $\nu (X)=1$.

We shall see that continuous valuations are the same thing as linear
previsions.  Before that, we related the notion with the better-known
notion of measure.

\subsection{Continuous Valuations and Measures}
\label{sec:cont-valu-meas}

A measure $\mu$ on the Borel subsets of a topological space is called
\emph{$\tau$-smooth} if and only if, for every directed family of open
sets ${(U_i)}_{i \in I}$,
$\mu (\bigcup_{i \in I} U_i) = \sup_{i \in I} \mu (U_i)$
\cite{Topsoe:topmes}.  A $\tau$-smooth measure $\mu$ gives rise to a
unique continuous valuation by restriction to the open sets.

In general, the restriction of a measure to the open sets is only
\emph{countably continuous}, that is, for every monotone
\emph{sequence}
$U_0 \subseteq U_1 \subseteq \cdots \subseteq U_n \subseteq \cdots$ of
open sets,
$\mu (\bigcup_{n \in \nat} U_n) = \sup_{n \in \nat} \mu (U_n)$.
Following Alvarez-Manilla \emph{et al.\@}  \cite{AMESD:ext:val}, we
call \emph{countably continuous} any valuation that satisfies this
property.

A poset is \emph{$\omega$-continuous} if and only if it is continuous
and has a countable basis $B$: that is, every element $x$ is the
supremum of a directed family of elements $y$ way-below $x$ and in the
countable set $B$.  A continuous poset is $\omega$-continuous if and
only if its Scott topology is countably based
\cite[Proposition~3.1]{Norberg:randomsets}.

Let us call a quasi-metric space $X, d$ \emph{$\omega$-continuous} if
and only if $\mathbf B (X, d)$ is an $\omega$-continuous poset.  As
shown by Edalat and Heckmann, when $X, d$ is a metric space,
$\mathbf B (X, d)$ is a continuous poset, and is $\omega$-continuous
if and only if $X, d$ is separable
\cite[Corollary~10]{EH:comp:metric}.  Hence, for metric spaces,
$\omega$-continuous is synonymous with separable.

\begin{lem}
  \label{lemma:mes=>val}
  Let $X, d$ be an $\omega$-continuous quasi-metric space.  Every
  measure on $X$, with the Borel $\sigma$-algebra of its $d$-Scott
  topology, is $\tau$-smooth, i.e., restricts to a continuous
  valuation on the open sets of $X$.
\end{lem}
\proof Let $\mu$ be a measure on $X$.  Consider the image measure
$\eta_X [\mu]$ of $\mu$, namely
$\eta_X [\mu] (E) = \mu (\eta_X^{-1} (E))$ for every Borel subset $E$
of $\mathbf B (X, d)$.  Its restriction to the Scott-open subsets of
$\mathbf B (X, d)$ yields a countably continuous valuation.  Lemma~2.5
of \cite{AMESD:ext:val} states that, on an $\omega$-continuous dcpo, a
valuation is continuous if and only if it is countably continuous.
(R. Heckmann observed that this in fact holds on any second countable
topological space.)  Hence $\eta_X [\mu]$ restricts to a continuous
valuation on $\mathbf B (X, d)$.

Let ${(U_i)}_{i \in I}$ be a directed family of $d$-Scott open subsets
of $X$, and $U$ be their union.  We have:
\begin{eqnarray*}
  \mu (U) & = & \eta_X [\mu] (\bigcup_{i \in I} \widehat U_i) \\
  & = & \sup_{i \in I} \eta_X [\mu] (\widehat U_i) = \sup_{i \in I}
        \mu (U_i).
\end{eqnarray*}
\qed

In the reverse direction, we use the following theorem, due to Keimel
and Lawson \cite[Theorem~5.3]{KL:measureext}: every locally finite
continuous valuation on a locally compact sober space extends to a
Borel measure on the Borel $\sigma$-algebra.  A continuous valuation
$\nu$ is \emph{locally finite} if and only if every point has an open
neighborhood $U$ such that $\nu (U) < +\infty$.  We say that it is
\emph{finite} if and only if the measure of every open subset
(equivalently, of the whole space), is finite, and that it is
\emph{$\sigma$-finite} if and only if one can write the whole space as
a countable union of open subsets of finite measure.

\begin{lem}
  \label{lemma:val=>mes}
  Let $X, d$ be a continuous Yoneda-complete quasi-metric space.
  Every locally finite continuous valuation on $X$ extends to a
  measure on $X$ with the Borel $\sigma$-algebra of its $d$-Scott
  topology.
\end{lem}
\proof We start with a finite continuous valuation $\nu$ on $X$.
Consider its image $\eta_X [\nu]$, defined by
$\eta_X [\nu] (\mathcal U) = \nu (\eta_X^{-1} (\mathcal U))$ for every
Scott-open subset $\mathcal U$ of $\mathbf B (X, d)$.  This is a
finite continuous valuation on a continuous dcpo.  Every continuous dcpo is
locally compact and sober, so we can apply Keimel and Lawson's
theorem.  Let $\mu$ be an extension of $\eta_X [\nu]$ to the Borel
subsets of $\mathbf B (X, d)$.  Since $X, d$ is Yoneda-complete, it is
standard, so $X$ is a $G_\delta$ subset of $\mathbf B (X, d)$
(Lemma~\ref{lemma:Veps}).  Every Borel subset of $X$ is then also
Borel in $\mathbf B (X, d)$, hence $\mu$ restricts to a measure
$\mu_{|X}$ on $X$.  For every open subset $U$ of $X$,
$\mu_{|X} (U) = \eta_X [\nu] (U) = \nu (U)$, so $\mu_{|X}$ is a
measure extending $\nu$.

We now invoke Corollary~3.7 of \cite{KL:measureext}: for every locally
finite continuous valuation $\nu$ on a topological space, if the
finite continuous valuations $\nu_V$ defined by
$\nu_V (U) = \nu (U \cap V)$, where $V$ ranges over the open subsets
such that $\nu (V) < +\infty$, all extend to measures on the Borel
$\sigma$-algebra, then $\nu$ also extends to a measure on the Borel
$\sigma$-algebra.  That is enough to conclude.  \qed

Together with the fact that, by the $\lambda\pi$-theorem, any two
finite measures that agree on the $\pi$-system of open sets agree on
the whole Borel $\sigma$-algebra, we obtain:
\begin{prop}[Finite measures=finite continuous valuations]
  \label{prop:mes=val}
  On an $\omega$-continuous Yoneda-complete quasi-metric space, there
  is a bijective correspondence between finite continuous valuations
  and finite measures, defined by restriction to open sets in one
  direction and by extension to the Borel $\sigma$-algebra in the
  other direction.
\end{prop}
This holds in particular for (sub)normalized continuous valuations and
(sub)probability measures.  Note that Proposition~\ref{prop:mes=val}
includes the case of $\omega$-continuous complete metric spaces, that
is, of Polish spaces, since the open ball topology on metric spaces
coincides with the $d$-Scott topology.


\subsection{The Kantorovich-Rubinshte\u\i n-Hutchinson Quasi-Metric on
  Continuous Valuations}
\label{sec:kant-rubinsht-n}

There is an integral with respect to $\nu$, defined by the Choquet
formula
$\int_{x \in X} h (x) d\nu = \int_0^{+\infty} \nu (h^{-1} (]t,
+\infty])) dt$, for every $h \in \Lform X$, where the right-hand side
is an ordinary indefinite Riemann integral.  When $\nu$ is a
continuous valuation, the map $h \mapsto \int_{x \in X} h (x) d\nu$ is
a linear prevision, which is subnormalized, resp.\ normalized, as soon
as $\nu$ is.  This is stated in \cite{Gou-csl07}, and also follows
from \cite{Jones:proba,JP:proba}.  The use of the Choquet formula is
due to Tix \cite{Tix:bewertung}.

The mapping that sends $\nu$ to its associated linear prevision is
one-to-one, and the inverse map sends every linear prevision $F$ to
$\nu$, where $\nu (U) = F (\chi_U)$.  This allows us to treat
continuous valuations, interchangeably, as linear previsions, and
conversely.  This also holds in the subnormalized and normalized
cases.

This bijection is easily turned into an isometry, by defining a
quasi-metric on spaces of continuous valuations by:
\begin{equation}
  \label{eq:V:dKRH}
  \dKRH (\nu, \nu') = \sup_{h \in \Lform_1 X} \dreal \left(\int_{x \in X} h
(x) d\nu, \int_{x \in X} h (x) d\nu'\right),
\end{equation}
and a bounded version by:
\begin{equation}
  \label{eq:V:dKRHa}
  \dKRH^a (\nu, \nu') = \sup_{h \in \Lform_1^a X} \dreal \left(\int_{x \in X} h
(x) d\nu, \int_{x \in X} h (x) d\nu'\right).
\end{equation}

We write $\Val X$ for the space of continuous valuations on $X$, and
$\Val_{\leq 1} X$ and $\Val_1 X$ for the subspaces of subnormalized,
resp.\ normalized, continuous valuations.

\begin{rem}
  \label{rem:KRH:useless}
  For any two continuous valuations $\nu$ and $\nu'$ such that
  $\nu (X) > \nu' (X)$, $\dKRH (\nu, \nu') = +\infty$.  Indeed, let
  $h = a.\mathbf 1$, $a \in \Rp$: the right-hand side of
  (\ref{eq:V:dKRH}) is larger than or equal to
  $a (\nu (X) - \nu' (X))$, hence equal to $+\infty$, by taking
  suprema over $a$.  This incongruity disappears on the subspace
  $\Val_1 X$ of normalized continuous valuations (a.k.a.,
  \emph{probability} valuations).
\end{rem}

The specialization preordering of $\dKRH$, resp.\ $\dKRH^a$, on the
space of linear previsions is the pointwise ordering, as we have
already seen: $F \leq G$ if and only if $F (h) \leq G (h)$ for every
$h \in \Lform X$ (Lemma~\ref{lemma:KRH:qmet},
Lemma~\ref{lemma:KRHa:qmet}).  On $\Val X$, $\Val_{\leq 1} X$ and
$\Val_1 X$, the specialization preordering is then given by
$\nu \leq \nu'$ if and only if $\nu (U) \leq \nu' (U)$ for every open
subset $U$ of $X$.

Note that the map $h \mapsto \int_{x \in X} h (x) d\nu$ is also
Scott-continuous.  This follows from the Choquet formula, and the fact
that Riemann integrals of non-increasing functions commute with
arbitrary pointwise suprema, as noticed by Tix \cite{Tix:bewertung}.

\subsection{Yoneda-Completeness}
\label{sec:yoneda-completeness}

Theorem~\ref{thm:LPrev:sup} and Theorem~\ref{thm:LaPrev:sup} imply:
\begin{fact}
  \label{fact:V:complete}
  $\Val X$, $\Val_{\leq 1} X$ and $\Val_1 X$ are Yoneda-complete as
  soon as $X, d$ is standard and Lipschitz regular.
\end{fact}
We shall see that they are also Yoneda-complete if $X, d$ is
continuous Yoneda-complete, and not necessarily Lipschitz regular.

For that, considering the results of Section~\ref{sec:supports}, we
shall naturally look at supports.
\begin{lem}
  \label{lemma:support}
  For a continuous valuation $\nu$ on $Y$, $A$ is a support of $\nu$
  if and only if, for all open subsets $U$, $V$ of $Y$ such that $U
  \cap A = V \cap A$, $\nu (U)=\nu (V)$.
\end{lem}
\proof
Write $F (h)$ for $\int_{x \in X} h (x) d\nu$.
If $A$ is a support of $\nu$, and $U$ and $V$ are two open subsets
such that $U \cap A = V \cap A$, then $\chi_U$ and $\chi_V$ have the
same restriction to $A$, so $\nu (U) = F (\chi_U) = F (\chi_V) = \nu
(V)$.

Conversely, assume that for all open subsets $U$, $V$ of $Y$ such that
$U \cap A = V \cap A$, $\nu (U)=\nu (V)$.  Consider two maps
$g, h \in \Lform X$ with the same restriction to $A$.  For every
$t \in \Rp$,
$g^{-1} (]t, +\infty]) \cap A = g_{|A}^{-1} (]t, +\infty]) =
h_{|A}^{-1} (]t, +\infty]) = h^{-1} (]t, +\infty]) \cap A$, so
$F (g) = \int_0^{+\infty} \nu (g^{-1} (]t, +\infty])) dt =
\int_0^{+\infty} \nu (h^{-1} (]t, +\infty])) dt = F (h)$.  \qed

We recall that every locally finite continuous valuation on a locally
compact sober space extends to a Borel measure on the Borel
$\sigma$-algebra \cite[Theorem~5.3]{KL:measureext}.
\begin{lem}
  \label{lemma:V:supp}
  Let $Y$ be a locally compact sober space, and $\nu$ be a finite
  continuous valuation on $Y$.  Let
  $W_0 \supseteq W_1 \supseteq \cdots \supseteq W_n \supseteq \cdots$
  be a non-increasing sequence of open subsets of $Y$.  If $\nu$ is
  supported on each $W_n$, $n \in \nat$, then $\nu$ is supported on
  $\bigcap_{n \in \nat} W_n$.
\end{lem}
\proof Extend $\nu$ to a measure on the Borel $\sigma$-algebra of $Y$,
and continue to write $\nu$ for the extension.  Clearly, that
extension is a finite measure, namely $\nu (Y) < +\infty$.  Let
$A = \bigcap_{n \in \nat} W_n$.  $A$ is a $G_\delta$ set, hence is
Borel.  For every open subset $U$ of $Y$, $U$ and $U \cap W_n$ have
the same intersection with $W_n$, so, using the fact that $W_n$ is a
support of $\nu$, $\nu (U ) = \nu (U \cap W_n)$.  It follows that
$\nu (U) = \inf_{n \in \nat} \nu (U \cap W_n)$.  For a finite measure
$\nu$, the latter is equal to
$\nu (\bigcap_{n \in \nat} (U \cap W_n)) = \nu (U \cap A)$.

Now, if $U$ and $V$ are any two open subsets such that $U \cap A = V
\cap A$, then $\nu (U) = \nu (U \cap A) = \nu (V \cap A) = \nu (V)$.  \qed

\begin{thm}[Yoneda-completeness, valuations]
  \label{thm:V:complete}
  Let $X, d$ be a standard quasi-metric space, and assume that it is
  either Lipschitz regular, or continuous Yoneda-complete.  Then the
  spaces $\Val_{\leq 1} X$ and $\Val_1 X$, equipped with the $\dKRH$,
  resp.\ the $\dKRH^a$ quasi-metric ($a > 0$), are Yoneda-complete.

  Moreover, directed suprema $(\nu, r)$ of formal balls
  ${(\nu_i, r_i)}_{i \in I}$ are computed as naive suprema:
  $r = \inf_{i \in I} r_i$ and for every $h \in \Lform_\alpha (X, d)$
  (resp., in $\Lform_\alpha^a (X, d)$), 
  \begin{equation}
    \label{eq:V:sup}
    \int_{x \in X} h (x) d\nu =
    \sup_{i \in I} \left(\int_{x \in X} h (x) d\nu_i + \alpha r - \alpha r_i\right).
  \end{equation}
\end{thm}
\proof When $X, d$ is standard and Lipschitz regular, this follows
from Theorem~\ref{thm:LPrev:sup} and Theorem~\ref{thm:LaPrev:sup}.

When $X, d$ is continuous Yoneda-complete, $Y = \mathbf B (X, d)$ is a
continuous dcpo.  Every continuous dcpo is locally compact and sober
in its Scott topology: in fact, the continuous dcpos are exactly the
sober c-spaces \cite[Proposition~8.3.36]{JGL-topology}, and the
property of being a c-space is a strong form of local compactness.
Moreover, all subnormalized and all normalized continuous valuations
are finite.  We can therefore apply Lemma~\ref{lemma:V:supp}: every
element of $\Val_{\leq 1} (\mathbf B (X, d))$, resp.\
$\Val_1 (\mathbf B (X, d))$, that is supported on $V_{1/2^n}$ for
every $n \in \nat$ is in fact supported on
$\bigcap_{n \in \nat} V_{1/2^n}$, which happens to be $X$
(Lemma~\ref{lemma:Veps}).  The conclusion then follows from
Proposition~\ref{prop:supp:complete}.  \qed

\subsection{Algebraicity of $\Val_{\leq 1} X$}
\label{sec:algebr-cont-val_l}

A \emph{simple valuation} is a finite linear combination
$\sum_{i=1}^n a_i \delta_{x_i}$, where $a_1, \ldots, a_n \in \Rp$.  It
is subnormalized if $\sum_{i=1}^n a_i \leq 1$, normalized if
$\sum_{i=1}^n a_i = 1$.
\begin{lem}
  \label{lemma:V:simple:center}
  Let $X, d$ be a continuous Yoneda-complete quasi-metric space.

  For all center points $x_1$, \ldots, $x_n$ and all non-negative
  reals $a_1$, \ldots, $a_n$ with $\sum_{i=1}^n a_i \leq 1$, the
  simple valuation $\sum_{i=1}^n a_i \delta_{x_i}$ is a center point
  of $\Val_{\leq 1} X, \dKRH$ (resp., of $\Val_1, \dKRH$ if
  additionally $\sum_{i=1}^n a_i = 1$).

  The same result holds with $\dKRH^a$ instead of $\dKRH$, for any
  $a \in \Rp$, $a > 0$.
\end{lem}
\proof We only deal with the case of $\Val_{\leq 1} X$, and of
$\dKRH$, the other cases are similar.  Let
$\nu_0 = \sum_{i=1}^n a_i \delta_{x_i}$,
$U = B^{\dKRH^+}_{(\nu_0, 0), <\epsilon}$.  $U$ is upwards-closed: if
$(\nu, r) \leq^{\dKRH^+} (\nu', r')$ and $(\nu, r) \in U$, then
$\dKRH (\nu, \nu') \leq r-r'$ and $\dKRH (\nu_0, \nu) < \epsilon -r$,
so $\dKRH (\nu_0, \nu') < \epsilon-r'$ by the triangular inequality,
and that means that $(\nu', r')$ is in $U$.

Recall that $\Val_{\leq 1}, \dKRH$ is Yoneda-complete, and that
directed suprema of formal balls are computed as naive suprema, by
Theorem~\ref{thm:V:complete}.  To show that $U$ is Scott-open,
consider a directed family ${(\nu_i, r_i)}_{i \in I}$ in
$\mathbf B (\Val_{\leq 1} X, \dKRH)$, with supremum $(\nu, r)$, and
assume that $(\nu, r)$ is in $U$.  That supremum is given as in
(\ref{eq:V:sup}).  Hence $r = \inf_{i \in I} r_i$ and for every
$h \in \Lform_1 (X, d)$,
$\int_{x \in X} h (x) d\nu = \sup_{i \in I} (\int_{x \in X} h (x)
d\nu_i + r - r_i)$.  (In the case of $\dKRH^a$, the same formula
holds, this time for every $h \in \Lform_1^a (X, d)$.)  Since
$(\nu, r) \in U$, $\dKRH (\nu_0, \nu) < \epsilon - r$, that is,
$\epsilon > r$ and
$\int_{x \in X} h (x) d\nu_0 - \epsilon + r < \int_{x \in X} h (x)
d\nu$.  Therefore, for every $h \in \Lform_1 (X, d)$ (resp.,
$\Lform_1^a (X, d)$), there is an index $i \in I$ such that
$\int_{x \in X} h (x) d\nu_0 - \epsilon + r < \int_{x \in X} h (x)
d\nu_i + r - r_i$, or equivalently,
\begin{equation}
  \label{eq:V:A}
  \sum_{i=1}^n a_i h (x_i) < \int_{x \in X} h (x) d\nu_i + \epsilon - r_i.
\end{equation}
Moreover, since $\epsilon > r = \inf_{i \in I} r_i$, we may take $i$
so large that $\epsilon - r_i > 0$.

Let $V_i$ be the set of all $h \in \Lform_1 (X, d)$ (resp.,
$\Lform_1^a (X, d)$) satisfying (\ref{eq:V:A}).  We have just shown
that $\Lform_1 (X, d) $(resp., $\Lform_1^a (X, d)$) is included in
$\bigcup_{i \in I} V_i$.  $V_i$ is also the inverse image of
$[0, \epsilon - r_i[$ by the map
$h \mapsto \dreal (\sum_{i=1}^n a_i h (x_i), \int_{x \in X} h (x)
d\nu_i)$, which is continuous from $\Lform_1 (X, d)^\patch$ (resp.,
$\Lform_1^a (X, d)^\patch$) to $(\creal)^\dG$ by
Proposition~\ref{prop:dKRH:cont}~(1).  Therefore $V_i$ is open in
$\Lform_1 (X, d)^\patch$ (resp.,
$\Lform_1^a (X, d)^\patch$), itself a compact space by
Lemma~\ref{lemma:cont:Lalpha:retr}~(4).

Hence there is a finite subset $J$ of $I$ such that
${(V_i)}_{i \in J}$ is also an open cover of $\Lform_1 (X, d)^\patch$
(resp., $\Lform_1^a (X, d)$).  That means that for every
$h \in \Lform_1 (X, d)$ (resp., $\Lform_1^a (X, d)$), there is an
index $i \in J$ (not just in $I$) such that (\ref{eq:V:A}) holds.  By
directedness, there is a single index
$i \in I$
such that (\ref{eq:V:A}) holds for every $h \in \Lform_1 (X, d)$
(resp., $\Lform_1^a (X, d)$).  That implies that $(\nu_i, r_i)$ is in
$U$, proving the claim.  \qed

\begin{rem}
  \label{rem:V:simple:center}
  Using the same proof, but relying on Fact~\ref{fact:V:complete}
  instead of Theorem~\ref{thm:V:complete}, we obtain that simple
  valuations supported on center points are center points of $\Val X,
  \dKRH$ (and similarly for subnormalized, resp.\ normalized
  valuations, and for $\dKRH^a$ instead of $\dKRH$), under the
  alternative assumption that $X, d$ is standard and Lipschitz regular.
\end{rem}

Our proof of algebraicity for spaces of valuations
(Theorem~\ref{thm:V:alg}) will make use of the following lemma, which
is folklore.
\begin{lem}
  \label{lemma:val:cont}
  Let $Y$ be a continuous dcpo, with a basis $\mathcal B$.  $\Val Y$
  (resp., $\Val_{\leq 1} Y$) is a continuous dcpo, and a basis is
  given by simple valuations supported on $\mathcal B$, viz., of the
  form $\sum_{i=1}^n a_i \delta_{y_i}$ with $y_i \in \mathcal B$
  (resp., and with $\sum_{i=1}^n a_i \leq 1$).
\end{lem}
\proof  The first part of the Lemma, that $\Val Y$ is a continuous dcpo, holds for general continuous dcpos
\cite[Theorem~IV-9.16]{GHKLMS:contlatt}, and is an extension of a
theorem by C. Jones that shows a similar result for $\Val_{\leq 1} Y$
\cite[Corollary~5.4]{Jones:proba}\footnote{The fact that Jones only
considers subnormalized continuous valuations is somewhat hidden.
Jones states that she considers continuous valuations as maps from
$\Open X$ to $\Rp$, not $\creal$ (Section~3.9, loc.cit.).  It is fair
to call those finite valuations.  Finite valuations do not form a
dcpo, as one can check easily.  To repair this, Jones says ``We shall
initially consider the set of continuous evaluations on any topological
space with the additional property that $\nu (X)\leq 1$'' at the
beginning of Section~4.1, loc.cit.  The new condition $\nu (X) \leq 1$
is necessary to prove that her space of continuous valuations is a
dcpo (Theorem~4.1, loc.cit.), and seems to have been assumed silently
for the rest of the thesis.}.

Every simple valuation $\sum_{i=1}^n b_i \delta_{z_i}$ on $Y$ (where,
without loss of generality, $b_i > 0$ for every $i$; and
$\sum_{i=1}^n b_i \leq 1$ in the case of $\Val_{\leq 1} Y$) is the
supremum of the family $D$ of simple valuations of the form
$\sum_{i=1}^n a_i \delta_{y_i}$, where $y_i$ ranges over the elements
of $\mathcal B$ way-below $z_i$ for each $i$, and $a_i < b_i$.
Indeed, for each open set $U$, for every
$r < (\sum_{i=1}^n b_i \delta_{z_i}) (U)$, let $A$ be the finite set
of indices $i$ such that $z_i \in U$.  Note that $A$ is non-empty.
Find $y_i \in \mathcal B$ way-below $z_i$ so that $y_i \in U$ for each
$i \in A$, and $a_i = b_i - \epsilon$ for each $i \in A$, with
$0 < \epsilon < 1/|A| \; ((\sum_{i=1}^n b_i \delta_{z_i}) (U) - r)$.
Then $r < (\sum_{i=1}^n a_i \delta_{y_i}) (U)$.

Every element of $D$ is way-below $\sum_{i=1}^n b_i \delta_{z_i}$:
Theorem~IV-9.16 of \cite{GHKLMS:contlatt}, already cited, states that
$\sum_{i=1}^n a_i \delta_{y_i} \ll \xi$, for any continuous valuation
$\xi$, if and only if for every non-empty subset $A$ of
$\{1, \cdots, n\}$,
$\sum_{i \in A} a_i < \xi (\bigcup_{i \in A} \uuarrow y_i)$, and for
$\xi = \sum_{i=1}^n b_i \delta_{z_i}$, this is obvious.  Note that
we only need the if part of Theorem~IV-9.16 of op.cit., which one can
check by elementary means.  A similar statement holds in the case
of subnormalized valuations, and we conclude similarly that
every element of $D$ is way-below $\sum_{i=1}^n b_i \delta_{z_i}$ in
the case of subnormalized valuations.

$D$ also forms a directed family: for two simple valuations
$\sum_{i=1}^n a_i \delta_{y_i}$ and $\sum_{i=1}^n a'_i \delta_{y'_i}$
in $D$, pick $y''_i \in \mathcal B$ way-below $z_i$ and above both
$y_i$ and $y'_i$ for each $i$, then
$\sum_{i=1}^n \max (a_i, a'_i) \delta_{y''_i}$ is above the two given
simple valuations, and in $D$.

In a poset, if $\xi$ is the supremum of a directed family
${(\xi_i)}_{i \in I}$, and each $\xi_i$ is the supremum of a directed
family of elements $\xi_{ij}$, $j \in J_i$, way-below $\xi_i$, then
the family ${(\xi_{ij})}_{i \in I, j \in J_i}$ is directed and admits
$\xi$ as supremum, see e.g.\ \cite[Exercise~5.1.13]{JGL-topology}.
Here we know that every continuous valuation $\xi$ is the supremum of
some directed family of simple valuations $\xi_i$ (way-)below $\xi$,
and we have just proved that each $\xi_i$ is the directed supremum of
simple valuations $\xi_{ij} \ll \xi_i$ supported on $\mathcal B$: the
result follows.  \qed

The proof of the following theorem will proceed through the study of
$\Val (\mathbf B (X, d))$.  We recall the canonical embedding
$\eta_X \colon x \mapsto (x, 0)$ of $X$, with its $d$-Scott topology,
into $\mathbf B (X, d)$, with its Scott topology.
  Every continuous valuation
$\nu$ on $X$ gives rise to an image valuation $\eta_X [\nu]$, which maps
every open subset $U$ to $\nu (\eta_X^{-1} (U))$.

For every $h \in \Lform_1 (X, d)$, and assuming again that $X, d$ is
standard, recall that $h' \colon (x, r) \mapsto h (x) - r$ is
Scott-continuous from $\mathbf B (X, d)$ to $\real \cup \{+\infty\}$.
It follows that $h'' \colon (x, r) \mapsto \max (h (x) - r, 0)$ is
Scott-continuous from $\mathbf B (X, d)$ to $\creal$.  Then we have:
\begin{lem}
  \label{lemma:h'':<}
  Let $X, d$ be a standard quasi-metric space, $\nu$ be a continuous
  valuation on $X$, and $h \in \Lform_1 (X,
  d)$.  Define $h'' (x, r) = \max (h (x) - r, 0)$.  Then:
  \begin{eqnarray*}
    \int_{(x, r) \in \mathbf B (X, d)} h'' (x, r) d \eta_X[\nu]
    & = &
    \int_{x \in X} h (x) d\nu.
  \end{eqnarray*}
\end{lem}
\proof 
$\int_{(x, r) \in \mathbf B (X, d)} h'' (x, r) d \eta_X[\nu] = \int_{x
  \in X} h'' (\eta_X (x)) d\nu$,
by a change of variables formula.  Explicitly,
$\int_{x \in X} h'' (x, r) d \eta_X [\nu] = \int_0^{+\infty} \eta_X [\nu]
({h''}^{-1} (]t, +\infty])) dt = \int_0^{+\infty} \nu (\eta_X^{-1}
({h''}^{-1} (]t, +\infty])) dt = \int_0^{+\infty} \nu ((h'' \circ
\eta_X)^{-1} (]t, +\infty])) dt = \int_{x \in X} h'' (\eta_X (x)) d\nu$.
Since $h'' \circ \eta_X = h$, we conclude.
\qed

Recall that a strong basis of a standard quasi-metric space $X, d$ is
any set $B$ of center points of $X$ such that, for every $x \in X$,
$(x, 0)$ is the supremum of a directed family of formal balls with
center points in $B$.  $X, d$ is algebraic if and only if it has a
strong basis.  The largest strong basis is simply the set of all
center points.
\begin{lem}
  \label{lemma:B:basis}
  Let $X, d$ be a standard algebraic quasi-metric space, with a strong basis
  $\mathcal B$.  Then $\mathbf B (X, d)$ is a continuous dcpo, with a
  basis consisting of the formal balls $(x, r)$ with
  $x \in \mathcal B$.  For $x \in \mathcal B$, $(x, r) \ll (y, s)$ if
  and only if $d (x, y) < r-s$.
\end{lem}
\proof Since $X, d$ is standard algebraic hence continuous,
$\mathbf B (X, d)$ is a continuous poset
\cite[Proposition~5.18]{JGL:formalballs}.  The fact that
$(x, r) \ll (y, s)$ if and only if $d (x, y) < r-s$ is true, whenever
$x$ is a center point, is also mentioned in that proposition.  Now let
$(y, s) \in \mathbf B (X, d)$.  Write $(y, s)$ as the supremum of a
directed family ${(y_i, s_i)}_{i \in I}$ of formal balls way-below
$(y, s)$.  By assumption each $(y_i, s_i)$ is the supremum of a
directed family ${(x_{ij}, r_{ij})}_{j \in J_i}$, where
$x_{ij} \in \mathcal B$.  The family
${(x_{ij}, r_{ij}+1/2^n)}_{j \in J_i, n \in \nat}$ is also directed,
and its supremum is also equal to $(y_i, s_i)$, as one easily checks
by looking at the upper bounds of the family.  Additionally, since
$d (x_{ij}, y_i) \leq r_{ij} - s_i$,
$d (x_{ij}, y_i) < r_{ij}+1/2^n - s_i$, so each
$(x_{ij}, r_{ij}+1/2^n)$ is way-below $(y_i, s_i)$.  This is exactly
what we need to conclude that
${(x_{ij}, r_{ij}+1/2^n)}_{i \in I, j \in J_i, n \in \nat}$ is a
directed family whose supremum is $(y, s)$: if, in a poset, $a$ is the
supremum of a directed family $D$ and each element $b$ of $D$ is the
supremum of a directed family $D_b$ of elements way-below $b$, then
$\bigcup_{b \in D} D_b$ is a directed family whose supremum is $a$,
see Exercise~5.1.13 of \cite{JGL-topology} for example.  \qed

\begin{thm}[Algebraicity for spaces of subprobabilities]
  \label{thm:V:alg}
  Let $X, d$ be an algebraic Yoneda-complete quasi-metric space, with
  a strong basis $\mathcal B$.

  The spaces $\Val_{\leq 1} X, \dKRH$ and $\Val_{\leq 1} X, \dKRH^a$
  (for any $a \in \Rp$, $a > 0$) are algebraic Yoneda-complete, and a
  strong basis is given by the simple valuations
  $\sum_{i=1}^n a_i \delta_{x_i}$ with $x_i \in \mathcal B$, and
  $\sum_{i=1}^n a_i \leq 1$.
\end{thm}
\proof Let $\nu \in \Val_{\leq 1} X$.  We wish to show that $(\nu, 0)$
is the supremum of a directed family of formal balls
$(\mathring\nu_i, R (\nu_i))$ below $(\nu, 0)$, and where each
$\mathring\nu_i$ is a simple valuation of the form
$\sum_{j=1}^n a_j \delta_{x_j}$, with $x_j \in \mathcal B$, and
$\sum_{j=1}^n a_j \leq 1$.

Finding $\mathring\nu_i$ and $R (\nu_i)$ is obvious if $\nu$ is the
zero valuation, so we assume for the rest of the proof that
$\nu (X) \neq 0$.

Since $X, d$ is algebraic Yoneda-complete, $\mathbf B (X, d)$ is a
continuous dcpo and the formal balls with centers in $\mathcal B$ are
a basis, by Lemma~\ref{lemma:B:basis}.

We profit from the fact that, since $\mathbf B (X, d)$ is a continuous
poset, $\Val_{\leq 1} (\mathbf B (X, d))$ is a continuous dcpo.
Lemma~\ref{lemma:val:cont} even allows us to say that a basis of the
latter consists in the simple valuations
$\mu = \sum_{j=1}^n a_j \delta_{(x_j, r_j)}$, where each $x_j$ is in
$\mathcal B$, and $\sum_{j=1}^{n_i} a_j \leq 1$.  In particular, for
every $\nu \in \Val_{\leq 1} X$, $\eta_X [\nu]$ is the supremum of a
directed family of simple valuations
$\nu_i = \sum_{j=1}^{n_i} a_{ij} \delta_{(x_{ij}, r_{ij})}$,
$i \in I$, where each $x_{ij}$ is a center point, and
$\sum_{j=1}^{n_i} a_{ij} \leq 1$.

We can require that $r_{ij} < 1$ for all $i, j$, by the following
argument.  Define $\nu'_i (U) = \nu_i (U \cap V_1)$, where
$V_1 = \{(x, r) \in \mathbf B (X, d) \mid r < 1\}$.  $V_1$ is open,
since $X, d$ is standard (Lemma~\ref{lemma:Veps}).  $\nu'_i$ is the
restriction of $\nu_i$ to $V_1$, and is again a subnormalized
valuation.  Explicitly,
$\nu'_i = \sum_{\substack{1\leq j \leq n_i\\r_{ij} < 1}} a_{ij}
\delta_{(x_{ij}, r_{ij})}$.  Note that all the radii involved in the
latter sum are strictly less than $1$.  If $\nu_i \leq \nu_{i'}$ then
$\nu'_i \leq \nu'_{i'}$, so the family ${(\nu'_i)}_{i \in I}$ is
directed as well.  Since $\nu'_i \leq \nu_i \ll \eta_X [\nu]$, this is
a family of (subnormalized) simple valuations way-below
$\eta_X [\nu]$.  Moreover, for every open subset $U$ of
$\mathbf B (X, d)$,
$\eta_X [\nu] (U) = \nu (U \cap X) = \nu (U \cap V_1 \cap X) = \eta_X
[\nu] (U \cap V_1)$ is equal to
$\sup_{i \in I} \nu_i (U \cap V_1) = \sup_{i \in I} \nu'_i (U)$, so
$\eta_X [\nu]$ is also the supremum of ${(\nu'_i)}_{i \in I}$.  All
this concurs to show that we may assume that $\nu_i$ satisfies
$r_{ij} < 1$ for all $i, j$, replacing $\nu_i$ by $\nu'_i$ if needed.

For every subnormalized simple valuation
$\mu = \sum_{j=1}^n a_j \delta_{(x_j, r_j)}$ on $\mathbf B (X, d)$
such that $\mu \leq \eta_X [\nu]$ and with $r_j < 1$ for every $j$,
let:
\begin{eqnarray}
  \label{eq:ringmu}
  \mathring\mu & = & \sum_{i=1}^n a_j \delta_{x_j} \\ 
  \label{eq:Rmu}
  R (\mu) & = & \sum_{j=1}^n a_j r_j + \nu (X) - \sum_{j=1}^n a_j.  
\end{eqnarray}
$R (\mu)$ is a non-negative number, owing to the fact that
$\mu \leq \eta_X [\nu]$: indeed
$\sum_{j=1}^n a_j = \mu (\mathbf B (X, d)) \leq \eta_X [\nu] (\mathbf
B (X, d)) = \nu (X)$.  Therefore
$\beta (\mu) = (\mathring \mu, R (\mu))$ is a well-defined formal ball
on $\Val_{\leq 1} X, \dKRH$ (resp., $\dKRH^a$).

For every $h \in \Lform_1 (X, d)$ (resp., $\Lform_1^a (X, d)$), recall
the construction $h''$ mentioned in Lemma~\ref{lemma:h'':<}.  However,
apply it to $h+\mathbf 1$, not $h$.  In other words,
$(h + \mathbf 1)''$ maps $(x, r)$ to $\max (h (x) - r +1, 0)$.  We
have
$\int_{(x, r) \in \mathbf B (X, d)} (h + \mathbf 1)'' (x, r) d\mu =
\sum_{j=1}^n a_j (h (x_j) - r_j + 1)$, because $r_j < 1$ for all
$i, j$.
Applying this to $\mu=\nu_i$ and $\mu=\nu_{i'}$, we obtain
that if $\nu_i \leq \nu_{i'}$, then
$\sum_{j=1}^{n_i} a_{ij} h (x_{ij}) - \sum_{j=1}^{n_i} a_{ij} r_j +
\sum_{j=1}^{n_i} a_{ij} \leq \sum_{j=1}^{n_{i'}} a_{i'j} h (x_{i'j}) -
\sum_{j=1}^{n_{i'}} a_{i'j} r_j + \sum_{j=1}^{n_{i'}} a_{i'j}$, namely
$\sum_{j=1}^{n_i} a_{ij} h (x_{ij}) \leq \sum_{j=1}^{n_{i'}} a_{i'j} h
(x_{i'j}) + R(\nu_i) - R(\nu_{i'})$.
This can be rewritten as
$\dreal (\int_{x \in X} h (x) d\mathring\nu_i, \int_{x \in X} h (x)
d\mathring\nu_{i'}) \leq R (\nu_i) - R (\nu_{i'})$.  Since $h$ is
arbitrary, $\beta (\nu_i) \leq^{\dKRH^+} \beta (\nu_{i'})$ (resp.,
$\leq^{\dKRH^{a+}}$%
, by multiplying $h''$ by $1/a$ first%
).  This shows that the family
${(\beta (\nu_i))}_{i \in I}$ is directed.

We now claim that $\beta (\nu_i) \leq^{\dKRH^+} (\nu, 0)$ for every
$i \in I$ (resp., $\leq^{\dKRH^{a+}}$).  Fix $h \in \Lform_1 (X, d)$
(resp., in $\Lform_1^a (X, d)$).  We wish to show that
$\sum_{j=1}^{n_i} a_{ij} h (x_{ij}) \leq \int_{x \in X} h (x) d\nu + R
(\nu_i)$.  To this end, recall that $\nu_i \ll \eta_X [\nu]$, in
particular $\nu_i \leq \eta_X [\nu]$.  Integrate $(h+\mathbf 1)''$ with
respect to each side of the inequality:
$\int_{(x, r) \in \mathbf B (X, d)} (h+\mathbf 1)'' (x, r) d\nu_i =
\sum_{j=1}^{n_i} a_{ij} h (x_{ij}) - \sum_{j=1}^{n_i} a_{ij} r_{ij} +
\sum_{j=1}^{n_i} a_{ij}$, and
$\int_{(x, r) \in \mathbf B (X, d)} (h+\mathbf 1)'' (x, r) d\eta_X[\nu]
= \int_{x \in X} h (x) d \nu + \nu (X)$, by
Lemma~\ref{lemma:h'':<}. 
Therefore
$\sum_{j=1}^{n_i} a_{ij} h (x_{ij}) \leq \int_{x \in X} h (x) d \nu +
\nu (X) + \sum_{j=1}^{n_i} a_{ij} r_{ij} - \sum_{j=1}^{n_i} a_{ij} =
\int_{x \in X} h (x) d \nu + R (\nu_i)$.

We finally claim that $\sup_{i \in I} \beta (\nu_i) = (\nu, 0)$.  Let
$R = \inf_{i \in I} R (\nu_i)$.  Let $h = \mathbf 1$, and compute
$\int_{(x, r) \in \mathbf B (X, d)} h'' (x, r) d\nu_i =
\sum_{j=1}^{n_i} a_{ij} (1 - r_{ij}) = \nu (X) - R (\nu_i)$.  Since
$\sup_{i \in I} \nu_i = \eta_X [\nu]$ and integration is
Scott-continuous in the valuation,
$\sup_{i \in I} (\nu (X) - R (\nu_i)) = \int_{(x, r) \in \mathbf B (X,
  d)} h'' (x, r) \allowbreak d\eta_X [\nu] = \int_{x \in X} h (x) d\nu =
\nu (X)$.  Hence $R=0$.

By Theorem~\ref{thm:V:complete}, directed suprema are computed as
naive suprema.  That is to say, $\sup_{i \in I} \beta (\nu_i)$ is a
formal ball $(G, R)$, with $R=0$ as we have just seen, and where
(equating $G$ with a linear prevision) $G$ maps every
$h \in \Lform_1 (X, d)$ (resp., $\Lform_1^a (X, d)$) to
$\sup_{i \in I} (\sum_{j=1}^{n_i} a_{ij} h (x_{ij}) - R (\nu_i))$.  We
have already noticed that
$\int_{(x, r) \in \mathbf B (X, d)} (h + \mathbf 1)'' (x, r) d\nu_i$
is equal to
$\sum_{j=1}^{n_i} a_{ij} h (x_{ij}) - \sum_{j=1}^{n_i} a_{ij} r_{ij} +
\sum_{j=1}^{n_i} a_{ij}$, that is, to
$\sum_{j=1}^{n_i} a_{ij} h (x_{ij}) + \nu (X) - R (\nu_i)$, so
$G (h) = \sup_{i \in I} \int_{(x, r) \in \mathbf B (X, d)} (h +
\mathbf 1)'' (x, r) d\nu_i) - \nu (X)$.  This is equal to
$\int_{(x, r) \in \mathbf B (X, d)} (h + \mathbf 1)'' (x, r) d \eta_X
[\nu] - \nu (X)$, since $\sup_{i \in I} \nu_i = \eta_X [\nu]$ and
integration is Scott-continuous in the valuation.  In turn, this is
equal to $\int_{x \in X} (h + \mathbf 1) (x) d\nu - \nu (X)$ by
Lemma~\ref{lemma:h'':<}, 
namely to $\int_{x \in X} h (x) d\nu$.  It follows that
$G (h) = \int_{x \in X} h (x) d\nu$ for every $h \in \Lform_1 (X, d)$
(resp., $\Lform_1^a (X, d)$), and this suffices to show that $G$
coincides with $h \mapsto \int_{x \in X} h (x) d\nu$, by
Corollary~\ref{corl:=:L1} (resp., Corollary~\ref{corl:=:Lbnd1}).
%
%
\qed

\subsection{Continuity of $\Val_{\leq 1} X$}
\label{sec:cont-val-leq-1}

We deduce a similar theorem for the larger class of continuous
Yoneda-complete quasi-metric spaces, by relying on the fact that the
continuous Yoneda-complete quasi-metric spaces are exactly the
$1$-Lipschitz continuous retracts of algebraic Yoneda-complete
quasi-metric spaces \cite[Theorem~7.9]{JGL:formalballs}.

Recall the map $\Prev f$ from Lemma~\ref{lemma:Pf:lip}.
\begin{lem}
  \label{lemma:Vleq1:functor}
  Let $X, d$ and $Y, \partial$ be two continuous Yoneda-complete
  quasi-metric spaces, and $f \colon X, d \mapsto Y, \partial$ be a
  $1$-Lipschitz continuous map.  The restriction of $\Prev f$ to
  $\Val_{\leq 1} X$ is a $1$-Lipschitz continuous map from
  $\Val_{\leq 1} X, \dKRH$ to $\Val_{\leq 1} Y, \KRH\partial$, and
  also from $\Val_{\leq 1} X, \dKRH^a$ to
  $\Val_{\leq 1} Y, \KRH\partial^a$ for every $a \in \Rp$, $a > 0$.

  Similarly with $\Val_1$ instead of $\Val_{\leq 1}$.
\end{lem}
\proof Lemma~\ref{lemma:Pf:lip} says that $\Prev f$ is $1$-Lipschitz,
so $\mathbf B^1 (\Prev f)$ is monotonic.  The same Lemma shows that
$\Prev f$ maps subnormalized linear previsions (i.e., elements of
$\Val_{\leq 1} X$) to subnormalized linear previsions, and normalized
linear previsions (i.e., elements of $\Val_1 X$) to normalized linear
previsions.  Again we confuse linear previsions and continuous
valuations.

By Theorem~\ref{thm:V:complete}, $\Val_{\leq 1} X, \dKRH$ and
$\Val_{\leq 1} Y, \KRH\partial$ are Yoneda-complete, and directed
suprema in their spaces of formal balls are naive suprema.  Similarly
with $\dKRH^a$ and $\KRH\partial^a$ in lieu of $\dKRH$ and
$\KRH\partial$, or with $\Val_1$ instead of $\Val_{\leq 1}$.  By
Lemma~\ref{lemma:Pf:lipcont}, $\mathbf B^1 (\Prev f)$ preserves naive
suprema, hence all directed suprema.  It must therefore be
Scott-continuous, which shows the claim.  \qed

Let $X, d$ be a continuous Yoneda-complete quasi-metric space.
There is an algebraic Yoneda-complete quasi-metric space $Y, \partial$
and
there are
two $1$-Lipschitz continuous maps $r \colon Y, \partial \to X, d$ and
$s \colon X, d \to Y, \partial$ such that $r \circ s = \identity X$.

By Lemma~\ref{lemma:Vleq1:functor}, $\Prev r$ and $\Prev s$ are also
$1$-Lipschitz continuous, and clearly
$\Prev r \circ \Prev s = \identity {\Val_{\leq 1} X}$, so
$\Val_{\leq 1} X, \dKRH$ is a $1$-Lipschitz continuous retract of
$\Val_{\leq 1} Y, \KRH\partial$.  (Similarly with $\dKRH^a$ and
$\KRH\partial^a$.)  Theorem~\ref{thm:V:alg} states that
$\Val_{\leq 1} Y, \KRH\partial$ (resp., $\KRH\partial^a$) is algebraic
Yoneda-complete, whence:
\begin{thm}[Continuity for spaces of subprobabilities]
  \label{thm:Vleq1:cont}
  Let $X, d$ be a continuous Yoneda-complete quasi-metric space.  The
  quasi-metric spaces $\Val_{\leq 1} X, \dKRH$ and
  $\Val_{\leq 1} X, \dKRH^a$ ($a \in \Rp$, $a > 0$) are continuous
  Yoneda-complete.  \qed
\end{thm}

Together with Lemma~\ref{lemma:Vleq1:functor}, and
Theorem~\ref{thm:V:alg} for the algebraic case, we obtain.
\begin{cor}
  \label{cor:V:functor}
  $\Val_{\leq 1}, \dKRH$ defines an endofunctor on the category of
  continuous Yoneda-complete quasi-metric spaces and $1$-Lipschitz
  continuous map.  Similarly with $\dKRH^a$ instead of $\dKRH$
  ($a > 0$), or with algebraic instead of continuous.  \qed
\end{cor}

\subsection{Probabilities on Metric Spaces}
\label{sec:prob-metr-spac}

We would like to embark on the study of \emph{normalized} continuous valuations, and
the space $\Val_1 X$.  The situation is more complicated than in the
sub-normalized case, but the case of \emph{metric} spaces stands out.
We deal with them here, before we examine the more general algebraic
quasi-metric cases in subsection~\ref{sec:algebr-cont-prob}.


\begin{lem}
  \label{lemma:dKRH:metric}
  Let $X, d$ be a metric space.  Then $\dKRH$ is a metric, not just a
  quasi-metric, on $\Val_1 X$.  The same holds for $\dKRH^a$, for
  every $a \in \Rp$, $a > 0$.
\end{lem}
\proof
We consider, equivalently, the space of normalized linear previsions
on $X$ instead of $\Val_1 X$.

As a matter of simplification, recall that on a metric space $X, d$,
the $d$-Scott topology coincides with the open ball topology.  As a
result, every $1$-Lipschitz map is automatically continuous, hence
$1$-Lipschitz continuous since every metric space is standard.

We start with the case of $\dKRH^a$.  For every $h \in \Lform_1^a (X,
d)$, $a-h$ is $1$-Lipschitz: for all $x, y \in X$, $(a-h (x)) - (a-h
(y)) = h (y) - h (x) \leq d (y, x) = d (x,y)$, using the fact that $h$
is $1$-Lipschitz and that $d$ is a metric.  Moreover, $a-h$ is bounded
from above by $a$, so $a-h$ is in $\Lform_1^a (X, d)$.

For every normalized linear prevision $F$ on $X$,
$F (a - h) = a - F (h)$.  Indeed, by linearity
$F (a - h) + F (h) = F (a.\mathbf 1)$, and this is equal to $a$ since
$F$ is normalized.

For all normalized linear previsions $F$, $F'$ on $X$,
$\dreal (F' (a - h), F (a - h))$ is therefore equal to
$\dreal (a - F' (h), a - F (h)) = \dreal (F (h), F' (h))$.  It follows
that for every $h \in \Lform_1^a (X, d)$, there is an
$h' \in \Lform_1^a (X, d)$, namely $h'=a-h$, such that
$\dreal (F' (h'), F (h')) = \dreal (F (h), F' (h))$.  Therefore
$\dKRH^a (F, F') \leq \dKRH^a (F', F)$.  By symmetry, we conclude that
$\dKRH^a (F, F') = \dKRH^a (F', F)$: $\dKRH^a$ is a metric.

To show that $\dKRH$ is a metric, it is enough to observe that
$\dKRH (F, F')$ is equal to $\sup_{a \in \Rp, a > 0} \dKRH^a (F, F')$
(Lemma~\ref{lemma:KRH:KRHa}), and to use the fact that $\dKRH^a$ is a
metric.  \qed

Since $\dKRH (F, F') = \dKRH (F', F)$ in that case, $\dKRH (F, F')$ is
also equal to
$\max (\dKRH (F, \allowbreak F'), \allowbreak \dKRH (F', F))$, which
is easily seen to be equal to
$\sup_{h \in \Lform_1 (X, d)} \max (\dreal (F (h), F' (h)),
\allowbreak \dreal (F' (h), F (h)))$.  The inner maximum is the
symmetrized
metric
$\dreal^{sym}$, defined by $\dreal^{sym} (a, \allowbreak b) = |a-b|$
for all $a, b \in \Rp$.  Therefore, we obtain the formula:
\begin{eqnarray}
  \label{eq:dKRH:metric}
  \dKRH (\nu, \nu') & = & \sup_{h \text{ 1-Lipschitz}} 
                          \left|\int_{x \in X} h (x) d\nu - \int_{x
                          \in X} h (x) d\nu'\right| \\
  \nonumber
  & = & \sup_{h \text{ 1-Lipschitz bounded}}
                          \left|\int_{x \in X} h (x) d\nu - \int_{x
                          \in X} h (x) d\nu'\right|,
\end{eqnarray}
for $\nu, \nu' \in \Val_1 X$, in the case where $d$ is a metric on
$X$.  (The second equality is by Lemma~\ref{lemma:KRH:bounded}.)

Similarly,
\begin{eqnarray}
  \label{eq:dKRHa:metric}
  \dKRH^a (\nu, \nu') & = & \sup_{h \text{ 1-Lipschitz bounded by $a$}}
                          \left|\int_{x \in X} h (x) d\nu - \int_{x
                          \in X} h (x) d\nu'\right|,
\end{eqnarray}
for $\nu, \nu' \in \Val_1 X$, assuming again that $d$ is a metric.   We
recognize the usual formula for the Kantorovich-Rubinshte\u\i n metric
when $a=1$.

Theorem~\ref{thm:V:complete} then states:
\begin{thm}
  \label{thm:V1:complete}
  For every complete metric space $X, d$, the space $\Val_1 X$ with
  the Kantorovich-Rubinshte\u\i n-Hutchinson metric
  (\ref{eq:dKRH:metric}), or with the $a$-bounded
  Kantorovich-Rubinshte\u\i n-Hutchinson metric (\ref{eq:dKRHa:metric}),
  $a \in \Rp$, $a > 0$, is a complete metric space.  \qed
\end{thm}
We haven't cared to mention that the simple normalized valuations
$\sum_{i=1}^n a_i \delta_{x_i}$ with $x_i$ center points in $X, d$,
$\sum_{i=1}^n a_i = 1$, are center points.  This is trivially true,
because every point of a metric space is a center point.  In fact, in
the case of Theorem~\ref{thm:V1:complete}, \emph{every} normalized
continuous valuation is a center point.

Theorem~\ref{thm:V1:complete} resembles the classical result that, if
$X, d$ is a complete separable metric space, then $\Val_1 X, \dKRH$ is
a complete (separable) metric space.  Note that we do not need $X$ to
be separable for Theorem~\ref{thm:V1:complete} to hold.

We will not show that the simple normalized valuations
$\sum_{i=1}^n a_i \delta_{x_i}$ with $x_i$ taken from a given strong
basis $\mathcal B$ also form a strong basis for the $\dKRH^a$ metric,
that is, a dense subset (Remark~\ref{rem:strong:basis}): we shall
prove that in the more general case of the $\dKRH^a$
\emph{quasi-}metric (Theorem~\ref{thm:V1:alg}).

\begin{rem}
  \label{rem:not:dense}
  As a special case of the upcoming Theorem~\ref{thm:V:weak=dScott},
  applied to complete metric spaces $X, d$, the $\dKRH^a$-Scott
  topology on $\Val_1 X$ will coincide with the weak topology.  Since
  the former is the usual open ball topology of the metric $\dKRH^a$,
  $\dKRH^a$ metrizes the weak topology on spaces of normalized
  continuous valuations on complete metric spaces.  This subsumes the
  well-known result that it metrizes the weak topology on spaces of
  probability measures on complete separable (Polish) spaces.

  However, the unbounded $\dKRH$ metric does \emph{not} metrize the
  weak topology.  For a counterexample, we reuse one due to Kravchenko
  \cite[Lemma~3.7]{Kravchenko:complete:K}.  Take any complete metric
  space $X, d$, with points ${(x_n)}_{n \in \nat \smallsetminus \{0\}}$ such
  that $n \leq d (x_0, x_n) < + \infty$.  Let
  $\nu_n = \frac 1 n \delta_{x_n} + (1-\frac 1 n) \delta_{x_0}$.  Then
  ${(\nu_n)}_{n \in \nat \smallsetminus \{0\}}$ converges to $\delta_{x_0}$ in
  the weak topology, since for every subbasic weak open set $[f > r]$
  containing $\delta_{x_0}$ (i.e., $f (x_0) > r$),
  $\int_{x \in X} f (x) d\nu_n = \frac 1 n f (x_n) + (1-\frac 1 n) f
  (x_0)$ is strictly larger than $r$ for $n$ large enough (this is
  easy if $r \geq 0$ since then $f (x_0) > 0$, and is trivial if
  $r < 0$ since $f \in \Lform X$ takes its values in $\creal$).
  However,
  $\dKRH (\delta_{x_0}, \nu_n) = \sup_{h \in \Lform_1 X} \dreal (h
  (x_0), \frac 1 n h (x_n) + (1-\frac 1 n) h (x_0)) \geq \dreal (d
  (x_0, x_n), (1-\frac 1 n) d (x_0, x_n))$ (using
  $d (\_, x_n) \in \Lform_1 (X, d)$, see Lemma~\ref{lemma:d(_,x)})
  $=\frac 1 n d (x_0, x_n) \geq 1$, which shows that
  ${(\nu_n)}_{n \in \nat \smallsetminus \{0\}}$ does not converge to
  $\delta_{x_0}$ in the open ball topology of $\dKRH$.

  We reassure ourselves in checking that
  ${(\nu_n)}_{n \in \nat \smallsetminus \{0\}}$ does converge to $\delta_{x_0}$
  in the open ball topology of $\dKRH^a$, though (this must be, since
  we announced that it would coincide with the weak topology):
  $\dKRH^a (\delta_{x_0}, \nu_n) = \sup_{h \in \Lform_1^a X} \max
  (\frac 1 n (h (x_0) - h (x_n))) \leq \frac a n$.
\end{rem}

\subsection{Algebraicity and Continuity of the Probabilistic
  Powerdomain $\Val_1 X$}
\label{sec:algebr-cont-prob}

We now turn to spaces of normalized continuous valuations in the
general case of algebraic, then continuous, 
quasi-metric spaces.  The situation is less neat than with
sub-normalized continuous valuations (Theorem~\ref{thm:V:alg}) or with
metric spaces, and we shall distinguish two settings where we can
conclude: for the bounded $\dKRH^a$ quasi-metrics
(Theorem~\ref{thm:V1:alg}), or for $\dKRH$ assuming the existence of a
so-called root in $X$ (Theorem~\ref{thm:V1:alg:root}).

\begin{thm}[Algebraicity for spaces of probabilities, $\dKRH^a$]
  \label{thm:V1:alg}
  Let $X, d$ be an algebraic Yoneda-complete quasi-metric space, with
  a strong basis $\mathcal B$.

  The space $\Val_1 X, \dKRH^a$ is algebraic Yoneda-complete, for
  every $a \in \Rp$, $a > 0$.

  All the normalized simple normalized valuations
  $\sum_{i=1}^n a_i \delta_{x_i}$ with $x_i$ center points in $X, d$,
  $\sum_{i=1}^n a_i = 1$, are center points, and form a strong basis,
  even when each $x_i$ is taken from $\mathcal B$.
\end{thm}
\proof Let $\nu \in \Val_1 X$.  We wish to exhibit a directed family
of formal balls $(\nu'_i, r'_i)$, $i \in I$, below $(\nu, 0)$, whose
supremum is $(\nu, 0)$, where each $\nu'_i$ is a simple valuation of
the form $\sum_{j=1}^n a_j \delta_{x_j}$, with $x_j$ in $\mathcal B$,
and $\sum_{j=1}^n a_j = 1$.

As in the proof of Theorem~\ref{thm:V:alg}, $\eta_X [\nu]$ is the
supremum of a directed family of simple valuations
$\nu_i = \sum_{j=1}^{n_i} a_{ij} \delta_{(x_{ij}, r_{ij})}$,
$i \in I$, where each $x_{ij}$ is in $\mathcal B$, and
$\sum_{j=1}^{n_i} a_{ij} \leq 1$.

We can require that $r_{ij} < 1$ for all $i, j$, as in the proof of
Theorem~\ref{thm:V:alg}, and we similarly define $\mathring\nu_i$ as
$\sum_{j=1}^{n_i} a_{ij} \delta_{x_{ij}}$,
$R (\nu_i) = \sum_{j=1}^{n_i} a_{ij} r_{ij} + 1 - \sum_{j=1}^{n_i}
a_{ij}$, $\beta (\nu_i) = (\mathring\nu_i, R (\nu_i))$.  Define
$i \preceq i'$ if and only if $\nu_i \leq \nu_{i'}$, for all
$i, i' \in I$.  As before, we show that:
\begin{quote}
  $(*)$ for all $i \preceq i'$ in $I$,
  $\beta (\nu_i) \leq^{\dKRH^{a+}} \beta (\nu_{i'})$.
\end{quote}
This implies that the family ${(\beta (\nu_i))}_{i \in I}$ is
directed.  We also show that $\beta (\nu_i) \leq^{\dKRH^{a+}} (\nu, 0)$ for every
$i \in I$, and that $\sup_{i \in I} \beta (\nu_i) = (\nu, 0)$.
However $\mathring\nu_i$ is not normalized, only subnormalized.

Fix a center point $x_0$.  (There must exist one: the only algebraic
quasi-metric space without a center point is the empty space, and if
$X$ is empty then $\Val_1 X$ is empty as well, so we would not have
had a $\nu \in \Val_1 X$ to start with.)  For each $i \in I$, let
$\lambda_i = 1-\mathring\nu_i (X)$,
$\nu'_i = \mathring\nu_i + \lambda_i \delta_{x_0}$, and
$r'_i = R (\nu_i) + a \lambda_i$.  Now $\nu'_i$ is normalized.

We check that ${(\nu'_i, r'_i)}_{i \in I}$ is directed.  To this end,
it is enough to show that for all $i \preceq i'$ in $I$,
$(\nu'_i, r'_i) \leq{\dKRH^{a+}} (\nu'_{i'}, r'_{i'})$.  Before we do
so, we note that if $i \preceq i'$, then
$\nu_i (\mathbf B (X, d)) \leq \nu_{i'} (\mathbf B (X, d))$, that is,
$\sum_{j=1}^{n_i} a_{ij} \leq \sum_{j=1}^{n_{i'}} a_{i'j}$;
equivalently, $\mathring\nu_i (X) \leq \mathring\nu_{i'} (X)$.
Therefore:
\begin{quote}
  $(**)$ for all $i \preceq i'$ in $I$, $\lambda_i \geq \lambda_{i'}$.
\end{quote}
We also note that for all
$i \preceq i'$ in $I$, since by $(*)$
$\dKRH^a (\mathring\nu_i, \mathring\nu_{i'}) \leq R (\nu_i) - R
(\nu_{i'})$, the inequality
$\int_{x \in X} h (x) d\mathring\nu_i \leq \int_{x \in X} h (x)
d\mathring\nu_{i'} + R (\nu_i) - R (\nu_{i'})$ holds for $h$
equal to
the constant fonction equal to $a$; hence
$a \mathring\nu_i (X) \leq a \mathring\nu_{i'} (X) + R (\nu_i) - R
(\nu_{i'})$.  This implies:
\begin{quote}
  $(*{*}*)$ for all $i \preceq i'$ in $I$, $r'_i \geq r'_{i'}$.
\end{quote}
Let now $h$ be an arbitrary element from
$\Lform_1^a (X, d)$, and assume $i \preceq i'$.  We have the
following:
\begin{eqnarray*}
  \dreal (\int_{x \in X} \mskip-20mu h (x) d\nu'_i, \int_{x \in X} \mskip-20mu h (x)
  d\nu'_{i'})
  & = & \max (\int_{x \in X} \mskip-20mu h (x) d\mathring\nu_i + h (x_0) \lambda_i
        - \int_{x \in X} \mskip-20mu h (x) d\mathring\nu_{i'} - h (x_0)
        \lambda_{i'}, 0) \\
  & \leq & \max (R (\nu_i) - R (\nu_{i'}) + h (x_0) (\lambda_i -
           \lambda_{i'}), 0) \\
  && \quad\text{since }\dKRH^a (\mathring\nu_i, \mathring\nu_{i'}) \leq R (\nu_i) - R
(\nu_{i'})\text{, using }(*) \\
  & \leq & \max (R (\nu_i) - R (\nu_{i'}) + a (\lambda_i -
           \lambda_{i'}), 0) \\
  && \quad\text{since }h (x_0) \leq a\text{, and }\lambda_i \geq
     \lambda_{i'}\text{ by }(**) \\
  & = & \max (r'_i - r'_{i'}, 0) = r'_i - r'_{i'} \quad\text{ by }(*{*}*).
\end{eqnarray*}
This shows that for all $i \preceq i'$ in $I$,
$\dKRH^a (\nu'_i, \nu'_{i'}) \leq r'_i - r'_{i'}$, hence that
$(\nu'_i, r'_i) \leq^{\dKRH^{a+}} (\nu'_{i'}, r'_{i'})$.  In
particular, ${(\nu'_i, r'_i)}_{i \in I}$ is directed.

We know that $\beta (\nu_i) \leq^{\dKRH^{a+}} (\nu, 0)$ for every
$i \in I$, and we need to show that
$(\nu'_i, r'_i) \leq^{\dKRH^{a+}} (\nu, 0)$ as well.  For every
$h \in \Lform_1^a (X, d)$, this means showing that
$\int_{x \in X} h (x) d\mathring\nu_i + h (x_0) \lambda_i \leq \int_{x
  \in X} h (x) d \nu + R (\nu_i) + a \lambda_i$.  We know that
$\int_{x \in X} h (x) d\mathring\nu_i \leq \int_{x \in X} h (x) d \nu
+ R (\nu_i)$ since $\beta (\nu_i) \leq^{\dKRH^{a+}} (\nu, 0)$, and we
conclude since $h (x_0) \leq a$.

Finally, we know that $(\nu, 0)$ is the supremum of the directed
family ${(\beta (\nu_i))}_{i \in I}$ in $\mathbf B (\Val_{\leq 1} X,
\dKRH^a)$, and we must show that it is also the supremum of the
directed family ${(\nu'_i, r'_i)}_{i \in I}$ in $\mathbf B (\Val_1 X,
\dKRH^a)$.

We claim that:
\begin{quote}
  $(\dagger)$ $\inf_{i \in I} \lambda_i = 0$.
\end{quote}
Indeed, recall that $\sup_{i \in I} \beta (\nu_i) = (\nu, 0)$.  Since
directed suprema are computed as naive suprema
(Theorem~\ref{thm:V:complete}), this means that
$\inf_{i \in I} R (\nu_i) = 0$, and that $\int_{x \in X} h (x) d\nu$
is equal to
$\sup_{i \in I} (\int_{x \in X} h (x) d\mathring\nu_i - R (\nu_i))$
for every $h \in \Lform_1^a (X, d)$.  Taking $h = a.\mathbf 1$, the
latter yields $a = \sup_{i \in I} (a \mathring\nu_i (X) - R (\nu_i))$,
or equivalently $\sup_{i \in I} (-a \lambda_i - R (\nu_i)) = 0$.
Since ${(-R (\nu_i))}_{i \in I}$ is a directed family (because
$i \preceq i'$ implies $\beta (\nu_i) \leq \beta (\nu_{i'})$, hence
$R (\nu_i) \geq R (\nu_{i'})$), and ${(-\lambda_i)}_{i \in I}$ is
directed as well by $(**)$, we may use the Scott-continuity of
addition and of multiplication by $a$, and rewrite this as
$a \sup_{i \in I} (-\lambda_i) + \sup_{i \in I} (-R (\nu_i)) = 0$,
hence to $\inf_{i \in I} \lambda _i = 0$, considering that
$\inf_{i \in I} R (\nu_i) = 0$.

Knowing this, and using again that directed suprema are naive suprema,
our task consists in showing that $\inf_{i \in I} r'_i = 0$, and that
for every $h \in \Lform_1^a (X, d)$,
$\int_{x \in X} h (x) d\nu = \sup_{i \in I} (\int_{x \in X} h (x)
d\nu'_i - r'_i)$.  The first equality is proved by rewriting
$\inf_{i \in I} r'_i$ as
$\inf_{i \in I} R (\nu_i) + a \inf_{i \in I} \lambda_i$, invoking
Scott-continuity as above, and then using the equality
$\inf_{i \in I} R (\nu_i) = 0$ and $(\dagger)$.

For the second equality,
$\sup_{i \in I} (\int_{x \in X} h (x) d\nu'_i - r'_i)$ is equal to
$\sup_{i \in I} (\int_{x \in X} h (x) d\mathring\nu_i + h (x_0)
\lambda_i - R (\nu_i) - a \lambda_i)$, which is equal to the sum of
$\sup_{i \in I} (\int_{x \in X} h (x) d\mathring\nu_i - R (\nu_i)$ and
of $(a - h (x_0)) \sup_{i \in I} (-\lambda_i)$, by Scott-continuity of
addition and of multiplication by $a - h (x_0)$, again.  The first of
those summands is equal to $\int_{x \in X} h (x) d\nu$ since
$(\nu, 0)$ is the supremum of ${(\beta (\nu_i))}_{i \in I}$ in
$\mathbf B (\Val_{\leq 1} X, \dKRH^a)$, and the second one is equal to
$0$ by $(\dagger)$.  \qed

A quasi-metric $d$ on a space $X$ is \emph{$a$-bounded} if and only if
$d (x, y) \leq a$ for all $x, y \in X$.  It is \emph{bounded} if and
only if it is $a$-bounded for some $a \in \Rp$.
\begin{rem}
  \label{rem:dKRHa:bounded}
  The quasi-metric $\dKRH^a$ is $a$-bounded on any space of previsions.
\end{rem}

\begin{lem}
  \label{lemma:dKRH=dKRHa}
  If $d$ is an $a$-bounded quasi-metric on $X$, then $\dKRH$ and
  $\dKRH^a$ coincide on any space of previsions.
\end{lem}
\proof
Clearly, $\dKRH^a \leq \dKRH$.  In order to show that $\dKRH \leq
\dKRH^a$, let $F, F'$ be any two previsions, and let us consider any
$h \in \Lform_1 (X, d)$.  We claim that there is an $h' \in \Lform_1^a
(X, d)$ such that $\dreal (F (h), F' (h)) = \dreal (F (h'), F'
(h'))$.  This will imply that $\dreal (F (h), F' (h)) \leq \dKRH^a (F,
F')$, hence, as $h$ is arbitrary, that $\dKRH (F, F') \leq \dKRH^a (F,
F')$.

For any two points $x, y \in X$,
$\dreal (h (x), h (y)) \leq d (x, y) \leq a$.  This implies that the
range of $h$ is included in an interval of length at most $a$, namely
an interval $[b, b+a]$ for some $b \in \Rp$, or $\{+\infty\}$.  In the
first case, $h' = h - b.\mathbf 1$ fits.  In the second case,
$\dreal (F (h), F' (h)) = 0$ and the constant $0$ map fits.  \qed

\begin{cor}
  \label{corl:V1:alg}
  Let $X, d$ be an algebraic Yoneda-complete quasi-metric space with
  strong basis $\mathcal B$, where $d$ is bounded.  The space
  $\Val_1 X, \dKRH$ is algebraic Yoneda-complete.  All the simple
  normalized valuations $\sum_{i=1}^n a_i \delta_{x_i}$ with
  $x_i \in \mathcal B$ and $\sum_{i=1}^n a_i = 1$, are center points,
  and form a strong basis.
\end{cor}
\proof By Lemma~\ref{lemma:dKRH=dKRHa}, $\dKRH=\dKRH^a$, where $a > 0$
is some non-negative real such that $d$ is $a$-bounded.  Now apply
Theorem~\ref{thm:V1:alg}.  \qed

Corollary~\ref{corl:V1:alg} is rather restrictive, and we give a
better result below.  The proof reuses some of the ideas used above.

In a quasi-metric space $X, d$, say that $x \in X$ is an
\emph{$a$-root} if and only if $d (x, y) \leq a$ for every $y \in X$.
As a special case, a $0$-root is an element below all others in the
$\leq^{d^+}$ ordering.  For example, $0$ is a $0$-root in $\creal$.

We call \emph{root} any $a$-root, for some $a \in \Rp$, $a > 0$.
$\real$ has no root.

\begin{lem}
  \label{lemma:root:center}
  Let $X, d$ be a standard algebraic quasi-metric space with strong
  basis $\mathcal B$, and $a \in \Rp$, $a > 0$.  If $X, d$ has an
  $a$-root, then, for every $\epsilon > 0$, it also has an
  $(a+\epsilon)$-root in $\mathcal B$.
\end{lem}
\proof Assume an $a$-root $x$.  By Lemma~\ref{lemma:B:basis}, $(x, 0)$
is the supremum of a directed family ${(x_i, r_i)}_{i \in I}$ where
each $x_i$ is in $\mathcal B$.  Since $X, d$ is standard,
$\inf_{i \in I} r_i = 0$, so there is an $i \in I$ such that
$r_i < \epsilon$.  We use the fact that $(x_i, r_i) \leq^{d^+} (x, 0)$
to infer that $d (x_i, x) \leq r_i < \epsilon$.  It follows that, for
every $y \in X$, $d (x_i, y) \leq d (x_i, x) + d (x, y) < a+\epsilon$.
\qed

\begin{thm}[Algebraicity for spaces of probabilities, $\dKRH$]
  \label{thm:V1:alg:root}
  Let $X, d$ be an algebraic Yoneda-complete quasi-metric space, with
  strong basis $\mathcal B$, and with a root $x_0$.  The space
  $\Val_1 X, \dKRH$ is algebraic Yoneda-complete.

  All the simple normalized valuations $\sum_{i=1}^n a_i \delta_{x_i}$
  where $x_i$ are center points and $\sum_{i=1}^n a_i = 1$, are center
  points, and form a strong basis, even when each $x_i$ is taken from
  $\mathcal B$.
\end{thm}
\proof By Lemma~\ref{lemma:root:center}, we may assume that $x_0$ is
also in $\mathcal B$.  Let now $a \in \Rp$, $a > 0$, be such that
$d (x_0, x) \leq a$ for every $x \in X$.

For all normalized previsions $F$ and $F'$ on $X$, for every
$h \in \Lform_1 (X, d)$, we claim that there is a map
$h' \in \Lform_1 (X, d)$ such that $h' (x_0) \leq a$, and such that
$\dreal (F (h), F' (h)) = \dreal (F (h'), F' (h'))$.  The argument is
similar to Lemma~\ref{lemma:dKRH=dKRHa}.  
For every $x \in X$, $h (x_0) \leq h (x) + d (x_0, x)$ since $h$ is
$1$-Lipschitz, so $h (x) \geq h (x_0) - a$ for every $x \in X$.  If
$h (x_0)=+\infty$, this implies that $h$ is the constant $+\infty$
map.  Then
$F (h) = F (\sup_{k \in \nat} k.\mathbf 1) = \sup_{k \in \nat} k =
+\infty$, and similarly $F' (h) = +\infty$, so
$\dreal (F (h), F' (h)) = 0$: we can take $h'=0$.  If
$h (x_0) < +\infty$, then let $h' (x) = h (x) - h (x_0) + a$ for every
$x \in X$.  Since $h (x) \geq h (x_0) - a$ for every $x \in X$,
$h' (x)$ is in $\creal$.  Clearly $h' \in \Lform_1 (X, d)$, and
$h' (x_0) = a$.

Therefore $\dKRH (F, F') = \sup_{h \in \Lform_1 (X, d), h (x_0) \leq
  a} \dreal (F (h), F' (h))$.

The proof is now very similar to Theorem~\ref{thm:V1:alg}.  Let
$\nu \in \Val_1 X$.  Then $\eta_X [\nu]$ is the supremum of a directed
family of simple valuations
$\nu_i = \sum_{j=1}^{n_i} a_{ij} \delta_{(x_{ij}, r_{ij})}$,
$i \in I$, where each $x_{ij}$ is in $\mathcal B$, and
$\sum_{j=1}^{n_i} a_{ij} \leq 1$.  We require $r_{ij} < 1$, define
$\mathring\nu_i$ as $\sum_{j=1}^{n_i} a_{ij} \delta_{x_{ij}}$,
$R (\nu_i) = \sum_{j=1}^{n_i} a_{ij} r_{ij} + 1 - \sum_{j=1}^{n_i}
a_{ij}$, and $\beta (\nu_i) = (\mathring\nu_i, R (\nu_i))$.  Define
$i \preceq i'$ if and only if $\nu_i \leq \nu_{i'}$, for all
$i, i' \in I$.  We have again: $(*)$ for all $i \preceq i'$ in $I$,
$\beta (\nu_i) \leq^{\dKRH^+} \beta (\nu_{i'})$.

Now define $\nu'_i$ as $\mathring\nu_i + \lambda_i \delta_{x_0}$, and
$r'_i = R (\nu_i) + a \lambda_i$, where
$\lambda_i = 1-\mathring\nu_i (X)$, exactly as in the proof of
Theorem~\ref{thm:V1:alg}.  We show, as before: $(**)$ for all
$i \preceq i'$ in $I$, $\lambda_i \geq \lambda_{i'}$; $(*{*}*)$ for
all $i \preceq i'$ in $I$, $r'_i \geq r'_{i'}$.

Let now $h$ be an arbitrary element from $\Lform_1 (X, d)$ such that
$h (x_0) \leq a$, and assume $i \preceq i'$.  We have the following
chain of inequalities---exactly the same as in Theorem~\ref{thm:V1:alg}:
\begin{eqnarray*}
  \dreal (\int_{x \in X} \mskip-20mu h (x) d\nu'_i, \int_{x \in X} \mskip-20mu h (x)
  d\nu'_{i'})
  & = & \max (\int_{x \in X} \mskip-20mu h (x) d\mathring\nu_i + h (x_0) \lambda_i
        - \int_{x \in X} \mskip-20mu h (x) d\mathring\nu_{i'} - h (x_0)
        \lambda_{i'}, 0) \\
  & \leq & \max (R (\nu_i) - R (\nu_{i'}) + h (x_0) (\lambda_i -
           \lambda_{i'}), 0) \\
  && \quad\text{since }\dKRH (\mathring\nu_i, \mathring\nu_{i'}) \leq R (\nu_i) - R
(\nu_{i'})\text{, using }(*) \\
  & \leq & \max (R (\nu_i) - R (\nu_{i'}) + a (\lambda_i -
           \lambda_{i'}), 0) \\
  && \quad\text{since }h (x_0) \leq a\text{, and }\lambda_i \geq
     \lambda_{i'}\text{ by }(**) \\
  & = & \max (r'_i - r'_{i'}, 0) = r'_i - r'_{i'} \quad\text{ by }(*{*}*).
\end{eqnarray*}
This shows that for all $i \preceq i'$ in $I$,
$\dKRH (\nu'_i, \nu'_{i'}) \leq r'_i - r'_{i'}$, since we have taken
the precaution to show that $\dKRH (\nu'_i, \nu'_{i'})$ is the
supremum of
$\dreal (\int_{x \in X} \mskip-20mu h (x) d\nu'_i, \int_{x \in X}
\mskip-20mu h (x) d\nu'_{i'})$ over all $h \in \Lform_1 (X, d)$
\emph{such that $h (x_0) \leq a$}.  Therefore
$(\nu'_i, r'_i) \leq^{\dKRH^+} (\nu'_{i'}, r'_{i'})$ for all
$i \preceq i'$ in $I$.  In particular, ${(\nu'_i, r'_i)}_{i \in I}$ is
directed.

The proof that $(\nu'_i, r'_i) \leq^{\dKRH^+} (\nu, 0)$ is also as in
the proof of Theorem~\ref{thm:V1:alg}.  It suffices to show that for
every $h \in \Lform_1 (X, d)$ with $h (x_0) \leq a$,
$\int_{x \in X} h (x) d\mathring\nu_i + h (x_0) \lambda_i \leq \int_{x
  \in X} h (x) d \nu + R (\nu_i) + a \lambda_i$.  We know that
$\int_{x \in X} h (x) d\mathring\nu_i \leq \int_{x \in X} h (x) d \nu
+ R (\nu_i)$ since $\beta (\nu_i) \leq^{\dKRH^{a+}} (\nu, 0)$, and we
conclude since $h (x_0) \leq a$.

Finally, we show that $(\nu, 0)$ is also the supremum of the directed
family ${(\nu'_i, r'_i)}_{i \in I}$ in $\mathbf B (\Val_1 X, \dKRH)$
by first showing: $(\dagger)$ $\inf_{i \in I} \lambda_i = 0$; then by
showing that it is the naive supremum of the family, that is, by
showing that $\inf_{i \in I} r'_i = 0$ and that for every
$h \in \Lform_1 (X, d)$,
$\int_{x \in X} h (x) d\nu = \sup_{i \in I} (\int_{x \in X} h (x)
d\nu'_i - r'_i)$.  The first equality is easy, considering
$(\dagger)$.  For the second equality,
$\sup_{i \in I} (\int_{x \in X} h (x) d\nu'_i - r'_i)$ is equal to
$\sup_{i \in I} (\int_{x \in X} h (x) d\mathring\nu_i + h (x_0)
\lambda_i - R (\nu_i) - a \lambda_i)$, which is equal to the sum of
$\sup_{i \in I} (\int_{x \in X} h (x) d\mathring\nu_i - R (\nu_i)$ and
of $(a - h (x_0)) \sup_{i \in I} (-\lambda_i)$, by Scott-continuity of
addition and of multiplication by $a - h (x_0)$, again.  (We use
$h (x_0) \leq a$ again here.)  The first of those summands is equal to
$\int_{x \in X} h (x) d\nu$ since $(\nu, 0)$ is the (naive) supremum
of ${(\beta (\nu_i))}_{i \in I}$ in
$\mathbf B (\Val_{\leq 1} X, \dKRH^a)$, and the second one is equal to
$0$ by $(\dagger)$.  \qed

The case of continuous Yoneda-complete quasi-metric spaces follows by
the same pattern as in Section~\ref{sec:cont-val-leq-1}, but we need
an additional argument in the case of spaces with a root.

Let $X, d$ be a continuous Yoneda-complete quasi-metric space.  There
is an algebraic Yoneda-complete quasi-metric space $Y, \partial$ and
two $1$-Lipschitz continuous maps $r \colon Y, \partial \to X, d$ and
$s \colon X, d \to Y, \partial$ such that $r \circ s = \identity X$.
This is again by \cite[Theorem~7.9]{JGL:formalballs}.  We need to know
that $Y, \partial$ is built as the formal ball completion of $X, d$,
defined as follows.  Let $(x, r) \prec (y, s)$ if and only if
$d (x, y) < r - s$.  A \emph{rounded ideal} of formal balls on $X, d$
is a set $D \subseteq \mathbf B (X, d)$ that is $\prec$-directed in
the sense that every finite subset of $D$ is $\prec$-below some
element of $D$, and $\prec$-downwards-closed, in the sense that any
element $\prec$-below some element of $D$ is in $D$.  The
\emph{aperture} of such a set $D$ is
$\alpha (D) = \inf \{r \mid (x, r) \in D\}$.  The elements of $Y$ are
exactly the rounded ideals in this sense that have aperture $0$.  The
quasi-metric $\partial$ is defined by
$\partial (D, D') = \sup_{b \in D} \inf_{b' \in D'} d^+ (b, b')$.
The set $\Dc (x, 0) = \{(y, r) \in \mathbf B (X, d) \mid (y, r) \prec (x,
0)\}$ is an element of $Y$ for every $x \in X$.  We need to know these
details to show:
\begin{lem}
  \label{lemma:root:completion}
  The formal ball completion of a quasi-metric space $X, d$ with an
  $a$-root $x$ ($a \in \Rp$, $a > 0$) has an $a$-root, namely
  $\Dc (x, 0)$.
\end{lem}
\proof Let $D$ be an element of the formal ball completion.  We must
show that $\partial (\Dc (x, 0), D) \leq a$, where $\partial$ is
defined as above.  For every $b = (z, t) \in \Dc (x, 0)$, by
definition $d (z, x) < t$.  For every $b' = (y, s) \in D$,
$d^+ (b, b') = \max (d (z, y) -t+s, 0) \leq \max (d (z, x) + d (x, y)
- t + s, 0) \leq \max (d (x, y) + s, 0)$.  Since $x$ is an $a$-root,
this is less than or equal to $a+s$.  We use the fact that
$\alpha (D)=0$ to make $s$ arbitrarily small.  Therefore
$\inf_{b' \in D} d^+ (b, b') \leq a$.  Taking suprema over
$b \in \Dc (x, 0)$, we obtain $\partial (\Dc (x, 0), D) \leq a$.  \qed

We return to our argument.  By Lemma~\ref{lemma:Vleq1:functor},
$\Prev r$ and $\Prev s$ are also $1$-Lipschitz continuous, and clearly
$\Prev r \circ \Prev s = \identity {\Val_1 X}$, so $\Val_1 X, \dKRH$
is a $1$-Lipschitz continuous retract of $\Val_1 Y, \KRH\partial$.
(Similarly with $\dKRH^a$ and $\KRH\partial^a$.)  By
Theorem~\ref{thm:V1:alg}, $\Val_1 Y, \KRH\partial^a$ is algebraic
Yoneda-complete.  If $X, d$ has a root, then $Y, \partial$ has a root,
too, by Lemma~\ref{lemma:root:completion}, and in that case
Theorem~\ref{thm:V1:alg:root} tells us that $\Val_1 Y, \KRH\partial$
is algebraic Yoneda-complete.  In any case, we obtain:
\begin{thm}[Continuity for spaces of probabilities]
  \label{thm:V1:cont}
  Let $X, d$ be a continuous Yoneda-complete quasi-metric space.  The
  quasi-metric space $\Val_1 X, \dKRH^a$ is continuous Yoneda-complete
  for every $a \in \Rp$, $a >0$.  If $X, d$ has a root, then
  $\Val_1 X, \dKRH$ is continuous Yoneda-complete as well.  \qed
\end{thm}

Together with Lemma~\ref{lemma:Vleq1:functor}, and
Theorem~\ref{thm:V1:alg} for the algebraic case, we obtain:
\begin{cor}
  \label{cor:V1:functor:a}
  $\Val_1, \dKRH^a$ defines an endofunctor on the category of
  continuous (resp., algebraic) Yoneda-complete quasi-metric spaces
  and $1$-Lipschitz continuous map.  \qed
\end{cor}

The case of $\dKRH$ requires the following additional lemma.
\begin{lem}
  \label{lemma:V1:root}
  For every quasi-metric space $X, d$ with an $a$-root $x$,
  ($a \in \Rp$, $a > 0$), $\Val_1 X, \dKRH$ has an $a$-root, namely
  $\delta_x$.
\end{lem}
\proof Let $\nu \in \Val_1 X$, and $h \in \Lform_1 X$.  Then:
\begin{eqnarray*}
  h (x)
  & = & \int_{y \in X} h (x) d\nu \qquad \text{since $\nu (X)=1$} \\
  & \leq & \int_{y \in X} h (y) + d (x,y) d\nu \qquad \text{since
           $h$ is $1$-Lipschitz} \\
  & \leq & \int_{y \in X} h(y) d\nu + a,
\end{eqnarray*}
since $d (x, y) \leq a$ for every $a \in X$, and $\nu (X)=1$.  It
follows that
$\dreal (\int_{y \in X} h (y) d\delta_x, \int_{y \in x} h (y) d\nu)
\leq a$.  Taking suprema over $h$, we obtain
$\dKRH (\delta_x, \nu) \leq a$.  \qed

\begin{cor}
  \label{cor:V1:functor}
  $\Val_1, \dKRH$ defines an endofunctor on the category of continuous
  (resp., algebraic) Yoneda-complete quasi-metric spaces with a root
  and $1$-Lipschitz continuous map.  \qed
\end{cor}

\begin{rem}
  \label{lemma:Vleq1:root}
  For every quasi-metric space $X, d$, $\Val_{\leq 1} X, \dKRH$ and
  $\Val X, \dKRH$ have a root, namely the zero valuation.  That holds
  even when $X, d$ does not have a root.  Indeed, $0 \leq \nu$ for
  every continuous valuation $\nu$, namely $d (0, \nu) = 0$.
\end{rem}

\subsection{The Weak Topology}
\label{sec:weak-topology}

The weak topology on $\Val_{\leq 1} X$ and on $\Val_1 X$ is defined as
a special case of what it is on general spaces of previsions, through
subbasic open sets
$[h > a] = \{\nu \in \Val_{\leq 1} X \mid \int_{x \in X} h (x) d\nu >
a\}$, $h \in \Lform X$, $a \in \Rp$.

When $X, d$ is continuous Yoneda-complete, because of
Theorem~\ref{thm:V:complete}, the assumptions of
Proposition~\ref{prop:weak:dScott:a} are satisfied, so:
\begin{fact}
  \label{fact:V:weak}
  Let $X, d$ be a continuous Yoneda-complete quasi-metric space.  We
  have the following inclusions of topologies on $\Val_{\leq 1} X$,
  resp.\ $\Val_1 X$, where $0 < a \leq a'$:
  \begin{quote}
    weak $\subseteq$ $\dKRH^a$-Scott $\subseteq$ $\dKRH^{a'}$-Scott
    $\subseteq$ $\dKRH$-Scott.
  \end{quote}
\end{fact}
We will see that the $\dKRH$-Scott topology is in general strictly
finer than the other topologies.
We will also see that the other topologies are all equal, assuming
$X, d$ algebraic.

\begin{lem}
  \label{lemma:alg:balls}
  Let $Y, \partial$ be a standard algebraic quasi-metric space, and
  $\mathcal B$ be a strong basis of $Y, \partial$.
  Fix also $\epsilon > 0$.
  Every $\partial$-Scott open subset of $Y$ is a union of open balls
  $B^\partial_{b, < r}$, where $b \in \mathcal B$ and
  $0 < r <\epsilon$.
\end{lem}
\proof Let $V$ be a $\partial$-Scott open subset of $Y$, and
$y \in V$.  Our task is to show that
$y \in B^\partial_{b, <r} \subseteq V$ for some $b \in \mathcal B$ and
$r > 0$ such that $r < \epsilon$.  By definition of the
$\partial$-Scott topology, $V$ is the intersection of the Scott-open
subset $\widehat V$ of $\mathbf B (Y, \partial)$ with $Y$.  Since
$X, d$ is standard, $V_\epsilon$ is also Scott-open by
Lemma~\ref{lemma:Veps}, and $V$ is also the intersection of
$\widehat V \cap V_\epsilon$ with $Y$.  We use
Lemma~\ref{lemma:B:basis} and the fact that $(y, 0)$ is in
$\widehat V \cap V_\epsilon$ to conclude that there is a
$b \in \mathcal B$ and an $r \in \Rp$ such that $(b, r) \ll (y, 0)$
and $(b, r)$ is in $\widehat V \cap V_\epsilon$.  Since
$(b, r) \ll (y, 0)$, $\partial (b, y) < r$, which implies in
particular that $r > 0$, and also that $y$ is in $B^\partial_{b, <r}$.
Since $(b, r)$ is in $V_\epsilon$, $r < \epsilon$.  Finally,
$B^\partial_{b, <r}$ is included in $V$: for every
$z \in B^\partial_{b, <r}$, $d (b, z) < r$, so $(b, r) \ll (z, 0)$,
and that shows that $(z, 0)$ is in
$\uuarrow (b, r) \subseteq \widehat V$.  \qed


\begin{prop}
  \label{prop:V:weak=dScott:alg}
  Let $X, d$ be an algebraic Yoneda-complete quasi-metric space.  For
  every $a \in \Rp$, $a > 0$, the $\dKRH^a$-Scott topology coincides
  with the weak topology on $\Val_{\leq 1} X$, and on $\Val_1 X$.
\end{prop}
\proof The $\dKRH^a$-Scott topology is finer by
Fact~\ref{fact:V:weak}.  Conversely, let $\nu \in \Val_{\leq 1} X$
(resp., $\Val_1 X$) and $\mathcal U$ be a $\dKRH^a$-Scott open
neighborhood of $\nu$.  We wish to show that $\nu$ is in some weak
open subset of $\mathcal U$.  By Theorem~\ref{thm:V:alg} (resp.,
Theorem~\ref{thm:V1:alg}), $\Val_{\leq 1} X$ (resp., $\Val_1 X$) is
algebraic, and we can then apply Lemma~\ref{lemma:alg:balls}, taking
$\mathcal B$ equal to the subset of simple (subnormalized, resp.\
normalized) valuations $\sum_{i=1}^n a_i \delta_{x_i}$ where each
$x_i$ is a center point.  Therefore we can assume that
$\mathcal U = B^{\dKRH^a}_{\sum_{i=1}^n a_i \delta_{x_i}, < \epsilon}$
for some such simple valuation $\sum_{i=1}^n a_i \delta_{x_i}$ and
some $\epsilon > 0$.

By assumption, $\dKRH^a (\sum_{i=1}^n a_i \delta_{x_i}, \nu) <
\epsilon$.  Let $\eta > 0$ be such that $\dKRH^a (\sum_{i=1}^n a_i
\delta_{x_i}, \nu) < \epsilon - \eta$.  Let $N$ be a natural number such
that $a/N < \eta$, and consider the collection $\mathcal H$ of maps of
the form $\bigvee_{i=1}^n \sea {x_i} {b_i}$ (see
Proposition~\ref{prop:dKRH:max:d}) where each $b_i$ is an integer
multiple of $a/N$ in $[0, a]$.  Note that $\mathcal H$ is a finite
family, and that for each $h \in \mathcal H$, $\sum_{i=1}^n a_i h
(x_i) < +\infty$.

Let $\mathcal V$ be the weak open
$\bigcap_{h \in \mathcal H} [h > \sum_{i=1}^n a_i h (x_i) - \epsilon +
\eta]$.  (We extend the notation $[h > b]$ to the case where $b < 0$
in the obvious way, as the set of continuous valuations $\nu'$ such
that $\int_{x \in X} h (x) d\nu' > b$; when $b < 0$, this is the whole
set, hence is again open.)  For every $h \in \mathcal H$, $h$ is in
$\Lform_1^a (X, d)$: it is in $\Lform_1 (X, d)$ by
Lemma~\ref{lemma:sea:cont}, and clearly bounded from above by $a$.
Since
$\dKRH^a (\sum_{i=1}^n a_i \delta_{x_i}, \nu) < \epsilon - \eta$,
$\sum_{i=1}^n a_i h (x_i) < \int_{x \in X} h (x) d \nu + \epsilon -
\eta$.  Hence $\nu$ is in $\mathcal V$.

Next, we show that $\mathcal V$ is included in
$\mathcal U = B^{\dKRH^a}_{\sum_{i=1}^n a_i \delta_{x_i}, <
  \epsilon}$.  Let $\nu'$ be an arbitrary element of $\mathcal V$.  By
Proposition~\ref{prop:dKRH:max:d}, there are numbers
$b'_i \in [0, a]$, $1\leq i \leq n$, such that
$\dKRH^a (\sum_{i=1}^n a_i \delta_{x_i}, \nu') = \dreal (\sum_{i=1}^n
a_i h' (x_i), \int_{x \in X} h' (x) d\nu')$, where
$h' = \bigvee_{i=1}^n \sea {x_i} {b'_i}$.  For each $i$, let $b_i$ be
the largest integer multiple of $a/N$ below $b'_i$, and let
$h = \bigvee_{i=1}^n \sea {x_i} {b_i}$.  Since $h$ is in $\mathcal H$,
and $\nu' \in \mathcal V$,
$\sum_{i=1}^n a_i h (x_i) < \int_{x \in X} h (x) d\nu' + \epsilon -
\eta$.  For each $i$, $b'_i \leq b_i + a/N \leq b_i + \eta$.  It
follows that, for every $x \in X$,
$(\sea {x_i} {b'_i}) (x) = \max (b'_i - d (x_i, x), 0) \leq \max (b_i
+ \eta - d (x_i, x), 0) \leq \max (b_i - d (x_i, x), 0) + \eta = (\sea
{x_i} {b_i}) (x) + \eta$.  In turn, we obtain that for every
$x \in X$, $h' (x) \leq h (x) + \eta$, so
$\sum_{i=1}^n a_i h' (x_i) \leq \sum_{i=1}^n a_i h (x_i) + \eta$,
using the fact that $\sum_{i=1}^n a_i \leq 1$.  Therefore
$\sum_{i=1}^n a_i h' (x_i) < \int_{x \in X} h (x) d\nu' + \epsilon$.
This implies that
$\dKRH^a (\sum_{i=1}^n a_i \delta_{x_i}, \nu') < \epsilon$.  Therefore
$\nu'$ is in $\mathcal U$.  \qed

\begin{rem}[$\dKRH$ does not quasi-metrize the weak topology]
  \label{rem:Kravchenko}
  The similar result where $\dKRH^a$ is replaced by $\dKRH$ is wrong.
  In other words, there are algebraic Yoneda-complete quasi-metric
  spaces $X, d$ such that the $\dKRH$-Scott topology is strictly finer
  than the weak topology, on $\Val_{\leq 1} X$, and on $\Val_1 X$.  We
  can even choose $X, d$ metric: see Remark~\ref{rem:not:dense} for a
  counterexample, due to Kravchenko, remembering that since $\dKRH$ is
  a metric in that case, the open ball topology coincides with the
  $\dKRH$-Scott topology.
\end{rem}

We now use the following fact:
\begin{fact}
  \label{fact:retract:two}
  If $A$ is a space with two topologies $\mathcal O_1$ and
  $\mathcal O_2$, and both embed into a topological space $B$ by the
  same topological embedding, then $\mathcal O_1 = \mathcal O_2$.
  \qed
\end{fact}
We shall apply this when $A$ is actually a retract, which is a
stronger condition than just a topological embedding.  We shall also
use the following.
\begin{fact}
  \label{fact:Pf:weak}
  Let $f \colon X \to Y$ be a continuous map between topological
  spaces.  Then $\Prev f$ is continuous from $\Prev X$ to $\Prev Y$,
  both spaces being equipped with the weak topology.
\end{fact}
Indeed, $\Prev f^{-1} ([k > b]) = [k \circ f > b]$.

\begin{thm}[$\dKRH^a$ quasi-metrizes the weak topology]
  \label{thm:V:weak=dScott}
  Let $X, d$ be a continuous Yoneda-complete quasi-metric space.  For
  every $a \in \Rp$, $a > 0$, the $\dKRH^a$-Scott topology coincides
  with the weak topology on $\Val_{\leq 1} X$, and on $\Val_1 X$.
\end{thm}
\proof We reduce to the case of algebraic Yoneda-complete quasi-metric
spaces.

We invoke \cite[Theorem~7.9]{JGL:formalballs} again: $X, d$ is the
$1$-Lipschitz continuous retract of an algebraic Yoneda-complete
quasi-metric space $Y, \partial$.  Call $s \colon X \to Y$ the section
and $r \colon Y \to X$ the retraction.  Then $\Prev s$ and $\Prev r$
form a $1$-Lipschitz continuous section-retraction pair by
Lemma~\ref{lemma:Pf:lip}, and in particular $\Prev s$ is an embedding
of $\Val_{\leq 1} X$ into $\Val_{\leq 1} Y$ with their $\dKRH^a$-Scott
topologies (similarly with $\Val_1$).  However, $s$ and $r$ are also
just continuous, by Proposition~\ref{prop:cont}, so $\Prev s$ and
$\Prev r$ also form a section-retraction pair between the same spaces,
this time with their weak topologies, by Fact~\ref{fact:Pf:weak}.  By
Proposition~\ref{prop:V:weak=dScott:alg}, the two topologies on
$\Val_{\leq 1} Y$ (resp., $\Val_1 Y$) are the same.
Fact~\ref{fact:retract:two} then implies that the two topologies on
$\Val_{\leq 1} X$ (resp., $\Val_1 X$) are the same as well.
\qed

\subsection{Splitting Lemmas}

This section will not be used in the following, and can be skipped if
you have no specific interest in it.

Jones showed that, on a dcpo $X$, given a finite subset $A$,
$\sum_{x \in A} a_x \delta_x \leq \sum_{y \in A} b_y \delta_y$ if and
only if there is a matrix ${(t_{xy})}_{x, y \in A}$ of non-negative
reals such that: $(i)$ if $x \not\leq y$ then $t_{xy}=0$, $(ii)$
$\sum_{y \in A} t_{xy} = a_x$, and $(iii)$
$\sum_{x \in A} t_{xy} \leq b_y$.  This is Jones' \emph{splitting
  lemma} \cite[Theorem~4.10]{Jones:proba}.  Additionally, if the two
valuations are normalized, the the inequality $(iii)$ can be taken to
be an equality.

The splitting lemma for simple normalized valuations can be cast in an
equivalent way as: $\mu = \sum_{x \in A} a_x \delta_x$ is less than or
equal to $\nu = \sum_{y \in A} b_y \delta_y$ if and only if there is a
simple normalized valuation
$\tau = \sum_{x, y \in A} t_{xy} \delta_{(x, y)}$ on $X^2$, supported
on the graph of $\leq$ $(i)$, whose first marginal
$\pi_1 [\tau] = \Val (\pi_1) (\tau) = \sum_{x \in A} (\sum_{y \in A}
t_{xy}) \delta_x$ is equal to $\mu$ $(ii)$ and whose second marginal
$\pi_2 [\tau] = \Val (\pi_2) (\tau) = \sum_{y \in A} (\sum_{x \in A}
t_{xy}) \delta_y$ is equal to $\nu$ $(iii)$.

This theorem can be seen as a special case of results
of Edwards \cite{Edwards:marginals}, which says that, in certain
classes of ordered topological spaces, $\mu \leq \nu$ for probability
measures $\mu$, $\nu$ if and only if there is a probability measure
$\tau$ on the product space, supported on the graph of $\leq$, with
first marginal equal to $\mu$ and with second marginal equal to $\nu$.

\begin{lem}
  \label{lemma:to:<=}
  Let $X, d$ be a quasi-metric space.  For two continuous valuations
  $\mu$ and $\nu$ on $X$, write $\xymatrix{\mu \ar[r]^w & \nu}$ if and
  only if there is a center point $x \in X$, a point $y \in X$, and a
  real number $t \in \Rp$ such that $\mu \geq t \delta_x$, $\nu = \mu
  + t \delta_y - t \delta_x$, and $w = t d (x, y)$.  (We agree that $0
  . (+\infty) = 0$ in case $t=0$ and $d (x, y)=+\infty$.)

  If $\xymatrix{\mu \ar[r]^w & \nu}$, then $\dKRH (\mu, \nu) = w$.
\end{lem}
In more elementary terms: going from $\mu$ to $\nu$ by moving an
amount $t$ of mass from $x$ to $y$ makes you travel distance $t d (x,
y)$.

\proof For every bounded $h \in \Lform_1 (X, d)$,
$\int_{z \in X} h (z) d(\nu + t \delta_x) = \int_{z \in X} h (z) d
(\mu + t \delta_y)$, which implies that
$\int_{z \in X} h (z) d\nu + t h (x) = \int_{z \in X} d\mu + t h (y)$,
hence
$\dreal (\int_{z \in X} h (x) d\mu, \int_{z \in X} h (x) d\nu) = \max
(t (h (x) - h (y)), 0)$.  (The latter makes sense since $h$ is
bounded.)  Since $h$ is $1$-Lipschitz,
$\dreal (\int_{z \in X} h (x) d\mu, \int_{z \in X} h (x) d\nu) \leq t
d (x, y)$.  Taking suprema over all bounded maps
$h \in \Lform_1 (X, d)$, and using Lemma~\ref{lemma:KRH:bounded},
$\dKRH (\mu, \nu) \leq t d (x, y) = w$.

Conversely, if $d (x, y) < +\infty$, then let $b = d (x, y)$ and
consider $h = \sea x b$.  Then $h (x) = d (x, y)$ and $h (y) = 0$.
The function $h$ $h$ is in $\Lform_1 (X, d)$, by
Lemma~\ref{lemma:sea:cont}, and is bounded by $b$.  The equality
$\int_{z \in X} h (z) d\nu + t h (x) = \int_{z \in X} d\mu + t h (y)$
then rewrites as
$\dreal (\int_{z \in X} h (x) d\mu, \int_{z \in X} h (x) d\nu) = \max
(t (h (x) - h (y)), 0) = t d (x, y)$.  So
$\dKRH (\mu, \nu) \geq t d (x, y) = w$, whence $\dKRH (\mu, \nu) = w$.

If $d (x, y) = +\infty$, then let $b \in \Rp$ be arbitrary and
consider again $h = \sea x b$.  Again, $h$ is in $\Lform_1^b (X, d)$,
but now $h (x) = b$, $h (y) = 0$.  We obtain
$\dreal (\int_{z \in X} h (x) d\mu, \int_{z \in X} h (x) d\nu) = \max
(t (h (x) - h (y)), 0) = tb$, so $\dKRH (\mu, \nu) \geq tb$ for every
$b \in \Rp$.  If $t > 0$, it follows that
$\dKRH (\mu, \nu) = +\infty = w$.  Otherwise, we know already that
$\dKRH (\mu, \nu) \leq w = 0$, so $\dKRH (\mu, \nu) = 0 = w$.  \qed

\begin{rem}
  \label{rem:to:<=}
  Let $a \in \Rp$, $a > 0$.  A similar argument shows that under the
  same assumptions that $x$ is a center point, $\mu \geq t \delta_x$,
  and $\nu = \mu + t \delta_y - t \delta_x$, then
  $\dKRH^a (\mu, \nu) = t \min (a, d (x,y))$.

  Showing $\dKRH (\mu, \nu) \leq t \min (a, d (x, y))$ means showing
  that $\dKRH (\mu, \nu) \leq t d (x, y)$, which is done as above, and
  $\dKRH (\mu, \nu) \leq t a$, which follows from the fact that
  $h (x) - h (y) \leq a$ for every $h \in \Lform_1^a (X, d)$.  To show
  the reverse inequality, we take
  $h = \sea x a \in \Lform_1^a (X, d)$.  Then $h (x) = a$,
  $h (y) = \max (a - d (x, y), 0) = a - \min (a, d (x, y))$, so that
  $\dKRH (\mu, \nu) \geq \max (t (h (x) - h (y)), 0) = \min (a, d (x,
  y))$.
\end{rem}

For two simple valuations $\mu = \sum_{x \in A} a_x \delta_x$,
$\nu = \sum_{y \in B} b_y \delta_y$, say that
$T = {(t_{xy})}_{x \in A, y \in B}$ is a \emph{transition matrix} from
$\mu$ to $\nu$ if and only if:
\begin{enumerate}
\item for every $y \in B$, $\sum_{x \in A} t_{xy} = b_y$;
\item for every $x \in A$, $\sum_{y \in B} t_{xy} = a_x$.
\end{enumerate}
The \emph{weight} of $T$ is $\sum_{x \in A, y \in B} t_{xy} d (x, y)$.

\begin{lem}
  \label{lem:txy:<=}
  Let $X, d$ be a quasi-metric space, and
  $\mu = \sum_{x \in A} a_x \delta_x$,
  $\nu = \sum_{y \in B} b_y \delta_y$ be two simple valuations on $X$.
  The weight of any transition matrix from $\mu$ to $\nu$ is larger
  than or equal to $\dKRH (\mu, \nu)$.
\end{lem}
\proof Let $h \in \Lform_1 (X, d)$.  Then:
\begin{eqnarray*}
  \int_{x \in X} h (x) d\mu
  & = & \sum_{x \in A} a_x h (x) \\
  & = & \sum_{x \in A, y \in B} t_{xy} h (x) \\
  & \leq & \sum_{x \in A, y \in B} t_{xy} (d (x, y) + h (y))
           \quad\text{since $h$ is $1$-Lipschitz} \\
  & = & \sum_{x \in A, y \in B} t_{xy} d (x, y) + \sum_{x \in A, y \in
        B} t_{xy} h (y) \\
  & = & \sum_{x \in A, y \in B} t_{xy} d (x, y) + \sum_{y \in
        B} b_y h (y) \\
  & = & \sum_{x \in A, y \in B} t_{xy} d (x, y) + \int_{y \in X} h (y)
  d\nu,
\end{eqnarray*}
from which the result follows.  \qed


Refine the notation $\xymatrix{\mu \ar[r]^w & \nu}$, and define
$\xymatrix{\mu \ar[r]^w_{A, B} & \nu}$ if and only if we can pick $x$
from $A$ and $y$ from $B$; explicitly, if and only if there is a
center point $x \in A$, a point $y \in B$, and a real number
$t \in \Rp$ such that $\mu \geq t \delta_x$,
$\nu = \mu + t \delta_y - t \delta_x$, and $w = t d (x, y)$.
\begin{lem}
  \label{lemma:V:<=:simple}
  Let $X, d$ be a standard quasi-metric space, $A$ be a non-empty
  finite set of center points of $X, d$, and $B$ be a non-empty finite
  set of points of $X$.

  For all simple normalized valuations
  $\mu = \sum_{x \in A} a_x \delta_x$ and
  $\nu = \sum_{y \in B} b_y \delta_y$, and $r \in \Rp$, the following
  are equivalent:
  \begin{enumerate}
  \item $\dKRH (\mu, \nu) \leq r$;
  \item for every map $f \colon A \cup B \to \Rp$ such that
    $f (x) \leq f (y) + d (x, y)$ for all $x, y \in A \cup B$,
    $\sum_{x \in A} a_x f (x) \leq \sum_{y \in B} b_y f (y) + r$;
  \item for all maps $f_1 \colon A \to \Rp$ and $f_2 \colon B \to \Rp$
    such that $f_1 (x) \leq f_2 (y) + d (x, y)$ for all $x \in A$ and
    $y \in B$,
    $\sum_{x \in A} a_x f_1 (x) \leq \sum_{y \in B} b_y f_2 (y) + r$;
  \item there is a transition matrix ${(t_{xy})}_{x \in A, y \in B}$
    from $\mu$ to $\nu$ of weight
    $\sum_{x \in A, y \in B} t_{xy} d (x, y) \leq r$;
  \item there is a finite path
    $\mu = \xymatrix{\mu_0 \ar[r]^{w_1}_{A, B} & \mu_1
      \ar[r]^{w_2}_{A, B} & \cdots \ar[r]^{w_n}_{A, B} & \mu_n} =
    \nu$, with $\sum_{i=1}^n w_i \leq r$.
  \end{enumerate}
\end{lem}
\proof $1 \limp 2$.  Let $f \colon A \cup B \to \Rp$ be such that
$f (x) \leq f (y) + d (x, y)$ for all $x, y \in A \cup B$ (i.e., $f$
is $1$-Lipschitz).  Build
$h = \bigvee_{x \in A} \sea x {f (x)} \colon X \to \creal$.  By
Lemma~\ref{lemma:sea:min}, $h$ is the smallest $1$-Lipschitz map such
that $h (x) \geq f (x)$ for every $x \in A$, and is in
$\Lform_1 (X, d)$.  By $1$,
$\int_{x \in X} h (x) d\mu \leq \int_{y \in X} h (y) d\nu + r$, namely
$\sum_{x \in A} a_x h (x) \leq \sum_{y \in A} b_y h (y) + r$.

We claim that $h (z) \leq f (z)$ for every $z \in A \cup B$.
That reduces to showing that for every $x \in A$, $(\sea x {f (x)})
(z) \leq f (z)$, namely that $\max (f (x) - d (x, z), 0) \leq f (z)$.
Since $f$ is $1$-Lipschitz, $f (x) - d (x, z) \leq f (z)$, which
concludes the proof of the claim.

For every $x \in A$, we have just seen that $h (x) \leq f (x)$.
Recalling that we also have $h (x) \geq f (x)$, it follows that
$h (x) = f (x)$.  Therefore
$\sum_{x \in A} a_x f (x) \leq \sum_{y \in A} b_y h (y) + r$.  Since
$h (y) \leq f (y)$ for every $y \in B$, we conclude that
$\sum_{x \in A} a_x f (x) \leq \sum_{y \in A} b_y f (y) + r$.

$2 \limp 3$.  Assume that $f_1 (x) \leq f_2 (y) + d (x, y)$ for all
$x \in A$ and $y \in B$, and let
$f (z) = \inf_{y \in B} f_2 (y) + d (z, y)$ for every
$z \in A \cup B$.  This formula defines an element of $\Rp$ since $B$
is non-empty.  By assumption, $f_1 (x) \leq f (x)$ for every
$x \in A$.  Since we can take $y=z$ in the formula defining $f$,
$f (y) \leq f_2 (y)$ for every $y \in B$.  Finally, since for all
$z, z' \in A$, $d (z, y) \leq d (z, z') + d (z', y)$,
$f (z) \leq d (z, z') + f (z')$, so $f$ is $1$-Lipschitz on
$A \cup B$.  By $2$,
$\sum_{x \in A} a_x f (x) \leq \sum_{y \in B} b_y f (y) + r$.  Since
$f_1 (x) \leq f (x)$ for every $x \in A$ and $f (y) \leq f_2 (y)$ for
every $y \in B$,
$\sum_{x \in A} a_x f_1 (x) \leq \sum_{y \in B} b_y f_2 (y) + r$.

$3 \limp 4$.  We use linear programming duality.  (Using the max-flow
min-cut theorem would suffice, too.)  Recall that a (so-called primal)
linear programming problem consists in maximizing
$Z = \sum_{j=1}^n c_j X_j$, where every $X_j$ varies, subject to
constraints of the form $\sum_{j=1}^n \alpha_{ij} X_j \leq \beta_i$
($1\leq i \leq m$) and $X_j \geq 0$ ($1\leq j\leq n$).  The dual
problem consists in minimizing $W = \sum_{i=1}^m \beta_i Y_i$ subject
to $\sum_{i=1}^m \alpha_{ij} Y_i \geq c_j$ ($1\leq j\leq n$) and
$Y_i \geq 0$ ($1\leq i\leq m$).  The two problems are dual, meaning
that $Z$ is always less than or equal to $W$, and that if the maximum
of the $Z$ values is different from $+\infty$, then it is equal to the
minimum of the $W$ values.  In particular, if $Z \leq r$ for all
values of the variables $X_j$ satisfying the constraints, then
$W \leq r$ for some tuple of values of the variables $Y_i$ satisfying
the constraints.

In our case, $j$ ranges over the disjoint sum $A+B$ (up to a
bijection), and $i$ enumerates the pairs $(x, y) \in A \times B$ such
that $d (x, y) < +\infty$.  For $j=x\in A$, $X_j=f_1(x)$ and
$c_j=a_x$; for $j=y\in B$, $X_j=f_2(y)$ and $c_j=-b_y$.  For
$i=(x,y) \in A \times B$ such that $d (x,y) < +\infty$, the associated
constraint $\sum_{j=1}^n \alpha_{ij} X_j \leq \beta_i$ reads
$f_1 (x) \leq f_2 (y) + d (x, y)$; that is, $\beta_i = d (x, y)$, and
$\alpha_{ij}=1$ if $j=x\in A$, $\alpha_{ij} = -1$ if $j=y\in B$, and
$\alpha_{ij}=0$ otherwise.

The dual variables $Y_i$ can then be given values which we write
$t_{xy}$, where $i = (x, y)$ with $d (x, y) < +\infty$, in such a way
that $W$, that is,
$\sum_{(x, y) \in A \times B, d (x, y) < +\infty} t_{xy} d (x, y)$, is
less than or equal to $r$, and the dual constraints are satisfied.
Those are: for every $x \in A$,
$\sum_{y \in B, d (x, y) < +\infty} t_{xy} \geq a_x$; for every
$y \in B$, $\sum_{x \in A, d (x, y) < +\infty} -t_{xy} \geq -b_y$; and
$t_{xy} \geq 0$ for all $(x, y)$ such that $d (x, y) < +\infty$.  To
simplify notations, define $t_{xy}$ as $0$ when $d (x, y) = +\infty$.
Then $\sum_{x \in A, y \in B} t_{xy} d (x, y) \leq r$, which establishes
$(iii)$, and: $(i')$ $\sum_{y \in B} t_{xy} \geq a_x$ for every
$x \in A$, $(ii')$ $\sum_{x \in A} t_{xy} \leq b_y$ for every
$y \in B$.

For every $x \in A$, let
$\epsilon_x = \sum_{y \in B} t_{xy} - a_x \geq 0$.  We have
$\sum_{x \in A} \epsilon_x = \sum_{x \in A, y \in B} t_{xy} - \sum_{x
  \in A} a_x = \sum_{x \in A, y \in B} t_{xy} - 1$ (since $\mu$ is
normalized)
$= \sum_{y \in B} \sum_{x \in A} t_{xy} - 1 \leq \sum_{y \in B} b_y -
1 \leq 0$, using $(ii')$.  Since the numbers $\epsilon_x$ are
non-negative and sum to at most $0$, they are all equal to $0$,
showing that $\sum_{y \in B} t_{xy} = a_x$ for every $x \in A$.  To
show that $\sum_{x \in A} t_{xy} = b_y$ for every $y \in B$, we reason
symmetrically, by letting
$\eta_y = b_y - \sum_{x \in A} t_{xy} \geq 0$, and realizing that
$\sum_{y \in B} \eta_y = 1 - \sum_{x \in A, y \in B} t_{xy} \leq 1 -
\sum_{x \in A} a_x = 0$, using $(i')$, we show that every $\eta_y$ is
equal to $0$.

$4 \limp 5$.  Enumerate the pairs $(x, y)$ such that $x \in A$ and
$y \in B$ in any order.  It is easy to see that one can go from $\mu$
to $\nu$ by a series of transitions $\xymatrix{\ar[r]^w_{A, B} &}$,
one for each such pair $(x, y)$, where we transfer some mass $t_{xy}$
from $x$ to $y$.  In that case, $w = t_{xy} d (x, y)$.  Each of these
numbers $w$ is different from $+\infty$, since
$\sum_{x \in A, y \in B} t_{xy} d (x, y) \leq r$.  Moreover the sum of
these numbers $w$ is less than or equal to $r$.

$5 \limp 1$.  This follows from Lemma~\ref{lemma:to:<=} and the
triangular inequality.  \qed

\begin{cor}
  \label{corl:V:<=:simple}
  Let $X, d$ be a standard quasi-metric space, $A$ be a non-empty
  finite set of center points of $X, d$, and $B$ be a non-empty finite
  set of points of $X$.

  For all simple normalized valuations
  $\mu = \sum_{x \in A} a_x \delta_x$ and
  $\nu = \sum_{y \in B} b_y \delta_y$, if
  $\dKRH (\mu, \nu) < +\infty$, then:
  \begin{enumerate}
  \item there is a transition matrix ${(t_{xy})}_{x \in A, y \in B}$
    from $\mu$ to $\nu$ of weight
    $\sum_{x \in A, y \in B} t_{xy} d (x, y)$ equal to
    $\dKRH (\mu, \nu)$;
  \item there is a finite path
    $\mu = \xymatrix{\mu_0 \ar[r]^{w_1}_{A, B} & \mu_1
      \ar[r]^{w_2}_{A, B} & \cdots \ar[r]^{w_n}_{A, B} & \mu_n} =
    \nu$, with $\dKRH (\mu, \nu) = \sum_{i=1}^n w_i$.
  \end{enumerate}
\end{cor}
\proof Take $r = \dKRH (\mu, \nu)$ in
Lemma~\ref{lemma:V:<=:simple}~(4): then
$\sum_{x \in A, y \in B} t_{xy} d (x, y) \leq \dKRH (\mu, \nu)$.  The
reverse inequality is by Lemma~\ref{lem:txy:<=}.  Now use
Lemma~\ref{lemma:V:<=:simple}~(5) instead: there is a finite path
$\mu = \xymatrix{\mu_0 \ar[r]^{w_1}_{A, B} & \mu_1 \ar[r]^{w_2}_{A, B}
  & \cdots \ar[r]^{w_n}_{A, B} & \mu_n} = \nu$, with
$\sum_{i=1}^n w_i \leq \dKRH (\mu, \nu)$.  By Lemma~\ref{lemma:to:<=}
and the triangular inequality,
$\dKRH (\mu, \nu) \leq \sum_{i=1}^n w_i$, from which we obtain the
equality.  \qed

We have a similar result for the bounded KRH metric $\dKRH^a$.
The resulting formula is slightly less elegant.
\begin{lem}
  \label{lemma:V:<=:simple:a}
  Let $X, d$ be a standard quasi-metric space, $a > 0$, $A$ be a
  non-empty finite set of center points of $X, d$, and $B$ be a
  non-empty finite set of points of $X$.

  For all simple normalized valuations
  $\mu = \sum_{x \in A} a_x \delta_x$ and
  $\nu = \sum_{y \in B} b_y \delta_y$, and $r \in \Rp$, the following
  are equivalent:
  \begin{enumerate}
  \item $\dKRH^a (\mu, \nu) \leq r$;
  \item for every map $f \colon A \cup B \to [0, a]$ such that
    $f (x) \leq f (y) + d (x, y)$ for all $x, y \in A \cup B$,
    $\sum_{x \in A} a_x f (x) \leq \sum_{y \in B} b_y f (y) + r$;
  \item for all maps $f_1 \colon A \to [0, a]$ and
    $f_2 \colon B \to [0, a]$ such that
    $f_1 (x) \leq f_2 (y) + d (x, y)$ for all $x \in A$ and $y \in B$,
    $\sum_{x \in A} a_x f_1 (x) \leq \sum_{y \in B} b_y f_2 (y) + r$;
  \item there is a matrix ${(t_{xy})}_{x \in A, y \in B}$, and two
    vectors ${(u_x)}_{x \in A}$ and ${(v_y)}_{y \in B}$ with
    non-negative entries such that:
    \begin{enumerate}
    \item $\sum_{x \in A, y \in B} t_{xy} d (x, y) + a \sum_{x \in A}
      u_x + a \sum_{y \in B} v_y \leq r$,
    \item for every $x \in A$, $\sum_{y \in B} t_{xy} + u_x \geq a_x$,
    \item for every $y \in B$, $\sum_{x \in A} t_{xy} - v_y \leq b_y$.
    \end{enumerate}
  \end{enumerate}
\end{lem}
$1 \limp 2$.  Given $f$ as in (1), we build
$h = \bigvee_{x \in A} \sea x {f (x)} \colon X \to \creal$.  We then
reason as in Lemma~\ref{lemma:V:<=:simple}, additionally noting that
$h$ takes its values in $[0, a]$.  Indeed, for every $x \in A$, for
every $z$,
$(\sea x {f (x)}) (z) = \max (f (x) - d (x, z), 0) \leq f (x) \leq a$.

$2 \limp 3$.  We reason again as in Lemma~\ref{lemma:V:<=:simple}, but
we define $f (z)$ instead as
$\min (a, \inf_{y \in B} f_2 (y) + d (z, y))$, so that $f$ takes its
values in $[0, a]$.  For every $x \in A$, we obtain that
$f_1 (x) \leq \inf_{y \in B} f_2 (y) + d (x, y)$ as before, and
$f_1 (x) \leq a$ by assumption, so $f_1 (x) \leq f (x)$.  We have
$f (y) \leq f_2 (y)$ for every $y \in B$ by taking $z=y$ in the
formula defining $f$.  The rest is as before.

$3 \limp 4$.  We again use linear programming duality.  We now have
the extra constraints $f_1 (x) \leq a$ for each $x \in A$ and
$f_2 (y) \leq a$ for each $y \in B$.  For that, we must create new
variables $u_x$, $x \in A$, and $v_y$, $y \in A$, and the dual problem
now reads: minimize
$W = \sum_{x \in A, y \in B} t_{xy} d (x, y) + \sum_{x \in A} a u_x +
\sum_{y \in B} a v_y$ under the constraints
$\sum_{y \in B} t_{xy} + u_x \geq a_x$, $x \in A$, and
$\sum_{x \in A} -t_{xy} + v_y \geq -b_y$.

$4 \limp 1$.  Let $h \in \Lform^a_1 (X, d)$.  Then:
\begin{eqnarray*}
  \int_{x \in X} h (x) d\mu
  & = & \sum_{x \in A} a_x h (x) \\
  & \leq & \sum_{x \in A} u_x h (x) + \sum_{x \in A, y \in B} t_{xy} h
           (x) \qquad \text{by (b)} \\
  & \leq & \sum_{x \in A} u_x h (x) + \sum_{x \in A, y \in B} t_{xy} h
           (y) + \sum_{x \in A, y \in B} t_{xy} d (x, y)
           \qquad \text{since $h$ is $1$-Lipschitz} \\
  & \leq & \sum_{x \in A} u_x h (x) + \sum_{y \in B} v_y h (y) +
           \sum_{y \in B} b_y h (y) + \sum_{x \in A, y \in B} t_{xy} d (x, y)
           \qquad \text{by (c)} \\
  & \leq & a \sum_{x \in A} u_x + a \sum_{y \in V} v_y  +
           \sum_{y \in B} b_y h (y) + \sum_{x \in A, y \in B} t_{xy} d
           (x, y) \\
  &&   \qquad \text{since $h$ is bounded by $a$ from above} \\
  & \leq & r + \sum_{y \in B} b_y h (y) \qquad \text{by (a)} \\
  & = & r + \int_{x \in X} h (x) d\nu.
\end{eqnarray*}
\qed

Those splitting lemmas can be seen as variants of the
Kantorovich-Rubinshte\u\i n Theorem \cite{Kantorovich:1942}.  The
latter says that, on a complete separable metric space $X, d$ (i.e., a
Polish metric space), for any two probability measures $\mu$, $\nu$,
$\dKRH (\mu, \nu) = \min_\tau \int_{(x, y) \in X^2} d (x, y) d\tau$,
where $\tau$ ranges over the probability measures on $X^2$ whose first
marginal is equal to $\mu$ and whose second marginal is equal to
$\nu$.  See \cite{Fernique:KR} for an elementary proof.

This is what Lemma~\ref{corl:V:<=:simple} says, in the more general
asymmetric setting where we have a quasi-metric, not a metric, but
applied to the less general situation where $\mu$ and $\nu$ are
\emph{simple} normalized valuations and $\mu$ is supported on a set of
center points.  Translating that lemma, we obtain that
$\dKRH (\mu, \nu)$ is equal to the infimum of all the integrals
$\int_{(x, y) \in X^2} d (x, y) d \tau$, where $\tau$ ranges over
certain simple normalized valuations with first marginal $\mu$ and
second marginal $\nu$.  Moreover, the infimum is attained in case
$\dKRH (\mu, \nu) < +\infty$.

\subsection{Linear Duality}
\label{sec:minimal-transport}

We wish to show a general linear duality theorem, in the style of
Kantorovich and Rubinshte\u\i n.  That would say that, for any
$\mu, \nu \in \Val_1 X$,
$\dKRH (\mu, \nu) = \min_\tau \int_{(x, y) \in X^2} d (x, y) d\tau$,
where $\tau \in \Val_1 (X^2)$ satisfies $\pi_1 [\tau] = \mu$ and
$\pi_2 [\tau] = \nu$.

That runs into a silly problem: the map $(x, y) \mapsto d (x, y)$ is
not lower semicontinuous, unless $d$ is a metric.  Indeed, the map $y
\mapsto d (x, y)$ is even antitonic (with respect to $\leq^d$),
whereas lower semicontinuous maps are monotonic.

One may think of taking $\tau$ to be a probability measure, and that
would probably solve part of the issue, but one would need to show
that $(x, y) \mapsto d (x, y)$ is a measurable map, where $X^2$ comes
with the Borel $\sigma$-algebra over the $d^2$-Scott topology, where
$d^2$ is introduced below.

Instead, we observe that each probability distribution
$\mu \in \Val_1 (X)$ gives rise to a linear prevision
$F_\mu \colon h \mapsto \int_{x \in X} h (x) d\mu$, then to a linear,
normalized monotonic map
$\extF {F_\mu} \colon L_\infty (X, d) \to \creal$ whose restriction
to $\Lform_\infty (X, d)$ is lower semicontinuous, by
Lemma~\ref{lemma:extF:prev}, assuming $X, d$ continuous
Yoneda-complete.  Recall that $L_\infty (X, d)$ is the space of
Lipschitz, not necessarily continuous, maps from $X, d$ to $\creal$.
We write $L_\infty^\bnd (X, d)$ for the subspace of bounded Lipschitz
maps.

Since $F_\mu$ is normalized, it restricts to a map from
$L_\infty^\bnd (X, d)$ to $\Rp$, that is, not taking the value
$+\infty$.
\begin{defi}[$L^\pm (X, d)$]
  \label{defn:L+-}
  Let $X, d$ be a quasi-metric space.  The maps of the form $f-g$,
  where $f, g \in L_\infty^\bnd (X, d)$ form a real vector space
  $L^\pm (X, d)$.
\end{defi}
\begin{rem}
  \label{rem:L+-:op}
  $L^\pm (X, d)$ not only contains $L_\infty^\bnd (X, d)$, but also
  $L_\infty^\bnd (X, d^{op})$.  Indeed, for every bounded
  $\alpha$-Lipschitz map $g \colon X, d^{op} \to \real$, $-g$ is a
  bounded $\alpha$-Lipschitz map from $X, d$ to $\real$, and therefore
  $g = 0 - (-g)$ is in $L^\pm (X, d)$.
\end{rem}

We can now extend $\extF {F_\mu}$ to integrate functions taking real
values, not just non-negative values.  This is ${\extF F}^\pm_\mu$,
defined below.
\begin{lem}
  \label{lemma:L+-:extF}
  Let $X, d$ be a continuous quasi-metric space.  For every monotonic
  linear map $F \colon L_\infty^\bnd (X, d) \to \real$, there is a
  unique monotonic linear map $F^\pm \colon L^\pm (X, d) \to \real$
  such that $F^\pm (h) = F (h)$ for every
  $h \in L_\infty^\bnd (X, d)$.
\end{lem}
That $F^\pm$ is linear should be taken in the usual
vector-space-theoretic sense: it preserves sums and products by
arbitrary real numbers.  That $F$ is linear should be taken in the
usual sense of this paper: it preserves sums and products by
non-negative real numbers.

\proof
Necessarily $F^\pm (f-g) =  F (f) -  F
(g)$ for all $f, g \in L_\infty^\bnd (X, d)$, which shows uniqueness.

To show existence, we first note that for all
$f, g \in L_\infty^\bnd (X, d)$, say $f$ and $g$ take their values in
$[0, a]$, then $ F (f)$ and $ F (g)$ are in
$[0, a]$, and are certainly different from $+\infty$.  Therefore the
difference $ F (f) -  F (g)$ makes sense, and
is a real number.

If $f-g \leq f'-g'$, where $f, f', g, g' \in L_\infty^\bnd (X, d)$, we
claim that $F (f) - F (g) \leq F (f') - F (g')$.  Indeed, this is
equivalent to $ F (f) + F (g') \leq F (g) + F (f')$, which follows
from $f+g' \leq g+f'$ and the fact that $F$ is monotonic and linear.

From that claim we obtain that $f-g=f'-g'$ implies
$F (f) - F (g) = F (f') - F (g')$, whence defining $F^\pm (f-g)$ as
$ F (f) - F (g)$ for all $f, g \in L_\infty^\bnd (X, d)$ is
non-ambiguous.  It also follows that $F^\pm$ is monotonic.

$F^\pm$ is then clearly linear.  In particular, for $a \in \Rp$,
$F^\pm (a (f-g)) = F (af) - F (ag) = a F (f) - a F (g) = a F^\pm
y(f-g)$, and for $a \in \real$ negative,
$F^\pm (a (f-g)) = F^\pm ((-a) (g-f)) = F ((-a) g) - F ((-a) f) = (-a)
F (g) - (-a) F (f) = a F (g) - a F (g) = a F^\pm (f-g)$.  \qed

Let $d^{op}$ be the opposite quasi-metric of $d$, defined by
$d^{op} (x, y) = d (y, x)$.  We equip $X \times X$ with the
quasi-metric $d^2$ defined by
$d^2 ((x, y), (x', y')) = d (x, x') + d (y', y) = d (x, x') + d^{op}
(y, y')$---note the reversal of the $y$ components.
\begin{lem}
  \label{lemma:d2:lip}
  Let $X, d$ be a quasi-metric space.  The distance map
  $d \colon X \times X, d^2 \to \creal$ is $1$-Lipschitz.
\end{lem}
\proof It suffices to show that
$d (x, y) \leq d (x', y') + d^2 ((x, y), (x', y'))$, which easily
follows from the triangular inequality.  \qed

\begin{defi}[Coupling]
  \label{defn:coupling}
  Let $X, d$ be a continuous quasi-metric space, and $F$, $G$ be two
  linear, monotonic maps from $L_\infty^\bnd (X, d)$ to $\real$.  A
  \emph{coupling} between $F$ and $G$ is a linear, monotonic map
  $\Gamma \colon L_\infty^\bnd (X \times X, d^2) \to \real$ such that
  $\Gamma (\mathbf 1) = 1$, $\Gamma (f \circ \pi_1) = F (f)$ for every
  $f \in L_\infty^\bnd (X, d)$, and $\Gamma (g \circ \pi_2) = G^\pm (g)$
  for every $g \in L_\infty^\bnd (X, d^{op})$.

  By extension, we say that $\Gamma$ is a coupling between normalized
  continuous valuations $\mu$ and $\nu$ if and only if $\Gamma$ is a
  coupling between $\extF {F_\mu}$ and $\extF {F_\nu}$.
\end{defi}
Note that $G^\pm (g)$ makes sense by Remark~\ref{rem:L+-:op}.  It is
easy to see that $\pi_1 \colon X \times X, d^2 \to X, d$ and
$\pi_2 \colon X \times X, d^2 \to X, d^{op}$ are $1$-Lipschitz, so
that $f \circ \pi_1$ and $g \circ \pi_2$ are Lipschitz from
$X \times X, d^2$ to $\real$, hence $\Gamma (f \circ \pi_1)$ and
$\Gamma (g \circ \pi_2)$ make sense.

Here is a trivial example of coupling, which is an analogue of the
product measure between $F$ and $G$.
\begin{lem}
  \label{lemma:coupling:x}
  Let $X, d$ be a continuous quasi-metric space, and $F$, $G$ be two
  linear, monotonic maps from $L_\infty^\bnd (X, d)$ to $\creal$.  The map
  $F \ltimes G$ defined by
  $(F \ltimes G) (h) = F (x \mapsto G^\pm (y \mapsto h (x, y)))$ for
  every $h \in L_\infty^\bnd (X \times X, d^2)$ is a coupling between $F$
  and $G$.
\end{lem}
\proof The map $y \mapsto h (x, y)$ is Lipschitz from $X, d^{op}$ to
$\real$, hence is in $L^\pm (X, d)$ by Remark~\ref{rem:L+-:op}.
Therefore $G^\pm (y \mapsto h (x, y))$ makes sense.  We claim that the
map $x \mapsto G^\pm (y \mapsto h (x, y))$ is $\alpha$-Lipschitz,
assuming that $h$ is $\alpha$-Lipschitz.  Indeed, for every $y$,
$h (x, y) \leq h (x', y) + \alpha d (x, x')$, so
$G^\pm (y \mapsto h (x, y)) \leq G^\pm (y \mapsto h (x', y) + \alpha d
(x, x')) = G^\pm (y \mapsto h (x', y)) + \alpha d (x, x')$, using the
fact that $G^\pm$ maps $\mathbf 1$ to $1$.  It follows that our
definition of $F \ltimes G$ makes sense.

$F \ltimes G$ is linear and monotonic because $F$ and $G^\pm$ are.
$(F \ltimes G) (\mathbf 1) = F (x \mapsto G^\pm (y \mapsto 1)) = F (x
\mapsto 1) = 1$.  For every $f \in L_\infty^\bnd (X, d)$,
$(F \ltimes G) (f \circ \pi_1) = F (x \mapsto G^\pm (y \mapsto f (x)))
= F (x \mapsto f (x)) = F (f)$, and for every
$g \in L_\infty^\bnd (X, d^{op})$,
$(F \ltimes G) (g \circ \pi_2) = F (x \mapsto G^\pm (y \mapsto g (y)))
= F (x \mapsto G^\pm (g)) = G^\pm (g)$.  \qed

\begin{rem}
  \label{rem:coupling:alt}
  The map $F \rtimes G \colon h \mapsto G^\pm (y \mapsto F (x \mapsto
  h (x, y)))$ is also a coupling between $F$ and $G$.
\end{rem}

Lemma~\ref{lemma:d2:lip} allows us to make sense of $\Gamma (d)$,
for any coupling $\Gamma$, where $d$ is the underlying metric.

\begin{lem}
  \label{lemma:coupling:<=}
  Let $X, d$ be a continuous quasi-metric space, and $\mu, \nu \in
  \Val_1 X$.  For every coupling $\Gamma$ between $\mu$ and $\nu$,
  $\dKRH (\mu, \nu) \leq \Gamma (d)$.
\end{lem}
\proof Using Lemma~\ref{lemma:KRH:bounded}, it suffices to show that
for every bounded map $h$ in $\Lform_1 (X, d)$,
$F_\mu (h) \leq F_\nu (h) + \Gamma (d)$.  Since $\pi_1$ is
$1$-Lipschitz, $h \circ \pi_1$ is $1$-Lipschitz, and since $\pi_2$ is
$1$-Lipschitz from $X \times X, d^2$ to $X, d^{op}$, $-h \circ \pi_2$
is $1$-Lipschitz from $X \times X, d^2$ to $\real$.  By
Lemma~\ref{lemma:d2:lip}, the map $(h \circ \pi_2) + d$ is then in
$L^\pm (X \times X, d^2)$.

We observe that $(h \circ \pi_1) \leq (h \circ \pi_2) + d$, in other
words, for all $x, y \in X$, $h (x) \leq h (y) + d (x, y)$: this is
because $h$ is $1$-Lipschitz.

Since $\Gamma$ is monotonic and linear, $\Gamma^\pm$ is monotonic and
linear (Lemma~\ref{lemma:L+-:extF}), so
$\Gamma^\pm (h \circ \pi_1) \leq \Gamma^\pm (h \circ \pi_2) +
\Gamma^\pm (d)$.

Since $h \circ \pi_1$ is in $L_1 (X, d)$, $\Gamma^\pm (h \circ \pi_1) =
\Gamma (h \circ \pi_1)$.  This is equal to $\extF {F_\mu} (h)$ by the
definition of a coupling, and then to $F_\mu (h)$ by
Lemma~\ref{lemma:extF:prev}~(3).

Let $a$ be some upper bound on the values of $h$.  Then
$\Gamma^\pm (h \circ \pi_2) = a - \Gamma^\pm (a.\mathbf 1 - h \circ
\pi_2)$.  Since $-h \circ \pi_2$ is $1$-Lipschitz,
$a.\mathbf 1 - h \circ \pi_2$ is also $1$-Lipschitz, so
$\Gamma^\pm (a.\mathbf 1 - h \circ \pi_2) = \Gamma (a.\mathbf 1 - h
\circ \pi_2) = \Gamma (g\circ \pi_2)$ where $g = a.\mathbf 1 - h$.  By
the definition of a coupling, and since $g$ is $1$-Lipschitz from
$X, d^{op}$ to $\Rp$,
$\Gamma (g \circ \pi_2) = \extF {F_\nu}^\pm (g)$.  Therefore
$\Gamma^\pm (h \circ \pi_2) = a - \extF {F_\nu}^\pm (g) = \extF
{F_\nu}^\pm (a.\mathbf 1 - g) = \extF {F_\nu}^\pm (h)$.  This is equal
to $\extF {F_\nu} (h)$ since $h \in L_1 (X, d)$, hence to $F_\nu (h)$
by Lemma~\ref{lemma:extF:prev}~(3).  \qed

\begin{prop}
  \label{prop:coupling:=}
  Let $X, d$ be a continuous quasi-metric space, and
  $\mu, \nu \in \Val_1 X$.  There is a coupling $\Gamma$ between $\mu$
  and $\nu$ such that $\dKRH (\mu, \nu) = \Gamma (d)$.
\end{prop}
\proof Let $B$ be the vector space $L^\pm (X \times X, d^2)$.  This is
a normed vector space with $||h|| = \sup_{(x, y) \in X \times X} |h (x, y)|$.
 
For all $f, g \in L_\infty^\bnd (X, d)$, the map $f \ominus g$ defined
as mapping $(x, y)$ to $f (x) - g (y)$, is Lipschitz from
$X \times X, d^2$ to $\real$.  Indeed, if $f$ and $g$ are
$\alpha$-Lipschitz, then $f (x) \leq f (x') + \alpha d (x, x')$ and
$g (y') \leq g (y) + \alpha d (y', y)$, so then
$f (x) - g (y) \leq f (x') + \alpha d (x, x') - g (y') + \alpha d (y',
y) = (f (x') - g (y')) + \alpha d^2 ((x, y), (x', y'))$.

In particular, the maps $f \ominus g$, with
$f, g \in L_\infty^\bnd (X, d)$, are all in $B$.  It is easy to see
that the subspace $A$ generated by those maps are the maps of the form
$f \ominus g$, where $f$ and $g$ are now taken from $L^\pm (X, d)$.

Define a map $\Lambda \colon A \to \real$ by
$\Lambda (f \ominus g) = \extF {F_\mu}^\pm (f) - \extF {F_\nu}^\pm
(g)$.  If $f \ominus g = f' \ominus g'$, then for all $x, y \in X$,
$f (x) - g (y) = f' (x) - g' (y)$, so $f=f'$ and $g=g'$.  It follows
that $\Lambda$ is defined unambiguously.  It is easy to see that
$\Lambda$ is linear.

We claim that $|\Lambda (f \ominus g)| \leq ||f \ominus g||$.  To show
that, we consider a function similar to the coupling
$\extF {F_\mu} \ltimes \extF {F_\nu}$ of Lemma~\ref{lemma:coupling:x}.
Let $\Phi (h)$ be defined as
$\extF {F_\mu}^\pm (x \mapsto \extF {F_\nu}^\pm (y \mapsto h (x, y)))$
for every $h \in B$.  Then:
\begin{eqnarray*}
  \Phi (f \ominus g)
  & = & \extF {F_\mu}^\pm (x \mapsto \extF {F_\nu}^\pm (y \mapsto f (x) -
        g (y))) \\
  & = & \extF {F_\mu}^\pm (x \mapsto f (x) - \extF {F_\nu}^\pm (g)) \\
  & = & \extF {F_\mu}^\pm (f) - \extF {F_\nu}^\pm (g) = \Lambda (f \ominus g).
\end{eqnarray*}
Also, $|\Phi (h)| \leq ||h||$ for every $h \in B$.  Indeed,
$\Phi (h) \leq \extF {F_\mu}^\pm (x \mapsto \extF {F_\nu}^\pm (y
\mapsto ||h||))) = \extF {F_\mu}^\pm (x \mapsto ||h||) = ||h||$, and
similarly for $-\Phi (h) = \Phi (-h)$.  As a consequence,
$|\Lambda (f \ominus g)| = |\Phi (f \ominus g)| \leq ||f \ominus g||$.

By the Hahn-Banach Theorem, $\Lambda$ has a linear extension $\Gamma$
to the whole of $B$, such that $\Gamma (h) \leq ||h||$ for every
$h \in B$.

We note that $\Gamma (\mathbf 1)=1$.  Indeed, $\mathbf 1$ can be
written as $f \ominus g$ with $f = \mathbf 1$ and $g = 0$, whence
$\Gamma (\mathbf 1) = \Lambda (f \ominus g) = 1$.

We claim that, for every $h \in B$ that takes its values in $\Rp$,
then $\Gamma (h) \geq 0$.  Otherwise, there is an $h \in B$ that takes
its values in $\Rp$ such that $\Gamma (h) < 0$.  By multiplying by an
appropriate positive scalar, we may assume that $||h|| \leq 1$.  Then
$\Gamma (\mathbf 1-h) \leq ||\mathbf 1-h|| \leq 1$, although
$\Gamma (\mathbf 1-h) = 1 - \Gamma (h) > 1$.

It follows that $\Gamma$ is monotonic: if $h \leq h'$, then $h' - h$
takes its values in $\Rp$, so $\Gamma (h' - h) \geq 0$, and that is
equivalent to $\Gamma (h) \leq \Gamma (h')$.

Finally, for every $f \in L_\infty^\bnd (X, d)$,
$\Gamma (f \circ \pi_1) = \Gamma (f \ominus 0) = \Lambda (f \ominus 0)
= \extF {F_\mu}^\pm (f) = \extF {F_\mu} (f)$, and for every
$g \in L_\infty^\bnd (X, d^{op})$,
$\Gamma (g \circ \pi_2) = \Gamma (- (0 \ominus g)) = - \Lambda (0
\ominus g) = \extF {F_\nu}^\pm (g)$.  \qed

As a corollary of Lemma~\ref{lemma:coupling:<=} and
Proposition~\ref{prop:coupling:=}, we obtain the following.
\begin{thm}[Coupling]
  \label{thm:coupling}
  Let $X, d$ be a continuous quasi-metric space, and
  $\mu, \nu \in \Val_1 X$.  Then:
  \begin{equation}
    \label{eq:coupling}
    \dKRH (\mu, \nu) = \min_\Gamma \Gamma (d),
  \end{equation}
  where $\Gamma$ ranges over the couplings between $\mu$ and $\nu$.  \qed
\end{thm}
This holds, in particular, for every metric space, since every metric
space is continuous.


\section{The Hoare Powerdomain}
\label{sec:hoare-powerdomain}

\subsection{The $\dH$ Quasi-Metric}
\label{sec:hoare-quasi-metric}

For every open subset $U$ of $X$ (with its $d$-Scott topology), and
every point $x \in X$, we can define the distance from $x$ to the
complement $\overline U$ of $U$ by
$d (x, \overline U) = \sup \{r \in \Rp \mid (x, r) \in \widehat U\}$
\cite[Definition~6.10]{JGL:formalballs}, where $\widehat U$ is the
largest open subset of $\mathbf B (X, d)$ such that
$\widehat U \cap X \subseteq U$ (equivalently,
$\widehat U \cap X = U$).  Given that $X, d$ is standard, the
following laws hold (Lemma~6.11, loc.cit.): $d (x, \overline U)=0$ if
and only if $x \not\in U$;
$d (x, \overline U) \leq d (x, y) + d (y, \overline U)$; the map
$d (\_, \overline U)$ is $1$-Lipschitz continuous from $X, d$ to
$\creal, \dreal$.  Additionally,
$d (x, \overline U) \leq \inf_{y \in \overline U} d (x, y)$, with
equality if $\overline U = \dc E$ for some finite set $E$, or if $x$
is a center point (Proposition~6.12, loc.cit.).

An equivalent definition, working directly with closed sets, is as
follows.  For a closed subset $C$ of $X$, let $\cl (C)$ be its closure
in $\mathbf B (X, d)$---the smallest closed subset of $\mathbf B (X, d)$
that contains $C$.  If $U$ is the complement of $C$, then $\cl (C)$ is
the complement of $\widehat U$ in $\mathbf B (X, d)$.  Then
$d (x, C) = \sup \{r \in \Rp \mid (x, r) \not\in \cl (C)\}$, and
$d (x, C)=0$ if and only if $x \in C$.
\begin{lem}
  \label{lemma:dxC}
  Let $X, d$ be a quasi-metric space.  For every $x \in X$ and every
  closed subset $C$ of $X$,
  $d (x, C) = \inf \{r \in \Rp \mid (x, r) \in \cl (C)\}$, where the
  infimum of an empty set of elements of $\Rp$ is taken to be
  $+\infty$.

  Moreover,
  \begin{enumerate}
  \item the infimum is attained if $d (x, C) < +\infty$;
  \item for every $r \in \Rp$, $d (x, C) \leq r$ if and only if $(x,
    r) \in \cl (C)$.
  \end{enumerate}
\end{lem}
\proof Let $A = \{r \in \Rp \mid (x, r) \not\in \cl (C)\}$,
$B = \{r \in \Rp \mid (x, r) \in \cl (C)\}$.  $A$ and $B$ partition
$\Rp$, and since $\cl (C)$ is downwards-closed, $A$ is
downwards-closed, $B$ is upwards-closed, and for all $a \in A$,
$b \in B$, we must have $a < b$.  It follows that $\sup A = \inf B$,
proving the first claim.

$B$ is also closed under (filtered) infima.  Indeed, for any filtered
family ${(r_i)}_{i \in I}$ of elements of $B$ with
$\inf_{i \in I} r_i = r$, ${(x, r_i)}_{i \in I}$ is a directed family
in $\mathbf B (X, d)$, whose supremum is $(x, r)$.  (The upper bounds
of ${(x, r_i)}_{i \in I}$ are those formal balls $(y, s)$ such that
$d (x, y) \leq r_i - s$ for every $i \in I$, namely the elements above
$(x, r)$.)  That implies that $r$ is in $B$.

It follows that, if $B$ is non-empty, then it is an interval of the
form $[b, +\infty[$, where $b = \inf B = d (x, C)$.  This shows (1).

For (2), if $d (x, C) \leq r \in \Rp$, then $B$ is non-empty and $r
\geq b$, so $r$ is in $B$, that is, $(x, r) \in \cl (C)$.  Conversely,
if $(x, r) \in \cl (C)$, then $r$ is in $B$, so $B$ is non-empty and $r
\geq b$, meaning that $d (x, C) \leq r$.  \qed

In the following, we shall need to know a few elementary facts about
the closure operation $\cl$.  For a subset $\mathcal A$ of $\mathbf B
(X, d)$, and every $a \in \Rp$, we agree that $\mathcal A+a$ denotes
$\{(x, r+a) \mid (x, r) \in \mathcal A\}$.
\begin{lem}
  \label{lemma:cl:1}
  Let $X, d$ be a standard quasi-metric space.
  \begin{enumerate}
  \item For every Scott-closed subset $\mathcal C$ of
    $\mathbf B (X, d)$, $\mathcal C+a$ is Scott-closed.
  \item For every subset $\mathcal A$ of $\mathbf B (X, d)$, $\cl
    (\mathcal A)+a = \cl (\mathcal A+a)$.
  \end{enumerate}
\end{lem}
\proof (1) For every formal ball $(x, r)$ below some element
$(y, s+a)$ of $\mathcal C+a$ (viz., $(y, s)$ is in $\mathcal C$),
$d (x, y) \leq r - s - a$, so $r \geq s+a \geq a$, and
$(x, r-a) \leq^{d^+} (y, s)$.  Since $\mathcal C$ is downwards-closed,
$(x, r-a)$ is in $\mathcal C$, hence $(x, r)$ is in $\mathcal C+a$.
This proves that $\mathcal C+a$ is downwards-closed.  For every
directed family ${(x_i, r_i+a)}_{i \in I}$ in $\mathcal C+a$,
${(x_i, r_i)}_{i \in I}$ is also directed, and in $\mathcal C$.  If
${(x_i, r_i+a)}_{i \in I}$ has a supremum, then
${(x_i, r_i)}_{i \in I}$ also has a supremum $(x, r)$, by the
definition of standardness, and the supremum of
${(x_i, r_i+a)}_{i \in I}$ is $(x, r+a)$.
Since $\mathcal C$ is Scott-closed, $(x, r)$ is in $\mathcal C$, so
$(x, r+a)$ is in $\mathcal C+a$.

(2) Since $\cl (\mathcal A)+a$ is Scott-closed, and since it obviously
contains $\mathcal A+a$, it contains $\cl (\mathcal A+a)$.  For the reverse
inclusion, we use the fact that the map $\_ + a$ is Scott-continuous,
since $X, d$ is standard,
and that for a continuous map, the closure of the image of a set
contains the image of the closure.  \qed

In accordance with our convention that $X$ is equated with a subspace
of $\mathbf B (X, d)$ through the embedding $x \mapsto (x, 0)$, we
also agree that, for any subset $C$ of $X$, $C+a$ denotes
$\{(x, a) \mid x \in C\}$.
\begin{lem}
  \label{lemma:cl}
  Let $X, d$ be a standard quasi-metric space.  For every closed
  subset $C$ of $X$, for every $a \in \Rp$:
  \begin{enumerate}
  \item $\cl (C) \cap X = C$;
  \item $\cl (C)+a = \cl (C+a)$.
  \end{enumerate}
\end{lem}
\proof (1) Let $U$ be the complement of $C$ in $X$, so that $\cl (C)$
is the complement of $\widehat U$ in $\mathbf B (X, d)$.
Since $\widehat U \cap X = U$, by taking complements $\cl (C) \cap X = C$.

(2) By Lemma~\ref{lemma:cl:1}~(2).  \qed

We now define a quasi-metric on $\Hoarez X$, the set of all closed
subsets of $X$, in a manner that looks like one half of the Hausdorff
metric.  We reserve the notation $\Hoare X$ for the set of
\emph{non-empty} closed subsets of $X$.
\begin{defi}[$\dH$]
  \label{defn:dH}
  Let $X, d$ be a standard quasi-metric space.  For any two closed
  subsets $C$, $C'$ of $X$, let $\dH (C, C') = \sup_{x \in C} d (x, C')$.
\end{defi}

\begin{lem}
  \label{lemma:dH}
  Let $X, d$ be a standard quasi-metric space.
  \begin{enumerate}
  \item For every $C \in \Hoarez X$, for every closed set $C'$, for
    every $a \in \Rp$, $\dH (C, C') \leq a$ if and only if $C+a
    \subseteq \cl (C')$;
  \item $\dH$ is a quasi-metric on $\Hoarez X$, and its
    specialization ordering is inclusion;
  \item the order $\leq^{\dH^+}$ on $\mathbf B (\Hoarez X,
    \dH)$ is given by $(C, r) \leq^{\dH^+} (C', r')$ if
    and only if $C+r \subseteq \cl (C' + r')$.
  \end{enumerate}
\end{lem}
\proof (1) By definition, $\dH (C, C') \leq a$ if and only if, for every
$x \in C$, $d (x, C') \leq a$, if and only if, for every $x \in C$,
$(x, a)$ is in $\cl (C')$, by Lemma~\ref{lemma:dxC}~(2).  That is
equivalent to $C+a \subseteq \cl (C')$.

(2) By (1), $\dH (C, C')=0$ if and only if
$C \subseteq \cl (C')$.  The latter is equivalent to
$C \subseteq \cl (C') \cap X$, hence to $C \subseteq C'$, using
Lemma~\ref{lemma:cl}~(1).  This shows immediately that
$\dH (C, C')=\dH (C', C)=0$ implies $C=C'$, and that
inclusion will be the specialization ordering of $\dH$.

It still remains to show the triangular inequality, namely
$\dH (C, C'') \leq \dH (C, C') + \dH (C', C'')$.  If the right-hand
side is equal to $+\infty$, then the inequality is obvious.
Otherwise, let $a = \dH (C, C')$, $a' = \dH (C', C'')$, both elements
of $\Rp$.  By (1), $C+a \subseteq \cl (C')$ and
$C'+a' \subseteq \cl (C'')$.  Therefore
$C+a+a' \subseteq \cl (C') + a' = \cl (C' + a')$ (by
Lemma~\ref{lemma:cl}~(2)) $\subseteq \cl (C'')$.  By (1) again, this
implies that $\dH (C, C'') \leq a+a'$, which proves our claim.

(3) $(C, r) \leq^{\dH^+} (C', r')$ if and only if $r \geq r'$ and
$C+r-r' \subseteq \cl (C')$, by (1).  Adding $r'$ to both sides,
$C+r-r' \subseteq \cl (C')$ if and only if
$C+r \subseteq \cl (C')+r'$, and $\cl (C')+r' = \cl (C'+r')$ by
Lemma~\ref{lemma:cl}~(2).  \qed

\begin{rem}
  \label{rem:H:root}
  For every standard quasi-metric space $X, d$ with an $a$-root $x$,
  ($a \in \Rp$, $a > 0$), $\Hoare X, \dH$ has an $a$-root, namely
  $\dc x$.  Indeed, fix $C \in \Hoarez X$.  For every $y \in X$,
  $d (x, y) \leq a$, and that certainly holds for any fixed $y \in C$.
  We recall that $d (x, C) \leq \inf_{y \in C} d (x,y)$, so
  $d (x, C) \leq a$.  By Lemma~\ref{lemma:dxC}~(2), $(x, a)$ is then
  in $\cl (C)$, so $\dc x + a = \dc (x, a)$ is included in $\cl (C)$.
  by Lemma~\ref{lemma:dH}~(1), $\dH (\dc x, C) \leq a$.

  $\Hoarez X, \dH$ has a root even when $X, d$ has no root, namely the
  empty closed subset: since $\emptyset \subseteq C$ for every $C \in
  \Hoarez X, \dH$, $\dH (\emptyset, C) = 0$.
\end{rem}

We relate that construction with previsions now.  This is similar to
Lemma~4.7, item~1, of \cite{jgl-jlap14}.
\begin{lem}
  \label{lemma:eH}
  Let $X$ be a topological space.  For each subset $C$ of $X$, let
  $F^C (h) = \sup_{x \in C} h (x)$.  (We agree that this is equal to
  $0$ if $C = \emptyset$.)
  \begin{enumerate}
  \item If $C$ is closed, then $F^C$ is a discrete sublinear
    prevision.
  \item Conversely, every discrete sublinear prevision is of the form
    $F^C$ for some unique closed set $C$.
  \end{enumerate}
  Moreover, $F$ is normalized if and only if $C$ is non-empty.
\end{lem}
\proof We start with the last claim, and show that $F^C$ is normalized
if and only if $C$ is non-empty.  If $C = \emptyset$, $F^C = 0$ is not
normalized.  Otherwise, $F^C (\alpha.\mathbf 1 + h) = \sup_{x \in C}
(\alpha + h (x)) = \alpha + \sup_{x \in C} h (x) = \alpha + F^C (h)$.

1. $F^C$ is clearly a prevision.  We check that $F^C$ is discrete.  If
$C = \emptyset$, then $F^C = 0$ and for every strict
$f \in \Lform \creal$, $F^C (f \circ h) = 0 = f (F^C (h))$---because
$f$ is strict.  Otherwise, $f (F^C (h)) = \sup_{x \in C} f (h (x))$
because every non-empty family in $\creal$ is directed, and because
$f$ is Scott-continuous.  Since
$\sup_{x \in C} f (h (x)) = F^C (f \circ h)$, we conclude that $F^C$
is discrete.

Clearly,
$F^C (h+h') = \sup_{x \in C} (h (x) + h' (x)) \leq \sup_{x \in C} h
(x) + \sup_{x \in C} h' (x)$, so $F^C$ is sublinear.

2. Let now $F$ be an arbitrary discrete sublinear prevision on $X$.
Let $\mathcal F$ be the family of open subsets $U$ of $X$ such that
$F (\chi_U)=0$.  $\mathcal F$ contains the empty set, hence is
non-empty.  If $U \subseteq V$ and $V \in \mathcal F$, then $U$ is in
$\mathcal F$ because $F$ is monotonic.  For every directed family
${(U_i)}_{i \in I}$ of elements of $\mathcal F$, letting
$U = \bigcup_{i \in I} U_i$, ${(\chi_{U_i})}_{i \in I}$ is also a
directed family of lower semicontinuous maps whose supremum is
$\chi_U$.  Since $F$ is Scott-continuous,
$F (\chi_U) = \sup_{i \in I} F (\chi_{U_i}) = 0$, so $U$ is in
$\mathcal F$.  We have shown that $\mathcal F$ is Scott-closed.

We observe that $\mathcal F$ is directed.  Take any two elements
$U_1$, $U_2$ of $\mathcal F$.  Since $F$ is sublinear,
$F (\chi_{U_1}) + F (\chi_{U_2}) \geq F (\chi_{U_1} + \chi_{U_2}) = F
(\chi_{U_1 \cup U_2} + \chi_{U_1 \cap U_2}) \geq F (\chi_{U_1 \cup
  U_2})$, so the latter is equal to $0$, meaning that $U_1 \cup U_2$
is in $\mathcal F$.

$\mathcal F$ being Scott-closed and directed, it has a largest element
$U_0$, and therefore $\mathcal F$ is the set of open subsets of $U_0$.

We can now prove (2).  If $F = F^C$, then $\mathcal F$ is the family
of open subsets that do not intersect $C$, so $U_0$ is the complement
of $C$.  This shows the uniqueness of $C$.

As far as existence is concerned, we are led to define $C$ as the
complement of $U_0$.

We check that $F = F^C$.  Let $h$ be a lower semicontinuous map from
$X$ to $\creal$.  For every $t \in \Rp$, let $f = \chi_{]t, +\infty]}$
and notice that this is a strict lower semicontinuous map.  Since $F$
is discrete, $F (f \circ h) = f (F (h))$.  Noting that
$f \circ h = \chi_U$ where $U$ is the open set
$h^{-1} (]t, +\infty])$, $U \in \mathcal F$ is equivalent to
$f (F (h)) = 0$.  In other words,
$h^{-1} (]t, +\infty]) \subseteq U_0$ if and only if $F (h) \leq t$.
We notice that $F^C (h) \leq t$ if and only if $h (x) \leq t$ for
every $x \in C$, if and only if every $x \in X$ such that $h (x) > t$
is in $U_0$ (by contraposition), if and only if
$h^{-1} (]t, +\infty]) \subseteq U_0$.  Therefore $F^C (h) \leq t$ if
and only if $F (h) \leq t$, for every $t \in \Rp$.  It follows that
$F^C (h) = F (h)$.  \qed

\begin{prop}
  \label{prop:H:prev}
  Let $X, d$ be a standard quasi-metric space.  The bijection
  $C \mapsto F^C$ is an isometry between $\Hoarez X$ and the set of
  discrete sublinear previsions on $X$: for all closed subsets $C$,
  $C'$, $\dH (C, C') = \dKRH (F^C, F^{C'})$.
\end{prop}
The indicated bijection is, of course, the one exhibited in
Lemma~\ref{lemma:eH}.

\proof Take $h = d (\_, C')$ in the definition of $\dKRH$
(Equation~(\ref{eq:dKRH})), and recall that this is $1$-Lipschitz
continuous.  We obtain that
$\dKRH (F^C, F^{C'}) \geq \dreal (\sup_{x \in C} d (x, C'), \sup_{y
  \in C'} d (y, C')) = \dreal (\dH (C, C'), \dH (C', C')) = \dreal
(\dH (C, C'), 0)$ (using Lemma~\ref{lemma:dH}~(2)) $= \dH (C, C')$.

For the reverse inclusion, fix $h \in \Lform_1 (X, d)$, and let us
check that
$\sup_{x \in C} h (x) \leq \sup_{y \in C'} h (y) + \dH (C, C')$.  If
$\dH (C, C') = +\infty$, this is clear, so assume
$a = \dH (C, C') < +\infty$.  Fix $x \in C$: our aim is to show that
$h (x) - a \leq \sup_{y \in C'} h (y)$.  Note that the left-hand side
is $h' (x, a)$, where $h' \colon (z, t) \mapsto h (z) - t$ is
Scott-continuous.  For every $b < h' (x, a)$, the formal ball $(x, a)$
is in the Scott-open set ${h'}^{-1} (]b, +\infty])$.  Using
Lemma~\ref{lemma:dH}~(1) and the fact that $a = \dH (C, C')$, $C+a$ is
included in $\cl (C')$, so $(x, a)$ is in $\cl (C')$.  Therefore
${h'}^{-1} (]b, +\infty])$ intersects $\cl (C')$, hence also $C'$.
Let $y$ be in the intersection.  Then $h' (y, 0) = h (y) > b$, which
shows that $\sup_{y \in C'} h (y) > b$.  As $b < h' (x, a)$ is
arbitrary, we conclude.  \qed

\begin{rem}
  \label{rem:HX:dKRHa}
  By using similar arguments, one can show that
  $\dKRH^a (F^C, F^{C'}) = \dH^a (C, C')$, where $\dH^a (C, C')$ is
  defined as $\min (a, \allowbreak \dH (C, C'))$, for every
  $a \in \Rp$, $a > 0$.  In one direction,
  $\dKRH^a (F^C, F^{C'}) \leq \dKRH (F^C, F^{C'}) = \dH (C, C')$ by
  Proposition~\ref{prop:H:prev}, and $\dKRH^a (F^C, F^{C'}) \leq a$ by
  definition.  In the other direction, we take
  $h = \min (a, d (\_, C'))$ in the definition of $\dKRH^a$
  (\ref{eq:dKRHa}), check that $h \in \Lform_1^a (X, d)$, and
  therefore obtain that
  $\dKRH^a (F^C, F^{C'}) \geq \dreal (\sup_{x \in C} \min (a, d (x,
  C')), \allowbreak \sup_{y \in C'} \min (a, d (y, C'))) = \min (a, d
  (C, C'))$.
\end{rem}

\subsection{Completeness}
\label{sec:H:completeness}

We shall need the following easy, folklore lemma.
\begin{lem}
  \label{lemma:H:cl:sup}
  For every topological space $Y$, for every subset $A$ of $Y$, for
  every lower semicontinuous map
  $h \colon Y \to \real \cup \{-\infty, +\infty\}$,
  $\sup_{y \in A} h (y) = \sup_{y \in cl (A)} h (y)$.
\end{lem}
\proof Since $A \subseteq cl (A)$,
$\sup_{y \in A} h (y) \leq \sup_{y \in cl (A)} h (y)$.  If that
inequality were strict, then there would be a real number $a$ such
that $\sup_{y \in A} h (y) \leq a < \sup_{y \in cl (A)} h (y)$.  In
particular, $h^{-1} (]a, +\infty])$ would intersect $cl (A)$, hence
also $A$.  Therefore, there would be an $y \in A$ such that
$h (y) > a$, which contradicts $\sup_{y \in A} h (y) \leq a$.  \qed

\begin{lem}
  \label{lemma:H:clC}
  Let $X, d$ be a continuous Yoneda-complete quasi-metric space.  For
  every closed subset $\mathfrak C$ of $\mathbf B (X, d)$ such that
  $F^{\mathfrak C}$ is supported on $V_{1/2^n}$ for every $n \in \nat$,
  \begin{enumerate}
  \item $\mathfrak C$ is equal to $\cl (C)$, where $C$ is the closed
    set $\mathfrak C \cap X$;
  \item For every $k \in \Lform (\mathbf B (X, d))$, $F^{\mathfrak C}
    (k) = F^C (k \circ \eta_X)$;
  \item $F^{\mathfrak C}$ is supported on $X$.
  \end{enumerate}
\end{lem}
\proof We first note that for every (Scott-)open subset $\mathcal U$
of $\mathbf B (X, d)$ that intersects $\mathfrak C$, for every
$\eta > 0$, $\mathfrak C$ also intersects $\mathcal U \cap V_{\eta}$.
Indeed, since $\mathcal U \cap \mathfrak C \neq \emptyset$,
$F^{\mathfrak C} (\chi_{\mathcal U}) = 1$.  We use our assumption that
$F^{\mathfrak C}$ is supported on $V_{1/2^n}$, where we choose $n$ so
large that $1/2^n < \eta$: since $\chi_{\mathcal U}$ and
$\chi_{\mathcal U \cap V_{\eta}}$ coincide on $V_{1/2^n}$,
$F^{\mathfrak C} (\chi_{\mathcal U \cap V_{\eta}}) = 1$.  Therefore
$\mathfrak C$ also intersects $\mathcal U \cap V_{\eta}$.

(1) Clearly $\cl (C) \subseteq \mathfrak C$.  In the reverse
direction, assume for the sake of contradiction that $\mathfrak C$ is
not included in $\cl (C)$.  There is a formal ball $(x, r)$ in
$\mathfrak C$ that is in some Scott-open subset $\mathcal U$ (namely,
the complement of $\cl (C)$) such that $\mathcal U \cap X$ does not
intersect $C$.

By our preliminary remark, $\mathfrak C$ must also intersect
$\mathcal U \cap V_1$, say at $(y_0, s_0)$.  Since $\mathbf B (X, d)$
is a continuous dcpo, $(y_0, s_0)$ is the supremum of a directed
family of formal balls way-below $(y_0, s_0)$, so one of them, call it
$(x_0, r_0)$, is in $\mathcal U \cap V_1$.  It is also in
$\mathfrak C$, since every Scott-closed set is downwards-closed.

Since $(x_0, r_0) \ll (y_0, s_0)$, $\mathfrak C \cap \uuarrow (x_0, r_0)$
is non-empty.  By our preliminary remark again, $\mathfrak C$ must
intersect $\uuarrow (x_0, r_0) \cap V_{1/2}$, say at $(y_1, s_1)$.  As
above, there is formal ball $(x_1, r_1) \ll (y_1, s_1)$ in $\mathfrak C
\cap \uuarrow (x_0, r_0) \cap V_{1/2}$.

We proceed in the same way: $\mathfrak C \cap \uuarrow (x_1, r_1)$ is
non-empty, hence $\mathfrak C \cap \uuarrow (x_1, r_1) \cap V_{1/4}$
contains a formal ball $(y_2, s_2)$, and that allows us to find
$(x_2, r_2) \ll (y_2, s_2)$ in
$\mathfrak C \cap \uuarrow (x_1, r_1) \cap V_{1/4}$.

Inductively, we obtain formal balls $(x_n, r_n)$ in
$\mathfrak C \cap V_{1/2^n}$ with the property that
$(x_0, r_0) \ll (x_1, r_1) \ll \cdots \ll (x_n, r_n) \ll \cdots$, and
$(x_0, r_0) \in \mathcal U$.  Let
$(y, s) = \sup_{n \in \nat} (x_n, r_n)$.  Since $s \leq r_n$ and
$(x_n, r_n) \in V_{1/2^n}$, $s$ is less than or equal to $1/2^n$ for
every $n \in \nat$, hence zero.  Since $\mathfrak C$ is Scott-closed,
$(y, s) = (y, 0)$ is in $\mathfrak C$, so $y$ is in $C$.  Because
$\mathcal U$ is upwards-closed and $(x_0, r_0)$ is in $\mathcal U$,
$(y, 0)$ must also be in $\mathcal U$.  But that contradicts the fact
that $\mathcal U \cap X$ does not intersect $C$.

(2) Let $k \in \Lform (\mathbf B (X, d))$.  We have
$F^{\mathfrak C} (k) = \sup_{(x, r) \in \mathfrak C} k (x, r) =
\sup_{(x, r) \in \cl (C)} k (x, r)$ by (1), so
$F^{\mathfrak C} (k) = \sup_{x \in C} k (x, 0)$ by
Lemma~\ref{lemma:H:cl:sup}, and that is exactly
$F^C (k \circ \eta_X)$.

(3) Let $g, f \in \Lform (\mathbf B (X, d))$ be such that $g_{|X} =
f_{|X}$, that is, $g \circ \eta_X = f \circ \eta_X$.  By (2),
$F^{\mathfrak C} (g) = F^C (g \circ \eta_X) = F^C (f \circ \eta_X) =
F^{\mathfrak C} (f)$.  \qed

\begin{prop}[Yoneda-completeness, Hoare powerdomain]
  \label{prop:HX:Ycomplete}
  Let $X, d$ be a standard Lipschitz regular quasi-metric space, or a
  continuous Yoneda-complete quasi-metric space.  Then
  $\Hoarez X, \dH$ and $\Hoare X, \dH$ are Yoneda-complete.

  Suprema of directed families of formal balls of previsions are naive
  suprema.
\end{prop}
\proof Through the isometry of Proposition~\ref{prop:H:prev}, we only
have to show that the spaces of discrete sublinear previsions (resp.,
the normalized discrete sublinear previsions) are Yoneda-complete
under $\dKRH$.  If $X, d$ is standard and Lipschitz regular, this
follows from Theorem~\ref{thm:LPrev:sup}.  Otherwise, this follows from
Proposition~\ref{prop:supp:complete}, since the assumption on supports
is verified by Lemma~\ref{lemma:H:clC}.  \qed

\begin{rem}
  \label{rem:HX:dKRHa:Ycomplete}
  Under the same assumptions, for every $a > 0$, $\Hoarez X, \dH^a$
  and $\Hoare X, \dH^a$ are Yoneda-complete, and directed suprema of
  formal balls are naive suprema.  The quasi-metric $\dH^a$ was
  introduced in Remark~\ref{rem:HX:dKRHa}.  The argument is similar to
  Proposition~\ref{prop:HX:Ycomplete}, replacing
  Proposition~\ref{prop:H:prev} by Remark~\ref{rem:HX:dKRHa}.
\end{rem}

\begin{rem}
  \label{rem:Ha}
  For $a > 0$, it is clear that the specialization ordering
  $\leq^{\dH^a}$ is again inclusion: $\dH^a (C, C') = 0$ if and only
  if $\min (a, \dH (C, C'))=0$ if and only if $\dH (C, C')=0$, if and
  only if $C \subseteq C'$ by Lemma~\ref{lemma:dH}~(2).
\end{rem}

\subsection{$\dH$-Limits}
\label{sec:dh-limits}

Suprema of directed families of formal balls in $\Hoare_0 X$ and in
$\Hoare X$ are naive suprema, but the detour through previsions makes
that less than explicit.  Let us give a more concrete description of
those suprema, i.e., of $\dH$-limits.
\begin{lem}
  \label{lemma:H:sup}
  Let $X, d$ be a continuous Yoneda-complete quasi-metric space.  For
  every directed family ${(C_i, r_i)}_{i \in I}$ in
  $\mathbf B (\Hoare_0 (X, d), \dH)$, its supremum $(C, r)$ is given
  by $r = \inf_{i \in I} r_i$ and
  $C = \cl (\bigcup_{i \in I} C_i + r_i - r) \cap X$.
\end{lem}
\proof By Proposition~\ref{prop:HX:Ycomplete}, that supremum is the
naive supremum, so $r = \inf_{i \in I} r_i$ in particular.  One might
think of writing down the naive supremum explicitly and simplifying
the expression, but a direct verification seems easier.

Let $\mathcal A = \bigcup_{i \in I} C_i + r_i - r$ and
$\mathfrak C = \cl (\mathcal A)$.  Let also $i \sqsubseteq j$ if and
only if $(C_i, r_i) \leq^{\dH^+} (C_j, r_j)$, making
${(C_i, r_i)}_{i \in I, \sqsubseteq}$ a monotone net.  The core of the
proof consists in showing that $\mathfrak C = \cl (C)$ for
$C = \mathfrak C \cap X$, and we shall obtain that by showing that
$F^{\mathfrak C}$ is supported on every $V_\epsilon$, $\epsilon > 0$.

For every $k \in \Lform (\mathbf B (X, d))$,
$F^{\mathfrak C} (k) = \sup_{(x, s) \in \mathcal A} k (x, s)$ by
Lemma~\ref{lemma:H:cl:sup}.  By definition of $\mathcal A$, and
another recourse to Lemma~\ref{lemma:H:cl:sup},
$F^{\mathfrak C} (k) = \sup_{i \in I} \sup_{x \in C_i} k (x, r_i-r) =
\sup_{i \in I} \sup_{(x, s) \in \cl (C_i+r_i-r)} k (x, s)$.  If
$i \sqsubseteq j$ then
$\dH (C_i, C_j) \leq r_i-r_j = (r_i-r)-(r_j-r)$, whence
$(C_i, r_i-r) \leq^{\dH^+} (C_j, r_j-r)$ and therefore
$\cl (C_i+r_i-r) \subseteq \cl (C_j+r_j-r)$ by
Lemma~\ref{lemma:dH}~(3).  In particular,
$\sup_{x \in C_i} k (x, r_i-r) = \sup_{(x, s) \in \cl (C_i+r_i-r)} k
(x, s)$ is smaller than or equal to
$\sup_{x \in C_j} k (x, r_j-r) = \sup_{(x, s) \in \cl (C_j+r_j-r)} k
(x, s)$.  This shows that the outer supremum in
$F^{\mathfrak C} (k) = \sup_{i \in I} \sup_{x \in C_i} k (x, r_i-r)$
is a directed supremum.

For any $i_0 \in I$, the subfamily of all $i \in I$ such that
$i_0 \sqsubseteq i$ is cofinal in $I$, hence $F^{\mathfrak C} (k)$ is
also equal to
$\sup_{i \in I, i_0 \sqsubseteq i} \sup_{x \in C_i} k (x, r_i-r)$.

For every $\epsilon > 0$, if $g$ and $f$ are two elements of
$\Lform (\mathbf B (X, d))$ that coincide on $V_\epsilon$, and since
$r = \inf_{i \in r_i}$, there is an $i_0 \in I$ such that
$r_i - r < \epsilon$ for every $i \in I$, $i_0 \sqsubseteq i$.  Then
$g (x, r_i-r) = f (x, r_i-r)$, and the formula we have just obtained
for $F^{\mathfrak C} (k)$ shows that
$F^{\mathfrak C} (g) = F^{\mathfrak C} (f)$.  In other words,
$F^{\mathfrak C}$ is supported on $V_\epsilon$.  By
Lemma~\ref{lemma:H:clC}, $\mathfrak C$ is equal to $\cl (C)$, where
$C = \mathfrak C \cap X$.

For every $i \in I$, $C_i + r_i \subseteq \cl (C + r)$, since
$\cl (C+r) = \cl (C) + r$ (by Lemma~\ref{lemma:cl}~(2))
$ = \mathfrak C + r = \cl (\bigcup_{i \in I} C_i + r_i)$ (by
Lemma~\ref{lemma:cl:1}~(2)).  Hence $(C, r)$ is an upper bound of
${(C_i, r_i)}_{i \in I}$ by Lemma~\ref{lemma:dH}, item~3.

Assume $(C', r')$ is another upper bound of ${(C_i, r_i)}_{i \in I}$.
Then $r' \leq \inf_{i \in I} r_i = r$, and
$C_i+r_i \subseteq \cl (C'+r')$ for every $i \in I$, by
Lemma~\ref{lemma:dH}~(3).  Let $\mathfrak D$ be the set of formal balls
$(y, s)$ such that $(y, s+r) \in \cl (C'+r')$.  Since $X, d$ is
standard, the map $(y, s) \mapsto (y, s+r)$ is Scott-continuous, so
$\mathfrak D$ is closed in $\mathbf B (X, d)$.  Since $r_i \geq r$, the
inclusion $C_i+r_i \subseteq \cl (C'+r')$ rewrites as
$C_i+r_i-r \subseteq \mathfrak D$.  Taking unions over $i \in I$, then
closures, we obtain $\mathfrak C \subseteq \mathfrak D$.

Since $\mathfrak C = \cl (C)$, $C$ is included in $\mathfrak D$, too,
and that means that $C+r \subseteq \cl (C'+r')$.  By
Lemma~\ref{lemma:dH}~(3), $(C, r) \leq^{\dH^+} (C', r')$, so $(C, r)$
is the least upper bound of ${(C_i, r_i)}_{i \in I}$.  \qed

\subsection{Algebraicity}
\label{sec:H:algebraicity}

\begin{lem}
  \label{lemma:H:simple:center}
  Let $X, d$ be a continuous Yoneda-complete quasi-metric space.  For
  all center points $x_1$, \ldots, $x_n$, $\dc \{x_1, \cdots, x_n\}$
  is a center point of $\Hoarez X, \dH$, and also of $\Hoare X, \dH$
  if $n \geq 1$.
\end{lem}
\proof The proof is very similar to Lemma~\ref{lemma:V:simple:center}.
Let $C_0 = \dc \{x_1, \cdots, x_n\}$.  For every $h \in \Lform
X$, $F^{C_0} (h) = \max_{j=1}^n h (x_j)$, where the maximum is taken to be
$0$ if $n=0$.  Let $U = B^{\dKRH^+}_{(F^{C_0}, 0), <\epsilon}$.  $U$ is
upwards-closed: if $(F^C, r) \leq^{\dKRH^+} (F^{C'}, r')$ and $(F^C, r) \in
U$, then $\dKRH (F^C, F^{C'}) \leq r-r'$ and $\dKRH (F^{C_0}, F^C) <
\epsilon - r$, so $\dKRH (F^{C_0}, F^{C'}) < \epsilon - r'$ by the
triangular inequality, and that means that $(F^{C'}, r')$ is in $U$.

To show that $U$ is Scott-open, consider a directed family
${(C_i, r_i)}_{i \in I}$ in $\mathbf B (\Hoarez X, \dH)$ (resp.,
$\mathbf B (\Hoare X, \dH)$) with supremum $(C, r)$, and assume that
$(F^C, r)$ is in $U$.  By Proposition~\ref{prop:HX:Ycomplete},
directed suprema of formal balls are naive suprema, so
$r = \inf_{i \in I} r_i$ and $C$ is characterized by the fact that,
for every $h \in \Lform_1 (X, d)$,
$F^C (h) = \sup_{i \in I} (F^{C_i} (h) + r - r_i)$.  Since $(F^C, r)$
is in $U$, $\dKRH (F^{C_0}, F^C) < \epsilon - r$, so $\epsilon > r$
and $F^{C_0} (h) - \epsilon + r < F^C (h)$.  Therefore, for every
$h \in \Lform_1 (X, d)$, there is an index $i \in I$ such that
$F^{C_0} (h) - \epsilon + r < F^{C_i} (h) + r - r_i$, or equivalently:
\begin{equation}
  \label{eq:H:A}
  \max_{j=1}^n h (x_j) < F^{C_i} (h) + \epsilon - r_i.
\end{equation}
Moreover, since $\epsilon > r = \inf_{i \in I} r_i$, we may assume
that $i$ is so large that $\epsilon > r_i$.

Let $V_i$ be the set of all $h \in \Lform_1 (X, d)$ satisfying
(\ref{eq:H:A}).  Each $V_i$ is the inverse image of $[0, \epsilon -
r_i[$ by the map $h \mapsto \dreal (\max_{i=1}^n h (x_i), F^{C_i}
(h))$, which is continuous from $\Lform_1 (X,
d)^\patch$ to $(\creal)^\dG$ by Proposition~\ref{prop:dKRH:cont}~(2).
Hence ${(V_i)}_{i \in I}$ is an open cover of $\Lform_1 (X,
d)^\patch$.  The latter is a compact
space by Lemma~\ref{lemma:cont:Lalpha:retr}~(4).  Hence we can
extract a finite subcover ${(V_i)}_{i \in J}$: for every $h \in
\Lform_1 (X, d)$, there is an index $i$ in the finite set $J$ such
that (\ref{eq:H:A}) holds.  By directedness, one can require that $i$
be the same for all $h$.  This shows that $(F^{C_i}, r_i)$ is in $U$,
proving the claim.  \qed

\begin{rem}
  \label{rem:Ha:simple:center}
  A similar result holds for $\dH^a$ instead of $\dH$, $a \in \Rp$,
  $a > 0$: on a continuous Yoneda-complete quasi-metric space $X, d$,
  and for all center points $x_1$, \ldots, $x_n$,
  $\dc \{x_1, \cdots, x_n\}$ is a center point of $\Hoarez X, \dH^a$,
  and also of $\Hoare X, \dH^a$ if $n \geq 1$.  The proof is as for
  Lemma~\ref{lemma:H:simple:center}.
\end{rem}

\begin{thm}[Algebraicity of Hoare powerdomains]
  \label{thm:H:alg}
  Let $X, d$ be an algebraic Yoneda-complete quasi-metric space, with
  strong basis $\mathcal B$.  The
  spaces $\Hoarez X, \dH$ and $\Hoare X, \dH$ are algebraic
  Yoneda-complete.

  Every closed set of the form $\dc \{x_1, \cdots, x_n\}$ where every
  $x_i$ is a center point (and with $n \geq 1$ in the case of
  $\Hoare X$) is a center point of $\Hoarez X, \dH$ (resp.,
  $\Hoare X, \dH$), and those form a strong basis, even when each
  $x_i$ is taken from $\mathcal B$.
\end{thm}
\proof $\Hoarez X, \dH$ and $\Hoare X, \dH$ are Yoneda-complete by
Proposition~\ref{prop:HX:Ycomplete}, and $\dc \{x_1, \cdots, x_n\}$ is
a center point as soon as every $x_i$ is a center point, by
Lemma~\ref{lemma:H:simple:center}.

Let $C \in \Hoarez X$.  We wish to show that $(C, 0)$ is the
supremum of a directed family of formal balls $(C_i, r_i)$ below
$(C, 0)$, where every $C_i$ is the downward closure of finitely many
center points.  In case $C = \emptyset$, this is easy: take the family
$\{(\emptyset, 0)\}$.  Hence assume that $C$ is non-empty.

Consider the family $D$ of all formal balls
$(\dc \{x_1, x_2, \cdots, x_n\}, r)$ where $n \geq 1$, every $x_j$ is
in $\mathcal B$, and $d (x_j, C) < r$.

$D$ is a non-empty family.  Indeed, since $C$ is non-empty, pick $x$
from $C$.  By Lemma~\ref{lemma:B:basis}, there is a formal ball
$(x_1, r) \ll (x, 0)$ (i.e., $d (x_1, x) < r$) where $x_1$ is in
$\mathcal B$.  Since $x_1$ is a center point,
$d (x_1, C) = \inf_{z \in C} d (x_1, z)$
\cite[Proposition~6.12]{JGL:formalballs}, so $(\dc x_1, r)$ is in $D$.

To show that $D$ is directed, let $(\dc \{x_1, \cdots, x_m\}, r)$ and
$(\dc \{y_1, \cdots, y_n\}, s)$ be two elements of $D$.  By
definition, $d (x_j, C) < r$ and $d (y_k, C) < s$ for all indices $j$
and $k$.  Since there are finitely many points $x_j$ and $y_k$, we can
find $\epsilon > 0$ such that $d (x_j, C) < r-\epsilon$ and
$d (y_k, C) < s-\epsilon$ for all indices $j$ and $k$.  For each $j$,
since $d (x_j, C) < r-\epsilon$, there is a point $z_j \in C$ such
that $d (x_j, z_j) < r-\epsilon$.  (We use Proposition~6.12 of
loc.cit.\ once again.)  Since $x_j$ is a center point,
$B^{d^+}_{(x_j, 0), <r-\epsilon}$ is Scott-open.  It also contains
$(z_j, 0)$.  By Lemma~\ref{lemma:B:basis}, there is a formal ball
$(x'_j, r'_j) \ll (z_j, 0)$, where $x'_j$ is in $\mathcal B$ and
$r'_j < \epsilon$.  Indeed, one of the formal balls way-below
$(z_j, 0)$ must be in the open set $V_\epsilon$
(Lemma~\ref{lemma:Veps}).  By construction,
$d (x'_j, C) \leq d (x'_j, z_j) < r'_j < \epsilon$: the first
inequality comes from $d (x'_j, C) = \inf_{z \in C} d (x'_j, z)$,
which holds since $x'_j$ is a center point, and the second one is
justified by the fact that $(x'_j, r'_j) \ll (z_j, 0)$.  Also
$d (x_j, x'_j) \leq r-\epsilon-r'_j \leq r-\epsilon$, because
$(x'_j, r'_j) \in B^{d^+}_{(x_j, 0), <r-\epsilon}$.  Similarly, we
find center points $y'_k$ such that $d (y'_k, C) < \epsilon$ and
$d (y_k, y'_k) \leq s-\epsilon$ for every $k$.  Let
$E = \{x'_1, \cdots, x'_m, y'_1, \cdots, y'_n\}$, then
$(\dc E, \epsilon)$ is in $D$, and we claim that
$(\dc \{x_1, \cdots, x_m\}, r)$ and $(\dc \{y_1, \cdots, y_n\}, s)$
are both below $(\dc E, \epsilon)$.  We check the first of those
claims:
$\dH (\dc \{x_1, \cdots, x_m\}, \dc E) = \sup_{j=1}^m d (x_j, \dc E)
\leq \sup_{j=1}^m d (x_j, x'_j) \leq r-\epsilon$, and this proves the
claim.

We may use Lemma~\ref{lemma:H:sup} to compute the supremum of $D$.
Instead we verify directly that this is $(C, 0)$.

Every element $(\dc \{x_1, \cdots, x_m\}, r)$ of $D$ is below
$(C, 0)$.  To verify this, we compute:
\begin{eqnarray*}
  \dH (\dc \{x_1, \cdots, x_m\}, C)
  & = & \sup_{x' \in \dc \{x_1, \cdots, x_m\}} d (x', C) \\
  & = & \sup\nolimits_{j=1}^m \sup_{x' \leq x_j} d (x', C) \\
  & \leq & \sup\nolimits_{j=1}^m \sup_{x' \leq x_j} \underbrace {d (x', x_j)}_0 + d (x_j,
C) \\
  & = & \sup\nolimits_{j=1}^m d (x_j, C),
\end{eqnarray*}
and we use the fact that $d (x_j, C) < r$ for every $j$.

This means that $(C, 0)$ is an upper bound of $D$, and it remains to
show that it is the least.

We first note that there are elements of $D$ with arbitrary small
radius.  For that, we return to the argument we have given that $D$ is
non-empty: pick $x \in C$, and use Lemma~\ref{lemma:B:basis} to find a
formal ball $(x_1, r) \ll (x, 0)$ (i.e., $d (x_1, x) < r$) where $x_1$
is in $\mathcal B$.  That lemma also guarantees that $r$ can be chosen
as small as we wish, since we can choose $(x_1, r)$ in the open set
$V_\epsilon$ for $\epsilon > 0$ as we wish, and we have seen that
$(\dc x_1, r)$ is in $D$.

Let $(C', r')$ be another upper bound of $D$.  Since $D$ contains
formal balls of arbitrary small radius, and all those radii are larger
than or equal to $r'$ (because all the elements of $D$ are below
$(C', r')$), $r'=0$.  It remains to show that $C$ is included in $C'$,
and this will prove that $(C, 0) \leq^{\dH^+} (C', 0)$, using
Lemma~\ref{lemma:dH}~(2).  To this end, we consider an arbitrary open
subset $U$ that intersects $C$, say at $z$, and we assume for the sake
of contradiction that it does not intersect $C'$.  Since $z$ is not in
$C'$, $(z, 0)$ is not in $\cl (C')$, by Lemma~\ref{lemma:cl}~(1), and
Lemma~\ref{lemma:B:basis} gives us a formal ball
$(x, \epsilon) \ll (z, 0)$ (i.e., $d (x, z) < \epsilon$) where $x$ is
in $\mathcal B$.  Then
$B^d_{x, <\epsilon} = \uuarrow (x, \epsilon) \cap X$ contains $z$ and
does not intersect $C'$, since $\uuarrow (x, \epsilon)$ does not
intersect $\cl (C')$.  Since $d (x, z) < \epsilon$,
$d (x, C) < \epsilon$.  Find $\epsilon' > 0$ such that
$d (x, C) < \epsilon' < \epsilon$.  The formal ball
$(\dc x, \epsilon')$ is then in $D$, and since $(C', 0)$ is an upper
bound of $D$, $(\dc x, \epsilon') \leq^{\dH^+} (C', 0)$.  This
rewrites as $d (x, C') \leq \epsilon' < \epsilon$.  Using the fact
that for a center point $x$, $d (x, C') = \inf_{z' \in C'} d (x, z')$,
there is a point $z' \in C'$ such that $d (x, z') < \epsilon$.  This
contradicts the fact that $B^d_{x, <\epsilon}$ does not intersect
$C'$.  \qed

\begin{rem}
  \label{rem:Ha:alg}
  The same result holds for $\dH^a$, for every $a \in \Rp$, $a > 0$:
  when $X, d$ is algebraic Yoneda-complete, $\Hoarez X, \dH^a$ and
  $\Hoare X, \dH^a$ are algebraic Yoneda-complete, with the same
  strong basis.
  The proof is the same, except in the final step, where we consider
  another upper bound $(C', r')$ of $D$.  We need to make sure that we
  can take $\epsilon \leq a$.  This is guaranteed by
  Lemma~\ref{lemma:B:basis}, which states that we can choose
  $\epsilon$ as small as we wish.
  In the last lines of the proof, we obtain that
  $(\dc x, \epsilon') \leq^{\dH^{a+}} (C', 0)$, hence
  $\min (a, d (x, C')) \leq \epsilon' < \epsilon$.  Since
  $\epsilon \leq a$, this implies that $d (x, C') < \epsilon$, leading
  to a contradiction as above.
%
%
\end{rem}

\subsection{Continuity}
\label{sec:continuity-hoare}

We proceed exactly as in Section~\ref{sec:cont-val-leq-1} for
subnormalized continuous valuations.
\begin{lem}
  \label{lemma:H:functor}
  Let $X, d$ and $Y, \partial$ be two continuous Yoneda-complete
  quasi-metric spaces, and $f \colon X, d \mapsto Y, \partial$ be a
  $1$-Lipschitz continuous map.  The map
  $\Hoare f \colon \Hoarez X, \dH \to \Hoarez Y, \dH$ defined by
  $\Hoare f (C) = cl (f [C])$ is $1$-Lipschitz continuous.  Moreover,
  $F^{\Hoare f (C)} = \Prev f (F^C)$ for every $C \in \Hoarez X$.

  Similarly with $\Hoare$ instead of $\Hoarez$, with $\dH^a$ instead
  of $\dH$.
\end{lem}
Recall that $f [C]$ denotes the image of $C$ by $f$.

\proof We first check that $F^{\Hoare f (C)} = \Prev f (F^C)$.  For
every $h \in \Lform X$,
$F^{\Hoare f (C)} (h) = \sup_{y \in cl (f [C])} h (y) = \sup_{y \in f
  [C]} h (y)$ by Lemma~\ref{lemma:H:cl:sup}.  That is equal to
$\sup_{x \in X} h (f (x))$.  We also have
$\Prev f (F^C) (h) = F^C (h \circ f) = \sup_{x \in X} h (f (x))$.
Therefore $F^{\Hoare f (C)} = \Prev f (F^C)$, which shows that the
isometry of Proposition~\ref{prop:H:prev} is natural.

By Lemma~\ref{lemma:Pf:lip}, $\Prev f$ is $1$-Lipschitz, so
$\mathbf B^1 (\Prev f)$ is monotonic.  Also, $\Prev f$ maps discrete
sublinear previsions to discrete sublinear previsions, and normalized
previsions to normalized previsions.  By
Proposition~\ref{prop:HX:Ycomplete}, $\Hoarez X, \dH$ and
$\Hoare X, \dH$ are Yoneda-complete, hence through the isometry of
Proposition~\ref{prop:H:prev}, the corresponding spaces of discrete
sublinear previsions are Yoneda-complete as well.  Moreover, directed
suprema of formal balls are computed as naive suprema.  By
Lemma~\ref{lemma:Pf:lipcont}, $\mathbf B^1 (\Prev f)$ preserves naive
suprema, hence all directed suprema.  It must therefore be
Scott-continuous, and using the (natural) isometry of
Proposition~\ref{prop:H:prev}, $\mathbf B^1 (\Hoare f)$ must also be
Scott-continuous.  Hence $\Hoare f$ is $1$-Lipschitz continuous.

In the case of $\dH^a$, the argument is the same, except that we use
Remark~\ref{rem:HX:dKRHa:Ycomplete} instead of
Proposition~\ref{prop:HX:Ycomplete}.  \qed

Let $X, d$ be a continuous Yoneda-complete quasi-metric space.
There is an algebraic Yoneda-complete quasi-metric space $Y, \partial$
and two $1$-Lipschitz continuous maps $r \colon Y, \partial \to X, d$
and $s \colon X, d \to Y, \partial$ such that $r \circ s = \identity
X$.

By Lemma~\ref{lemma:H:functor}, $\Hoarez r$ and $\Hoarez s$ are also
$1$-Lipschitz continuous.  Also,
$\Hoarez r \circ \Hoarez s = \identity {\Hoarez X}$: through the
isometry $C \mapsto F^C$, that boils down to
$\Prev r \circ \Prev s = \identity \relax$, which is easily checked
since
$\Prev r (\Prev s (F^C)) (h) = \Prev s (F^C) (h \circ r) = F^C (h
\circ r \circ s) = F^C (h)$, for all $C$ and $h$.  Therefore
$\Hoarez X, \dH$ is a $1$-Lipschitz continuous retract of
$\Hoarez Y, \mH\partial$.  (Similarly with $\dH^a$ and
$\mH\partial^a$.)  Theorem~\ref{thm:H:alg} states that
$\Hoarez Y, \mH\partial$ and $\Hoare Y, \mH\partial$ (resp.,
$\mH\partial^a$, using Remark~\ref{rem:Ha:alg} instead) is algebraic
Yoneda-complete, whence:
\begin{thm}[Continuity for Hoare powerdomains]
  \label{thm:H:cont}
  Let $X, d$ be a continuous Yoneda-complete quasi-metric space.  The
  quasi-metric spaces $\Hoarez X, \dH$ and $\Hoarez X, \dH^a$
  ($a \in \Rp$, $a > 0$) are continuous Yoneda-complete.  Similarly
  with $\Hoare X$. \qed
\end{thm}

Together with Lemma~\ref{lemma:H:functor}, and
Theorem~\ref{thm:H:alg} for the algebraic case, we obtain the following.
\begin{cor}
  \label{cor:H:functor}
  $\Hoarez, \dH$ defines an endofunctor on the category of continuous
  Yoneda-complete quasi-metric spaces and $1$-Lipschitz continuous
  map.  Similarly with $\Hoare$ instead of $\Hoarez$, with $\dH^a$
  instead of $\dH$ ($a > 0$), or with algebraic instead of continuous.
  \qed
\end{cor}

\subsection{The Lower Vietoris Topology}
\label{sec:lower-viet-topol}

The lower Vietoris topology on $\Hoare X$, resp.\ $\Hoarez X$, is
generated by subbasic open sets $\Diamond U = \{C \mid C \cap U \neq
\emptyset\}$, where $U$ ranges over the open subsets of $X$.

\begin{lem}
  \label{lemma:H:V=weak}
  Let $X, d$ be a standard quasi-metric space.  The map $C \mapsto
  F^C$ is a homeomorphism of $\Hoarez X$ (resp., $\Hoare X$) with the
  lower Vietoris topology onto the space of discrete sublinear
  previsions on $X$ (resp., that are normalized) with the weak topology.
\end{lem}
\proof
This is a bijection by Lemma~\ref{lemma:eH}.
Fix $h \in \Lform X$, $a \in \Rp$.
For every $C \in \Hoarez X$, $F^C$ is in $[h > a]$ if and only if
$\sup_{x \in C} h (x) > a$, if and only if $h (x) > a$ for some $x \in
C$, if and only if $C$ intersects $h^{-1} (]a, +\infty])$, namely, $C
\in \Diamond h^{-1} (]a, +\infty])$.  Therefore the bijection is
continuous.

In the other direction, for every open subset $U$, $\Diamond U$ is the
set of all $C \in \Hoarez X$ such that $C$ intersects
$\chi_U^{-1} (]1/2, +\infty])$, i.e., such that $F^C$ is in
$[\chi_U > 1/2]$.  The case of $\Hoare X$ is similar.  \qed

\begin{lem}
  \label{lemma:H:Vietoris}
  Let $X, d$ be a standard and Lipschitz regular quasi-metric space,
  or a continuous Yoneda-complete quasi-metric space, and let
  $a, a' > 0$ with $a \leq a'$.  We have the following inclusions of
  topologies on $\Hoarez X$, resp.\ $\Hoare X$:
  \begin{quote}
    lower Vietoris $\subseteq$ $\dH^a$-Scott $\subseteq$
    $\dH^{a'}$-Scott $\subseteq$ $\dH$-Scott.
  \end{quote}
\end{lem}
\proof Considering Lemma~\ref{lemma:H:V=weak},
Remark~\ref{rem:HX:dKRHa}, and Proposition~\ref{prop:HX:Ycomplete},
this is a consequence of Proposition~\ref{prop:weak:dScott:a}.  \qed


\begin{prop}
  \label{prop:H:Vietoris}
  Let $X, d$ be an algebraic Yoneda-complete quasi-metric space.  The
  $\dH$-Scott topology, the $\dH^a$-Scott topology, for every
  $a \in \Rp$, $a > 0$, and the lower Vietoris topology all coincide
  on $\Hoarez X$, resp.\ $\Hoare X$.
\end{prop}
\proof Considering Lemma~\ref{lemma:H:Vietoris}, it remains to show
that every $\dH$-Scott open subset is open in the lower Vietoris
topology.  Again we deal with $\Hoarez X$, and note that the case of
$\Hoare X$ is completely analogous.

The subsets $\mathcal V = \uuarrow (\dc \{x_1, \cdots, x_n\}, r)$,
with $x_1$, \ldots, $x_n$ center points form a base of the Scott
topology on $\mathbf B (\Hoarez X, \dH)$, using
Theorem~\ref{thm:H:alg}.  It suffices to show that
$\mathcal V \cap \Hoarez X$ is open in the lower Vietoris topology.
For that, we use the implication (1)$\limp$(3) of Lemma~5.8 of
\cite{JGL:formalballs}, which states that, in a continuous
quasi-metric space $Y, \partial$, for every $\epsilon > 0$, for every
center point $y \in Y$,
$B^{\partial^+}_{(y, 0), <\epsilon} = \uuarrow (y, \epsilon)$.  Here
$H = \Hoarez X$, $\partial = \dH$, is algebraic complete by
Theorem~\ref{thm:H:alg}, $\mathcal V = \dc \{x_1, \cdots, x_n\}$ is a
center point by Lemma~\ref{lemma:H:simple:center}, so
$\uuarrow (\dc \{x_1, \cdots, x_n\}, r)$ is the open ball
$B^{\dH^+}_{\dc \{x_1, \cdots, x_n\}, < r}$.  Then, $C \in \Hoarez X$
is in $\mathcal V$ if and only if it is in
$B^{\dH^+}_{\dc \{x_1, \cdots, x_n\}, < r}$, if and only if
$\dH (\dc \{x_1, \cdots, x_n\}, C) < r$, if and only if
$d (x_i, C) < r$ for every $i$, $1\leq i\leq n$, if and only if $C$
intersects all the open balls $B^d_{x_i, <r}$.  Since each $x_i$ is a
center point, those open balls are open in the $d$-Scott topology, so
$\mathcal V \cap \Hoarez X = \bigcap_{i=1}^n \Diamond B^d_{x_i, <r}$
is open in the lower Vietoris topology.  \qed

\begin{thm}[$\dH$ quasi-metrizes the lower Vietoris topology]
  \label{thm:H:Vietoris}
  Let $X, d$ be a continuous Yoneda-complete quasi-metric space.  The
  $\dH$-Scott topology, the $\dH^a$-Scott topology, for every
  $a \in \Rp$, $a > 0$, and the lower Vietoris topology all coincide
  on $\Hoarez X$, resp.\ $\Hoare X$.
\end{thm}
\proof The proof is as for Theorem~\ref{thm:V:weak=dScott}.  By
\cite[Theorem~7.9]{JGL:formalballs}, $X, d$ is the $1$-Lipschitz
continuous retract of an algebraic Yoneda-complete quasi-metric space
$Y, \partial$.  Call $s \colon X \to Y$ the section and
$r \colon Y \to X$ the retraction.  Using
Proposition~\ref{prop:H:prev}, we confuse $\Hoarez X$ with the
corresponding space of discrete sublinear previsions, and similarly
for $\Hoare X$, $\Hoarez Y$, $\Hoare Y$.  Then $\Prev s$ and $\Prev r$
form a $1$-Lipschitz continuous section-retraction pair by
Lemma~\ref{lemma:Pf:lip}, and in particular $\Prev s$ is an embedding
of $\Hoarez X$ into $\Hoarez Y$ with their $\dH$-Scott topologies
(similarly with $\Hoare$, or with $\dH^a$ in place of $\dH$).
However, $s$ and $r$ are also just continuous, by
Proposition~\ref{prop:cont}, so $\Prev s$ and $\Prev r$ also form a
section-retraction pair between the same spaces, this time with their
weak topologies (as spaces of previsions), by Fact~\ref{fact:Pf:weak},
that is, with their lower Vietoris topologies, by
Lemma~\ref{lemma:H:V=weak}.  By Proposition~\ref{prop:H:Vietoris}, the
two topologies on $\Hoarez Y$ (resp., $\Hoare Y$) are the same.
Fact~\ref{fact:retract:two} then implies that the two topologies on
$\Hoarez X$ (resp., $\Hoare X$) are the same as well.  \qed

\section{The Smyth Powerdomain}
\label{sec:smyth-powerdomain}






\subsection{The $\dQ$ Quasi-Metric}
\label{sec:dq-quasi-metric}

The Smyth powerdomain of $X$ is the space of all non-empty compact
saturated subsets $Q$ of $X$.  Instead of defining a specific
quasi-metric on such subsets, as we did with $\dH$ on the Hoare
powerdomain, we shall reuse $\dKRH$, on the isomorphic space of
discrete superlinear previsions on $X$.  The following is inspired
from \cite[Lemma~4.7, item~2]{jgl-jlap14}, as was Lemma~\ref{lemma:eH}.
\begin{lem}
  \label{lemma:uQ}
  Let $X$ be a topological space.  For each non-empty subset $Q$ of
  $X$, let $F_Q (h) = \inf_{x \in Q} h (x)$.
  \begin{enumerate}
  \item If $Q$ is compact saturated and non-empty, then $F_Q$ is a
    normalized discrete superlinear prevision; moreover,
    $F_Q (h) = \min_{x \in Q} h (x)$ for every $h \in \Lform X$.
  \item Conversely, if $X$ is sober, then every normalized discrete
    superlinear prevision is of the form $F_Q$ for some unique
    non-empty compact saturated set $Q$.
  \end{enumerate}
\end{lem}
\proof 1. For every $h \in \Lform X$, $\{h (x) \mid x \in Q\}$ is
compact in $\creal$, and non-empty, hence has a least element.  It
follows that there is an $x \in Q$ such that $h (x) = F_Q (h)$.  This
justifies the second subclaim, that $F_Q (h) = \min_{x \in Q} h (x)$.

$F_Q$ is clearly positively homogeneous and monotonic.  To show
Scott-continuity, let $h$ be written as the supremum of a directed
family ${(h_i)}_{i \in I}$ in $\Lform X$.  $F_Q (h) \geq \sup_{i \in
  I} F_Q (h_i)$ is a consequence of monotonicity.  If the inequality
were strict, then there would be a number $r \in \Rp$ such that $F_Q
(h) > r \geq \sup_{i \in I} F_Q (h_i)$.  The inequality $F_Q (h) > r$
means that the image of $Q$ by $h$ is contained in the Scott-open $]r,
+\infty]$.  For every $x \in Q$, $h (x) = \sup_{i \in I} h_i (x) > r$,
so $h_i (x) > r$ for some $i \in I$.  It follows that $Q$ is included
in the union of the open subsets $h_i^{-1} (]r, +\infty])$.  Since $Q$
is compact and the family of the latter open subsets is directed, $Q$
is included in one of them, say $h_i$.  That means that the image of
$Q$ by $h_i$ is included in $]r, +\infty]$, and that implies that $F_Q
(h_i) = \min_{x \in Q} h_i (x) > r$, contradicting $r \geq \sup_{i \in I} F_Q (h_i)$.

We check that $F_Q$ is discrete.  Let $x \in Q$ be such that
$h (x) = F_Q (h)$.  For every $f \in \Lform \creal$ (strict or not),
$F_Q (f \circ h)$ is trivially less than or equal to $f (h (x))$.  For
every $x' \in Q$, $h (x') \geq h (x)$, so
$(f \circ h) (x') \geq f (h (x))$, from which it follows that
$F_Q (f \circ h) = \min_{x' \in Q} (f \circ h) (x') \geq f (h (x)) = f
(F_Q (h))$.

Since this holds even when $f$ is not strict, it holds for the
map $f (x) = \alpha + x$, for any $\alpha \in \Rp$.  In other words,
$F_Q (\alpha.\mathbf 1 + h) = \alpha + F_Q (h)$, showing that $F_Q$ is
normalized.

Finally, we show that $F_Q$ is superlinear.  Fix $h, h' \in \Lform X$.
Then $F_Q (h+h') = \min_{x \in Q} (h (x) + h (x')) \geq \min_{x \in Q}
h (x) + \min_{x \in Q'} h' (x) = F_Q (h) + F_Q (h')$.

2. Now assume $X$ sober, and let $F$ be a normalized discrete
superlinear prevision on $X$.  Let $\mathcal F$ be the family of open
subsets $U$ of $X$ such that $F (\chi_U)=1$.  $\mathcal F$ is upwards-closed.
Since $F$ is normalized,
$\mathcal F$ contains $X$ itself, hence is non-empty.

Note that, for every open subset $U$ of $X$, $F (\chi_U)$ is either
equal to $1$ or to $0$.  This is because $F$ is discrete: for every
$t \in ]0, 1]$,
$F (\chi_U) = F (\chi_{]t, +\infty]} \circ \chi_U) = \chi_{]t,
  +\infty]} (F (\chi_U)) \in \{0, 1\}$.

Using discreteness again and superlinearity, we show that, if
$U \subseteq V$, then $F (\chi_U) + F (\chi_V) = F (\chi_U+\chi_V)$.
The right-hand side is larger than or equal to the left-hand side by
superlinearity.  Imagine it is strictly larger.  Let
$f \in \Lform \creal$ map any $t \leq 1/2$ to $0$, any
$t \in ]1/2, 3/2]$ to $1$, and any $t > 3/2$ to $2$.  By discreteness,
$F (\chi_U + \chi_V) = F (f \circ (\chi_U + \chi_V)) = f (F (\chi_U +
\chi_V))$, which shows that $F (\chi_U + \chi_V) \in \{0, 1, 2\}$.
Since $F (\chi_U) + F (\chi_V) < F (\chi_U+\chi_V)$, $F (\chi_U)$ and
$F (\chi_V)$ cannot both be equal to $1$, and since
$F (\chi_U) \leq F (\chi_V)$, at least $F (\chi_U)$ is equal to $0$.
Hence $F (\chi_V) < F (\chi_U+\chi_V)$.  Since
$F (\chi_U + \chi_V) \leq F (2 \chi_V)$, we obtain that
$F (\chi_U + \chi_V)$ is a natural number between $F (\chi_V)$
(exclusive) and $2 F (\chi_V)$ (inclusive).  This implies that
$F (\chi_V)=1$ and $F (\chi_U + \chi_V) = 2$.  Now let
$f = \chi_{]3/2, +\infty]}$, and observe that
$f \circ (\chi_U + \chi_V) = \chi_U$.  Therefore
$F (f \circ (\chi_U + \chi_V)) = F (\chi_U) = 0$.  By discreteness,
$F (f \circ (\chi_U + \chi_V)) = f (F (\chi_U + \chi_V)) = f (2) = 1$,
contradiction.

Using that, we show that $\mathcal F$ is a filter.  It remains to show
that for any two elements $U_1$, $U_2$ of $\mathcal F$, $U_1 \cap U_2$
is also in $\mathcal F$.  Since $U_1 \cap U_2 \subseteq U_1 \cup U_2$,
by the equality we have just shown,
$F (\chi_{U_1 \cup U_2}) + F (\chi_{U_1 \cap U_2}) = F (\chi_{U_1 \cup
  U_2} + \chi_{U_1 \cap U_2}) = F (\chi_{U_1} + \chi_{U_2})$, which is
larger than or equal to $F (\chi_{U_1}) + F (\chi_{U_2})$ by
superlinearity, that is, to $2$.  Hence
$F (\chi_{U_1 \cap U_2}) \geq 2 - F (\chi_{U_1 \cup U_2}) \geq 1$,
where the latter inequality comes from the fact that
$F (\chi_{U_1 \cup U_2}) $ can only be equal to $0$ or to $1$.  Since
$F (\chi_{U_1 \cap U_2}) $ can only be equal to $0$ or to $1$, its
value is $1$, so $U_1 \cap U_2$ is in $\mathcal F$.

Finally, $\mathcal F$ is Scott-open, meaning that for every directed
family ${(U_i)}_{i \in I}$ of open subsets whose union is in $\mathcal
F$, some $U_i$ is in $\mathcal F$.  This follows from the
Scott-continuity of $F$.

We now appeal to the Hofmann-Mislove Theorem: every Scott-open filter
of open subsets of a sober space is the filter of open neighborhoods
of a unique compact saturated subset $Q$.  Then, for every open subset
$U$, $Q \subseteq U$ if and only if $U \in \mathcal F$ if and only if
$F (\chi_U) = 1$.  Since $F (0)=0$, the empty set is not in $\mathcal
F$, so $Q$ is non-empty.

We check that $F = F_Q$.  Let $h \in \Lform X$.  For every $t \in
\Rp$, let $f = \chi_{]t, +\infty]}$.  Since $F$ is discrete, $F (f
\circ h) = f (F (h))$.   Let $U = h^{-1} (]t, +\infty])$, so that $f
\circ h = \chi_U$.  Then $F (h) > t$ if and only if $f (F (h)) = 1$, if
and only if $F (\chi_U)=1$ if and only if $Q \subseteq U$.
If $Q \subseteq U$, then $F_Q (h) = \min_{x \in Q} h (x) > t$, and conversely.
Therefore $F (h) > t$ if and only if $F_Q (h) > t$, for all $h$ and
$t$, whence $F = F_Q$.

Finally, uniqueness is obvious: if $F = F_Q$ for some non-empty
compact saturated subset $Q$, then $Q$ must be such that for every
open subset $U$, $Q \subseteq U$ if and only if $F (\chi_U)=1$.  \qed

\begin{defi}[$\dQ$]
  \label{defn:dQ}
  Let $X, d$ be a quasi-metric space.  For any two non-empty compact
  saturated subsets $Q$, $Q'$ of $X$, let
  $\dQ (Q, Q') = \dKRH (F_Q, F_{Q'})$.
\end{defi}
\begin{rem}
  \label{defn:dQa}
  It of course makes sense to define $\dQ^a (Q, Q')$ as $\dKRH^a (F_Q,
  F_{Q'})$ as well.
\end{rem}
We give a more concrete description of $\dQ$ in Lemma~\ref{lemma:dQ}.
We start with an easy observation.
\begin{lem}
  \label{lemma:d:lip}
  Let $X, d$ be a standard quasi-metric space.  For every point $x'
  \in X$, the map $d (\_, x')$ is in $\Lform_1 (X, d)$.
\end{lem}
\proof
Let $h = d (\_, x')$.  This is a $1$-Lipschitz
Yoneda-continuous map \cite[Exercise~7.4.36]{JGL-topology}, hence it
is $1$-Lipschitz continuous, since $X, d$ is standard.  Alternatively,
$h$ is also equal to the map $d (\_, C)$ where $C$ is the closed set
$\dc x'$, and that is $1$-Lipschitz continuous by \cite[Lemma~6.11]{JGL:formalballs}.
\qed

\begin{lem}
  \label{lemma:dQ}
  Let $X, d$ be a standard quasi-metric space.  For all non-empty
  compact saturated subsets $Q$, $Q'$,
  $\dQ (Q, Q') = \sup_{x' \in Q'} d (Q, x')$, where
  $d (Q, x') = \inf_{x \in Q} d (x, x')$.
\end{lem}
\proof We first show that
$\dQ (Q, Q') \leq \sup_{x' \in Q'} d (Q, x')$.  It suffices to show
that for every $r \in \Rp$ such that $r < \dQ (Q, Q')$,
$r \leq d (Q, x')$ for some $x' \in Q'$.  Since
$r < \dQ (Q, Q') = \dKRH (F_Q, F_{Q'})$, there is an
$h \in \Lform_1 (X, d)$ such that $F_Q (h) > F_{Q'} (h) + r$.  Let
$x' \in Q'$ be such that $F_{Q'} (h) = h (x')$, using
Lemma~\ref{lemma:uQ}, item~1.  For every $x \in Q$,
$h (x) > h (x') + r $, so, since $h$ is $1$-Lipschitz,
$d (x, x') > r$.  It follows that $d (Q, x') \geq r$.

In the reverse direction, we fix $x' \in Q'$, $r < d (Q, x')$, and we
show that there is an $h \in \Lform_1 (X, d)$ such that
$F_Q (h) \geq r + F_{Q'} (h)$.  Since $r < d (Q, x')$, $d (x, x') > r$
for every $x \in Q$.  Let $h = d (\_, x')$.  This is in
$\Lform_1 (X, d) $ by Lemma~\ref{lemma:d:lip}, and $h (x) > r$ for
every $x \in Q$.  It follows that
$F_Q (h) = \min_{x \in Q} h (x) > r$, whereas
$F_{Q'} (h) \leq h (x')=0$.  \qed

Lemma~\ref{lemma:dQ} means that $\dQ (Q, Q')$ is given by one half of
the familiar Hausdorff formula $\dQ (Q, Q') = \sup_{x' \in Q'} \inf_{x
  \in Q} d (x, x')$.  The quasi-metric $\dH$ is given by a formula
that is almost the other half, $\sup_{x \in Q} \inf_{x' \in Q'} d (x, x')$.

\begin{rem}
  \label{rem:dQa}
  We also have $\dQ^a (Q, Q') = \min (a, \dQ (Q, Q'))$, for every
  $a \in \Rp$, $a > 0$.  In one direction,
  $\dQ^a (Q, Q') = \dKRH^a (F_Q, F_{Q'}) \leq a$ and
  $\dQ^a (Q, Q') = \dKRH^a (F_Q, F_{Q'}) \leq \dKRH (F_Q, F_{Q'}) =
  \dQ (Q, Q')$, so $\dQ^a (Q, Q') \leq \min (a, \dQ (Q, Q'))$.  In the
  other direction, we fix $r < \min (a, \dQ (Q, Q'))$.  In particular,
  $r < a$ and there is an $x' \in Q'$ such that $r < d (Q, x')$,
  whence $d (x, x') > r$ for every $x \in Q$.  Let
  $h = \min (a, d (\_, x')) \in \Lform_1^a (X, d)$:
  $F_Q (h) = \min (a, \min_{x \in Q} d (x, x')) > r$, and
  $F_{Q'} (h) = 0$, so $\dQ^a (Q, Q') = \dKRH^a (F_Q, F_{Q'}) \geq r$.
  Since $r$ is arbitrary, $\min (a, \dQ (Q, Q')) \leq \dQ^a (Q, Q')$.
\end{rem}

\begin{rem}
  \label{rem:Q:root}
  For every standard quasi-metric space $X, d$ with an $a$-root $x$,
  ($a \in \Rp$, $a > 0$), $\Smyth X, \dQ$ has an $a$-root, namely
  $\upc x$.  Indeed, fix $Q \in \Smyth X$.  For every $y \in X$,
  $d (x, y) \leq a$.  For every $x' \in \upc x$,
  $d (x, y) \leq d (x, x') + d (x', y) = d (x', y)$, so
  $d (\upc x, y) = \inf_{x' \in \upc x} d (x', y) = d (x, y)$.  By
  Lemma~\ref{lemma:dQ},
  $d (\upc x, Q) = \sup_{y \in Q} d (\upc x, y) = \sup_{y \in Q} d (x,
  y) \leq a$.
\end{rem}

\begin{lem}
  \label{lemma:d(Q,x):props}
  Let $X, d$ be a standard quasi-metric space.  For every non-empty
  compact saturated subset $Q$ of $X$, for every $x' \in X$, the
  following hold:
  \begin{enumerate}
  \item there is an $x \in Q$ such that $d (Q, x') = d (x, x')$;
    $d (Q, x') = \min_{x \in Q} d (x, x')$;
  \item $d (Q, x') = 0$ if and only if $x' \in Q$;
  \item $d (Q, x') \leq d (Q, y') + d (y', x')$ for every $y' \in X$.
  \end{enumerate}
\end{lem}
\proof 1. Because $d (\_, x')$ is in
$\Lform_1 (X, d) \subseteq \Lform X$ (Lemma~\ref{lemma:d:lip}), it
reaches its minimum on the compact set $Q$.

2. Let $x \in Q$ be such that $d (Q, x') = d (x, x')$, by (1).  If
$d (Q, x')=0$, then $d (x, x')=0$, so $x \leq x'$, and since $Q$ is
saturated, $x'$ is in $Q$.  Conversely, if $x' \in Q$, then
$d (x', x') = 0$ implies that $d (Q, x')=0$.

3. Let $y'$ be an arbitrary point of $X$.  For every $x \in Q$,
$d (x, x') \leq d (x, y') + d (y', y)$, and we obtain (3) by taking
minima over $x \in Q$ on both sides.  \qed

\begin{lem}
  \label{lemma:dQ:spec}
  Let $X, d$ be a standard quasi-metric space.  For all non-empty
  compact saturated subsets $Q$, $Q'$ of $X$, $\dQ (Q, Q')=0$ if and
  only if $Q \supseteq Q'$.
\end{lem}
\proof $\dQ (Q, Q') = 0$ if and only if $\sup_{x' \in Q'} d (Q, x')=0$
by Lemma~\ref{lemma:dQ}, if and only if $d (Q, x')=0$ for every $x'
\in Q'$.  By Lemma~\ref{lemma:d(Q,x):props}, this is equivalent to
requiring that $x' \in Q$ for every $x' \in Q'$.  \qed

\begin{rem}
  \label{rem:dQa:spec}
  Because of Remark~\ref{rem:dQa}, it also follows that for every $a
  \in \Rp$, $a > 0$, $\dQ^a (Q, Q') = 0$ if and only if $Q \supseteq Q'$.
\end{rem}

\subsection{Completeness}
\label{sec:Q:completeness}

\begin{lem}
  \label{lemma:Q:supp}
  Let $X, d$ be a standard quasi-metric space.  For every compact
  saturated subset $\mathcal Q$ of $\mathbf B (X, d)$ such that
  $F_{\mathcal Q}$ is supported on $V_{1/2^n}$ for every $n \in \nat$,
  \begin{enumerate}
  \item $\mathcal Q$ is included in $X$;
  \item $\mathcal Q$ is a compact saturated subset of $X$;
  \item $F_{\mathcal Q}$ is supported on $X$.
  \end{enumerate}
\end{lem}
\proof
If $\mathcal Q$ is empty, this is obvious, so let us assume that
$\mathcal Q$ is non-empty.

For any two real numbers $r, s > 0$, $\chi_{V_r}$ and $\chi_{V_s}$
coincide on $V_{1/2^n}$, where $n$ is any natural number large enough
that $1/2^n \leq r, s$.  Therefore
$F_{\mathcal Q} (\chi_{V_r}) = F_{\mathcal Q} (\chi_{V_s})$.  It
follows that $F_{\mathcal Q} (\chi_{V_r})$ is a value $a$ that does
not depend on $r > 0$.  The union
$\bigcup_{r \in \Rp \smallsetminus \{0\}} V_r$ is the whole space of
formal balls, so
$\sup_{r \in \Rp \smallsetminus \{0\}} \chi_{V_r} = \mathbf 1$.  Since
$F_{\mathcal Q}$ is Scott-continuous,
$\sup_{r \in \Rp \smallsetminus \{0\}} F_{\mathcal Q} (\chi_{V_r}) =
F_{\mathcal Q} (\mathbf 1) = 1$ (since $\mathcal Q$ is non-empty), and
that is also equal to $\sup_{r \in \Rp \smallsetminus \{0\}} a = a$.
We have obtained that $F_{\mathcal Q} (\chi_{V_r}) = 1$ for every
$r > 0$, namely that $\mathcal Q \subseteq V_r$ for every $r > 0$.  As
a consequence $\mathcal Q$ is included in $\bigcap_{r > 0} V_r = X$,
showing (1).  (2) follows from Lemma~\ref{lemma:compact:subspace}.
(3) is now obvious.  \qed

Let $\Smyth X$ be the set of all non-empty compact saturated subsets
of $X$, with the $\dQ$ quasi-metric.  $\Smyth X$ is the \emph{Smyth
  powerdomain} of $X$.  By Lemma~\ref{lemma:uQ}, if $X$ is sober in
its $d$-Scott topology, then the bijection $Q \mapsto F_Q$ is an
isometry of $\Smyth X, \dQ$ onto the set of normalized discrete
superlinear previsions on $X$, with the usual $\dKRH$ quasi-metric.
\begin{prop}[Yoneda-completeness, Smyth powerdomain]
  \label{prop:QX:Ycomplete}
  Let $X, d$ be a sober standard quasi-metric space.  Then
  $\Smyth X, \dQ$ is Yoneda-complete.

  Suprema of directed families of formal balls of previsions are naive
  suprema.
\end{prop}
\proof This follows from Proposition~\ref{prop:supp:complete}, since
the assumption on supports is verified by Lemma~\ref{lemma:Q:supp}.
\qed

\begin{rem}
  \label{rem:QX:Ycomplete:a}
  Similarly, when $X, d$ is standard and sober, $\Smyth X, \dQ^a$ is
  Yoneda-complete for every $a \in \Rp$, $a > 0$.
\end{rem}
As a first example of sober standard quasi-metric spaces, one can find
all metric spaces.  They are all standard, and they are sober since
$T_2$.  They are also continuous.

A second family of sober standard quasi-metric spaces consists in the
continuous Yoneda-complete quasi-metric spaces, in particular all
algebraic Yoneda-complete quasi-metric spaces.  We have already
mentioned that all Yoneda-complete quasi-metric spaces are standard
\cite[Proposition~2.2]{JGL:formalballs}; that continuous
Yoneda-complete quasi-metric spaces are sober is Proposition~4.1 of
loc.cit.  Sobriety could of course be dispensed with it we worked
directly on normalized discrete superlinear previsions, instead of the
more familiar elements of $\Smyth X$.

\subsection{$\dQ$-Limits}
\label{sec:dq-limits}

As for Lemma~\ref{lemma:H:sup}, there is a more direct expression of
directed suprema of formal balls over $\Smyth X, \dQ$, i.e., of
$\dQ$-limits, than by relying on naive suprema.  Recall that $Q+r$ is
the set $\{(x, r) \mid x \in Q\}$.
\begin{lem}
  \label{lemma:Q:sup}
  Let $X, d$ be a sober standard quasi-metric space.  Then:
  \begin{enumerate}
  \item In $\mathbf B (\Smyth X, \dQ)$, $(Q, r) \leq^{\dQ^+} (Q', r')$
    if and only if 
    $Q'+r' \subseteq \upB (Q+r)$, where $\upB$ is upward closure in
    $\mathbf B (X, d)$.
  \item If $X, d$ is continuous Yoneda-complete, then for every
    directed family ${(Q_i, r_i)}_{i \in I}$, the supremum $(Q, r)$ is
    given by $r = \inf_{i \in I} r_i$ and
    $Q = \bigcap_{i \in I} \upB (Q_i + r_i - r)$.
  \end{enumerate}
\end{lem}
\proof (1) If $(Q, r) \leq^{\dQ^+} (Q', r')$, then
$\dQ (Q, Q') \leq r-r'$, so $r \geq r'$ and for every $x' \in Q'$,
$d (Q, x') \leq r-r'$, by Lemma~\ref{lemma:dQ}.  Using
Lemma~\ref{lemma:d(Q,x):props}~(1), there is an $x \in Q$ such that
$d (x, x') \leq r-r'$.  In particular, $(x, r) \leq^{d^+} (x', r')$.
This shows that every element of $Q'+r'$ is in $\upB (Q, r)$.

Conversely, if every element of $Q'+r'$ is in $\upB (Q, r)$, then in
particular $r \geq r'$.  Indeed, since $Q'$ is non-empty, we can find
$x' \in Q'$, and $(x', r') \in \upB (Q, r)$.  There is an $x \in Q$ such that
$(x, r) \leq^{d^+} (x', r')$, and that implies $r \geq r'$.

We rewrite the inclusion $Q'+r' \subseteq \upB (Q+r)$ as: for every
$x' \in Q'$, there is an $x \in Q$ such that
$(x, r) \leq^{d^+} (x', r')$, i.e., $d (x, x') \leq r-r'$.  It follows
that $\dQ (Q, Q') \leq r-r'$ by Lemma~\ref{lemma:dQ}, whence
$(Q, r) \leq^{\dQ^+} (Q', r')$.

(2) Let $r = \inf_{i \in I} r_i$ and
$Q = \bigcap_{i \in I} \upB (Q_i + r_i - r)$.  Since $X, d$ is
standard, the map $\_ + r_i - r$ is Scott-continuous, and since
$\eta_X$ is also continuous, the image of $Q_i + r_i - r$ of $Q_i$ by
their composition is compact in $\mathbf B (X, d)$.  Hence
$\upB (Q_i + r_i - r)$ is compact saturated in $\mathbf B (X, d)$.

Let $i \sqsubseteq j$ if and only if
$(Q_i, r_i) \leq^{\dQ^+} (Q_j, r_j)$.  If $i \sqsubseteq j$ then
$(Q_i, r_i-r) \leq^{\dQ^+} (Q_j, r_j-r)$, so
$\upB (Q_i + r_i - r) \supseteq \upB (Q_j + r_j - r)$ by (1).  It
follows that the family ${(\upB (Q_i + r_i - r))}_{i \in I}$ is
filtered.

The radius of $\upB (Q_i + r_i - r)$, as defined in
Lemma~\ref{lemma:Q:radius}, is equal to $r_i - r$, and
$\inf_{i \in I} (r_i-r)=0$.  Therefore, by
Lemma~\ref{lemma:Xd:compact:1}, $Q$ is non-empty, compact, and
saturated in $\mathbf B (X, d)$, hence an element of $\Smyth X$.

We claim that $\upB (Q+r) = \bigcap_{i \in I} \upB (Q_i + r_i)$.  For
every element $(x, r)$ of $Q+r$ (i.e., $x \in Q$), for every
$i \in I$, $(x, 0)$ lies above some element $(x_i, r_i-r)$ where
$x_i \in Q_i$; so $d (x_i, x) \leq r_i-r$, which implies
$(x_i, r_i) \leq^{d^+} (x, r)$, hence $(x, r) \in \upB (Q_i + r_i)$.
That shows $Q+r \subseteq \bigcap_{i \in I} \upB (Q_i + r_i)$.  Since
the right-hand side is upwards-closed, $\upB (Q+r)$ is also included
in $\bigcap_{i \in I} \upB (Q_i + r_i)$.  In the reverse direction,
it is enough to show that every Scott-open neighborhood $\mathcal U$
of $\upB (Q+r)$ contains $\bigcap_{i \in I} \upB (Q_i + r_i)$.  Since
$X, d$ is standard, the map $\_ + r$ is Scott-continuous, so
$(\_ + r)^{-1} (\mathcal U)$ is Scott-open.  Since $Q+r$ is included
in $\mathcal U$, $Q$ is included in $(\_ + r)^{-1} (\mathcal U)$.
Since $X, d$ is continuous Yoneda-complete, $\mathbf B (X, d)$ is a
continuous dcpo, hence is sober, hence well-filtered, so, using the
definition of $Q$, $Q_i+r_i-r$ is included in
$(\_ + r)^{-1} (\mathcal U)$ for some $i \in I$.  That means that
$Q_i + r_i$ is included in $\mathcal U$, hence also $\upB (Q_i + r_i)$
since every Scott-open set is upwards-closed.  Therefore, $\mathcal U$
indeed contains $\bigcap_{i \in I} \upB (Q_i + r_i)$.

We can now finish the proof.  Since $Q+r$ is included in
$\upB (Q_i + r_i)$, $(Q_i, r_i) \leq^{\dQ^+} (Q, r)$ by (1).  For
every upper bound $(Q', r')$ of ${(Q_i, r_i)}_{i \in I}$,
$r' \leq \inf_{i \in I} r_i = r$, and $Q'+r'$ is included in
$\upB (Q_i + r_i)$ for every $i \in I$, by (1).  Hence $Q'+r'$ is
included in $\bigcap_{i \in I} \upB (Q_i + r_i) = \upB (Q+r)$, showing
that $(Q, r) \leq^{\dQ^+} (Q', r')$, by (1) again.  \qed

\begin{rem}
  \label{rem:Q:sup}
  One can rephrase Lemma~\ref{lemma:Q:sup}~(2) as follows.  Let $X, d$
  be a continuous Yoneda-complete quasi-metric space.  Let
  ${(Q_i)}_{i \in I, \sqsubseteq}$ be a Cauchy-weightable net, and
  ${(Q_i, r_i)}_{i \in I, \sqsubseteq}$ be some corresponding
  Cauchy-weighted net.  Then the $\dQ$-limit of
  ${(Q_i)}_{i \in I, \sqsubseteq}$ is $Q$ is the filtered intersection
  $\bigcap_{i \in I} \upB (Q_i+r_i)$.  Since $Q$ is included in $X$,
  this is also equal to $\bigcap_{i \in I} (X \cap \upB (Q_i+r_i))$.
  Note that $X \cap \upB (Q_i+r_i)$ is the set $Q_i^{+r_i}$ of points
  $\{x \in X \mid \exists y \in Q_i. d (y, x) \leq r_i\}$ that are at
  distance at most $r_i$ from $Q_i$.  (Beware that there is no reason
  to believe that $Q_i^{+r_i}$ would be compact.)  Hence $Q$ is the
  filtered intersection $\bigcap_{i \in I} Q_i^{+r_i}$.
\end{rem}

\subsection{Algebraicity}
\label{sec:Q:algebraicity}

\begin{lem}
  \label{lemma:Q:simple:center}
  Let $X, d$ be a continuous Yoneda-complete quasi-metric space.  For
  all $n \geq 1$, and center points $x_1$, \ldots, $x_n$,
  $\upc \{x_1, \cdots, x_n\}$ is a center point of $\Smyth X, \dQ$.
\end{lem}
\proof Recall that every continuous Yoneda-complete quasi-metric space
is sober.

Let $Q_0 = \upc \{x_1, \cdots, x_n\}$.  For every $h \in \Lform X$,
$F_{Q_0} (h) = \min_{j=1}^n h (x_j)$.  Let
$U = B^{\dKRH^+}_{(F_{Q_0}, 0), <\epsilon}$.  $U$ is upwards-closed:
if $(F_Q, r) \leq^{\dKRH^+} (F_{Q'}, r')$ and $(F_Q, r)$ is in $U$,
then $\dKRH (F_Q, F_{Q'}) \leq r-r'$, and
$\dKRH (F_{Q_0}, \allowbreak F_Q) < \epsilon - r$, so
$\dKRH (F_{Q_0}, F_{Q'}) < r-r'+\epsilon-r = \epsilon-r'$, showing
that $(F_{Q'}, r')$ is in $U$.

To show that $U$ is Scott-open, consider a directed family
${(Q_i, r_i)}_{i \in I}$ in $\mathbf B (\Smyth X, \dQ)$, with supremum
$(Q, r)$, and assume that $(F_Q, r)$ is in $U$.  By
Proposition~\ref{prop:QX:Ycomplete}, this is a naive supremum, so
$r = \inf_{i \in I} r_i$ and $Q$ is characterized by the fact that,
for every $h \in \Lform_1 (X, d)$,
$F_Q (h) = \sup_{i \in I} (F_{Q_i} (h) + r - r_i)$.  Since $(F_Q, r)$
is in $U$, $\dKRH (F_{Q_0}, F_Q) < \epsilon - r$, so $\epsilon > r$
and $F_{Q_0} (h) - \epsilon + r < F_Q (h)$ for every
$h \in \Lform_1 (X, d)$.  Therefore, for every $h \in \Lform_1 (X, d)$,
there is an index $i \in I$ such that
$F_{Q_0} (h) - \epsilon + r < F_{Q_i} (h) + r - r_i$, or equivalently:
\begin{equation}
  \label{eq:Q:A}
  \min_{j=1}^n h (x_j) < F_{Q_i} (h) + \epsilon - r_i.
\end{equation}
Moreover, since $\epsilon > r = \inf_{i \in I} r_i$, we may assume
that $i$ is so large that $\epsilon > r_i$.

Let $V_i$ be the set of all $h \in \Lform_1 (X, d)$ such that
(\ref{eq:Q:A}) holds.  Each $V_i$ is the inverse image of
$[0, \epsilon - r_i[$ by the map
$h \mapsto \dreal (\min_{j=1}^n h (x_j), F^{C_i} (h))$, which is
continuous from $\Lform_1 (X, d)^\patch$ to $(\creal)^\dG$ by
Proposition~\ref{prop:dKRH:cont}~(3).  Hence ${(V_i)}_{i \in I}$ is an
open cover of $\Lform_1 (X, d)^\patch$.  The latter is a compact space
by Lemma~\ref{lemma:cont:Lalpha:retr}~(4).  Hence we can extract
a finite subcover ${(V_i)}_{i \in J}$: for every
$h \in \Lform_1 (X, d)$, there is an index $i$ in the finite set $J$
such that (\ref{eq:Q:A}) holds.  By directedness, one can require that
$i$ be the same for all $h$.  This shows that $(F_{Q_i}, r_i)$ is in
$U$, proving the claim.  \qed

\begin{rem}
  \label{rem:Qa:simple:center}
  A similar result holds for $\Smyth X, \dQ^a$ for any $a \in \Rp$, $a
  > 0$, and the argument is the same as for Lemma~\ref{lemma:Q:simple:center}.
\end{rem}

\begin{lem}
  \label{lemma:Q:approx}
  Let $X, d$ be a standard algebraic quasi-metric space, with strong
  basis $\mathcal B$.  For every compact subset $Q$ of $X$, for every
  open neighborhood $U$ of $Q$, and for every $\epsilon > 0$, there
  are finitely many points $x_1$, \ldots, $x_n$ in $\mathcal B$ and
  radii $r_1, \ldots, r_n < \epsilon$ such that
  $Q \subseteq \bigcup_{j=1}^n B^d_{x_j, <r_j} \subseteq U$.
\end{lem}
\proof For every $y$ in $Q$, we can find a formal ball
$(x, r) \ll (y, 0)$ in $\widehat U$ such that $x \in \mathcal B$ and
such that $r < \epsilon$, by Lemma~\ref{lemma:B:basis}.  The
corresponding open balls $B^d_{x, <r}$ are open since $x$ is a center
point, hence form an open cover.  It remains to extract a finite
subcover, thanks to the compactness of $Q$.  \qed



\begin{lem}
  \label{lemma:Q:B:eps}
  Let $X, d$ be a standard algebraic quasi-metric space.  Let $Q$ be a
  compact saturated subset of $X$.  For all center points $x_1$,
  \ldots, $x_m$ and all $r_1, \cdots, r_m > 0$ such that
  $Q \subseteq \bigcup_{j=1}^m B^d_{x_j, <r_j}$, there is an
  $\epsilon > 0$, $\epsilon < r_1, \cdots, r_m$, such that
  $Q \subseteq \bigcup_{j=1}^m B^d_{x_j, <r_j-\epsilon}$.
\end{lem}
\proof For each $\epsilon > 0$ such that
$\epsilon < r_1, \cdots, r_m$, we consider the open subset
$U_\epsilon = \bigcup_{j=1}^m B^d_{x_j, < r_j-\epsilon}$.  For every
$x \in Q$, $x$ is in some open ball $B^d_{x_j, <r_j}$ by assumption,
so $d (x, x_j) < r_j$.  This implies that there is an $\epsilon > 0$
such that $d (x, x_j) < r_j-\epsilon$ (in particular
$\epsilon < r_j$).  That inequality is preserved by replacing
$\epsilon$ by a smaller positive number strictly less than $r_1$,
\ldots, $r_m$.  Then $x$ is in $U_\epsilon$.  The family
${(U_\epsilon)}_{0 < \epsilon < r_1, \cdots, r_m}$ is then a chain
that forms an open cover of $Q$.  Since $Q$ is compact, $Q$ is
included in $U_\epsilon$ for some $\epsilon > 0$ with
$\epsilon < r_1, \cdots, r_m$.  \qed

\begin{thm}[Algebraicity of Smyth powerdomains]
  \label{thm:Q:alg}
  Let $X, d$ be an algebraic Yoneda-complete quasi-metric space, with
  a strong basis $\mathcal B$.

  The space $\Smyth X, \dQ$ is algebraic Yoneda-complete.

  Every non-empty compact saturated set of the form
  $\upc \{x_1, \cdots, x_n\}$ where every $x_i$ is a center point is a
  center point of $\Smyth X, \dQ$, and those sets form a strong basis,
  even when each $x_i$ is taken from $\mathcal B$.
\end{thm}
\proof First recall that every algebraic Yoneda-complete quasi-metric
space is continuous Yoneda-complete, and that every continuous
quasi-metric space is sober.  Then $\Smyth, \dQ$ is Yoneda-complete by
Proposition~\ref{prop:QX:Ycomplete}, and $\upc \{x_1, \cdots, x_n\}$
is a center point as soon as every $x_i$ is a center point, by
Lemma~\ref{lemma:Q:simple:center}.

Let $Q \in \Smyth X$.  We wish to show that $(Q, 0)$ is the
supremum of a directed family $D$ of formal balls $(Q_i, r_i)$ below
$(Q, 0)$, where every $Q_i$ is the upward closure of finitely many
center points.

Let $D$ be the family of all formal balls
$(\upc \{x_1, \cdots, x_n\}, r)$, where $n \geq 1$, every $x_j$ is in
$\mathcal B$, and $Q \subseteq \bigcup_{j=1}^n B^d_{x_j, <r}$.

We start by showing that $D$ is non-empty, and that we can in fact
find elements of $D$ with arbitrary small radius.  By
Lemma~\ref{lemma:Q:approx} with $U=X$, for every $\epsilon > 0$, there
are finitely many points $x_1$, \ldots, $x_n$ in $\mathcal B$ and
radii $r_1, \ldots, r_n < \epsilon$ such that
$Q \subseteq \bigcup_{j=1}^n B^d_{x_j, <r_j}$, in particular
$Q \subseteq \bigcup_{j=1}^n B^d_{x_j, <\epsilon}$.  It follows that
$(\upc \{x_1, \cdots, x_n\}, \epsilon)$ is in $D$.

Before we proceed, we make the following observation: $(*)$ for every
$(\upc \{x_1, \cdots, x_m\}, r)$ in $D$, there is an $\epsilon > 0$
such that $\epsilon < r$ and
$(\upc \{x_1, \cdots, x_m\}, r - \epsilon)$ in $D$.  This is exactly
Lemma~\ref{lemma:Q:B:eps} with $r_1, \cdots, r_m = r$.  Note also that
if $(\upc \{x_1, \cdots, x_m\}, r - \epsilon)$ is in $D$, then
$(\upc \{x_1, \cdots, x_m\}, r - \epsilon')$ is also in $D$ for every
$\epsilon' > 0$ with $\epsilon' \leq \epsilon$.

We now claim that $D$ is directed.  Let
$(\upc \{x_1, \cdots, x_m\}, r)$ and $(\upc \{y_1, \cdots, y_n\}, s)$
be two elements of $D$.  Using $(*)$, there is an $\epsilon > 0$ such
that $\epsilon < r, s$ and such that
$(\upc \{x_1, \cdots, x_m\}, r-\epsilon)$ and
$(\upc \{y_1, \cdots, y_n\}, s-\epsilon)$ are again in $D$.  The open
subset
$U = \bigcup_{j=1}^m B^d_{x_j, <r-\epsilon} \cap \bigcup_{k=1}^n
B^d_{y_k, <s-\epsilon}$ contains $Q$.  By Lemma~\ref{lemma:Q:approx},
there is a finite family of points $z_1$, \ldots, $z_p$ in
$\mathcal B$ and associated radii $t_1, \cdots, t_p < \epsilon$ such
that
$Q \subseteq \bigcup_{\ell=1}^p B^d_{z_\ell, < t_\ell} \subseteq U$.
Let $E = \{z_1, \cdots, z_p\}$.  By construction,
$Q \subseteq \bigcup_{\ell=1}^p B^d_{z_\ell, < \epsilon}$, so
$(\upc E, \epsilon)$ is in $D$.  We claim that
$(\upc \{x_1, \cdots, x_m\}, r) \leq^{\dQ^+} (\upc E, \epsilon)$.  To
show that, we use Lemma~\ref{lemma:dQ} and reduce our problem to
showing that
$\sup_{\ell=1}^p \inf_{j=1}^m d (x_j, z_\ell) \leq r-\epsilon$.  For
every $\ell$, $z_\ell$ is in the open ball
$B^d_{z_\ell, <t_\ell} \subseteq U \subseteq \bigcup_{j=1}^m B^d_{x_j,
  <r-\epsilon}$, so $d (x_j, z_\ell) < r-\epsilon$ for some $j$, and
this proves the claim.  We show that
$(\upc \{y_1, \cdots, y_n\}, s) \leq^{\dQ^+} (\upc E, \epsilon)$
similarly.

By definition of $D$, $(Q, 0)$ is an upper bound of $D$: every element
$(\upc \{x_1, \cdots, x_m\}, r)$ of $D$ is such that $Q \subseteq
\bigcup_{j=1}^m B^d_{x_j, <r}$, so $\dQ (\upc \{x_1, \cdots, x_m\}, Q)
= \sup_{x \in Q} \inf_{j=1}^m d (x_j, x) \leq r$.

We claim that $(Q, 0)$ is the least upper bound of $D$.  Let
$(Q', r')$ be another upper bound.  Since $D$ contains formal balls of
arbitrarily small radius, $r'=0$.  Let us assume that $(Q, 0)$ is not
below $(Q', 0)$.  By definition, this means that $\dQ (Q, Q') > 0$.
Let $\epsilon > 0$ be such that $\epsilon < \dQ (Q, Q')$.  By
Lemma~\ref{lemma:dQ}, there is a point $x' \in Q'$ such that
$d (Q, x') > \epsilon$, so $d (x, x') > \epsilon$ for every $x \in Q$.
Let $U$ be the open hole $T^d_{x', >\epsilon}$, a $d$-Scott open set
by Lemma~\ref{lemma:hole}.  By construction, $U$ contains $Q$.  By
Lemma~\ref{lemma:Q:approx}, $U$ contains a finite union
$\bigcup_{j=1}^n B^d_{x_j, <r_j}$ of formal balls centered at points
of $\mathcal B$, that union contains $Q$, and $r_j < \epsilon$ for
every $j$.  Then $(\upc \{x_1, \cdots, x_n\}, \epsilon)$ is in $D$,
hence is below $(Q', 0)$.  This means that
$\dQ (\upc \{x_1, \cdots, x_n\}, Q') \leq \epsilon$.  In particular,
$d (\upc \{x_1, \cdots, x_n\}, x') \leq \epsilon$, so
$d (x_j, x') \leq \epsilon$ for some $j$.  (Indeed,
$d (x, x') \leq \epsilon$ for some $x \in \upc \{x_1, \cdots, x_n\}$,
say $x_j \leq x$.  Then $d (x_j, x') \leq d (x_j, x) + d (x, x')$, and
$d (x_j, x)=0$ since $x_j \leq x$.)  Since
$\bigcup_{j=1}^n B^d_{x_j, <r_j}$ is included in $U$, $x_j$ is in $U$,
so $d (x_j, x') > \epsilon$, contradiction.  \qed

\begin{rem}
  \label{rem:Qa:alg}
  The same result holds for $\dQ^a$, for every $a \in \Rp$, $a > 0$:
  when $X, d$ is standard algebraic, $\Smyth X, \dQ^a$ is algebraic
  Yoneda-complete, with the same strong basis.
  The proof is the same, except in the final step, where we consider
  another upper bound $(Q', r')$.  We must additionally require that
  $\epsilon \leq a$, and this is possible because
  Lemma~\ref{lemma:Q:approx} allows us to take $\epsilon$ as small as
  we wish.  The inequality
  $\dQ^a (\upc \{x_1, \cdots, x_n\}, Q') \leq \epsilon$ then implies
  $\dQ (\upc \{x_1, \cdots, x_n\}, Q') \leq \epsilon$, allowing us to
  conclude as above.
\end{rem}

\subsection{Continuity}
\label{sec:continuity-smyth}

We proceed exactly as in Section~\ref{sec:cont-val-leq-1} for
subnormalized continuous valuations, or as in
Section~\ref{sec:continuity-hoare} for the Hoare powerdomains.

We start with an easy fact.
\begin{fact}
  \label{fact:Q:upc:inf}
  For every topological space $Y$, for every subset $A$ of $Y$, for
  every monotonic map
  $h \colon Y \to \real \cup \{-\infty, +\infty\}$,
  $\inf_{y \in A} h (y) = \inf_{y \in \upc A} h (y)$.  \qed
\end{fact}

\begin{lem}
  \label{lemma:Q:functor}
  Let $X, d$ and $Y, \partial$ be two continuous Yoneda-complete
  quasi-metric spaces, and $f \colon X, d \mapsto Y, \partial$ be a
  $1$-Lipschitz continuous map.  The map
  $\Smyth f \colon \Smyth X, \dQ \to \Smyth Y, \dQ$ defined by
  $\Smyth f (Q) = \upc f [Q]$ is $1$-Lipschitz continuous.
  Moreover, $F_{\Smyth f (Q)} = \Prev f (F_Q)$ for every
  $Q \in \Smyth X$.

  Similarly with $\dQ^a$ instead of $\dQ$.
\end{lem}
Recall that $f [Q]$ denotes the image of $Q$ by $f$.

\proof We first check that $F_{\Smyth f (Q)} = \Prev f (F_Q)$.  For
every $h \in \Lform X$,
$F_{\Smyth f (Q)} (h) = \inf_{y \in \upc (f [Q])} h (y) = \inf_{y \in
  f [Q]} h (y)$ by Fact~\ref{fact:Q:upc:inf}.  That is equal to
$\inf_{x \in X} h (f (x))$.  We also have
$\Prev f (F_Q) (h) = F_Q (h \circ f) = \inf_{x \in X} h (f (x))$.
Therefore $F_{\Smyth f (Q)} = \Prev f (F_Q)$.

By Lemma~\ref{lemma:Pf:lip}, $\Prev f$ is $1$-Lipschitz, so
$\mathbf B^1 (\Prev f)$ is monotonic.  Also, $\Prev f$ maps normalized
discrete superlinear previsions to normalized discrete superlinear
previsions.  By Proposition~\ref{prop:QX:Ycomplete}, $\Smyth X, \dQ$
is Yoneda-complete, hence through the isometry of
Proposition~\ref{prop:H:prev}, the corresponding spaces of normalized
discrete superlinear previsions are Yoneda-complete as well.
Moreover, directed suprema of formal balls are computed as naive
suprema.  By Lemma~\ref{lemma:Pf:lipcont}, $\mathbf B^1 (\Prev f)$
preserves naive suprema, hence all directed suprema.  It must
therefore be Scott-continuous.  Hence $\Smyth f$ is $1$-Lipschitz
continuous.

In the case of $\dQ^a$, the argument is the same, except that we use
Remark~\ref{rem:QX:Ycomplete:a} instead of
Proposition~\ref{prop:QX:Ycomplete}.  \qed

Let $X, d$ be a continuous Yoneda-complete quasi-metric space.
There is an algebraic Yoneda-complete quasi-metric space $Y, \partial$
and two $1$-Lipschitz continuous maps $r \colon Y, \partial \to X, d$
and $s \colon X, d \to Y, \partial$ such that $r \circ s = \identity
X$.

By Lemma~\ref{lemma:Q:functor}, $\Smyth r$ and $\Smyth s$ are also
$1$-Lipschitz continuous, and clearly
$\Smyth r \circ \Smyth s = \identity {\Smyth X}$, so $\Smyth X, \dQ$ is
a $1$-Lipschitz continuous retract of $\Smyth Y, \mQ\partial$.
(Similarly with $\dQ^a$ and $\mQ\partial^a$.)  Theorem~\ref{thm:Q:alg}
states that $\Smyth Y, \mQ\partial$ (resp., $\mQ\partial^a$, using
Remark~\ref{rem:Qa:alg} instead) is algebraic Yoneda-complete, whence:
\begin{thm}[Continuity for the Smyth powerdomain]
  \label{thm:Q:cont}
  Let $X, d$ be a continuous Yoneda-complete quasi-metric space.  The
  quasi-metric spaces $\Smyth X, \dQ$ and $\Smyth X, \dQ^a$
  ($a \in \Rp$, $a > 0$) are continuous Yoneda-complete.  \qed
\end{thm}

Together with Lemma~\ref{lemma:Q:functor}, and
Theorem~\ref{thm:Q:alg} for the algebraic case, we obtain.
\begin{cor}
  \label{cor:Q:functor}
  $\Smyth, \dQ$ defines an endofunctor on the category of continuous
  Yoneda-complete quasi-metric spaces and $1$-Lipschitz continuous
  map.  Similarly with $\dQ^a$ instead of $\dQ$ ($a > 0$), or with
  algebraic instead of continuous.  \qed
\end{cor}

\subsection{The Upper Vietoris Topology}
\label{sec:upper-viet-topol}

The upper Vietoris topology on $\Smyth X$ is generated by basic open
subsets $\Box U = \{Q \in \Smyth X \mid Q \subseteq U\}$, where $U$
ranges over the open subsets of $X$.

\begin{lem}
  \label{lemma:Q:V=weak}
  Let $X, d$ be a standard quasi-metric space that is sober in its
  $d$-Scott topology.  The map $Q \mapsto F_Q$ is a homeomorphism of
  $\Smyth X$ with the upper Vietoris topology onto the space of
  normalized discrete superlinear previsions on $X$ with the weak
  topology.
\end{lem}
\proof This is a bijection by Lemma~\ref{lemma:uQ}.  Fix
$h \in \Lform X$, $a \in \Rp$.  For every $Q \in \Smyth X$, $F_Q$ is
in $[h > a]$ if and only if $h (x) > a$ for every $x \in Q$, if and
only if $Q \in \Box h^{-1} (]a, +\infty])$.  Therefore the bijection
is continuous.  In the other direction, for every open subset $U$,
$\Box U$ is the set of all $Q \in \Smyth X$ such that
$F_Q \in [\chi_U > 1/2]$.  \qed

\begin{lem}
  \label{lemma:Q:Vietoris}
  Let $X, d$ be a standard quasi-metric space that is sober in its
  $d$-Scott topology, and $a, a' > 0$ with $a \leq a'$.  We have the
  following inclusion of topologies:
  \begin{quote}
    upper Vietoris $\subseteq$ $\dQ^a$-Scott $\subseteq$ $\dQ^{a'}$-Scott
    $\subseteq$ $\dQ$-Scott.
  \end{quote}
\end{lem}
\proof Considering Lemma~\ref{lemma:Q:V=weak}, this is a consequence
of Proposition~\ref{prop:weak:dScott:a}.  Note that the latter applies
because directed suprema of formal balls are indeed naive suprema, due
to Proposition~\ref{prop:QX:Ycomplete}.  \qed

\begin{prop}
  \label{prop:Q:Vietoris}
  Let $X, d$ be an algebraic Yoneda-complete quasi-metric space.  The
  $\dQ$-Scott topology and the upper Vietoris topologies coincide on
  $\Smyth X$.
\end{prop}
\proof Recall that algebraic Yoneda-complete quasi-metric spaces are
sober, as a consequence of \cite[Proposition~4.1]{JGL:formalballs},
which says that all continuous Yoneda-complete quasi-metric spaces are
sober.  Hence Lemma~\ref{lemma:Q:Vietoris} cares for one direction.
It remains to show that every $\dQ$-Scott open subset is open in the
upper Vietoris topology.

Theorem~\ref{thm:Q:alg} tells us that the subsets
$\mathcal V = \uuarrow (\upc \{x_1, \cdots, x_n\}, r)$ with
$n \geq 1$, and $x_1$, \ldots, $x_n$ center points, form a base of the
Scott topology on $\mathbf B (\Smyth X, \dQ)$.  We show that
$\mathcal V \cap \Smyth X$ is open in the upper Vietoris topology.
For that, we use the implication (1)$\limp$(3) of Lemma~5.8 of
\cite{JGL:formalballs}, which states that, in a continuous
quasi-metric space $Y, \partial$, for every $\epsilon > 0$, for every
center point $y \in Y$,
$B^{\partial^+}_{(y, 0), <\epsilon} = \uuarrow (y, \epsilon)$.  Here
$H = \Smyth X$, $\partial = \dQ$, is algebraic complete by
Theorem~\ref{thm:Q:alg}, $\upc \{x_1, \cdots, x_n\}$ is a center point
by Lemma~\ref{lemma:Q:simple:center}, so
$\mathcal V = \uuarrow (\upc \{x_1, \cdots, x_n\}, r)$ is the open
ball $B^{\dQ^+}_{\upc \{x_1, \cdots, x_n\}, < r}$.  For every $Q$ in
$\mathcal V$, $\dQ (\upc \{x_1, \cdots, x_n\}, Q) < r$, so for every
$x \in Q$, there is an index $j$ such that $d (x_j, x) < r$.  This
implies that $Q \subseteq \bigcup_{j=1}^n B^d_{x_j, <r}$.  Conversely,
if $Q \subseteq \bigcup_{j=1}^n B^d_{x_j, <r}$, then by
Lemma~\ref{lemma:Q:B:eps}, there is an $\epsilon > 0$, $\epsilon < r$,
such that $Q \subseteq \bigcup_{j=1}^n B^d_{x_j, <r-\epsilon}$.  In
other words, for every $x \in Q$, there is a $j$ such that
$d (x_j, x) < r-\epsilon$.  Therefore
$\dQ (\upc \{x_1, \cdots, x_n\}, Q) \leq r-\epsilon < r$, so $Q$ is in
$\mathcal V$.  Summing up, $\mathcal V$ is equal to
$\Box \bigcup_{j=1}^n B^d_{x_j, <r}$, hence is upper Vietoris open.
\qed

\begin{thm}[$\dQ$ quasi-metrizes the upper Vietoris topology]
  \label{thm:Q:Vietoris}
  Let $X, d$ be a continuous Yoneda-complete quasi-metric space.  The
  $\dQ$-Scott topology, the $\dQ^a$-Scott topology, for every
  $a \in \Rp$, $a > 0$, and the upper Vietoris topology all coincide
  on $\Smyth X$.
\end{thm}
\proof The proof is as for Theorem~\ref{thm:V:weak=dScott} or
Theorem~\ref{thm:H:Vietoris}.  By \cite[Theorem~7.9]{JGL:formalballs},
$X, d$ is the $1$-Lipschitz continuous retract of an algebraic
Yoneda-complete quasi-metric space $Y, \partial$.  Call
$s \colon X \to Y$ the section and $r \colon Y \to X$ the retraction.
We confuse $\Smyth X$ with the corresponding space of discrete
sublinear previsions.  Then $\Prev s$ and $\Prev r$ form a
$1$-Lipschitz continuous section-retraction pair by
Lemma~\ref{lemma:Pf:lip}, and in particular $\Prev s$ is an embedding
of $\Smyth X$ into $\Smyth Y$ with their $\dQ$-Scott topologies
(similarly with $\dQ^a$ in place of $\dQ$).  However, $s$ and $r$ are
also just continuous, by Proposition~\ref{prop:cont}, so $\Prev s$ and
$\Prev r$ also form a section-retraction pair between the same spaces,
this time with their weak topologies (as spaces of previsions), by
Fact~\ref{fact:Pf:weak}, that is, with their lower Vietoris
topologies, by Lemma~\ref{lemma:Q:V=weak}.  (Recall that $X$ and $Y$
are sober since continuous Yoneda-complete.)  By
Proposition~\ref{prop:Q:Vietoris}, the two topologies on $\Smyth Y$
are the same.  Fact~\ref{fact:retract:two} then implies that the two
topologies on $\Smyth X$ are the same as well.  \qed

\section{Sublinear, Superlinear Previsions}
\label{sec:other-previsions}

\subsection{A Minimax Theorem}
\label{sec:minimax-theorem-1}

We establish a minimax theorem on non-Hausdorff spaces.  We are not
aware of any preexisting such theorem.  We shall need to apply it on
non-Hausdorff spaces in the case of superlinear previsions
(Section~\ref{sec:superl-prev}).  For sublinear previsions, several
existing minimax theorems, on Hausdorff spaces, could be used instead.

Our proof is a simple modification of Frenk and Kassay's Theorem~1.1
\cite{FK:minimax}, a variant of Ky Fan's celebrated minimax theorem
\cite{KyFan:minimax}.  One might say that this is overkill here, since
the point in those theorems is that no vector space structure has to
be assumed to introduce the required convexity assumptions, and the
functions we shall apply this theorem to will not only be convex, but
in fact linear.  Our notion of linearity will be the same as in the
notion of linear previsions, and that is a notion that applies on
cones that are not in general embeddable in vector spaces, so that one
might say that we need the added generality, at least partly.  Our
reason for presenting the theorem in its full generality, however, is
because it is no more complicated to do so.

The proof we shall give is really Frenk and Kassay's, up to minor
modifications.  It is not completely obvious at first sight where they
require the Hausdorff assumption, and this is hidden in a use
of the finite intersection property for compact sets,
which typically only holds in Hausdorff spaces.
We show that this can be completely be dispensed with.

Given two non-empty sets $A$, $B$, and a map
$f \colon A \times B \to \real$, we clearly have:
\[
\sup_{b \in B} \inf_{a \in A} f (a,b) \leq \inf_{a \in A} \sup_{b \in
  B} f (a,b).
\]
Our aim is to strengthen that to an equality, under some extra
assumptions.  We shall also consider the case of functions with values
in $\real \cup \{+\infty\}$, not just $\real$, although we shall not
need that.  (Including the value $-\infty$ as well seems to present
some difficulties, but $+\infty$ does not.)

The minimax theorem must rely on a separation theorem, in the style of
Hahn and Banach.  For our purposes, we shall use a result due to
R. Tix, K. Keimel, and G. Plotkin, which we reproduce in the
proposition below.

We write ${\creal}_\sigma$ for $\creal$ with its Scott topology.
Hence a lower semicontinuous map from $A$ to $\creal$ is a continuous
map from $A$ to ${\creal}_\sigma$.  A subset $A$ of $\creal^n$ is
\emph{convex} if and only if for all $a_1, a_2 \in A$ and for every
$\alpha \in ]0, 1[$, $\alpha a_1 + (1-\alpha) a_2$ is in $A$.  A map
is \emph{linear} if and only if it preserves sums and products by
scalars from $\Rp$.  Let $\vec 1$ be the all one vector in $\creal^n$.
\begin{prop}[\cite{TKP:nondet:prob}, Lemma~3.7]
  \label{prop:strictsep}
  Let $Q$ be a convex compact subset of ${\creal}_\sigma^n$ disjoint
  from $\dc \vec 1 = [0, 1]^n$. Then there is a linear continuous
  function $\Lambda$ from ${\creal}_\sigma^n$ to ${\creal}_\sigma$ and
  a $b > 1$ such that $\Lambda (\vec 1) \leq 1$ and $\Lambda (x) > b$
  for every $x \in Q$.  \qed
\end{prop}
Writing $e_i$ for the tuple with a $1$ at position $i$ and zeros
elsewhere, such a linear continuous map must map every
$(x_1, x_2, \cdots, x_n)$ to $\sum_{i=1}^n c_i x_i$, where
$c_i = \Lambda (e_i)$.  This is indeed certainly true when every $x_i$
is different from $+\infty$, by linearity, and (Scott-)continuity
gives us the result for the remaining cases.

\begin{cor}
  \label{corl:strictsep}
  Let $Q$ be a non-empty convex compact subset of ${\creal}_\sigma^n$
  disjoint from $\dc \vec 1 = [0, 1]^n$.  There are non-negative real
  numbers $c_1$, $c_2$, \ldots, $c_n$ such that $\sum_{i=1}^n c_i=1$
  and $\sum_{i=1}^n c_i x_i > 1$ for every $(x_1, x_2, \cdots, x_n)$
  in $Q$.
\end{cor}
\proof Find $\Lambda$ and $b > 1$ as above.  $\Lambda$ maps every
$(x_1, x_2, \cdots, x_n)$ to $\sum_{i=1}^n c_i x_i$, where each $c_i$
is in $\creal$.  Since $\Lambda (\vec 1) \leq 1$,
$\sum_{i=1}^n c_i \leq 1$.  In particular, no $c_i$ is equal to
$+\infty$.  Since $Q$ is non-empty, $\sum_{i=1}^n c_i x_i > b$ for
some point $(x_1, x_2, \cdots, x_n)$ in $Q$.  This implies that not
every $c_i$ is equal to $0$, hence $\sum_{i=1}^n c_i \neq 0$.  Let
$c'_i = c_i / \sum_{i=1}^n c_i$.  Then $\sum_{i=1}^n c'_i=1$, and for
every $(x_1, x_2, \cdots, x_n)$ in $Q$,
$\sum_{i=1}^n c'_i x_i > b / \sum_{i=1}^n c_i \geq b > 1$.  \qed

Following Frenk and Kassay, we say that a map
$f \colon A \times B \to \real \cup \{+\infty\}$ is
\emph{closely convex in its first argument} if and and only if for all
$a_1, a_2 \in A$, for every $\alpha \in ]0, 1[$, for every
$\epsilon > 0$, there is an $a \in A$ such that, for every $b \in B$:
\[
  f (a, b) \leq \alpha f (a_1, b) + (1-\alpha) f (a_2, b) + \epsilon.
\]
We say that $f$ is \emph{closely concave in its second argument} if
and only if for all $b_1, b_2 \in B$, for every $\alpha \in ]0, 1[$,
for all $\epsilon, M > 0$, there is an $b \in B$ such that, for every
$a \in A$:
\[
  f (a, b) \geq \min (M, \alpha f (a, b_1) + (1-\alpha) f (a, b_2) - \epsilon).
\]
The extra lower bound $M$ serves to handle the cases where $\alpha f
(a, b_1) + (1-\alpha) f (a, b_2)$ is infinite.

\begin{lem}
  \label{lemma:minimax:convex}
  Let $A$ be a non-empty compact topological space and $B$ be a
  non-empty set.  Let $f \colon A \times B \to \real \cup \{+\infty\}$
  be a map that is closely convex in its first argument, and such that
  $f (\_, b)$ is lower semicontinuous for every $b \in B$.

  Then, for every $t \in \real$, $\inf_{a \in A} \sup_{b \in B} f (a,
  b) \leq t$ if and only if, for every normalized simple valuation
  $\sum_{i=1}^n c_i \delta_{b_i}$ on $B$, $\inf_{a \in A} \sum_{i=1}^n
  c_i f (a, b_i) \leq t$.
\end{lem}
\proof Let $L \colon A \to \real \cup \{+\infty\}$ be defined by
$L (a) = \sup_{b \in B} f (a, b)$.  Being a pointwise supremum of
lower semicontinuous functions, $L$ is itself lower semicontinuous.

For every normalized simple valuation $\sum_{i=1}^n c_i \delta_{b_i}$
on $B$,
$\sum_{i=1}^n c_i f (a, b_i) \leq \sum_{i=1}^n c_i L (a) = L (a)$, so
$\inf_{a \in A} \sum_{i=1}^n c_i f (a, b_i) \leq \inf_{a \in A} L
(a)$.  In particular, if $\inf_{a \in A} L (a) \leq t$, then
$\inf_{a \in A} \sum_{i=1}^n c_i f (a, b_i) \leq t$.

In the reverse direction, assume that $\inf_{a \in A} L (a) > t$.  We
wish to show that
$\inf_{a \in A} \allowbreak \sum_{i=1}^n c_i f (a, b_i) > t$ for some
normalized simple valuation $\sum_{i=1}^n c_i \delta_{b_i}$.  To that
end, we first pick a real number $t'$ such that
$\inf_{a \in A} L (a) > t' > t$, and we shall find a normalized simple
valuation $\sum_{i=1}^n c_i \delta_{b_i}$ such that
$\sum_{i=1}^n c_i f (a, b_i) \geq t'$ for every $a \in A$.

For every $b \in B$, $U_b = \{a \in A \mid f (a, b) > t'\}$ is open
since $f (\_, b)$ is lower semicontinuous, and the sets $U_b$,
$b \in B$, cover $A$: for every $a \in A$, $L (a) > t'$, so
$f (a, b) > t'$ for some $b \in B$.  Since $A$ is compact, there are
finitely many points $b_1$, $b_2$, \ldots, $b_n$ such that
$A = \bigcup_{i=1}^n U_{b_i}$.  Note that, since $A$ is non-empty,
$n \geq 1$.

Since $f (\_, b_i)$ is lower semicontinuous on the compact space $A$,
it reaches its minimum $r_i$ in $\real \cup \{+\infty\}$.  Let
$r = \min (r_1, r_2, \cdots, r_n)$, and find $c > 0$ such that
$1 + c (r - t') \geq 0$.  If $r \geq t'$, we can take any $c > 0$,
otherwise any $c$ such that $0 < c < 1/(t'-r)$ will fit.
The map $h_i \colon a \mapsto 1 + c (f (a, b_i) - t')$ is then lower
semicontinuous from $A$ to $\creal$, and therefore
$h (a) = (h_1 (a), h_2 (a), \cdots, h_n (a))$ defines a continuous map
from $A$ to ${\creal}_\sigma^n$.

Let $K$ be the image of $A$ by $h$.  This is a compact subset of
${\creal}_\sigma^n$.  It is also non-empty, since $A$ is non-empty.
Hence $Q = \upc K$ is also non-empty, compact, and saturated.

We observe the following fact: $(*)$ for every $z \in \creal^n$, if
$z + \epsilon. \vec 1$ is in $Q$ for every $\epsilon > 0$, then $z$ is
also in $Q$.  Indeed, consider the Scott-closed set
$C_\epsilon = \dc (z + \epsilon.\vec 1)$ in $\Rp^n$, for every
$\epsilon \geq 0$.  Then
$C_0 = \dc z = \bigcap_{\epsilon > 0} C_\epsilon$.  If $z$ were not in
$Q$, then $C_0$ would not intersect $Q$, since $Q$ is upwards-closed.
Then there would be finitely many values
$\epsilon_1, \epsilon_2, \ldots, \epsilon_n > 0$ such that
$\bigcap_{i=1}^n C_{\epsilon_i}$ does not intersect $Q$ by
compactness, and by letting $\epsilon$ be the smallest $\epsilon_i$,
$C_\epsilon$ would not intersect $Q$, contradiction.

We claim that $Q$ is convex.  Otherwise, there are two points $z_1$
and $z_2$ of $Q$ and a real number $\alpha \in {]0, 1[}$ such that
$\alpha z_1 + (1-\alpha) z_2$ is not in $Q$.  By $(*)$,
$\alpha z_1 + (1-\alpha) z_2 +\epsilon.\vec 1$ is not in $Q$ for some
$\epsilon > 0$.  Since $z_1$ and $z_2$ are in $Q$, by definition there
are points $a_1$ and $a_2$ in $A$ such that $h (a_1) \leq z_1$ and
$h (a_2) \leq z_2$.  We finally use the fact that $f$ is closely
convex in its first argument: there is a point $a \in A$ such that,
for every $i$, $1\leq i\leq n$,
$f (a, b_i) \leq \alpha f (a_1, b_i) + (1-\alpha) f (a_2, b_i) +
\epsilon/c$.  That implies
$h (a) \leq \alpha h (a_1) + (1-\alpha) h (a_2) + \epsilon.\vec 1 \leq
\alpha z_1 + (1-\alpha) z_2 + \epsilon.\vec 1$.  Since $h (a)$ is in
$Q$ and $Q$ is upwards-closed, this contradicts the fact that
$\alpha z_1 + (1-\alpha) z_2 + \epsilon.\vec 1$ is not in $Q$.

We claim that $Q$ does not intersect $[0, 1]^n$.  Assume on the
contrary that there is an $a \in A$, and a $z \in [0, 1]^n$ such that
$h (a) \leq z$.  Since $A = \bigcup_{i=1}^n U_{b_i}$, there is an
index $i$ such that $f (a, b_i) > t'$, so $h_i (a) > 1$.  However
$h (a) \leq z \in [0, 1]^n$ entails that $h_i (a) \leq 1$,
contradiction.

Now we use Corollary~\ref{corl:strictsep} and obtain non-negative real
numbers $c_1$, $c_2$, \ldots, $c_n$ such that $\sum_{i=1}^n c_i=1$,
and $\sum_{i=1}^n c_i z_i \geq 1$ for every
$(z_1, z_2, \cdots, z_n) \in Q$.  This holds for the elements $h (a)$
of $K \subseteq Q$, hence
$\sum_{i=1}^n c_i (1 + c (f (a, b_i) - t')) \geq 1$ for every
$a \in A$.  It follows that $\sum_{i=1}^n c_i f (a, b_i) \geq t'$, as
required.  \qed

\begin{lem}
  \label{lemma:minimax:concave}
  Let $A$ and $B$ be two non-empty sets, and let
  $f \colon A \times B \to \real \cup \{+\infty\}$ be closely concave
  in its second argument.  For every normalized simple valuation
  $\sum_{i=1}^n c_i \delta_{b_i}$ on $B$,
  $\inf_{a \in A} \sum_{i=1}^n c_i f (a, b_i) \leq \sup_{b \in B}
  \inf_{a \in A} f (a, b)$.
\end{lem}
\proof We first show that, if $f$ is closely concave in its second
argument, then for every normalized simple valuation
$\sum_{i=1}^n c_i \delta_{b_i}$ on $B$, for all $\epsilon, M > 0$,
there is a $b \in B$ such that, for every $a \in A$:
\[
  f (a, b) \geq \min (M, \sum_{i=1}^n c_i f (a, b_i) - \epsilon).
\]
The definition of closely concave is the special case $n=2$.
The case $n=0$ is vacuous, and the case $n=1$ is trivial.

We show that claim by induction by induction on $n$.  Assume
$n\geq 3$.  If $c_n=0$, then the result is a direct appeal to the
induction hypothesis.  If $c_n=1$, then $c_1=c_2=\cdots=c_{n-1}=0$,
and we can take $b=b_n$.  Otherwise, let $c'_i = c_i / (1 - c_n)$ for
every $i$, and $\alpha=1-c_n$.  Note that $\alpha$ is in $]0, 1[$.
Also,
$\sum_{i=1}^n c_i f (a, b_i) = \alpha \sum_{i=1}^{n-1} c'_i f (a, b_i)
+ (1-\alpha) f (a, b_n)$.  By induction hypothesis, there is a point
$y' \in B$ such that, for every $a \in A$:
\[
  f (a, y') \geq \min \left(M/\alpha, \sum_{i=1}^{n-1} c'_i f (a, b_i) - \epsilon/2\alpha\right).
\]
Since $f$ is closely concave in its second argument, there is a point
$b \in B$ such that, for every $a \in A$:
\[
  f (a, b) \geq \min \left(M, \alpha f (a, y') + (1-\alpha) f (a, b_n) - \epsilon/2\right).
\]
Therefore, for every $a \in A$,
\begin{eqnarray*}
  f (a, b)
  & \geq
  & \min \left(M, \alpha \min \left(M/\alpha, \sum_{i=1}^{n-1} c'_i f (a, b_i) -
    \epsilon/2\alpha\right)
    + (1-\alpha) f (a, b_n) - \epsilon/2\right) \\
  & = & \min \left(M, \alpha \sum_{i=1}^{n-1} c'_i f (a, b_i) -
    \epsilon/2 + (1-\alpha) f (a, b_n) - \epsilon/2\right) \\
  & = & \min (M, \sum_{i=1}^n c_i f (a, b_i) - \epsilon).
\end{eqnarray*}

We now prove the lemma by showing that for every real number
$t < \inf_{a \in A} \sum_{i=1}^n c_i f (a, b_i)$, there is an element
$b \in B$ such that, for every $a \in A$, $f (a, b) \geq t$.  For
that, we pick $\epsilon > 0$ such that
$t + \epsilon \leq \inf_{a \in A} \sum_{i=1}^n c_i f (a, b_i)$, and we
let $M$ be any non-negative number larger than $t$.  The above
generalization of the notion of closely concave then yields the
existence of a point $b$ such that for every $a \in A$,
$f (a, b) \geq \min (M, \sum_{i=1}^n c_i f (a, b_i) - \epsilon) \geq
t$.  \qed

\begin{thm}[Minimax]
  \label{thm:minimax}
  Let $A$ be a non-empty compact topological space and $B$ be a
  non-empty set.  Let $f \colon A \times B \to \real \cup \{+\infty\}$
  be a map that is closely convex in its first argument, closely
  concave in its second argument, and such that $f (\_, b)$ is lower
  semicontinuous for every $b \in B$.  Then:
  \[
    \sup_{b \in B} \inf_{a \in A} f (a, b) = \inf_{a \in A} \sup_{b
      \in B} f (a, b),
  \]
  and the infimum on the right-hand side is attained.
\end{thm}
\proof The $\leq$ direction is obvious.  If
$\sup_{b \in B} \inf_{a \in A} f (a, b) = +\infty$, then the equality
is clear.  Otherwise, let
$t = \sup_{b \in B} \inf_{a \in A} f (a, b)$.  This is a real number
that is larger than or equal to
$\inf_{a \in A} \sum_{i=1}^n c_i f (a, b_i)$ for every normalized
simple valuation $\sum_{i=1}^n c_i \delta_{b_i}$ on $B$, by
Lemma~\ref{lemma:minimax:concave}.  Lemma~\ref{lemma:minimax:convex}
then tells us that $\inf_{a \in A} \sup_{b \in B} f (a, b) \leq t$,
which gives the $\geq$ direction of the inequality.

The fact that the infimum on the right-hand side is attained is due to
the fact that $A$ is compact, and that $a \mapsto \sup_{b \in B} f (a,
b)$ is lower semicontinuous, as a pointwise supremum of lower
semicontinuous maps.  \qed

\begin{rem}
  \label{rem:compact}
  Note that we do not require $A$ to be Hausdorff.  We have already
  said so, but one should stress it, as compactness without the
  Hausdorff separation axiom is a very weak property.  For example,
  any topological space with a least element in its specialization
  preordering is compact.
\end{rem}

\begin{rem}
  \label{rem:closely}
  Among the closely convex functions (in the first argument), one
  finds the \emph{convex} functions in the sense of Ky Fan
  \cite{KyFan:minimax}, that is, the maps such that for all $a_1, a_2
  \in A$, for every $\alpha \in {]0, 1[}$, there is a point $a \in A$
  such that for every $b \in B$, $f (a, b) \leq \alpha f (a_1, b) +
  (1-\alpha) f (a_2, b)$.  Similarly with closely concave functions and
  concave functions in the sense of Ky Fan, which satisfy that for all
  $b_1, b_2 \in B$, for every $\alpha \in {]0, 1[}$, there is a point
  $b \in B$ such that for every $a \in A$, $f (a, b) \geq \alpha f (a,
  b_1) + (1-\alpha) f (a, b_2)$.

  Among the convex functions (in the first argument) in the sense of
  Ky Fan, one simply finds the functions $f$ such that
  $f (\alpha a_1 + (1-\alpha) a_2, b) \leq \alpha f (a_1, b) + (1-\alpha)
  f (a_2, b)$, namely those that are convex in their first argument,
  in the ordinary sense, provided one can interpret scalar
  multiplication and addition on $A$, for example when $A$ is a convex
  subset of a cone.

  One also finds more unusual cases of Ky Fan convexity.  Consider for
  example a family of functions
  $f_i \colon A \to \real \cup \{+\infty\}$, $i \in I$.  We can see
  them as one function
  $f \colon A \times I \to \real \cup \{+\infty\}$ by letting
  $f (a, i) = f_i (a)$.  If the family ${(f_i)}_{i \in I}$ is
  directed, then $f$ is Ky Fan concave: for all $i, j \in I$, find
  $k \in I$ such that $f_i, f_j \leq f_k$, then for every
  $\alpha \in {[0, 1]}$, for every $a \in A$,
  $f_k (a) \geq \alpha f_i (a) + (1-\alpha) f_j (a)$.
\end{rem}

\subsection{Sublinear Previsions}
\label{sec:sublinear-previsions}

To study sublinear previsions, we recall from \cite{JGL-mscs16} that
sublinear previsions are in bijection with the non-empty closed convex
subsets of linear previsions, under some slight assumptions.

More precisely, let $r_\AN$ be the function that maps every set $E$ of
linear previsions to the prevision
$h \in \Lform X \mapsto \sup_{G \in E} G (h)$.  The latter is always a
sublinear prevision, which is subnormalized, resp.\ normalized, as
soon as $E$ is a non-empty set of subnormalized, resp.\ normalized
linear previsions.

In the reverse direction, for every sublinear prevision $F$, let
$s_\AN (F)$ be the set of all linear previsions $G$ such that
$G \leq F$.  If $F$ is subnormalized, we write $s^{\leq 1}_\AN (F)$
for the set of all subnormalized linear previsions $G \leq F$, and if
$F$ is normalized, we write $s^1_\AN (F)$ for the set of all
normalized linear previsions $G \leq F$.

We shall write $s^\bullet_\AN$ instead of $s_\AN$, $s^{\leq 1}_\AN$,
or $s^1_\AN$, when the superscript is meant to be implied.  Again, we
equate linear previsions with continuous valuations.  Correspondingly,
we shall write $\Val_\bullet X$ for the space of linear previsions
(resp., subnormalized, normalized, depending on the subscript) on $X$.

Corollary~3.12 of \cite{JGL-mscs16} states that, under the assumption
that $\Lform X$ is locally convex, that is, if every $h \in \Lform X$
has a base of convex open neighborhoods, then:
\begin{itemize}
\item $r_\AN$ is continuous from the space of sublinear (resp.,
  subnormalized sublinear, resp., normalized sublinear) previsions on
  $X$, with the weak topology, to $\Hoare (\Val_\bullet X)$ with its
  lower Vietoris topology, and where $\Val_\bullet X$ has the weak
  topology;
\item $s^\bullet_\AN$ is continuous in the reverse direction;
\item $r_\AN \circ s^\bullet_\AN = \identity \relax$;
\item $\identity \relax \leq s^\bullet_\AN \circ r_\AN$.
\end{itemize}
Hence the space of (possibly subnormalized, or normalized) sublinear
previsions, with the weak topology, occurs as a retract of
$\Hoare (\Val_\bullet X)$.

Letting $\Hoare^{cvx} (\Val_\bullet X)$ denote the subspace of those
closed sets in $\Hoare (\Val_\bullet X)$ that are convex, Theorem~4.11
of loc.\ cit.\ additionally states that $r_\AN$ and $s^\bullet_\AN$
define a homeomorphism between $\Hoare^{cvx} (\Val_\bullet X)$ and the
corresponding space of (possibly subnormalized, or normalized)
sublinear previsions, under the same local convexity assumption.  That
assumption is satisfied on continuous Yoneda-complete quasi-metric
spaces, because:
\begin{prop}
  \label{prop:locconvex}
  Let $X$ be a $\odot$-consonant space, for example, a continuous
  Yoneda-complete quasi-metric space.  $\Lform X$, with its Scott
  topology, is locally convex.
\end{prop}
\proof Proposition~\ref{prop:co=Scott} tells us that the Scott
topology on $\Lform X$ coincides with the compact-open topology.  One
observes easily that the latter is locally convex.  Explicitly, let
$h \in \Lform X$, let $\mathcal U$ be an open neighborhood of $h$ in
the compact-open topology.  We must show that $\mathcal U$ contains a
convex open neighborhood of $h$.  By definition of the compact-open
topology, $\mathcal U$ contains a finite intersection of open subsets
of the form $[Q > a]$, all containing $h$.  Since any intersection of
convex sets is convex, it suffices to show that $[Q > a]$, for $Q$
compact saturated and $a \in \real$, is convex.  For any two elements
$h_1, h_2 \in [Q > a]$ and any $\alpha \in [0, 1]$, for every
$x \in Q$, $\alpha h_1 (x) + (1-\alpha) h_2 (x) > a$ because
$h_1 (x) > a$ and $h_2 (x) > a$; so $\alpha h_1 + (1-\alpha) h_2$ is
in $[Q > a]$.  \qed

Hence simply assuming $X, d$ continuous Yoneda-complete will allow us
to use the facts stated above on $r_\AN$ and $s_\AN^\bullet$, which
will shall do without further mention.

We silently equate every element $G$ of $\Val_\bullet X$ (a continuous
valuation) with the corresponding linear prevision (as in $G (h)$).

\begin{lem}
  \label{lemma:H}
  Let $X, d$ be a continuous Yoneda-complete quasi-metric space, and
  $\alpha > 0$.  Let $\bullet$ be either ``$\leq 1$'' or ``$1$''.

  For each $h \in \Lform_\alpha (X, d)$, the map
  $H \colon G \in \Val_\bullet X \mapsto G (h)$ is $\alpha$-Lipschitz
  continuous, when $\Val_\bullet X$ is equipped with the $\dKRH$
  quasi-metric.

  If $h \in \Lform^a_\alpha (X, d)$, where $a > 0$, then it is also
  $\alpha$-Lipschitz continuous with respect to the $\dKRH^a$
  quasi-metric.
\end{lem}
\proof We first check that it is $\alpha$-Lipschitz: for all
$G, G' \in \Val_\bullet X$,
$\dreal (H (G), H (G')) = \dreal (G (h), G' (h)) \leq \dKRH (G, G')$,
by definition of $\dKRH$.  If additionally $h$ is bounded from above
by $a$, then $\dreal (G (h), G' (h)) \leq \dKRH^a (G, G')$, by
definition of $\dKRH^a$.

Now build
$H' \colon (G, r) \mapsto H (G) - \alpha r = G (h) - \alpha r$.  To
show $\alpha$-Lipschitz continuity, we first recall that
$\Val_\bullet X$ is Yoneda-complete (both with $\dKRH$ and with
$\dKRH^a$) by Theorem~\ref{thm:V:complete}.  In particular, it is
standard, so we only have to show that $H'$ is Scott-continuous,
thanks to Lemma~\ref{lemma:f'}.  Theorem~\ref{thm:V:complete} also
tells us that directed suprema of formal balls are computed as naive
suprema.  For every directed family ${(G_i, r_i)}_{i \in I}$ of formal
balls on $\Val_\bullet X$, with (naive) supremum $(G, r)$, we have
$r = \inf_{i \in I} r_i$ and
$G (h) = \sup_{i \in I} (G_i (h) -\alpha r_i + \alpha r)$.  It follows
that
$H' (G, r) = G (h) - \alpha r = \sup_{i \in I} (G_i (h) - \alpha r_i)
= \sup_{i \in I} H' (G_i, r_i)$.  \qed

The following is where we require a minimax theorem.
We silently equate every element $G'$ of $\mathcal C'$ (a continuous
valuation) with the corresponding linear prevision (as in $G' (h)$).

\begin{lem}
  \label{lemma:AN:supinf}
  Let $X, d$ be a continuous quasi-metric space.  Let $\bullet$ be
  either ``$\leq 1$'' or ``$1$'', and $a > 0$.  For every simple
  valuation $G$ supported on center points in $\Val_\bullet X$, for
  every $\alpha > 0$, for every convex subset $\mathcal C'$ of
  $\Val_\bullet X$,
  \[
    \sup_{h \in \Lform^a_\alpha X} \inf_{G' \in \mathcal C'}  (G
    (h) - G' (h))
    = \inf_{G' \in \mathcal C'} \sup_{h \in \Lform^a_\alpha X} (G
    (h) - G' (h)).
  \]
\end{lem}
\proof First, we note that $G (h) - G' (h)$ makes sense: since
$h \in \Lform^a_\alpha (X, d)$ and $G$, $G'$ are subnormalized,
$G (h)$ and $G' (h)$ are both in $[0, a]$, and never infinite.  This
also shows that the function
$f \colon (h, G') \mapsto -G (h) + G' (h)$ is real-valued.

We use $A = \Lform_\alpha^a (X, d)^\patch$.  That makes sense because
$\Lform_\alpha^a (X, d)$ is stably compact
(Lemma~\ref{lemma:cont:Lalpha:retr}~(4)).  Then $A$ is compact (and
Hausdorff).  We use $B = \mathcal C'$.  Then $f$ is a map from
$A \times B$ to $\real$, which is linear in both arguments, in the
sense that it preserves sums and scalar products by non-negative
reals.  In particular, $f$ is certainly convex in its first argument
and concave in its second argument.


We claim that $h \mapsto f (h, G')$ is lower semicontinuous.  Recall
that $G$ is a simple valuation $\sum_{i=1}^n a_i \delta_{x_i}$.  The
map $h \mapsto \sum_{i=1}^n a_i h (x_i) = G (h)$ is continuous from
$A = \Lform_\alpha^a (X, d)^\dG$ to $(\creal)^\dG$ by
Corollary~\ref{corl:cont:Lalpha:simple}.  Because the non-trivial open
subsets of $(\creal)^\dG$ are of the form $[0, t[$, $t \in \Rp$,
$h \mapsto G (h)$ is also continuous from $A$ to $\real$ with the
Scott topology of the reverse ordering.  Hence $h \mapsto - G (h)$ is
continuous from $A$ to $\real$ with its Scott topology, i.e., it is
lower semicontinuous.
The map $h \mapsto G' (h)$ is also lower semicontinuous, by definition
of a prevision.  So $h \mapsto -G (h) + G' (h)$ is also lower
semicontinuous, and that is the map $h \mapsto f (h, G')$.



Theorem~\ref{thm:minimax} then implies that
$\sup_{G' \in \mathcal C'} \allowbreak \inf_{h \in \Lform_\alpha^a (X,
  d)} (- G (h) + G' (h')) = \inf_{h \in \Lform_\alpha^a (X, d)} \allowbreak
\sup_{G' \in \mathcal C'} (- G (h) + G' (h'))$.  Taking opposites
yields the desired result.  \qed

\begin{lem}
  \label{lemma:cl:conv}
  Let $Y$ be a topological space.  Let $\bullet$ be either
  ``$\leq 1$'' or ``$1$''.  The closure of any convex subset of
  $\Val_\bullet Y$, in the weak topology, is convex.
\end{lem}
\proof This is a classical exercise, and follows from a similar result
on topological vector spaces.  For a fixed number $\alpha \in ]0, 1[$,
consider the map
$f \colon (H_1, H_2) \in \Val_\bullet Y \times \Val_\bullet Y \mapsto
\alpha H_1 + (1-\alpha) H_2 \in \Val_\bullet Y$.  The inverse image of
$[k > b]$, where $k \in \Lform Y$, is
$\{(H_1, H_2) \mid \alpha H_1 (k) + (1-\alpha) H_2 (k) > b\} = \{(H_1,
H_2) \mid \exists b_1, b_2 . \alpha b_1 + (1-\alpha) b_2 \geq b, H_1
(k) > b_1, H_2 (k) > b_2\} = \bigcup_{b_1, b_2 / \alpha b_1 +
  (1-\alpha) b_2 \geq b} [k > b_1] \times [k > b_2]$, so $f$ is
(jointly) continuous.

Let $\mathcal C$ be convex, and imagine $cl (\mathcal C)$ is not.
There are elements $H_1$, $H_2$ of $cl (\mathcal C)$ and a real number
$\alpha \in [0, 1]$ such that $\alpha H_1 + (1-\alpha) H_2$ is not in
$cl (\mathcal C)$.  Clearly, $\alpha$ is in $]0, 1[$.  Let
$\mathcal V$ be the complement of $cl (\mathcal C)$.  This is an open
set, and $(H_1, H_2)$ is in the open set $f^{-1} (\mathcal V)$, where
$f$ is as above.  By the definition of the product topology, there are
two open neighborhoods of $\mathcal U_1$ and $\mathcal U_2$ of $H_1$
and $H_2$ respectively such that
$\mathcal U_1 \times \mathcal U_2 \subseteq f^{-1} (\mathcal V)$.
Since $H_1$ is in $\mathcal U_1$, $\mathcal U_1$ intersects
$cl (\mathcal C)$, hence also $\mathcal C$, say at $G_1$.  Similarly,
there is an element $G_2$ in $\mathcal U_2 \cap \mathcal C$.  Since
$\mathcal C$ is convex, $f (G_1, G_2) = \alpha G_1 + (1-\alpha) G_2$
is in $\mathcal C$.  However, $(G_1, G_2)$ is also in
$\mathcal U_1 \times \mathcal U_2 \subseteq f^{-1} (\mathcal V)$, so
$f (G_1, G_2)$ is in $\mathcal V$, contradicting the fact that
$\mathcal V$ is the complement of $cl (\mathcal C)$.  \qed

\begin{prop}
  \label{prop:AN:dKRH}
  Let $X, d$ be a continuous Yoneda-complete quasi-metric space.  Let
  $\bullet$ be either ``$\leq 1$'' or ``$1$'', and $a > 0$.  For all
  $\mathcal C, \mathcal C' \in \Hoare (\Val_\bullet X)$,
  \begin{equation}
    \label{eq:AN:dKRH}
    \mH {(\dKRH^a)} (\mathcal C, \mathcal C')
    \geq \dKRH^a (r_\AN (\mathcal C), r_\AN (\mathcal C')),
  \end{equation}
  with equality if $\mathcal C'$ is convex.
\end{prop}
\proof
\emph{The inequality.}
We develop the left-hand side.  Since $X, d$ is continuous
Yoneda-complete, we can use Theorem~\ref{thm:V:complete}, so that
$\Val_\bullet X, \dKRH^a$ is Yoneda-complete.  In particular, it is
standard, so we can apply Proposition~\ref{prop:H:prev}.  That
justifies the first of the following equalities, the others are by definition:
\begin{eqnarray*}
  \mH {(\dKRH^a)} (\mathcal C, \mathcal C')
  & = & \KRH {(\dKRH^a)} (F^{\mathcal C}, F^{\mathcal C'}) \\
  & = & \sup_{H \in \Lform_1 (\Val_\bullet X, \dKRH^a)}
        \dreal (F^{\mathcal C} (H), F^{\mathcal C'} (H)) \\
  & = & \sup_{H \in \Lform_1 (\Val_\bullet X, \dKRH^a)}
        \dreal (\sup_{G \in \mathcal C} H (G), \sup_{G' \in \mathcal
        C'} H (G')).
\end{eqnarray*}
The right-hand side of (\ref{eq:AN:dKRH}) is equal to:
\begin{eqnarray*}
  \dKRH^a (r_\AN (\mathcal C), r_\AN (\mathcal C'))
  & = & \sup_{h \in \Lform_1^a X} \dreal (r_\AN (\mathcal C) (h), r_\AN
        (\mathcal C') (h)) \\
  & = & \sup_{h \in \Lform_1^a X} \dreal (\sup_{G \in \mathcal C} G (h),
        \sup_{G' \in \mathcal C'} G' (h)).
\end{eqnarray*}
Among the elements $H$ of $\Lform_1 (\Val_\bullet X, \dKRH^a)$, we
find those obtained from $h \in \Lform^a_1 (X, d)$ by letting
$H (G) = G (h)$, according to Lemma~\ref{lemma:H}.  Therefore the
left-hand side of (\ref{eq:AN:dKRH}) is larger than or equal to its
right-hand side.

\vskip0.5em
\emph{The equality, assuming $X, d$ algebraic.}
Now assume $\mathcal C'$ convex.  We will show that
$\mH {(\dKRH^a)} (\mathcal C, \allowbreak \mathcal C') = \dKRH^a
(r_\AN (\mathcal C), r_\AN (\mathcal C'))$, and we do that first in
the special case where $X, d$ algebraic Yoneda-complete.  If the
inequality were strict, then there would be a real number $t$ such
that
$\mH {(\dKRH^a)} (\mathcal C, \allowbreak \mathcal C') > t > \dKRH^a
(r_\AN (\mathcal C), r_\AN (\mathcal C'))$.  Note that the second
inequality implies that $t$ is (strictly) positive.

Using Definition~\ref{defn:dH},
$\mH {(\dKRH^a)} (\mathcal C, \mathcal C') = \sup_{G \in \mathcal C}
\dKRH^a (G, \mathcal C')$, so $\dKRH^a (G, \mathcal C') > t$ for some
$G \in \mathcal C$.

Recall that on a standard quasi-metric space $X, d$, the map
$d (\_, C)$ is $1$-Lipschitz continuous for every $d$-Scott closed set
$C$ (see the beginning of Section~\ref{sec:hoare-quasi-metric}, or
Lemma~6.11 of \cite{JGL:formalballs} directly).  Moreover,
$d (x, C) \leq \inf_{y \in C} d (x, y)$, with equality if $x$ is a
center point (Proposition~6.12, loc.\ cit.)

Since $\Val_\bullet X, \dKRH^a$ is Yoneda-complete hence standard, the
map $\dKRH^a (\_, \mathcal C')$ is $1$-Lipschitz continuous.  By
Theorem~\ref{thm:V:alg} (when $\bullet$ is ``$\leq 1$'') or
Theorem~\ref{thm:V1:alg} (when $\bullet$ is ``$1$''), $(G, 0)$ is the
(naive) supremum of a directed family ${(G_i, r_i)}_{i \in I}$ where
each $G_i$ is a simple (subnormalized, resp.\ normalized) valuation
supported on center points.  Since $\dKRH^a (\_, \mathcal C')$ is
$1$-Lipschitz continuous, the map
$(F, r) \mapsto \dKRH^a (F, \mathcal C') - r$ is Scott-continuous, so
$\dKRH^a (G, \mathcal C') = \sup_{i \in I} \dKRH^a (G_i, \mathcal C')
- r_i$.  It follows that $\dKRH^a (G_i, \mathcal C') - r_i > t$ for
some $i \in I$.  We recall that
$\dKRH^a (G_i, \mathcal C') \leq \inf_{G' \in \mathcal C'} \dKRH^a
(G_i, G')$; in fact, equality holds since $G_i$ is a center point.  In
any case, $\inf_{G' \in \mathcal C'} \dKRH^a (G_i, G') > r_i + t$.

Expanding the definition of $\dKRH^a$, it follows that
$\inf_{G' \in \mathcal C'} \sup_{h \in \Lform^a_1 X} \dreal (G_i (h),
G' (h)) > r_i + t$.  Then there is a real number $t' > t+r_i$ such
that, for every $G' \in \mathcal C'$, there is an
$h \in \Lform^a_1 (X, d)$ such that $\dreal (G_i (h), G' (h)) > t'$;
in particular, since $t' > t > 0$ and therefore $G' (h)$ cannot be
equal to $+\infty$, $G_i (h) - G' (h)$ makes sense and is strictly
larger than $t'$.  Working in reverse, we obtain that
$\inf_{G' \in \mathcal C'} \sup_{h \in \Lform^a_1 X} (G_i (h) - G'
(h)) - r_i > t$.

We now use Lemma~\ref{lemma:AN:supinf} and obtain that
$\sup_{h \in \Lform^a_1 X} \inf_{G' \in \mathcal C'} (G_i (h) - G'
(h)) - r_i > t$.  Hence there is an $h \in \Lform^a_1 (X, d)$ such
that $\inf_{G' \in \mathcal C'} (G_i (h) - G' (h)) - r_i > t$.  That
can be written equivalently as
$G_i (h) - \sup_{G' \in \mathcal C'} G' (h) - r_i > t$.  Since
$(G_i, r_i) \leq^{\dKRH^{a+}} (G, 0)$, in particular
$G_i (h) \leq G (h) + r_i$, so
$G (h) - \sup_{G' \in \mathcal C'} G' (h) > t$.  This certainly
implies that
$\sup_{G \in \mathcal C} G (h) - \sup_{G' \in \mathcal C'} G' (h) >
t$, and therefore that
$\dreal (\sup_{G \in \mathcal C} G (h), \sup_{G' \in \mathcal C'} G'
(h)) > t$.  However, we had assumed that
$t > \dKRH^a (r_\AN (\mathcal C), r_\AN (\mathcal C')) = \sup_{h \in
  \Lform_1^a X} \dreal (\sup_{G \in \mathcal C} G (h), \sup_{G' \in
  \mathcal C'} G' (h))$, leading to a contradiction.

This shows
$\mH {(\dKRH^a)} (\mathcal C, \allowbreak \mathcal C') = \dKRH^a
(r_\AN (\mathcal C), r_\AN (\mathcal C'))$ when $\mathcal C'$ is
convex, assuming that $X, d$ is algebraic Yoneda-complete.

\vskip0.5em
\emph{The equality, in the general case.}
We now deal with the general case, where $X, d$ is continuous
Yoneda-complete.  We use \cite[Theorem~7.9]{JGL:formalballs}: $X, d$
is a retract of some algebraic Yoneda-complete space $Y, \partial$
through $1$-Lipschitz continuous maps $r \colon Y \to X$ and
$s \colon X \to Y$.  We check that
$r_\AN \circ \Hoarez (\Prev s) = \Prev s \circ r_\AN$ (namely that
$r_\AN$ is a natural transformation): for every
$\mathcal C \in \Hoare (\Val_\bullet X)$, for every $k \in \Lform Y$,
\begin{eqnarray*}
  r_\AN (\Hoarez (\Prev s) (\mathcal C)) (k)
  & = & \sup_{H \in \Hoarez (\Prev s) (\mathcal C)} H (k) \\
  & = & \sup_{H \in cl (\Prev s [\mathcal C])} H (k) \qquad\text{(see
        Lemma~\ref{lemma:H:functor})} \\
  & = & \sup_{H \in \Prev s [\mathcal C]} H (k) \qquad \text{by
        Lemma~\ref{lemma:H:cl:sup}} \\
  & = & \sup_{G \in \mathcal C} \Prev s (G) (k) = \sup_{G \in \mathcal
        C} G (k \circ s).
\end{eqnarray*}
Note that Lemma~\ref{lemma:H:cl:sup} applies because $H \mapsto H (k)$
is lower semicontinuous: the inverse image of $]b, +\infty]$ is $[k >
b]$, which is open in $\Val_\bullet Y$ with the weak topology, and we
recall that this coincides with the $\KRH\partial^a$-Scott topology, since
$Y, \partial$ is continuous Yoneda-complete.

The other term we need to compute is $\Prev s (r_\AN (\mathcal C)) (k)
= r_\AN (\mathcal C) (k \circ s) = \sup_{G \in \mathcal C} G (k \circ
s)$.  Therefore $r_\AN \circ \Hoarez (\Prev s) = \Prev s \circ r_\AN$.

Another fact we need to observe is that, when $\mathcal C'$ is convex,
$\Hoarez (\Prev s) (\mathcal C') = cl (\Prev s [\mathcal C'])$ is
convex.  It is easy to see that $\Prev s [\mathcal C']$ is convex: for
any two elements $\Prev s (G_1)$, $\Prev s (G_2)$ where
$G_1, G_2 \in \mathcal C'$, for every $\alpha \in [0, 1]$,
$\alpha \Prev s (G_1) + (1-\alpha) \Prev s (G_2)$ maps every
$k \in \Lform Y$ to
$\alpha G_1 (k \circ s) + (1-\alpha) G_2 (k \circ s)$, and is
therefore equal to $\Prev s (\alpha G_1 + (1-\alpha) G_2)$, which is
in $\Prev s [\mathcal C']$.  Then we apply Lemma~\ref{lemma:cl:conv}.

We can now finish our proof.  In the first line, we use that
$r \circ s$ is the identity map, and that $\Hoarez$ and $\Val_\bullet$
are functorial (Corollary~\ref{cor:H:functor}, and
Corollary~\ref{cor:V:functor} or Corollary~\ref{cor:V1:functor:a}
depending on the value of $\bullet$).
\begin{eqnarray*}
  \mH {(\dKRH^a)} (\mathcal C, \mathcal C')
  & = &
        \mH {(\dKRH^a)} ((\Hoarez \circ\Prev) (r) ((\Hoarez \circ \Prev)
        (s) (\mathcal C)),
        (\Hoarez \circ \Prev) (r)
        ((\Hoarez \circ \Prev) (s) (\mathcal C'))) \\
  & \leq &
           \mH {(\dKRH^a)} (\Hoarez (\Prev s) (\mathcal C), \Hoarez (\Prev s) (\mathcal
           C'))
           \qquad\text{since $(\Hoarez \circ \Prev) (r)$ is $1$-Lipschitz}
  \\
  & = &
        \dKRH^a (r_\AN (\Hoarez (\Prev s) (\mathcal C)),
        r_\AN (\Hoarez (\Prev s) (\mathcal C'))),
\end{eqnarray*}
since we know that the equality version of (\ref{eq:AN:dKRH}) holds on
the algebraic Yoneda-complete space $Y, \partial$, and since
$\Hoarez (\Prev s) (\mathcal C') = cl (s [\mathcal C'])$ is convex, as
we have just seen.  We now use our previous remark that
$r_\AN \circ \Hoarez (\Prev s) = \Prev s \circ r_\AN$ to obtain
$\mH {(\dKRH^a)} (\mathcal C, \mathcal C') \leq \dKRH^a (\Prev s
(r_\AN (\mathcal C)), \allowbreak \Prev s (r_\AN (\mathcal C')))$, and
the latter is less than or equal to
$\dKRH^a (r_\AN (\mathcal C), r_\AN (\mathcal C'))$ since $\Prev s$ is
$1$-Lipschitz.  Since by (\ref{eq:AN:dKRH}), the other inequality
$\mH {(\dKRH^a)} (\mathcal C, \mathcal C') \geq \dKRH^a (r_\AN
(\mathcal C), r_\AN (\mathcal C'))$ holds, equality follows.  \qed

\begin{lem}
  \label{lemma:AN:s:lip}
  Let $X, d$ be a continuous Yoneda-complete quasi-metric space.  Let
  $\bullet$ be either ``$\leq 1$'' or ``$1$'', and $a > 0$.  Then:
  \begin{enumerate}
  \item $r_\AN$ is $1$-Lipschitz from
    $\Hoare (\Val_\bullet X), \mH {(\dKRH^a)}$ to the space of
    (sub)normalized sublinear previsions on $X$ with the $\dKRH^a$
    quasi-metric;
  \item $s^\bullet_\AN$ preserves distances on the nose; in
    particular, $s^\bullet_\AN$ is $1$-Lipschitz.
  \end{enumerate}
\end{lem}
\proof (1) The first part of Proposition~\ref{prop:AN:dKRH} implies that
$r_\AN$ is $1$-Lipschitz. (2) For any two (sub)normalized
sublinear previsions $F$ and $F'$,
$s^\bullet_\AN (F') = \{G \in \Val_\bullet X \mid G \leq F'\}$ is
always convex, so the second part of Proposition~\ref{prop:AN:dKRH}
implies that:
\[
  \dKRH^a (F, F') = \dKRH^a (r_\AN (s^\bullet_\AN (F)), r_\AN
  (s^\bullet_\AN (F'))) = \mH {(\dKRH^a)} (s^\bullet_\AN (F),
  s^\bullet_\AN (F')).
\]

We would like to note a possible, subtle issue.  In writing
``$\Hoare (\Val_\bullet X), \mH {(\dKRH^a)}$'', we silently assume
that $\Val_\bullet X$ has the $\dKRH^a$-Scott topology, and the
elements of $\Hoare (\Val_\bullet X)$ are closed subsets for that
topology.  However Proposition~\ref{prop:AN:dKRH} considers closed
subsets for the weak topology.  Fortunately, the two topologies are
the same, by Theorem~\ref{thm:V:weak=dScott}.  \qed


\begin{lem}
  \label{lemma:AN:r:lipcont}
  Let $X, d$ be a continuous Yoneda-complete quasi-metric space.  Let
  $\bullet$ be either ``$\leq 1$'' or ``$1$'', and $a > 0$.  Then
  $r_\AN$ is $1$-Lipschitz continuous from $\Hoare (\Val_\bullet X),
  \mH {(\dKRH^a)}$ to the space of (sub)normalized sublinear
  previsions on $X$ with the $\dKRH^a$ quasi-metric.
\end{lem}
\proof By Lemma~\ref{lemma:AN:s:lip}~(1), it is $1$-Lipschitz.  In
particular, $\mathbf B^1 (r_\AN)$ is monotonic.  Similarly
$\mathbf B^1 (s^\bullet_\AN)$ is monotonic.  Since
$r_\AN \circ s^\bullet_\AN = \identity\relax$,
$\mathbf B^1 (r_\AN) \circ \mathbf B^1 (s^\bullet_\AN) =
\identity\relax$.  Since
$\identity\relax \leq s^\bullet_\AN \circ r_\AN$, i.e., since
$\mathcal C \subseteq s^\bullet_\AN (r_\AN (\mathcal C))$ for every
$\mathcal C$, we also have that
$(\mathcal C, r) \leq^{\mH {(\dKRH^a)}^+} \mathbf B^1 (s^\bullet_\AN)
(\mathbf B^1 (r_\AN) (\mathcal C, r))$: indeed
$\mH {(\dKRH^a)} (\mathcal C, s^\bullet_\AN (r_\AN (\mathcal C))) = 0
\leq r - r$.

This shows that $\mathbf B^1 (r_\AN)$ is left adjoint to $\mathbf B^1
(s^\bullet_\AN)$.  We conclude because left adjoints preserve all
existing suprema, hence in particular directed suprema.  \qed


We require the following order-theoretic lemma.
\begin{lem}
  \label{lemma:orderretract}
  Every order-retract of a dcpo is a dcpo.  Precisely, for any two
  monotonic maps $p \colon B \to A$ and $e \colon A \to B$ such that
  $p \circ e = \identity A$, if $B$ is a dcpo, then $A$ is a dcpo.
  Moreover, the supremum of any directed family ${(a_i)}_{i \in I}$ in
  $A$ is $p (\sup_{i \in I} e (a_i))$.
\end{lem}
\proof This is an easy exercise, but since one may be surprised by the
fact that we do not need $p$ or $e$ to be Scott-continuous for that,
we give the complete proof.  First $p (\sup_{i \in I} e (a_i))$ is
larger than $p (e (a_i)) = a_i$ for every $i \in I$, since $p$ is
monotonic.  Then, if $b$ is any other upper bound of
${(a_i)}_{i \in I}$ in $A$, then $e (a_i) \leq e (b)$ for every
$i \in I$, because $e$ is monotonic, so
$\sup_{i \in I} e (a_i) \leq e (b)$.  It follows that
$p (\sup_{i \in I} e (a_i)) \leq p (e (b)) = b$.  \qed

\begin{prop}[Yoneda-completeness, sublinear previsions]
  \label{prop:AN:Ycomplete}
  Let $X, d$ be an algebraic Yoneda-complete quasi-metric space.  Let
  $\bullet$ be either ``$\leq 1$'' or ``$1$'', and $a > 0$.  The space
  of (sub)normalized sublinear previsions on $X$ with the $\dKRH^a$
  quasi-metric is Yoneda-complete, and directed suprema of formal
  balls are computed as naive suprema.
\end{prop}
\proof By Theorem~\ref{thm:V:alg}, resp.\ Theorem~\ref{thm:V1:alg},
$\Val_\bullet X, \dKRH^a$ is algebraic Yoneda-complete, so we can use
Proposition~\ref{prop:HX:Ycomplete} and conclude that
$\Hoare (\Val_\bullet X), \mH {(\dKRH^a)}$ is Yoneda-complete, and
that directed suprema of formal balls are computed as naive suprema.
In particular, $\mathbf B (\Hoare (\Val_\bullet X), \mH {(\dKRH^a)})$
is a dcpo.  Moreover, $p = \mathbf B^1 (r_\AN)$ and
$e = \mathbf B^1 (s^\bullet_\AN)$ are monotonic by
Lemma~\ref{lemma:AN:s:lip}, and form an order-retraction.  By
Lemma~\ref{lemma:orderretract}, the poset of formal balls of the space
of (sub)normalized sublinear previsions is a dcpo, and that shows that
the latter space is Yoneda-complete.

We compute the supremum of a directed family ${(F_i, r_i)}_{i \in I}$
of formal balls, where each $F_i$ is a (sub)normalized sublinear
prevision.  This is $p (\sup_{i \in I} e (F_i, r_i))$, where $p$ and
$e$ are as above.  For each $i \in I$,
$e (F_i, r_i) = (s^\bullet_\AN (F_i), r_i)$.  The closed set
$\mathcal C_i = s^\bullet_\AN (F_i)$ has an associated discrete
sublinear prevision $F^{\mathcal C_i}$, defined as mapping
$H \in \Lform (\Val_\bullet X)$ to
$\sup_{G \in \mathcal C_i} H (G) = \sup_{G \in \Val_\bullet X, G\leq
  F_i} H (G)$.  Since directed suprema are naive suprema in
$\mathbf B (\Hoare (\Val_\bullet X), \mH {(\dKRH^a)})$,
$\sup_{i \in I} e (F_i, r_i)$ is $(\mathcal C, r)$ where
$r = \inf_{i \in I} r_i$ and
$F^{\mathcal C} (H) = \sup_{i \in I} (F^{\mathcal C_i} (H) - \alpha
r_i + \alpha r) = \sup_{i \in I} (\sup_{G \in \Val_\bullet X, G \leq
  F_i} H (G) - \alpha r_i + \alpha r)$, for every
$H \in \Lform^a_\alpha (\Val_\bullet X, \dKRH^a)$.  Then
$p (\sup_{i \in I} e (F_i, r_i))$ is the formal ball $(F, r)$, where
$F = r_\AN (\mathcal C)$.  In particular, for every
$h \in \Lform^a_\alpha (X, d)$,
$F (h) = \sup_{G \in \mathcal C} G (h)$.  The map
$H \colon G \mapsto G (h)$ is in
$\Lform^a_\alpha (\Val_\bullet X, \dKRH^a)$ by Lemma~\ref{lemma:H},
and:
\begin{eqnarray*}
  F (h)
  & =
  & \sup_{G \in \mathcal C} H (G) \\
  & = & F^{\mathcal C} (H) \\
  & = & \sup_{i \in I} (\sup_{G \in \Val_\bullet X, G \leq F_i} H (G)
        - \alpha r_i + \alpha r) \\
  & = & \sup_{i \in I} (\sup_{G \in \Val_\bullet X, G \leq F_i} G (h)
        - \alpha r_i + \alpha r) \\
  & = & \sup_{i \in I} (F_i (h) - \alpha r_i + \alpha r).
\end{eqnarray*}
\qed

\begin{lem}
  \label{lemma:AN:s:lipcont}
  Let $X, d$ be an algebraic Yoneda-complete quasi-metric space.  Let
  $\bullet$ be either ``$\leq 1$'' or ``$1$'', and $a > 0$.  Then
  $s^\bullet_\AN$ is $1$-Lipschitz continuous from the space of
  (sub)normalized sublinear previsions on $X$ with the $\dKRH^a$
  quasi-metric to $\Hoare (\Val_\bullet X), \mH {(\dKRH^a)}$.
\end{lem}
\proof
By Theorem~\ref{thm:V:weak=dScott}, the $\dKRH^a$-Scott
topology coincides with the weak topology on $\Val_\bullet X$.  By
Theorem~\ref{thm:V:alg}, resp.\ Theorem~\ref{thm:V1:alg},
$\Val_\bullet X, \dKRH^a$ is algebraic Yoneda-complete, so we can use
Proposition~\ref{prop:H:Vietoris} and conclude that the
$\mH {(\dKRH^a)}$-Scott topology coincides with the lower Vietoris
topology on $\Hoare (\Val_\bullet X)$.

Proposition~\ref{prop:AN:Ycomplete} gives us the necessary properties
that allow us to use Proposition~\ref{prop:weak:dScott:a}: the
$\dKRH^a$-Scott topology on our space of sublinear previsions is finer
than the weak topology.  Since $s^\bullet_\AN$ is continuous from that
space of sublinear previsions with the weak topology to $\Hoare
(\Val_\bullet X)$ with its lower Vietoris topology, it is therefore
also continuous with respect to the $\dKRH^a$-topology on the former,
and the $\mH {(\dKRH^a)}$-Scott topology on the latter.

We now know that $s^\bullet_\AN$ is both $1$-Lipschitz and continuous
with respect to the $\dKRH^a$-Scott and $\mH {(\dKRH^a)}$-Scott
topologies.  By Proposition~\ref{prop:cont}, it must be $1$-Lipschitz
continuous.  That proposition applies because all the spaces involved
are Yoneda-complete hence standard.  In particular, our space of
sublinear previsions is Yoneda-complete by
Proposition~\ref{prop:AN:Ycomplete}, and $\Hoare (\Val_\bullet X)$ is
Yoneda-complete by Proposition~\ref{prop:HX:Ycomplete}.  \qed

\begin{lem}
  \label{lemma:retract:alg}
  Let $X, d$ and $Y, \partial$ be two quasi-metric spaces, and two
  $1$-Lipschitz continuous maps:
  \[
    \xymatrix{Y, \partial \ar@<1ex>[r]^{r} & X, d
      \ar@<1ex>[l]^{s}}
  \]
  with $r \circ s = \identity X$.  (I.e., $X, d$ is a
  \emph{$1$-Lipschitz continuous retract} of $Y, \partial$ through
  $r$, $s$.)

  If $Y, \partial$ is algebraic Yoneda-complete, then $X, d$ is
  continuous Yoneda-complete.  Every point of $X$ is a $d$-limit of a
  Cauchy-weightable net of points of the form $r (y)$ with $y$ taken
  from a strong basis of $Y, \partial$.

  If additionally $d (r (y), x) = \partial (y, s (x))$ for all
  $x \in X$ and $y \in Y$, then $X, d$ is algebraic Yoneda-complete,
  and a strong basis of $X, d$ is given by the image of a strong basis
  of $Y, \partial$ by $r$.
\end{lem}
\proof $\mathbf B^1 (r)$ and $\mathbf B^1 (s)$ form a retraction, in
the sense that those are (Scott-)continuous maps and
$\mathbf B^1 (r) \circ \mathbf B^1 (s) = \identity {\mathbf B (X,
  d)}$.  The continuous retracts of continuous dcpos are continuous
dcpos (see Corollary~8.3.37 of \cite{JGL-topology} for example), so
$\mathbf B (X, d)$ is a continuous dcpo.

Now assume that $B$ is a strong basis of $Y, \partial$.  Let $A$ be
the image of $B$ by $r$.  By Lemma~\ref{lemma:B:basis}, the set
$\mathring B$ of formal balls whose centers are in $B$ is a basis of
the continuous dcpo $\mathbf B (Y, \partial)$.  Then the image of
$\mathring B$ by $\mathbf B^1 (r)$ is a basis of $\mathbf B (X, d)$.
This can be found as Exercise~5.1.44 of loc.\ cit., for example: the
image of a basis by a retraction is a basis.  The image of
$\mathring B$ is the set $\mathring A$ of formal balls whose centers
are in $A$.  That $\mathring A$ is a basis means exactly that every
element of $X$ is the $d$-limit of a Cauchy-weightable net of points
in $A$, considering that $X, d$ is Yoneda-complete.

In order to show that $A$ is a strong basis, it remains to show that
the elements $a = r (b)$ of $A$ are center points, under our
additional assumption that $d (r (y), x) = \partial (y, s (x))$ for
all $x \in X$ and $y \in Y$.  Consider the open ball
$B^{d^+}_{(a, 0), <\epsilon}$.  This is the set of formal balls
$(x, t)$ on $X$ such that $d (a, x) < \epsilon -t$, or equivalently
$\partial (b, s (x)) < \epsilon - t$, using the additional assumption.
Hence
$B^{d^+}_{(a, 0), <\epsilon} = \mathbf B^1 (s)^{-1}
(B^{\partial^+}_{(b, 0), <\epsilon})$.  This is open since $b$ is a
center point and $\mathbf B^1 (s)$ is continuous.  \qed

\begin{thm}[Continuity for sublinear previsions]
  \label{thm:AN:cont}
  Let $X, d$ be a continuous Yoneda-complete quasi-metric space, and
  $a > 0$.

  The space of all subnormalized (resp., normalized) sublinear
  previsions on $X$ with the $\dKRH^a$ quasi-metric is continuous
  Yoneda-complete.

  It occurs as a $1$-Lipschitz continuous retract of
  $\Hoare (\Val_{\leq 1} X), \mH {(\dKRH^a)}$ (resp.,
  $\Hoare (\Val_1 X), \allowbreak \mH {(\dKRH^a)}$), and as an
  isometric copy of $\Hoare^{cvx} (\Val_{\leq 1} X), \mH {(\dKRH^a)}$
  (resp., $\Hoare^{cvx} (\Val_1 X), \mH {(\dKRH^a)}$) through $r_\AN$
  and $s^{\leq 1}_\AN$ (resp., $s^1_\AN$).

  The $\dKRH^a$-Scott topology on the space of subnormalized (resp.,
  normalized) sublinear previsions coincides with the weak topology.
\end{thm}
\proof We have already seen that $r_\AN$ and $s^\bullet_\AN$ are
continuous, with respect to the weak topologies.  By
Lemma~\ref{lemma:AN:r:lipcont} and Lemma~\ref{lemma:AN:s:lipcont},
they are also $1$-Lipschitz continuous with respect to
$\mH {(\dKRH^a)}$ on $\Hoare (\Val_\bullet X)$ and $\dKRH^a$ on the
space of (sub)normalized sublinear previsions.

Recall also that $r_\AN \circ s^\bullet_\AN = \identity \relax$.  This
shows the $1$-Lipschitz continuous retract part, and also the
isometric copy part, since $s^\bullet_\AN$ and $r_\AN$ restrict to
bijections with the subspace $\Hoare^{cvx} (\Val_\bullet X)$.

We turn to continuity.  By Theorem~\ref{thm:Vleq1:cont} or
Theorem~\ref{thm:V1:cont}, $\Val_\bullet X, \dKRH^a$ is continuous
Yoneda-complete.  We can then use Theorem~\ref{thm:H:cont} to conclude
that $\Hoare (\Val_\bullet X), \mH{(\dKRH^a)}$ is continuous
Yoneda-complete.  Since every $1$-Lipschitz continuous retract of a
continuous Yoneda-complete quasi-metric space is continuous
Yoneda-complete (Corollary~8.3.37 of \cite{JGL-topology}, already used
in the proof of Lemma~\ref{lemma:retract:alg}), the space of all
subnormalized (resp., normalized) sublinear previsions on $X$ with the
$\dKRH^a$ quasi-metric is continuous Yoneda-complete.

We show that the $\dKRH^a$-Scott topology on the space of
(sub)normalized sublinear previsions is the weak topology.  We use
Fact~\ref{fact:retract:two} with the space of (sub)normalized
sublinear previsions, either with the weak or with the $\dKRH^a$-Scott
topology, for $A$, and $\Hoare (\Val_\bullet X)$ with its
$\mH {(\dKRH^a)}$-Scott topology, or equivalently with its weak
topology, for $B$.  The latter two topologies coincide because of
Theorem~\ref{thm:V:weak=dScott} and Proposition~\ref{prop:H:Vietoris}.
Since $s^\bullet_\AN$ is the section part of a $1$-Lipschitz
continuous retraction, it is a topological embedding of $A$ with the
$\dKRH^a$-Scott topology into $B$.  We recall that $s^\bullet_\AN$ is
also a section, hence a topological embedding of $A$ with the weak
topology into $\Hoare (\Val_\bullet X)$ with its lower Vietoris
topology, that is, into $B$, and this allows us to conclude.  \qed

We did not need to show that directed suprema of formal balls were
computed as naive suprema to prove the latter.  Still, that is true.
\begin{lem}
  \label{lemma:AN:cont}
  Let $X, d$ be a continuous Yoneda-complete quasi-metric space, and
  $a > 0$.  The suprema of directed families ${(F_i, r_i)}_{i \in I}$
  of formal balls, where each $F_i$ is a (sub)normalized sublinear
  prevision on $X$, are naive suprema.
\end{lem}
\proof Let $\alpha > 0$, and $h$ be in $\Lform_\alpha^a (X, d)$.  By
Theorem~\ref{thm:AN:cont}, $\mathbf B^1 (r_\AN)$ is Scott-continuous
and $\mathbf B^1 (r_\AN) \circ \mathbf B^1 (s_\AN^\bullet)$ is the
identity.  The supremum of ${(F_i, r_i)}_{i \in I}$ is then
$\mathbf B^1 (r_\AN) (\mathcal C, r) = (r_\AN (\mathcal C), r)$, where
$(\mathcal C, r) = \sup_{i \in I} (s_\AN^\bullet (F_i), r_i)$.  By
Proposition~\ref{prop:HX:Ycomplete}, this is a naive supremum, meaning
that $r = \inf_{i \in I} r_i$ and
$F^{\mathcal C} (H) = \sup_{i \in I} (F^{s_\AN^\bullet (F_i)} (H) +
\alpha r - \alpha r_i)$ for every
$H \in \Lform_\alpha (\Val_\bullet X, \dKRH^a)$.  In other words,
$F^{\mathcal C} (H) = \sup_{i \in I} (\sup_{G \in s_\AN^\bullet (F_i)}
H (G) + \alpha r - \alpha r_i)$.  Taking $H (G) = G (h)$ (see
Lemma~\ref{lemma:H}), we obtain:
\begin{eqnarray*}
  F^{\mathcal C} (H) & = &
                           \sup_{i \in I} (\sup_{G \in s_\AN^\bullet (F_i)}
                           G (h) + \alpha r - \alpha r_i) \\
  & = & \sup_{i \in I} (r_\AN (s_\AN^\bullet (F_i)) (h) + \alpha r -
        \alpha r_i) \\
  & = & \sup_{i \in I} (F_i (h) + \alpha r - \alpha r_i),
\end{eqnarray*}
where the last equalities are by definition of $r_\AN$, and because
$r_\AN \circ s_\AN^\bullet = \identity\relax$.  By using the
definition of $F^{\mathcal C}$, then of $H$, then of $r_\AN$, we
obtain
$F^{\mathcal C} (H) = \sup_{G \in \mathcal C} H (G) = \sup_{G \in
  \mathcal C} G (h) = r_\AN (\mathcal C) (h)$, showing that
$r_\AN (\mathcal C) (h) = \sup_{i \in I} (F_i (h) + \alpha r - \alpha
r_i)$, as desired.  \qed

We turn to algebraic Yoneda-complete spaces.
\begin{thm}[Algebraicity for sublinear previsions]
  \label{thm:AN:alg}
  Let $X, d$ be an algebraic Yoneda-complete quasi-metric space, with
  a strong basis $\mathcal B$.  Let $a > 0$.

  The space of all subnormalized (resp., normalized) sublinear
  previsions on $X$ with the $\dKRH^a$ quasi-metric is algebraic
  Yoneda-complete.

  All the sublinear previsions of the form
  $h \mapsto \max_{i=1}^m \sum_{j=1}^{n_i} a_{ij} h (x_{ij})$, with
  $m \geq 1$, $\sum_{j=1}^{n_i} a_{ij} \leq 1$ (resp., $=1$) for every
  $i$, and where each $x_{ij}$ is a center point, are center points,
  and they form a strong basis, even when each $x_{ij}$ is constrained
  to lie in $\mathcal B$.
\end{thm}
\proof For every $\mathcal C \in \Hoare (\Val_\bullet X)$ and every
(sub)normalized sublinear prevision $F'$, we have:
\begin{eqnarray*}
  \dKRH^a (r_\AN (\mathcal C), F')
  & = & \dKRH^a (r_\AN (\mathcal C), r_\AN (\mathcal C'))
        \qquad \text{where }\mathcal C' = s^\bullet_\AN (F') \\
  & = & \mH {(\dKRH^a)} (\mathcal C, \mathcal C')
        \qquad \text{by Proposition~\ref{prop:AN:dKRH}, since
        $\mathcal C'$ is convex} \\
  & = & \mH {(\dKRH^a)} (\mathcal C, s^\bullet_\AN (F')).
\end{eqnarray*}
This is exactly the additional assumption we require to apply the
final part of Lemma~\ref{lemma:retract:alg}.  It follows that the
space of (sub)normalized sublinear previsions is algebraic
Yoneda-complete, with a strong basis of elements of the form
$r_\AN (\dc \{\sum_{j=1}^{n_i} a_{ij} \delta_{x_{ij}} \mid 1\leq i\leq
m\})$, where $m \geq 1$, each
$\sum_{j=1}^{n_i} a_{ij} \delta_{x_{ij}}$ is a (sub)normalized simple
valuation and each $x_{ij}$ is in $\mathcal B$.  An easy check shows
that those elements are exactly the maps
$h \mapsto \max_{i=1}^m \sum_{j=1}^{n_i} a_{ij} h ({x_{ij}})$.  \qed

\begin{lem}
  \label{lemma:AN:functor}
  Let $X, d$ and $Y, \partial$ be two continuous Yoneda-complete
  quasi-metric spaces, and $f \colon X, d \mapsto Y, \partial$ be a
  $1$-Lipschitz continuous map.  Let also $a > 0$.  $\Prev f$
  restricts to a $1$-Lipschitz continuous map from the space of
  normalized, resp.\ subnormalized sublinear previsions on $X$ to the
  space of normalized, resp.\ subnormalized sublinear previsions on
  $Y$, with the $\dKRH^a$ and $\KRH\partial^a$ quasi-metrics.
\end{lem}
\proof By Lemma~\ref{lemma:Pf:lip}, $\Prev f$ is $1$-Lipschitz and
maps (sub)normalized sublinear previsions to (sub)normalized sublinear
previsions.  Also, by Lemma~\ref{lemma:Pf:lipcont},
$\mathbf B^1 (\Prev f)$ preserves naive suprema.  By
Lemma~\ref{lemma:AN:cont}, suprema of directed families of formal
balls of (sub)normalized sublinear previsions are naive suprema, so
$\mathbf B^1 (\Prev f)$ is Scott-continuous.  It follows that
$\Prev f$ is $1$-Lipschitz continuous.  \qed

With Theorem~\ref{thm:AN:alg} and Theorem~\ref{thm:AN:cont}, we obtain
the following.
\begin{cor}
  \label{cor:AN:functor}
  There is an endofunctor on the category of continuous (resp.,
  algebraic) Yoneda-complete quasi-metric spaces and $1$-Lipschitz
  continuous maps, which sends every object $X, d$ to the space of
  (sub)normalized sublinear previsions on $X$ with the $\dKRH^a$-Scott
  topology ($a > 0$), and every $1$-Lipschitz continuous map $f$ to
  $\Prev f$.  \qed
\end{cor}

\subsection{Superlinear Previsions}
\label{sec:superl-prev}

Superlinear previsions are dealt with in a very similar fashion,
although, curiously, some of the steps will differ in essential ways.
We recall from \cite{JGL-mscs16} that superlinear previsions are in
bijection with the non-empty compact saturated convex subsets of
linear previsions.

Let $r_\DN$ be the function that maps every non-empty set $E$ of
linear previsions to the prevision
$h \in \Lform X \mapsto \inf_{G \in E} G (h)$.  The latter is a
superlinear prevision as soon as $E$ is compact, which is
subnormalized, resp.\ normalized, as soon as $E$ is a non-empty set of
subnormalized, resp.\ normalized linear previsions.

In the reverse direction, for every superlinear prevision $F$, let
$s_\DN (F)$ be the set of all linear previsions $G$ such that
$G \geq F$.  If $F$ is subnormalized, we write $s^{\leq 1}_\DN (F)$
for the set of all subnormalized linear previsions $G \geq F$, and if
$F$ is normalized, we write $s^1_\DN (F)$ for the set of all
normalized linear previsions $G \geq F$.  As in
Section~\ref{sec:sublinear-previsions}, we use the notation
$s^\bullet_\DN$.

Proposition~3.22 of loc.\ cit.\ states that:
\begin{itemize}
\item $r_\DN$ is continuous from the space of superlinear (resp.,
  subnormalized superlinear, resp., normalized superlinear) previsions
  on $X$ to $\Smyth (\Val_\bullet X)$ with its lower Vietoris
  topology, and where $\Val_\bullet X$ has the weak topology;
\item $s^\bullet_\DN$ is continuous in the reverse direction;
\item $r_\DN \circ s^\bullet_\DN = \identity \relax$;
\item $\identity \relax \geq s^\bullet_\DN \circ r_\DN$.
\end{itemize}
Hence the space of (possibly subnormalized, or normalized) superlinear
previsions occurs as a retract of $\Smyth (\Val_\bullet X)$.  This
holds for all topological spaces $X$, with no restriction.

Let $\Smyth^{cvx} (\Val_\bullet X)$ denote the subspace of
$\Smyth (\Val_\bullet X)$ consisting of convex sets.  Theorem~4.15 of
loc.\ cit.\ additionally states that $r_\DN$ and $s^\bullet_\DN$
define a homeomorphism between $\Smyth^{cvx} (\Val_\bullet X)$ and the
corresponding space of (possibly subnormalized, or normalized)
superlinear previsions.

The following is where we require a minimax theorem on a non-Hausdorff
space.  Although this looks very similar to
Lemma~\ref{lemma:AN:supinf}, the compact space we use in conjunction
with our minimax theorem is not $\Lform^a_\alpha (X, d)^\patch$, but a space
$\mathcal Q$ of linear previsions.  The latter is almost never Hausdorff.
\begin{lem}
  \label{lemma:DN:supinf}
  Let $X, d$ be a continuous quasi-metric space.  Let $\bullet$ be
  either ``$\leq 1$'' or ``$1$'', and $a > 0$.  For every continuous
  valuation $G'$ in $\Val_\bullet X$, for every $\alpha > 0$, for
  every compact convex subset $\mathcal Q$ of $\Val_\bullet X$ (with
  the weak topology),
  \[
    \sup_{h \in \Lform^a_\alpha X} \inf_{G \in \mathcal Q} (G
    (h) - G' (h))
    = \inf_{G \in \mathcal Q} \sup_{h \in \Lform^a_\alpha X} (G
    (h) - G' (h)).
  \]
\end{lem}
\proof Let $A = \mathcal Q$.  This is compact by assumption.  Let
$B = \Lform^a_\alpha (X, d)$, and let
$f \colon (G, h) \in A \times B \mapsto (G (h) - G' (h))$.  This is a
function that takes its values in $[-a, a]$, hence in $\real$.  It is
also lower semicontinuous in $G$, since the set of elements
$G \in \Val_\bullet X$ such that $f (G, h) > t$ is equal to
$[h > G' (h) + t]$, an open set by the definition of the weak
topology.

The function $f$ is certainly convex in its first argument and concave
in its second argument, being linear in both.
Theorem~\ref{thm:minimax} then implies that
$\sup_{h \in \Lform_\alpha^a X} \inf_{G \in \mathcal Q} (G (h) - G'
(h)) = \inf_{G \in \mathcal Q} \sup_{h \in \Lform_\alpha^a X} (G (h) -
G' (h))$.  \qed

\begin{prop}
  \label{prop:DN:dKRH}
  Let $X, d$ be a continuous Yoneda-complete quasi-metric space.  Let
  $\bullet$ be either ``$\leq 1$'' or ``$1$'', and $a > 0$.  For all
  $\mathcal Q, \mathcal Q' \in \Smyth (\Val_\bullet X)$, where
  $\Val_\bullet X$ has the weak topology,
  \begin{equation}
    \label{eq:DN:dKRH}
    \mQ {(\dKRH^a)} (\mathcal Q, \mathcal Q')
    \geq \dKRH^a (r_\DN (\mathcal Q), r_\DN (\mathcal Q')),
  \end{equation}
  with equality if $\mathcal Q$ is convex.
\end{prop}
\proof The proof is extremely close to Proposition~\ref{prop:AN:dKRH}.
We develop the left-hand side, using Definition~\ref{defn:dQ} in the
first line:
\begin{eqnarray*}
  \mQ {(\dKRH^a)} (\mathcal Q, \mathcal Q')
  & = & \KRH {(\dKRH^a)} (F_{\mathcal Q}, F_{\mathcal Q'}) \\
  & = & \sup_{H \in \Lform_1 (\Val_\bullet X, \dKRH^a)}
        \dreal (F_{\mathcal Q} (H), F_{\mathcal Q'} (H)) \\
  & = & \sup_{H \in \Lform_1 (\Val_\bullet X, \dKRH^a)}
        \dreal (\inf_{G \in \mathcal Q} H (G), \inf_{G' \in \mathcal
        Q'} H (G')).
\end{eqnarray*}

The right-hand side is equal to:
\begin{eqnarray*}
  \dKRH^a (r_\DN (\mathcal Q), r_\DN (\mathcal Q'))
  & = & \sup_{h \in \Lform_1^a X} \dreal (r_\DN (\mathcal Q) (h), r_\DN
        (\mathcal Q') (h)) \\
  & = & \sup_{h \in \Lform_1^a X} \dreal (\inf_{G \in \mathcal Q} G (h),
        \inf_{G' \in \mathcal Q'} G' (h)).
\end{eqnarray*}
Among the elements $H$ of $\Lform_1 (\Val_\bullet X, \dKRH^a)$, we
find those obtained from $h \in \Lform^a_1 (X, d)$ by letting
$H (G) = G (h)$, according to Lemma~\ref{lemma:H}.  That shows 
inequality (\ref{eq:DN:dKRH}).

We now assume $\mathcal Q$ convex.  If the inequality were strict,
then there would be a real number $t$ such that
$\mQ {(\dKRH^a)} (\mathcal Q, \allowbreak \mathcal Q') > t > \dKRH^a
(r_\DN (\mathcal Q), r_\DN (\mathcal Q'))$.  Note that the second
inequality implies that $t$ is (strictly) positive.

By Lemma~\ref{lemma:dQ}, there is a linear prevision $G'$ in
$\mathcal Q'$ such that
$\inf_{G \in \mathcal Q} \sup_{h \in \Lform^a_1 X} \dreal (G (h),
\allowbreak G' (h)) > t$.  This means that there is a real number
$t' > t$ such that for every $G \in \mathcal Q$, there is an
$h \in \Lform_a^1 (X, d)$ such that $\dreal (G (h), G' (h)) > t'$, and
since $t' > t > 0$, $G (h) - G' (h) > t'$.  This implies that
$\inf_{G \in \mathcal Q} \sup_{h \in \Lform^a_1 X} (G (h) - G' (h)) >
t$.  We now use Lemma~\ref{lemma:DN:supinf} and obtain
$\sup_{h \in \Lform^a_1 X} \inf_{G \in \mathcal Q} (G (h) - G' (h)) >
t$.  Therefore, there is a map $h$ in $\Lform^a_1 (X, d)$ such that
$\inf_{G \in \mathcal Q} (G (h) - G' (h)) > t$, namely,
$\inf_{G \in \mathcal Q} G (h) - G' (h) > t$.  This implies
$\inf_{G \in \mathcal Q} G (h) - \inf_{G' \in \mathcal Q'} G' (h) >
t$, that is, $r_\DN (\mathcal Q) (h) - r_\DN (\mathcal Q') (h) > t$.
Recalling that $t > 0$, this entails
$\dreal (r_\DN (\mathcal Q) (h), r_\DN (\mathcal Q') (h)) > t$.
However, this is impossible since
$\dKRH^a (r_\DN (\mathcal Q), r_\DN (\mathcal Q')) = \sup_{h \in
  \Lform^1_a X} \dreal (r_\DN (\mathcal Q) (h), r_\DN (\mathcal Q')
(h))$ is less than $t$.  \qed

\begin{lem}
  \label{lemma:DN:s:lip}
  Let $X, d$ be a continuous Yoneda-complete quasi-metric space.  Let
  $\bullet$ be either ``$\leq 1$'' or ``$1$'', and $a > 0$.  Then
  \begin{enumerate}
  \item $r_\DN$ is $1$-Lipschitz from
    $\Smyth (\Val_\bullet X), \mQ {(\dKRH^a)}$ to the space of
    (sub)normalized superlinear previsions on $X$ with the $\dKRH^a$
    quasi-metric;
  \item $s^\bullet_\DN$ is preserves distances on the nose; in particular,
    $s^\bullet_\AN$ is $1$-Lipschitz.
  \end{enumerate}
\end{lem}
\proof (1) is by the first part of Proposition~\ref{prop:DN:dKRH}.  By
the second part of Proposition~\ref{prop:DN:dKRH}, using the fact that
$s^\bullet_\DN (F)$ is convex,
\[
\mQ {(\dKRH^a)} (s^\bullet_\DN (F), s^\bullet_\DN (F'))
= \dKRH^a (r_\DN (s^\bullet_\DN (F)), r_\DN (s^\bullet_\DN (F')))
= \dKRH^a (F, F'),
\]
and this shows (2).

As in the proof of Lemma~\ref{lemma:AN:s:lip}, there is a subtle issue
here.  In writing ``$\Smyth (\Val_\bullet X), \mQ {(\dKRH^a)}$'', we
silently assume that $\Val_\bullet X$ has the $\dKRH^a$-Scott
topology, and the elements of $\Smyth (\Val_\bullet X)$ are compact
subsets for that topology.  However Proposition~\ref{prop:DN:dKRH}
considers compact subsets for the weak topology.  Fortunately, the two
topologies are the same, by Theorem~\ref{thm:V:weak=dScott}.  \qed

\begin{lem}
  \label{lemma:DN:s:lipcont}
  Let $X, d$ be a continuous Yoneda-complete quasi-metric space.  Let
  $\bullet$ be either ``$\leq 1$'' or ``$1$'', and $a > 0$.  Then
  $s^\bullet_\DN$ is $1$-Lipschitz continuous from the space of
  (sub)normalized sublinear previsions on $X$ with the $\dKRH^a$
  quasi-metric to $\Smyth (\Val_\bullet X), \mQ {(\dKRH^a)}$.
\end{lem}
\proof By Lemma~\ref{lemma:DN:s:lip}, $\mathbf B^1 (r_\DN)$ and
$\mathbf B^1 (s^\bullet_\DN)$ are monotonic.  Since
$r_\DN \circ s^\bullet_\DN = \identity\relax$,
$\mathbf B^1 (r_\DN) \circ \mathbf B^1 (s^\bullet_\DN) =
\identity\relax$.  Since
$\identity\relax \geq s^\bullet_\DN \circ r_\DN$, i.e., since
$\mathcal Q \subseteq s^\bullet_\DN (r_\DN (\mathcal Q))$ for every
$\mathcal Q$, we also have that
$\mathbf B^1 (s^\bullet_\DN) (\mathbf B^1 (r_\DN) (\mathcal Q, r))
\leq^{\mQ {(\dKRH^a)}^+} (\mathcal Q, r)$: indeed
$\mQ {(\dKRH^a)} (s^\bullet_\DN (r_\DN (\mathcal Q)), \mathcal Q) = 0$
since $s^\bullet_\DN (r_\DN (\mathcal Q)) \supseteq \mathcal Q$ (see
Lemma~\ref{lemma:dQ:spec}), and this is less than or equal to $r - r$.

This shows that $\mathbf B^1 (s^\bullet_\DN)$ is left adjoint to
$\mathbf B^1 (r_\DN)$.  We conclude because left adjoints
preserve all existing suprema, hence in particular directed suprema.
\qed

\begin{lem}
  \label{lemma:DN:r:lipcont}
  Let $X, d$ be a continuous Yoneda-complete quasi-metric space.  Let
  $\bullet$ be either ``$\leq 1$'' or ``$1$'', and $a > 0$.  Then
  $r_\DN$ is $1$-Lipschitz continuous from
  $\Smyth (\Val_\bullet X), \mQ {(\dKRH^a)}$ to the space of
  (sub)normalized sublinear previsions on $X$ with the $\dKRH^a$
  quasi-metric.
\end{lem}
\proof By Lemma~\ref{lemma:DN:s:lip}, $\mathbf B^1 (r_\DN)$ is
monotonic.  Consider a directed family of formal balls
${(\mathcal Q_i, r_i)}_{i \in I}$ on
$\Smyth (\Val_\bullet X), \mQ {(\dKRH^a)}$, with supremum
$(\mathcal Q, r)$.  We use Theorem~\ref{thm:Vleq1:cont}, resp.,
Theorem~\ref{thm:V1:cont} to claim that $\Val_\bullet X, \dKRH^a$ is
continuous Yoneda-complete, and then
Proposition~\ref{prop:QX:Ycomplete} to obtain that
$\Smyth (\Val_\bullet X), \mQ {(\dKRH^a)}$ is Yoneda-complete, and
that directed suprema of formal balls are naive suprema.  This means
that for every $H \in \Lform^a_\alpha (\Val_\bullet X, \dKRH^a)$
($\alpha > 0$),
$F_{\mathcal Q} (H) = \sup_{i \in I} (F_{\mathcal Q_i} (H) - \alpha
r_i + \alpha r)$, and $r = \inf_{i \in I} r_i$.  In particular, for
$H (G) = G (h)$, for any $h \in \Lform^a_1 (X, d)$
(Lemma~\ref{lemma:H}), so that
$F_{\mathcal Q} (H) = \min_{G \in \mathcal Q} H (G) = \min_{G \in
  \mathcal Q} G (h)$, we obtain
$\min_{G \in \mathcal Q} G (h) = \sup_{i \in I} (\min_{G \in \mathcal
  Q_i} G (h) - \alpha r_i + \alpha r)$, namely
$r_\DN (\mathcal Q) (h) = \sup_{i \in I} (r_\DN (\mathcal Q_i) (h) -
\alpha r_i + \alpha r)$.  This characterizes $(r_\DN (\mathcal Q), r)$
as the naive supremum of ${(r_\DN (\mathcal Q_i), r_i)}_{i \in I}$.
For every non-empty compact saturated set $\mathcal Q$,
$r_\DN (\mathcal Q)$ is a superlinear prevision.  It is in particular
continuous, so $r_\DN (\mathcal Q)$ is the supremum of
${(r_\DN (\mathcal Q_i), r_i)}_{i \in I}$ by
Proposition~\ref{prop:LPrev:simplesup}~(3) (and (4)).  Hence
$\mathbf B^1 (r_\DN)$ is Scott-continuous, which concludes the
argument.  \qed

\begin{thm}[Continuity for superlinear previsions]
  \label{thm:DN:cont}
  Let $X, d$ be a continuous Yoneda-complete quasi-metric space, and
  $a > 0$.

  The space of all subnormalized (resp., normalized) superlinear
  previsions on $X$ with the $\dKRH^a$ quasi-metric is continuous
  Yoneda-complete.

  It occurs as a $1$-Lipschitz continuous retract of
  $\Smyth (\Val_{\leq 1} X), \mQ {(\dKRH^a)}$ (resp.,
  $\Smyth (\Val_1 X), \allowbreak \mQ {(\dKRH^a)}$), and as an
  isometric copy of $\Smyth^{cvx} (\Val_{\leq 1} X), \mQ {(\dKRH^a)}$
  (resp., $\Smyth^{cvx} (\Val_1 X), \mQ {(\dKRH^a)}$) through $r_\DN$
  and $s^{\leq 1}_\DN$ (resp., $s^1_\DN$).

  The $\dKRH^a$-Scott topology on the space of subnormalized (resp.,
  normalized) superlinear previsions coincides with the weak topology.

  Finally, directed suprema of formal balls are computed as naive
  suprema.
\end{thm}
\proof We have already seen that $r_\AN$ and $s^\bullet_\AN$ are
continuous, with respect to the weak topologies.  By
Lemma~\ref{lemma:DN:r:lipcont} and Lemma~\ref{lemma:DN:s:lipcont},
they are also $1$-Lipschitz continuous with respect to
$\mQ {(\dKRH^a)}$ on $\Smyth (\Val_\bullet X)$ and $\dKRH^a$ on the
space of (sub)normalized superlinear previsions.

Recall also that $r_\DN \circ s^\bullet_\DN = \identity \relax$.  This
shows the $1$-Lipschitz continuous retract part, and also the
isometric copy part, since $s^\bullet_\DN$ and $r_\DN$ restrict to
bijections with the subspace $\Smyth^{cvx} (\Val_\bullet X)$.

Lemma~\ref{lemma:retract:alg} tells us that the space of
(sub)normalized superlinear previsions on $X$ is continuous
Yoneda-complete.

Let us show that the $\dKRH^a$-Scott topology on the space of
(sub)normalized superlinear previsions is the weak topology.  For
that, we recall that the $\mQ {(\dKRH^a)}$-Scott topology coincides
with the upper Vietoris topology on $\Smyth (\Val_\bullet X)$, and we
note that $s^\bullet_\DN$ is both a topological embedding for the weak
and upper Vietoris topologies, and the section part of a $1$-Lipschitz
continuous retraction, hence a topological embedding for the
$\dKRH^a$-Scott and $\mQ {(\dKRH^a)}$-Scott topologies.  The
conclusion follows from Fact~\ref{fact:retract:two}.

It remains to show that directed suprema of formal balls are computed
as naive suprema.  The argument is exactly as for
Proposition~\ref{prop:AN:Ycomplete}.  We use the fact that
$p = \mathbf B^1 (r_\DN)$ and $e = \mathbf B^1 (s^\bullet_\DN)$ form
an order-retraction, and that the supremum $(F, r)$ of a directed
family ${(F_k, r_k)}_{k \in K}$ is $p (\sup_{k \in K} e (F_k, r_k))$
by Lemma~\ref{lemma:orderretract}.  We invoke
Theorem~\ref{thm:Vleq1:cont}, resp.\ Theorem~\ref{thm:V1:cont},
together with Proposition~\ref{prop:QX:Ycomplete} to check that the
supremum $(\mathcal Q, r) = \sup_{k \in K} e (F_k, r_k)$, computed in
$\mathbf B (\Smyth (\Val_\bullet X), \mQ {(\dKRH^a)})$, is a naive
supremum.  In particular, $r = \inf_{k \in K} r_k$.  Let
$\mathcal Q_k = s^\bullet_\DN (F_k)$.  Then
$F_{\mathcal Q_k} (H) = \inf_{G \in \mathcal Q_k} H (G) = \inf_{G \in
  \Val_\bullet X, G \geq F_k} H (G)$ for every
$H \in \Lform^a_\alpha (\Val_\bullet X, \dKRH^a)$.  We have
$(F, r) = p (\mathcal Q, r)$, so $F = r_\DN (\mathcal Q)$, and that
implies that $F (h) = \inf_{G \in \mathcal Q} G (h)$ for every
$h \in \Lform^a_\alpha (X, d)$.
The map $H \colon G \mapsto G (h)$ is in
$\Lform^a_\alpha (\Val_\bullet X, \dKRH^a)$ by Lemma~\ref{lemma:H},
and:
\begin{eqnarray*}
  F (h)
  & =
  & \inf_{G \in \mathcal Q} H (G) \\
  & = & F_{\mathcal Q} (H) \\
  & = & \sup_{i \in I} (\inf_{G \in \Val_\bullet X, G \geq F_i} H (G)
        - \alpha r_i + \alpha r) \\
  & = & \sup_{i \in I} (\inf_{G \in \Val_\bullet X, G \geq F_i} G (h)
        - \alpha r_i + \alpha r) \\
  & = & \sup_{i \in I} (F_i (h) - \alpha r_i + \alpha r).
\end{eqnarray*}
\qed

\begin{lem}
  \label{lemma:DN:cont:removed}
  (This result removed.  It stated that suprema are naive suprema,
  repeating the end of the previous theorem.  This display kept in
  order to maintain the numbering of theorems.)
\end{lem}

As usual, algebraicity is preserved, too.
\begin{thm}[Algebraicity for superlinear previsions]
  \label{thm:DN:alg}
  Let $X, d$ be an algebraic Yoneda-complete quasi-metric space, with
  a strong basis $\mathcal B$.  Let $a > 0$.

  The space of all subnormalized (resp., normalized) superlinear
  previsions on $X$ with the $\dKRH^a$ quasi-metric is algebraic
  Yoneda-complete.

  All the superlinear previsions of the form
  $h \mapsto \min_{i=1}^m \sum_{j=1}^{n_i} a_{ij} h (x_{ij})$, with
  $m \geq 1$, $\sum_{j=1}^{n_i} a_{ij} \leq 1$ (resp., $=1$) for every
  $i$, and where each $x_{ij}$ is a center point, are center points,
  and they form a strong basis, even when each $x_{ij}$ is constrained
  to lie in $\mathcal B$.
\end{thm}

\proof We cannot appeal to the final part of
Lemma~\ref{lemma:retract:alg} in order to show algebraicity.  Indeed,
we would need to show
$\dKRH^a (r_\DN (\mathcal Q), F') = \mQ {(\dKRH^a)} (\mathcal Q,
s^\bullet_\DN (F'))$ to this end, but Proposition~\ref{prop:DN:dKRH}
only allows us to conclude the similarly looking, but different
equality
$\dKRH^a (F, r_\DN (\mathcal Q')) = \mQ {(\dKRH^a)} (s^\bullet_\DN
(F), \mathcal Q')$.

We can still use the fact that every (sub)normalized superlinear
prevision is a $\dKRH^a$-limit of a Cauchy-weightable net of points of
the form $r_\DN (\mathcal Q)$, where each $\mathcal Q$ is taken from a
strong basis of $\Smyth (\Val_\bullet X)$.  Using
Theorem~\ref{thm:V:alg} or Theorem~\ref{thm:V1:alg}, and
Theorem~\ref{thm:Q:alg}, we know that we can take $\mathcal Q$ of the
form
$\upc \{\sum_{j=1}^{n_1} a_{1j} \delta_{x_{1j}}, \cdots,
\sum_{j=1}^{n_m} a_{mj} \delta_{x_{mj}}\}$ where $m \geq 1$ and each
$x_{ij}$ is in $\mathcal B$.  Then $r_\DN (\mathcal Q)$ is exactly one
of the superlinear previsions mentioned in the theorem, namely it is
of the form
$h \mapsto \min_{i=1}^m \sum_{j=1}^{n_i} a_{ij} h (x_{ij})$.

To show that our spaces of superlinear previsions are algebraic, it
remains to show that
$\mathring F \colon h \mapsto \min_{i=1}^m \sum_{j=1}^{n_i} a_{ij} h
(x_{ij})$, where $m \geq 1$ and each $x_{ij}$ is a center point, is
itself a center point.  We cannot use the final part of
Lemma~\ref{lemma:retract:alg}.  Instead we shall show that directly,
in the manner of Lemma~\ref{lemma:V:simple:center} and related lemmas.
In other words, we show that
$B^{\dKRH^a}_{(\mathring F, 0), \epsilon}$ is Scott-open for every
$\epsilon > 0$.  This proceeds exactly as in
Lemma~\ref{lemma:V:simple:center} and related lemmas.

For that, we consider a directed family ${(F_k, r_k)}_{k \in K}$ where
each $F_k$ is a (sub)normalized superlinear prevision, with supremum
$(F, r)$ in $B^{\dKRH^a}_{(\mathring F, 0), \epsilon}$.  For every
$h \in \Lform^a_1 (X, d)$,
$\dreal (\mathring F (h), \allowbreak F (h)) < \epsilon - r$, that is,
$\epsilon > r$ and $\mathring F (h) - \epsilon + r < F (h)$.  In
particular, and using the fact that the supremum $(F, r)$ is a naive
supremum (final part of Theorem~\ref{thm:DN:cont}), there is a
$k \in K$ such that $\mathring F (h) - \epsilon + r_k < F_k (h)$, and
we can take $k$ so large that $\epsilon - r_k > 0$.

We note that $\mathring F$ is not just Scott-continuous, but that it
is also continuous from $\Lform_\alpha (X, d)^\dG$ (resp.,
$\Lform_\alpha^a (X, d)^\dG$) to $(\creal)^\dG$.  This is
Corollary~\ref{corl:cont:Lalpha:simple:2}.  It follows that the map
$h \in \Lform^a_1 (X, d)^\patch \mapsto F_k (h) - \mathring F (h)$ is
lower semicontinuous.  Therefore, the set $V_k$ of all
$h \in \Lform^a_1 (X, d)$ such that
$\mathring F (h) - \epsilon + r_k < F_k (h)$ is open in
$\Lform^a_1 (X, d)^\patch$.  The argument of the previous paragraph
shows that ${(V_k)}_{k \in K, \epsilon - r_k > 0}$ forms an open cover
of $\Lform^a_1 (X, d)^\patch$.  The latter is compact by
Lemma~\ref{lemma:cont:Lalpha:retr}~(4), so we can extract a finite
subcover.  This gives us a finite set $J$ of indices from $K$ such
that for every $h \in \Lform^a_1 (X, d)$, there is a $j \in J$ such
that $\mathring F (h) - \epsilon + r_j < F_j (h)$ and
$\epsilon - r_j > 0$.  Using the fact that ${(F_k, r_k)}_{k \in K}$ is
directed, there is even a single $k \in K$ such that
$(F_j, r_j) \leq^{\dKRH^{a+}} (F_k, r_k)$ for every $j \in J$, and
that implies that for every $h \in \Lform^a_1 (X, d)$,
$\mathring F (h) - \epsilon + r_k < F_k (h)$.  That shows that
$(F_k, r_k)$ is in $B^{\dKRH^a}_{(\mathring F, 0), \epsilon}$,
finishing our proof that $B^{\dKRH^a}_{(\mathring F, 0), \epsilon}$ is
Scott-open.  Therefore
$\mathring F \colon h \mapsto \min_{i=1}^m \sum_{j=1}^{n_i} a_{ij} h
(x_{ij})$ is a center point.  \qed

\begin{lem}
  \label{lemma:DN:functor}
  Let $X, d$ and $Y, \partial$ be two continuous Yoneda-complete
  quasi-metric spaces, and $f \colon X, d \mapsto Y, \partial$ be a
  $1$-Lipschitz continuous map.  Let also $a > 0$.  $\Prev f$
  restricts to a $1$-Lipschitz continuous map from the space of
  normalized, resp.\ subnormalized superlinear previsions on $X$ to the
  space of normalized, resp.\ subnormalized superlinear previsions on
  $Y$, with the $\dKRH^a$ and $\KRH\partial^a$ quasi-metrics.
\end{lem}
\proof
By Lemma~\ref{lemma:Pf:lip}, $\Prev f$ is $1$-Lipschitz and maps
(sub)normalized superlinear previsions to (sub)normalized superlinear
previsions.  Also, $\mathbf B^1 (\Prev f)$ preserves naive suprema.
By Theorem~\ref{thm:DN:cont}, suprema of directed families of formal
balls of (sub)normalized sublinear previsions are naive suprema, so
$\mathbf B^1 (\Prev f)$ is Scott-continuous.  It follows that $\Prev
f$ is $1$-Lipschitz continuous.
\qed

With Theorem~\ref{thm:DN:alg} and Theorem~\ref{thm:DN:cont}, we obtain
the following.
\begin{cor}
  \label{cor:DN:functor}
  There is an endofunctor on the category of continuous (resp.,
  algebraic) Yoneda-complete quasi-metric spaces and $1$-Lipschitz
  continuous maps, which sends every object $X, d$ to the space of
  (sub)normalized superlinear previsions on $X$ with the $\dKRH^a$-Scott
  topology ($a > 0$), and every $1$-Lipschitz continuous map $f$ to
  $\Prev f$.  \qed
\end{cor}

\section{Lenses, Quasi-Lenses, and Forks}
\label{sec:forks-quasi-lenses}

A natural model for mixed angelic and demonic non-determinism, without
probabilistic choice, on a space $X$, is the Plotkin powerdomain.
That consists in the space of \emph{lenses} on $X$.  A lens $L$ is an
intersection $Q \cap C$ of a compact saturated set $Q$ and a closed
set $C$, where $Q$ intersects $C$.  The space of lenses is equipped
with the Vietoris topology, generated by subsets
$\Box U = \{L \mid L \subseteq U\}$ and
$\Diamond U = \{L \mid L \cap U \neq \emptyset\}$, $U$ open in $X$.
For every lens $L$, one can write $L$ canonically as $Q \cap C$, by
taking $Q = \upc L$ and $C = cl (L)$.

There are several competing variants of the Plotkin powerdomain.  An
elegant one is given by Heckmann's \emph{continuous
  $\mathbf A$-valuations} \cite{Heckmann:absval}.  Those are
Scott-continuous maps $\alpha \colon \Open X \to \mathbf A$ such that
$\alpha (\emptyset) = 0$, $\alpha (X) = 1$, and for all opens $U$,
$V$, $\alpha (U)=0$ implies $\alpha (U \cup V) = \alpha (V)$, and
$\alpha (U) = 1$ implies $\alpha (U \cap V) = \alpha (V)$.  Here
$\mathbf A$ is the poset $\{0, \mathsf M, 1\}$ ordered by
$0 < \mathsf M < 1$.  The \emph{Vietoris topology} on the space of
continuous $\mathbf A$-valuations is generated by the subsets $\Box U$
of continuous $\mathbf A$-valuations such that $\alpha (U)=1$, and the
subsets $\Diamond U$ of continuous $\mathbf A$-valuations such that
$\alpha (U) \neq 0$, where $U$ ranges over the open subsets of $X$.

It was shown in \cite[Fact~5.2]{JGL-mscs09} that, provided $X$ is
sober, the space of continuous $\mathbf A$-valuations is homeomorphic
to a space of so-called \emph{quasi-lenses}, namely, pairs $(Q, C)$ of
a compact saturated set $Q$ and a closed set $C$ such that $Q$
intersects $C$, $Q \subseteq \upc (Q \cap C)$ (whence
$Q = \upc (Q \cap C)$), and for every open neighborhood $U$ of $Q$,
$C \subseteq cl (U \cap C)$.  This is topologized by a topology that
we shall again call the \emph{Vietoris topology}, generated by subsets
$\Box U = \{(Q, C) \mid Q \subseteq U\}$ and
$\Diamond U = \{(Q, C) \mid C \cap U \neq \emptyset\}$.  The
homeomorphism maps $(Q, C)$ to $\alpha$ where $\alpha (U)$ is equal to
$1$ if $Q \subseteq U$, to $0$ if $C$ does not intersect $U$, and to
$\mathsf M$ otherwise.

The space of quasi-lenses is homeomorphic to the space of lenses when
$X$ is stably compact, through the map $L \mapsto (\upc L, cl (L))$
(Proposition~5.3, loc.\ cit.), but is in general larger.  Note, that,
in case $X$ is not only stably compact but also Hausdorff, i.e., when
$X$ is compact Hausdorff, then the lenses are just the compact subsets
of $X$.  It should therefore come as no surprise that we shall equip
our space of (quasi-)lenses with a quasi-metric resembling the
Hausdorff metric.

Quasi-lenses and continuous $\mathbf A$-valuations are a special case
of the notion of a \emph{fork} \cite{Gou-csl07}.  A fork on $X$ is by
definition a pair $(F^-, F^+)$ of a superlinear prevision $F^-$ and a
sublinear prevision $F^+$ satisfying \emph{Walley's condition}:
\[
F^- (h+h') \leq F^- (h) + F^+ (h') \leq F^+ (h+h')
\]
for all $h, h' \in \Lform X$.  This condition was independently
discovered by Keimel and Plotkin \cite{KP:predtrans:pow}, who call
such a pair \emph{medial} \cite{KP:mixed}.

By taking $h'=0$, or $h=0$, Walley's condition in particular implies
that $F^- \leq F^+$.  A fork $(F^-, F^+)$ is normalized,
subnormalized, or discrete if and only if both $F^-$ and $F^+$ are.

We shall call \emph{weak topology} on the space of forks the topology
induced by the inclusion in the product of the spaces of sublinear and
superlinear previsions, each equipped with their weak topology.  Its
specialization quasi-ordering is given by
$(F^-, F^+) \leq ({F'}^-, {F'}^+)$ if and only if $F^- \leq {F'}^-$
and $F^+ \leq {F'}^+$.

\begin{defi}[$\dKRH$, $\dKRH^a$ on forks]
  \label{defn:KRH:fork}
  Let $X, d$ be a quasi-metric space.  The $\dKRH$ quasi-metric on the
  space of forks on $X$ is defined by:
  \[
    \dKRH ((F^-, F^+), ({F'}^-, {F'}^+)) = \max (\dKRH (F^-, {F'}^-),
    \dKRH (F^+, {F'}^+)),
  \]
  and similarly for the $\dKRH^a$ quasi-metric, $a > 0$.
\end{defi}

Those are quasi-metrics indeed.  Lemma~\ref{lemma:KRH:qmet} and
Lemma~\ref{lemma:KRHa:qmet} indeed imply:
\begin{lem}
  \label{lemma:KRH:fork:qmet}
  Let $X, d$ be a standard quasi-metric space.  For all forks
  $(F^-, F^+)$ and $({F'}^-, {F'}^+)$ on $X$, the following are
  equivalent: $(a)$ $(F^-, F^+) \leq ({F'}^-, {F'}^+)$; $(b)$
  $\dKRH ((F^-, \allowbreak F^+), ({F'}^-, {F'}^+))=0$; $(c)$
  $\dKRH^a((F^-, F^+), ({F'}^-, {F'}^+))=0$, for any $a > 0$.  \qed
\end{lem}

\subsection{The Plotkin Powerdomain}
\label{sec:plotkin-powerdomain}

To us, the Plotkin powerdomain will be the set of discrete normalized forks.
Equivalently, this is the space of quasi-lenses, which justifies the name.
\begin{prop}
  \label{prop:fork=lens}
  Let $X$ be a sober space.  There is a bijection between quasi-lenses
  on $X$ and discrete normalized forks on $X$, which maps the
  quasi-lens $(Q, C)$ to $(F_Q, F^C)$.
\end{prop}
\proof Fix a quasi-lens $(Q, C)$.  $F^C$ is a discrete normalized
sublinear prevision by Lemma~\ref{lemma:eH}, and $F_Q$ is a discrete
normalized superlinear prevision by Lemma~\ref{lemma:uQ}.  Moreover,
the maps $Q \mapsto F_Q$ and $C \mapsto F^C$ are bijections.  We need
to show that if $(Q, C)$ is a quasi-lens, then $(F_Q, F^C)$ is a fork,
and conversely.



We claim that $F_Q (h) = \inf_{x \in Q \cap C} h (x)$ for every
$h \in \Lform X$.  Indeed,
$F_Q (h) = \inf_{x \in Q} h (x) \leq \inf_{x \in Q \cap C} h (x)$
follows from the fact that $Q$ contains $Q \cap C$, and in the
reverse direction, we use the equality $Q = \upc (Q \cap C)$ to show
that for every $x \in Q$, there is a $y \in Q \cap C$ such that
$h (x) \geq h (y)$.

For all $h, h' \in \Lform X$, it follows that
$F_Q (h+h') = \inf_{x \in Q \cap C} (h (x) + h' (x)) \leq \inf_{x \in Q
  \cap C} (h (x) + \sup_{x' \in Q \cap C} h' (x')) \leq \inf_{x \in Q
  \cap C} (h (x) + \sup_{x' \in C} h' (x')) = F_Q (h) + F^C (h')$.

To show that $F_Q (h) + F^C (h') \leq F^C (h+h')$, we assume the
contrary.  There is a real number $t$ such that
$F_Q (h) + F^C (h') > t \geq F^C (h+h')$.  The second inequality
implies that $F^C (h')$, which is less than or equal to $F^C (h+h')$,
is a (non-negative) real number $a$.  The first inequality then
implies $F_Q (h) > t-a$.  It follows that there must be a real number
$b$ such that $F_Q (h) > b > t-a$.  Since $F_Q (h) > b$, every element
$x$ of $Q$ is such that $h (x) > b$.  Therefore $Q$ is included in
the open set $U = h^{-1} (]b, +\infty])$.  We know that
$C \subseteq cl (U \cap C)$, and this implies that
$F^C (h') = \sup_{x \in C} h' (x) \leq \sup_{x \in cl (U \cap C)} h'
(x) = \sup_{x \in U \cap C} h' (x)$ (using
Lemma~\ref{lemma:H:cl:sup}).  Since $b > t-a$ and $a = F^C (h')$,
there is a point $x$ in $U \cap C$ such that $h' (x) > t-b$.  Because
$x$ is in $U$, $h (x) > b$, so $h (x)+h'(x) > t$, contradicting the
inequality $t \geq F^C (h+h')$.

Hence if $(Q, C)$ is a quasi-lens, then $(F_Q, F^C)$ is a fork, which
is necessarily discrete and normalized.

Conversely, let us assume that $(F_Q, F^C)$ is a fork.  We have seen
that, in particular, $F_Q \leq F^C$.  Hence
$F_Q (\chi_U) \leq F^C (\chi_U)$ where $U$ is the complement of $C$.
That implies that $F_Q (\chi_U)=0$, i.e., $Q$ is not included in $U$,
so that $Q$ intersects $C$.

Let us use the inequality $F_Q (h+h') \leq F_Q (h) + F^C (h')$ on
$h' = \chi_U$ where $U$ is the complement of $C$ again.  Then
$F^C (h')=0$, so
$\inf_{y \in Q} (h (y) + \chi_U (y)) \leq \inf_{y \in Q} h (y)$.  In
other words,
$\min (\inf_{y \in Q \cap U} (h (y) + 1), \inf_{y \in Q \cap C} h (y))
\leq \inf_{y \in Q} h (y)$, or equivalently, one of the two numbers
$\inf_{y \in Q \cap U} (h (y) + 1)$, $\inf_{y \in Q \cap C} h (y)$ is
less than or equal to $\inf_{y \in Q} h (y)$.  That cannot be the
first one, which is larger than or equal to
$1+\inf_{y \in Q \cap U} h (y) \geq 1 + \inf_{y \in Q} h (y)$, so:
$(*)$ $\inf_{y \in Q \cap C} h (y) \leq \inf_{y \in Q} h (y)$.  Now,
for every $x \in Q$, $\dc x$ is closed, hence its complement $V$ is
open.  Take $h = \chi_V$.  Then $\inf_{y \in Q} h (y) \leq h (x) = 0$,
therefore $\inf_{y \in Q \cap C} h (y)=0$ by $(*)$.  Since $h$ only
takes the values $0$ or $1$ (or since $Q \cap C$ is compact), there is
a point $y$ in $Q \cap C$ such that $h (y)=0$, i.e., such that
$y \leq x$.  Therefore $Q$ is included in $\upc (Q \cap C)$.

Finally, let $U$ be any open neighborhood of $Q$.  We use the
inequality $F_Q (h) + F^C (h') \leq F^C (h + h')$ on $h = \chi_U$, so
that $1 + F^C (h') \leq F^C (h + h')$.  For every open subset $V$ that
intersects $C$, taking $h' = \chi_V$ we obtain that
$2 \leq F^C (\chi_U + \chi_V) = \sup_{x \in C} (\chi_U (x) + \chi_V
(x))$.  Since $\chi_U (x) + \chi_V (x)$ can only take the values $0$,
$1$, or $2$, there must be a point $x$ in $C$ such that
$\chi_U (x)+\chi_V (x)=2$.  Such a point must be both in $U$ and in
$V$.  We have shown that every open subset $V$ that intersects $C$
also intersects $U \cap C$.  Letting $V$ be the complement of
$cl (U \cap C)$, it follows that $C$ cannot intersect $V$, hence must
be included in $cl (U \cap C)$.  \qed

\begin{defi}[$\dP$, $\dP^a$]
  \label{defn:dP}
  Let $X, d$ be a quasi-metric space.  For any two quasi-lenses
  $(Q, C)$ and $(Q', C')$, let
  $\dP ((Q, C), (Q', C')) = \dKRH ((F_Q, F^C), (F_{Q'}, F^{C'}))$,
  and, for each $a > 0$, let
  $\dP^a ((Q, C), (Q', C')) = \dKRH^a ((F_Q, F^C), (F_{Q'}, F^{C'}))$.
  For every $a > 0$,
  $\dP^a ((Q, C), (Q', C')) = \max (\dQ^a (Q, Q'), \dH^a (C, C'))$.
\end{defi}

By definition of $\dQ$ (Definition~\ref{defn:dQ}) and of $\dQ^a$
(Remark~\ref{defn:dQa}), and by Proposition~\ref{prop:H:prev}
and Remark~\ref{rem:HX:dKRHa}, we obtain the following
characterization.
\begin{lem}
  \label{lem:dP}
  Let $X, d$ be a standard quasi-metric space.  Then $\dP ((Q, C),
  (Q', C')) = \max (\dQ (Q, Q'), \dH (C, C'))$.
\end{lem}

Together with Proposition~\ref{prop:fork=lens}, this implies the following.
\begin{cor}
  \label{corl:P:prev}
  Let $X, d$ be a sober standard quasi-metric space.  The bijection
  $(Q, C) \mapsto (F_Q, F^C)$ is an isometry between the space of
  quasi-lenses on $X$ with the $\dP$ quasi-metric, and the space of
  discrete normalized forks on $X$ with the $\dKRH$ quasi-metric.  \qed
\end{cor}

\begin{rem}[The Hausdorff quasi-metric]
  \label{rem:dP:lens}
  Let $X, d$ be a standard quasi-metric space.  If we equate a lens
  $L$ with the quasi-lens $(\upc L, cl (L))$, then we can reinterpret
  Lemma~\ref{lem:dP} by saying that
  $\dP (L, L') = \max (\dQ (\upc L, \upc L'), \dH (cl (L), cl (L'))$.

  We have
  $\dQ (\upc L, \upc L') = \sup_{x' \in \upc L'} \inf_{x \in \upc L} d
  (x, x')$ by Lemma~\ref{lemma:dQ}, and this is also equal to
  $\sup_{x' \in L'} \inf_{x \in L} d (x, x')$, since $d$ is monotonic
  in its first argument and antitonic in its second argument.  We also
  have $\dH (cl (L), cl (L')) = \sup_{x \in cl (L)} d (x, cl (L'))$ by
  Definition~\ref{defn:dH}, and this is equal to
  $\sup_{x \in L} d (x, cl (L'))$ by Lemma~\ref{lemma:H:cl:sup}, since
  $d (\_, cl (L'))$ is $1$-Lipschitz continuous (by
  \cite[Lemma~6.11]{JGL:formalballs}, recalled at the beginning of
  Section~\ref{sec:hoare-quasi-metric}) hence continuous.

  If $x$ is a center point, we recall that
  $d (x, cl (L')) = \inf_{x' \in cl (L')} d (x, x')$.  Also the map
  $-d (x, \_)$ is Scott-continuous, because the inverse image of
  $]t, +\infty]$ is the open ball $B^d_{x, <-t}$ (meaning the empty
  set if $t\leq 0$), so
  $\inf_{x' \in cl (L')} d (x, x') = -\sup_{x' \in cl (L')} -d (x, x')
  = - \sup_{x' \in L'} \allowbreak -d (x, x') = \inf_{x' \in L'} d (x,
  x')$ by Lemma~\ref{lemma:H:cl:sup}.

  All that implies that, on a standard quasi-metric space $X, d$ where
  every point is a center point, the $\dP$ quasi-metric coincides with
  the familiar Hausdorff quasi-metric:
  \[
    \dP (L, L') = \max (\sup_{x \in L} \inf_{x' \in L'} d (x, x'),
    \sup_{x' \in L'} \inf_{x \in L} d (x, x')).
  \]
  This is the case, notably, if $X, d$ is Smyth-complete.  This is
  also the case if $X, d$ is a \emph{metric} space, in which case
  $\dP$ is exactly the Hausdorff metric on the space of non-empty
  compact subsets of $X$.
\end{rem}

\begin{lem}
  \label{lemma:PX:naivesup}
  Let $X, d$ be a sober, standard quasi-metric space.  For every
  directed family of formal balls ${((F^-_i, F^+_i), r_i)}_{i \in I}$
  on the space of forks, if ${(F^-_i, r_i)}_{i \in I}$ has a naive
  supremum $(F^-, r)$ where $F^-$ is continuous, and if
  ${(F^+_i, r_i)}_{i \in I}$ has a naive supremum $(F^+, r)$ where
  $F^+$ is continuous, then $(F^-, F^+)$ is a fork.
\end{lem}
\proof Note that, because the quasi-metric on forks is defined as a
maximum, the two families ${(F^-_i, r_i)}_{i \in I}$ and
${(F^+_i, r_i)}_{i \in I}$ are directed, so taking their naive suprema
makes sense.  Explicitly, define $i \preceq j$ as
$((F^-_i, F^+_i), r_i) \leq^{\dKRH^+} ((F^-_j, F^+_j), r_j)$.  Then
$i \preceq j$ implies that
$\dKRH ((F^-_i, F^+_i), (F^-_j, F^+_j)) \leq r_i - r_j$, hence that
$\dKRH (F^-_i, F^-_j) \leq r_i - r_j$ and
$\dKRH (F^+_i, F^+_j) \leq r_i - r_j$, so that
$(F^-_i, r_i) \leq^{\dKRH^+} (F^-_j, r_j)$ and
$(F^+_i, r_i) \leq^{\dKRH^+} (F^+_j, r_j)$.  For all $i, j \in I$,
there is a $k \in I$ such that $i, j \preceq k$, and that implies that
$(F^-_i, r_i), (F^-_j, r_j) \leq^{\dKRH^+} (F^-_k, r_k)$, and
similarly with $+$ instead of $-$.

Note also that the naive suprema $(F^-, r)$ and $(F^+, r)$ must have
the same radius $r = \inf_{i \in I} r_i$, by definition.

By Proposition~\ref{prop:LPrev:simplesup}~(4), $F^-$ is a
superlinear prevision and $F^+$ is a sublinear prevision.  We must
prove Walley's condition.

For all $h \in \Lform_\alpha (X, d)$ (resp., $\Lform_\alpha^a (X, d)$)
and $h' \in \Lform_\beta (X, d)$ (resp., $\Lform_\beta^a (X, d)$),
for all $\alpha, \beta > 0$, we have:
\[
  F_i^- (h + h') \leq F_i^- (h) + F_i^+ (h') \leq F_i^+ (h+h')
\]
for every $i \in I$, since $(F_i^-, F_i^+)$ is a fork.  Therefore:
\begin{align*}
  F_i^- (h + h') + (\alpha+\beta) r - (\alpha+\beta) r_i
  & \leq (F_i^- (h) +\alpha r - \alpha r_i) + (F_i^+ (h')+\beta r-\beta
    r_i) \\
  & \leq F_i^+ (h+h')  + (\alpha+\beta) r - (\alpha+\beta) r_i
\end{align*}
We note that $h+h'$ is $(\alpha+\beta)$-Lipschitz continuous.  By
taking (directed) suprema over $i \in I$, and noting that addition is
Scott-continuous on $\creal$, we obtain the desired inequalities:
\[
  F^- (h + h') \leq F^- (h) + F^+ (h') \leq F^+ (h+h')
\]
for all $h \in \Lform_\alpha (X, d)$ (resp., $\Lform_\alpha^a (X, d)$)
and $h' \in \Lform_\beta (X, d)$ (resp., $\Lform_\beta^a (X, d)$),

For all $h, h' \in \Lform X$, for every $\alpha > 0$, we then obtain
that $F^- (h^{(\alpha)} + {h'}^{(\alpha)}) \leq F^- (h^{(\alpha)}) +
F^+ ({h'}^{(\alpha)}) \leq F^- (h) + F^+ (h')$, hence, taking suprema
over $\alpha$, $F^- (h+h') \leq F^- (h) + F^+ (h)$.
Similarly, $F^- (h^{(\alpha)}) + F^+ ({h'}^{(\alpha)}) \leq F^+
(h^{(\alpha)} + {h'}^{(\alpha)}) \leq F^+ (h+h')$, hence, taking
suprema over $\alpha$, $F^- (h) + F^+ (h') \leq F^+ (h+h')$.  \qed

\begin{prop}[Yoneda-completeness, Plotkin powerdomain]
  \label{prop:PX:Ycomplete}
  Let $X, d$ be a sober standard Lipschitz regular quasi-metric space,
  or a continuous Yoneda-complete quasi-metric space.  The space of
  quasi-lenses on $X$, with the $\dKRH$ quasi-metric, is
  Yoneda-complete.

  The supremum of a directed family of formal balls
  ${((Q_i, C_i), r_i)}_{i \in I}$ is $((Q, C), r)$ where
  $r = \inf_{i \in I} r_i$, $(F_Q, r)$ is the naive supremum of
  ${(F_{Q_i}, r_i)}_{i \in I}$, and $(F^C, r)$ is the naive supremum
  of ${(F^{C_i}, r_i)}_{i \in I}$.  In particular, $(Q, r)$ is the
  supremum of ${(Q_i, r_i)}_{i \in I}$ and $(C, r)$ is the supremum of
  ${(C_i, r_i)}_{i \in I}$.
\end{prop}
\proof Note that every continuous Yoneda-complete quasi-metric space is sober
\cite[Proposition~4.1]{JGL:formalballs}, so we do not have to assume
sobriety explicitly in that case.

Let ${((Q_i, C_i), r_i)}_{i \in I}$ be a directed family of formal
balls on the space of quasi-lenses.  The families
${(Q_i, r_i)}_{i \in I}$ and ${(C_i, r_i)}_{i \in I}$ are also
directed, because $\dP$ is defined as a maximum.  We also obtain a
directed family of formal balls $((F_{Q_i}, F^{C_i}), r_i)$ on the
space of discrete normalized forks, by
Proposition~\ref{prop:fork=lens}.  By
Proposition~\ref{prop:HX:Ycomplete} and
Proposition~\ref{prop:QX:Ycomplete}, the supremum $(C, r)$ of
${(C_i, r_i)}_{i \in I}$ exists and is a naive supremum, and the
supremum $(Q, r)$ (with the same $r$) of ${(Q_i, r_i)}_{i \in I}$
exists and is a naive supremum.  By Lemma~\ref{lemma:PX:naivesup},
$(F_Q, F^C)$ is a fork, and certainly it is discrete and normalized.

To show that $((Q, C), r)$ is the desired supremum, we observe that
the upper bounds $((Q', C'), r')$ of ${((Q_i, C_i), r_i)}_{i \in I}$
are exactly those formal balls such that $(Q', r')$ is an upper bound
of ${(Q_i, r_i)}_{i \in I}$ and $(C', r')$ is an upper bound of
${(C_i, r_i)}_{i \in I}$.  This is again because our $\dKRH$
quasi-metric on forks (or the $\dP$ quasi-metric on quasi-lenses) is
defined as a maximum.  Hence $((Q, C), r)$ is such an upper bound, and
if $((Q', C'), r')$ is any other upper bound, then
$(Q, r) \leq^{\dQ^+} (Q', r')$ and $(C, r) \leq^{\dH^+} (C', r')$, so
$((Q, C), r) \leq^{\dP^+} ((Q', C'), r')$.  \qed

\begin{rem}
  \label{rem:PX:Ycomplete}
  For complete metric spaces $X, d$,
  Proposition~\ref{prop:PX:Ycomplete} specializes to the well-known
  fact that the space of non-empty compact subsets of $X$, with the
  Hausdorff metric, is complete.
\end{rem}

In order to show that the Plotkin powerdomain of an algebraic
Yoneda-complete quasi-metric space is algebraic, we require the
following two lemmas.

For disambiguation purposes, we write $\Box_{\mathcal P} U$ for the set
of lenses $(Q, C)$ such that $Q \subseteq U$, reserving the notation
$\Box U$ for the set of non-empty compact saturated subsets $Q$ of $X$
such that $Q \subseteq U$.  We use a similar convention with
$\Diamond_{\mathcal P} U$ and $\Diamond U$.

\begin{customlemma}{\thesubsection.A}
  \label{lemma:QC:approx}
  Let $X, d$ be a standard algebraic quasi-metric space, with a strong
  basis $\mathcal B$.  Given any quasi-lens $(Q, C)$ on $X$, any
  $\epsilon > 0$, and given open subsets $U$, $U_1$, \ldots, $U_m$ of
  $X$ such that
  $(Q, C) \in \Box_{\mathcal P} U \cap \bigcap_{i=1}^m
  \Diamond_{\mathcal P} U_i$, there is a finite non-empty subset $E$
  of $\mathcal B$ and a number $r$ such that $0 < r \leq \epsilon$ and
  $(Q, C) \in \Box_{\mathcal P} (\bigcup_{x \in E} B^d_{x, <r}) \cap
  \bigcap_{x \in E} \Diamond_{\mathcal P} B^d_{x, <r} \subseteq
  \Box_{\mathcal P} U \cap \bigcap_{i=1}^m \Diamond_{\mathcal P} U_i$.
\end{customlemma}
\proof We use the following remark: (A) For every $\epsilon > 0$, the
open balls $B^d_{x, <r}$ with $x \in \mathcal B$ and
$0 < r < \epsilon$ form a base of the $d$-Scott topology on $X$.  In
order to prove this, let $U$ be any open subset of $X$, and $y \in U$.
Then $(y, 0)$ is in $\widehat U \cap V_\epsilon$, and is the supremum
of a directed family of formal balls of the form $(x, r)$ with
$x \in \mathcal B$ and $(x, r) \ll (y, 0)$ by
Lemma~\ref{lemma:B:basis}.  Hence one of them is in
$\widehat U \cap V_\epsilon$.  Since it is in $V_\epsilon$,
$r < \epsilon$, and since $(x, r) \ll (y, 0)$ where $x$ is a center
point in a standard algebraic space, $d (x, y) < r$, so $y$ is in
$B^d_{x, <r}$.  Moreover, $B^d_{x, <r}$ is open, still because $x$ is
a center point, and because
$B^d_{x, <r} = B^{d^+}_{(x, 0), <r} \cap X$.  Finally, $B^d_{x, <r}$
is included in $\upc (x, r) \cap X \subseteq \widehat U \cap X = U$.

Also, we will rely on the following: (B) for every center point $x$,
for every closed subset $C'$ of $X$, for every $r > 0$,
$d (x, C') < r$ if and only if $B^d_{x, <r}$ intersects $C'$.  Indeed,
by \cite[Proposition~6.12]{JGL:formalballs} and since $x$ is a center
point, $d (x, C') = \inf_{y \in C'} d (x, y)$, so $d (x, C') < r$ if
and only if $d (x, y) < r$ for some $y \in C'$.

For every $i$, $1\leq i\leq m$, $C$ intersects $U_i$.  Since
$Q \subseteq U$, by definition of quasi-lenses $C$ is included in
$cl (U \cap C)$, so $cl (U \cap C)$ intersects $U_i$.  It follows that
$U \cap C$ intersects $U_i$, so that $C$ is in
$\Diamond (U \cap U_i)$.  Hence $(Q, C)$ is in the open set
$\Box_{\mathcal P} U \cap \bigcap_{i=1}^m \Diamond_{\mathcal P} (U
\cap U_i)$.  Replacing $U_i$ by $U \cap U_i$ if necessary, we will
therefore assume that every $U_i$ is included in $U$.

For every $i$, $1\leq i\leq m$, $U_i$ intersects $C$.  Using (A),
there is an open ball $B^d_{x'_i, <r'_i}$ included in $U_i$ that
intersects $C$, with $x'_i \in \mathcal B$ and $0 < r'_i < \epsilon$.
By (B), $d (x'_i, C) < r'_i$.  Let $\epsilon_1$ be any positive number
such that $d (x'_i, C) < r'_i - \epsilon_1$ for every $i$,
$1\leq i\leq m$.  Using (B) again, this means that
$B^d_{x'_i, <r'_i - \epsilon_1}$ intersects $C$.

Since $Q$ is compact and $C$ is closed, $Q \cap C$ is compact.  Using
Lemma~\ref{lemma:Q:approx}, there are finitely many points
$y'_1, \cdots, y'_p \in \mathcal B$ and radii
$s'_1, \cdots, s'_p < \epsilon$ such that
$Q \cap C \subseteq \bigcup_{j=1}^p B^d_{y'_j, <s'_j} \subseteq U$.
By Lemma~\ref{lemma:Q:B:eps}, there is a positive number $\epsilon_2$
such that $\epsilon_2 < s'_j$ for every $j$, $1\leq j\leq n$, and such
that
$Q \cap C \subseteq \bigcup_{j=1}^p B^d_{y'_j, <s'_j-\epsilon_2}
\subseteq U$.

Let $r = \min (\epsilon, \epsilon_1, \epsilon_2)$.

We use (A) a second time.  For every $i$, $1\leq i\leq m$,
$B^d_{x'_i, <r'_i-\epsilon_1}$ intersects $C$, so there is an open
ball $B^d_{x_i, <r_i}$ included in $B^d_{x'_i, <r'_i-\epsilon_1}$ that
intersects $C$, with $x_i \in \mathcal B$ and $0 < r_i < r$.  Now the
larger open ball $B^d_{x_i, <r}$ also intersects $C$, and we claim
that it is included in $U_i$.  For every $z \in B^d_{x_i, <r}$,
$d (x_i, z) < r$, so
$d (x'_i, z) \leq d (x'_i, x_i) + d (x_i, z) < r'_i -\epsilon_1 + r
\leq r'_i - \epsilon_1 + \epsilon_1 = r'_i$; this shows that $z$ is in
$B^d_{x'_i, <r'_i}$, hence in $U_i$.

We now use Lemma~\ref{lemma:Q:approx} a second time.  Since
$Q \cap C \subseteq \bigcup_{j=1}^p B^d_{y'_j, <s'_j-\epsilon_2}$,
there are finitely many points $y_1, \cdots, y_q \in \mathcal B$ and
radii $s_1, \cdots, s_q < r$ such that
$Q \cap C \subseteq \bigcup_{k=1}^q B^d_{y_k, <s_k} \subseteq
\bigcup_{j=1}^p B^d_{y'_j, <s'_j-\epsilon_2}$.  We may assume that
every term $B^d_{y_k, <s_k}$ in the union
$\bigcup_{k=1}^q B^d_{y_k, <s_k}$ intersects $Q \cap C$: any term of
that form that does not intersect $Q \cap C$ can safely be removed
from the union.  In particular, $B^d_{y_k, <s_k}$ intersects $C$ for
every $k$, $1\leq k\leq q$.  It follows that the larger open ball
$B^d_{y_k, <r}$ also intersects $C$.

Since $Q \cap C$ is included in $\bigcup_{k=1}^q B^d_{y_k, <s_k}$, it
is included in the larger set $\bigcup_{k=1}^q B^d_{y_k, <r}$.  We
claim that the latter is included in $U$.  For every element $z$ of
$\bigcup_{k=1}^q B^d_{y_k, <r}$, there is an index $k$,
$1\leq k\leq q$, such that $d (y_k, z) < r$.  Since $y_k$ is in
$\bigcup_{j=1}^p B^d_{y'_j, <s'_j-\epsilon_2}$, there is an index $j$,
$1\leq j\leq p$, such that $d (y'_j, y_k) < s'_j-\epsilon_2$.  Then
$d (y'_j, z) \leq d (y'_j, y_k) + d (y_k, z) < s'_j - \epsilon_2 + r
\leq s'_j - \epsilon_2 + \epsilon_2 = s'_j$, so $z$ is in
$B^d_{y'_j, <s'_j}$, hence in $U$.

In summary, we have found finitely many points $x_1$, \ldots, $x_m$
and $y_1$, \ldots, $y_q$ in $\mathcal B$ and a number $r$ such that
$0 < r \leq \epsilon$ and: $(a)$ $B^d_{x_i, <r}$ intersects $C$ for
each $i$, $1\leq i\leq m$, $(b)$ $B^d_{x_i, <r}$ is included in $U_i$
for each $i$, $1\leq i\leq m$, $(c)$ $B^d_{y_k, <r}$ intersects $C$
for each $k$, $1\leq k\leq q$; $(d)$ $Q \cap C$ is included in
$\bigcup_{k=1}^q B^d_{y_k, <r}$, and $(e)$
$\bigcup_{k=1}^q B^d_{y_k, <r}$ is included in $U$.

Let $E = \{x_1, \cdots, x_m, y_1, \cdots, y_q\}$.  This is a finite
subset of $\mathcal B$.  By $(d)$ and since $Q \cap C$ is non-empty,
$q$ is different from $0$, so $E$ is non-empty.  Using $(d)$ again,
$Q \cap C$ is included in $\bigcup_{x \in E} B^d_{x, <r}$, and therefore
$Q$ also is included in $\bigcup_{x \in E} B^d_{x, <r}$, since
$Q \subseteq \upc (Q \cap C)$ by definition of quasi-lenses, and since
open sets such as $\bigcup_{x \in E} B^d_{x, <r}$ are upwards-closed.
For every $x \in E$, $B^d_{x, <r}$ intersects $C$, by $(a)$ and $(c)$.
Therefore $(Q, C)$ is in
$\Box_{\mathcal P} (\bigcup_{x \in E} B^d_{x, <r}) \cap \bigcap_{x \in
  E} \Diamond_{\mathcal P} B^d_{x, <r}$.

Finally, let $(Q', C')$ be any element of
$\Box_{\mathcal P} (\bigcup_{x \in E} B^d_{x, <r}) \cap \bigcap_{x \in
  E} \Diamond_{\mathcal P} B^d_{x, <r}$.  From $(b)$, $(e)$, and the
fact that each $U_i$ is included in $U$, we obtain that
$\bigcup_{x \in E} B^d_{x, <r} \subseteq U$.  Hence, and since
$Q' \subseteq \bigcup_{x \in E} B_{x, <r}$, $Q'$ is included in $U$.
Next, using (B), for every $x \in E$, $d (x, C') < r$.  In particular,
for every $i$, $1\leq i\leq m$, $d (x_i, C') < r$.  By (B), $C'$
intersects $B^d_{x_i, <r}$.  Since the latter is included in $U_i$ by
$(b)$, $C'$ intersects $U_i$.  It follows that $(Q', C')$ is in
$\Box_{\mathcal P} U \cap \bigcap_{i=1}^n \Diamond_{\mathcal P} U_i$.
\qed

\begin{customlemma}{\thesubsection.B}
  \label{lemma:QC:B:eps}
  Let $X, d$ be a standard algebraic quasi-metric space.  Given any
  quasi-lens $(Q, C)$ on $X$, for all center points $x_1$, \ldots,
  $x_m$ and all $r_1, \cdots, r_m > 0$ such that
  $(Q, C) \in \Box_{\mathcal P} (\bigcup_{j=1}^m B^d_{x_j, <r_j}) \cap
  \bigcap_{j=1}^m \Diamond_{\mathcal P} B^d_{x_j, <r_j}$, there is a
  number $\epsilon$ such that $0 < \epsilon < r_1, \cdots, r_m$ and
  $(Q, C) \in \Box_{\mathcal P} (\bigcup_{j=1}^m B^d_{x_j,
    <r_j-\epsilon}) \cap \bigcap_{j=1}^m \Diamond_{\mathcal P}
  B^d_{x_j, <r_j-\epsilon}$.
\end{customlemma}
\proof As in the proof of Lemma~\ref{lemma:QC:approx}, since $x_j$ is
a center point for every $j$, $d (x_j, C) < r_j$ if and only if $C$
intersects $B^d_{x_j, <r_j}$, by
\cite[Proposition~6.12]{JGL:formalballs}.  Therefore
$d (x_j, C) < r_j$ for every $j$, $1\leq j\leq m$.  Let
$\epsilon_1 > 0$ be such that $d (x_j, C) < r_j-\epsilon_1$ for every
$j$, $1\leq j\leq m$.  By Lemma~\ref{lemma:QC:B:eps},
$Q \subseteq \bigcup_{j=1}^m B^d_{x_j, <r_j-\epsilon_2}$ for some
$\epsilon_2$ such that $0 < \epsilon_2 < r_1, \cdots, r_m$.  It then
suffices to let $\epsilon$ be $\min (\epsilon_1, \epsilon_2)$.  \qed

Let us write $\langle E \rangle$ for the \emph{order-convex hull} of
the set $E$.  This is defined as $\dc E \cap \upc E$.  When $E$ is
finite and non-empty, $\dc E$ is closed and intersects $\upc E$, so
$\langle E \rangle$ is a lens.  Therefore
$(\upc \langle E \rangle, cl (\langle E \rangle))$ is a quasi-lens.
We observe that $\upc \langle E \rangle = \upc E$, and that
$cl (\langle E \rangle) = \dc E$.  (Only the latter needs an argument:
$\dc E$ is the closure of $E$ since $E$ is finite, and as
$E \subseteq \langle E \rangle$,
$\dc E \subseteq cl (\langle E \rangle)$; in the reverse direction,
$\langle E \rangle \subseteq \dc E$, and since $\dc E$ is closed, it
must also contain the closure of $\langle E \rangle$.)  It follows
that $(\upc E, \dc E)$ is a quasi-lens.

\begin{thm}[Algebraicity of Plotkin powerdomains]
  \label{thm:P:alg}
  Let $X, d$ be an algebraic Yoneda-complete quasi-metric space, with
  strong basis $\mathcal B$.  The space of quasi-lenses on $X, d$,
  with the $\dP$ quasi-metric, is algebraic Yoneda-complete.

  Every quasi-lens of the form $(\upc E, \dc E)$ where $E$ is a finite
  non-empty set of center points is a center point in the space of
  quasi-lenses, and those such that $E \subseteq \mathcal B$ form a
  strong basis.
\end{thm}
\proof For every $\epsilon > 0$,
$B^{\dP^+}_{((\upc E, \dc E), 0), < \epsilon}$ is the set of formal
balls $((Q, C), r)$ such that $\dQ (\upc E, Q) < \epsilon-r$ and
$\dH (\dc E, C) < \epsilon-r$.  We show that it is Scott-open as
follows.  For every directed family ${((Q_i, C_i), r_i)}_{i \in I}$
with supremum $((Q, C), r)$, we have seen in
Proposition~\ref{prop:PX:Ycomplete} that $(Q, r)$ is the supremum of
${(Q_i, r_i)}_{i \in I}$.  Since $\dQ (\upc E, Q) < \epsilon-r$,
$(Q, r)$ is in $B^{\dQ^+}_{(\upc E, 0), <\epsilon}$, and that is
Scott-open because $\upc E$ is a center point, by
Lemma~\ref{lemma:Q:simple:center}.  Therefore $(Q_i, r_i)$ is in
$B^{\dQ^+}_{(\upc E, 0), <\epsilon}$ for some $i \in I$, namely,
$\dQ (\upc E, Q_i) < \epsilon - r_i$, or equivalently
$(\upc E, \epsilon) \leq^{\dQ^+} (Q_i, r_i)$.  Similarly, and using
Lemma~\ref{lemma:H:simple:center},
$(\dc E, \epsilon) \leq^{\dH^+} (Q_j, r_j)$ for some $j \in I$. By
directedness we can take the same value for $i$ and $j$, and it
follows that $\dP ((\upc E, \dc E), (Q_i, C_i)) < \epsilon - r_i$,
i.e., that $((Q_i, C_i), r_i)$ is in
$B^{\dP^+}_{((\upc E, \dc E), 0), < \epsilon}$.  This shows that
$(\upc E, \dc E)$ is a center point.

Let $(Q, C)$ be a quasi-lens.  We wish to show that $((Q, C), 0)$ is
the supremum of some family $D$ of formal balls of the form
$((\upc E, \dc E), r)$, with $E$ non-empty, finite, and included in
$\mathcal B$.  We define $D$ as the set of those formal balls of that
form such that $d (x, C) < r$ for every $x \in E$ and
$Q \subseteq \bigcup_{x \in E} B^d_{x, <r}$.  This is simply the
conjunction of the conditions we took in the proofs of
Theorem~\ref{thm:H:alg} and Theorem~\ref{thm:Q:alg}
As in the proof of Lemma~\ref{lemma:QC:approx}, since $x$ is a center
point for every $x \in E$, $d (x, C) < r$ if and only if $C$
intersects $B^d_{x, <r}$, by \cite[Proposition~6.12]{JGL:formalballs},
so $D$ is the set of formal balls $((\upc E, \dc E), r)$ with $E$
non-empty, finite, and included in $\mathcal B$, such that
$(Q, C) \in \Box_{\mathcal P} (\bigcup_{x \in E} B^d_{x, <r}) \cap
\bigcap_{x \in E} \Diamond_{\mathcal P} B^d_{x, <r}$.

By Lemma~\ref{lemma:QC:approx} with $U = X$, $m=0$, and $\epsilon > 0$
arbitrary, $D$ is non-empty.  We embark on showing that $D$ is
directed.  Let $((\upc E, \dc E), r)$ and $((\upc F, \dc F), s)$ be
two elements of $D$.  Let $U = \bigcup_{x \in E} B^d_{x, <r}$ and
$V = \bigcup_{y \in F} B^d_{y, <s}$.  Using the fact that
$\Box_{\mathcal P} U \cap \Box_{\mathcal P} V = \Box_{\mathcal P} (U
\cap V)$, $(Q, C)$ is in
$\Box_{\mathcal P} (U \cap V) \cap \bigcap_{x \in E}
\Diamond_{\mathcal P} B^d_{x, <r} \cap \bigcap_{y \in F}
\Diamond_{\mathcal P} B^d_{y, <s}$.  By Lemma~\ref{lemma:QC:B:eps},
there is a number $\epsilon$ such that $0 < \epsilon < \min (r, s)$
and such that $(Q, C)$ is in
$\mathcal U = \Box_{\mathcal P} (\bigcup_{x \in E} B^d_{x,
  <r-\epsilon} \cap \bigcup_{y \in F} B^d_{y, <s-\epsilon}) \cap
\bigcap_{x \in E} \Diamond_{\mathcal P} B^d_{x, <r-\epsilon} \cap
\bigcap_{y \in F} \Diamond_{\mathcal P} B^d_{y, <s-\epsilon}$.  By
Lemma~\ref{lemma:QC:approx}, there is a finite non-empty subset $G$ of
$\mathcal B$ and a number $t$ such that $0 < t \leq \epsilon$ and
$(Q, C) \in \Box_{\mathcal P} (\bigcup_{z \in G} B^d_{z, <t}) \cap
\bigcap_{z \in G} \Diamond_{\mathcal P} B^d_{z, <t} \subseteq \mathcal
U$.  In particular, $((\upc G, \dc G), t)$ is in $D$.

Let us verify that
$((\upc E, \dc E), r) \leq^{\dP^+} ((\upc G, \dc G), t)$.  The
quasi-lens $(\upc G, \dc G)$ is in
$\Box_{\mathcal P} (\bigcup_{z \in G} B^d_{z, <t}) \cap \bigcap_{z \in
  G} \Diamond_{\mathcal P} B^d_{z, <t}$, hence in $\mathcal U$.  By
definition of $\mathcal U$,
$\upc G \subseteq \bigcup_{x \in E} B^d_{x, <r-\epsilon}$, so for
every $z \in \upc G$, there is an $x \in E \subseteq \upc E$ such that
$d (x, z) < r-\epsilon$.  If follows that
$\dQ (\upc E, \upc G) < r-\epsilon \leq r-t$, by Lemma~\ref{lemma:dQ},
so $\dQ (\upc E, \upc G) \leq^{\dQ^+} r-t$.  By definition of
$\mathcal U$ and since $(\upc G, \dc G) \in \mathcal U$, $\dc G$
intersects $B^d_{x, <r-\epsilon}$ for every $x \in E$.  In other
words, for every $x \in E$, there is a $z \in \dc G$ such that
$d (x, z) < r-\epsilon$.  For every $x' \in \dc E$, there is a point
$x \in E$ such that $x' \leq x$, and we have just seen that
$\inf_{z \in \dc G} d (x, z) < r-\epsilon$.  Using
\cite[Proposition~6.12]{JGL:formalballs}, and since $x$ is a center
point, $d (x, \dc G) < r-\epsilon$, and then
$d (x', \dc G) \leq d (x', x) + d (x, \dc G)$ (using Lemma~6.11 of
\cite{JGL:formalballs}) $< r-\epsilon$.  Taking suprema over the
elements $x'$ of $\dc E$, we obtain that
$\dH (\dc E, \dc G) < r-\epsilon \leq r-t$.  Combining this with
$\dQ (\upc E, \upc G) < r-t$ through Lemma~\ref{lem:dP}, we obtain that
$\dP ((\upc E, \dc E), (\upc G, \dc G)) < r-t$, hence that
$((\upc E, \dc E), r) \leq^{\dP^+} ((\upc G, \dc G), t)$, as promised.

We verify that
$((\upc F, \dc F), s) \leq^{\dP^+} ((\upc G, \dc G), t)$ in a similar
fashion.  It follows that $D$ is directed.

We claim that $((Q, C), 0)$ is the supremum of $D$.

It is an upper bound: for every element $((\upc E, \dc E), r)$ of $D$,
$Q$ is included in $\bigcup_{x \in E} B^d_{x, <r}$, so
$\dQ (\upc E, Q) = \sup_{z \in Q} \inf_{x \in \upc E} d (x, z) \leq
\sup_{z \in Q} \inf_{x \in E} d (x, z) \leq r$; and for every
$x \in E$, $B^d_{x, <r}$ intersects $C$, so $d (x, C) < r$, so
$\dH (\dc E, C) = \sup_{x' \in \dc E} d (x', C) = \sup_{x \in E}
\sup_{x' \leq x} d (x', C) \leq \sup_{x \in E} \sup_{x' \leq x}
\underbrace {d (x', x)}_0 + d (x, C) \leq r$.  Therefore
$\dP ((\upc E, \dc E), (Q, C)) \leq r$.

If $((Q', C'), r')$ is another upper bound of $D$, then $r'=0$ since
$D$ contains elements with arbitrary small radii, by
Lemma~\ref{lemma:QC:approx}.  We claim that $Q \supseteq Q'$ and $C
\subseteq C'$.  This will imply that $\dQ (Q, Q')=0$ and $\dH (C,
C')=0$, so that $\dP ((Q, C), (Q', C'))$ by Lemma~\ref{lem:dP}, and
therefore $((Q, C), 0) \leq^{\dP^+} ((Q', C'), 0)$.  This will prove
that $((Q, C), 0)$ is the least upper bound of $D$, which will
terminate the proof.

Let $U$ be any open neighborhood of $Q$, and let $U'$ be any open set
that intersects $C$.  Then $(Q, C)$ is in
$\Box_{\mathcal P} U \cap \Diamond_{\mathcal P} U'$.  There is a
finite non-empty subset $E$ of $\mathcal B$ and a number $r > 0$ such
that
$(Q, C) \in \Box_{\mathcal P} (\bigcup_{x \in E} B^d_{x, <r}) \cap
\bigcap_{x \in E} \Diamond_{\mathcal P} B^d_{x, <r} \subseteq
\Box_{\mathcal P} U \cap \Diamond_{\mathcal P} U'$, by
Lemma~\ref{lemma:QC:approx}.  By Lemma~\ref{lemma:QC:B:eps}, there is
a number $\epsilon$ such that $0 < \epsilon < r$ and
$(Q, C) \in \Box_{\mathcal P} (\bigcup_{x \in E} B^d_{x, <r-\epsilon})
\cap \bigcap_{x \in E} \Diamond_{\mathcal P} B^d_{x, <r-\epsilon}$.
In particular, $((\upc E, \dc E), r-\epsilon)$ is in $D$.  Since
$((Q', C'), 0)$ is an upper bound of $D$, we have
$((\upc E, \dc E), r-\epsilon) \leq^{\dP^+} ((Q', C'), 0)$.  Hence
$\dQ (\upc E, Q') \leq r-\epsilon$ and
$\dH (\dc E, C') \leq r-\epsilon$, using Lemma~\ref{lem:dP}.  Using
Lemma~\ref{lemma:dQ}, the first of these inequalities implies that for
every $z \in Q'$, there is a point $x' \in \upc E$ such that
$d (x', z) \leq r-\epsilon$; in particular, there is a point $x \in E$
such that $x \leq x'$, so
$d (x, z) \leq d (x', z) \leq r-\epsilon < r$.  Therefore
$Q' \subseteq \bigcup_{x \in E} B^d_{x, <r}$.  From the second
inequality, $\dH (\dc E, C') \leq r-\epsilon$, we obtain that for
every $x \in E$, $d (x, C') \leq r-\epsilon < r$, so there is a point
$y \in C'$ such that $d (x, y) < r$, since $x$ is a center point.  In
other words, for every $x \in E$, $C'$ intersects $B^d_{x, <r}$.
Together with $Q' \subseteq \bigcup_{x \in E} B^d_{x, <r}$, this shows
that $(Q', C')$ is in
$\Box_{\mathcal P} (\bigcup_{x \in E} B^d_{x, <r}) \cap \bigcap_{x \in
  E} \Diamond_{\mathcal P} B^d_{x, <r}$, hence in
$\Box_{\mathcal P} U \cap \Diamond_{\mathcal P} U'$.  Hence $Q'$ is
included in $U$ and $C'$ intersects $U'$.  Since that holds for every
open neighborhood $U$ of $Q$ and for every open set $U'$ that
intersects $C$, we obtain that $Q \supseteq Q'$ and $C \subseteq C'$,
as promised. \qed



As for the Hoare and Smyth powerdomains, we reduce the study of
continuity to algebraicity.
\begin{lem}
  \label{lemma:P:functor}
  Let $X, d$ and $Y, \partial$ be two continuous Yoneda-complete
  quasi-metric spaces, and $f \colon X, d \mapsto Y, \partial$ be a
  $1$-Lipschitz continuous map.  The map $\Plotkin f$ defined by
  $\Plotkin f (Q, C) = (\Smyth f (Q), \Hoare f (C)) = (\upc f [Q], cl
  (f [C]))$ is a $1$-Lipschitz continuous map from the space of
  quasi-lenses on $X$ with the $\dP$ quasi-metric to the space of
  quasi-lenses on $Y$ with the $\mP\partial$ quasi-metric.
\end{lem}
\proof We need to check that $\Plotkin f (Q, C)$ is a quasi-lens.  Let
$Q' = \upc f [Q]$, $C' = cl (f [C])$.  We rely on
Proposition~\ref{prop:fork=lens}: we know that $(F_Q, F^C)$ satisfies
Walley's condition, and we show that $(Q', C')$ is a quasi-lens by
showing that $(F_{Q'}, F^{C'})$ satisfies Walley's condition.  By
Lemma~\ref{lemma:H:functor} and Lemma~\ref{lemma:Q:functor},
$F^{C'} = \Prev f (F^C)$ and $F_{Q'} = \Prev f (F_Q)$.  Now
$F_{Q'} (k+k') = \Prev f (F_Q) (k+k') = F_Q ((k+k') \circ f) \leq F_Q
(k \circ f) + F^C (k' \circ f)$ (by Walley's condition, left
inequality, on $(F_Q, F^C)$)
$= \Prev f (F_Q) (k) + \Prev f (F^C) (k) = F_{Q'} (k) + F^{C'} (k)$.
The other part of Walley's condition is proved similarly.

$\Plotkin f$ is $1$-Lipschitz, because of Lemma~\ref{lem:dP} and the
fact that $\Smyth f$ and $\Hoare f$ are $1$-Lipschitz (special cases
of Lemma~\ref{lemma:H:functor} and Lemma~\ref{lemma:Q:functor}).
Given a directed family of formal balls
${((Q_i, C_i), r_i)}_{i \in I}$ on quasi-lenses, we have seen in
Proposition~\ref{prop:PX:Ycomplete} that its supremum $((Q, C), r)$ is
characterized by $r = \inf_{i \in I} r_i$, $(Q, r)$ is the supremum of
${(Q_i, r_i)}_{i \in I}$ in $\mathbf B (\Smyth X, \dQ)$ and $(C, r)$
is the supremum of ${(C_i, r_i)}_{i \in I}$ in
$\mathbf B (\Hoare X, \dH)$.  Since $\Smyth f$ and $\Hoarez f$ are
$1$-Lipschitz continuous (Lemma~\ref{lemma:H:functor} and
Lemma~\ref{lemma:Q:functor}), $(\Smyth f (Q), r)$ is the supremum of
${(\Smyth f (Q_i), r_i)}_{i \in I}$ and $(\Hoarez f (C), r)$ is the
supremum of ${(\Hoarez f (C_i), r_i)}_{i \in I}$.  By
Proposition~\ref{prop:PX:Ycomplete}, the supremum of
${(\Plotkin f (Q_i, C_i), r_i)}_{i \in I}$ is
$((\Smyth f (Q), \Hoarez f (C)), r) = (\Plotkin f (Q, C), r)$.
Therefore $\Plotkin f$ is $1$-Lipschitz continuous.  \qed

Let $X, d$ be a continuous Yoneda-complete quasi-metric space.  There
is an algebraic Yoneda-complete quasi-metric space $Y, \partial$ and
two $1$-Lipschitz continuous maps $r \colon Y, \partial \to X, d$ and
$s \colon X, d \to Y, \partial$ such that $r \circ s = \identity X$
\cite[Theorem~7.9]{JGL:formalballs}.

By Lemma~\ref{lemma:P:functor}, $\Plotkin r$ and $\Plotkin s$ are also
$1$-Lipschitz continuous, and clearly
$\Plotkin r \circ \Plotkin s = \identity \relax$, so the space of
quasi-lenses on $X$ with the $\dP$ quasi-metric is a $1$-Lipschitz
continuous retract of that on $Y$, with the $\mP\partial$
quasi-metric.  Theorem~\ref{thm:P:alg} states that the latter is
algebraic Yoneda-complete, whence:
\begin{thm}[Continuity for the Plotkin powerdomain]
  \label{thm:P:cont}
  Let $X, d$ be a continuous Yoneda-complete quasi-metric space.  The
  space of quasi-lenses on $X$ with the $\dP$ quasi-metric is
  continuous Yoneda-complete.  \qed
\end{thm}

\subsection{The Vietoris Topology}
\label{sec:vietoris-topology}

Recall that the Vietoris topology on the space of quasi-lenses on $X$
is generated by the subbasic open sets
$\Box U = \{(Q, C) \mid Q \subseteq U\}$ and
$\Diamond U = \{(Q, C) \mid C \cap U \neq \emptyset\}$, where $U$
ranges over the open subsets of $X$.  This generalizes the usual
Vietoris topology on spaces of lenses.

Recall also that the \emph{weak topology} on the space of forks is the
topology induced by the inclusion in the product of the spaces of
sublinear and superlinear previsions, each equipped with their weak
topology.  In other words, it has subbasic open sets
$[h > a]_- = \{(F^-, F^+) \mid F^- (h) > a\}$ and
$[h > a]_+ = \{(F^-, F^+) \mid F^+ (h) > a\}$, when $h$ ranges over
$\Lform X$ and $a \in \Rp$.

\begin{lem}
  \label{lemma:P:V=weak}
  Let $X, d$ be a sober, standard quasi-metric space.  The map $(Q, C)
  \mapsto (F_Q, F^C)$ is a homeomorphism of the space of quasi-lenses
  on $X$ with the Vietoris topology and the space of discrete
  normalized forks on $X$ with the weak topology.
\end{lem}
\proof This is a bijection by Proposition~\ref{prop:fork=lens}.  The
homeomorphism part is an easy consequence of
Lemma~\ref{lemma:H:V=weak} and Lemma~\ref{lemma:Q:V=weak}.  \qed

\begin{prop}
  \label{prop:P:Vietoris}
  Let $X, d$ be an algebraic Yoneda-complete quasi-metric space.  The
  $\dP$-Scott topology and the Vietoris topology coincide on the space
  of quasi-lenses on $X$.
\end{prop}
\proof
The Vietoris open set $\Box_{\mathcal P} U$, $U$ open in $X$, is the
inverse image of $\Box U$ by the first projection
$\pi_1 \colon (Q, C) \mapsto Q$.  Let us check that $\pi_1$ is
continuous with respect to the $\dP$-Scott and $\dQ$-Scott topologies.
First,
$\dQ (Q, Q') \leq \dP ((Q, C), (Q', C')) = \max (\dQ (Q, Q'), \dH (C,
C'))$, so $\pi_1$ is $1$-Lipschitz.  For every directed family
${((Q_i, C_i), r_i)}_{i \in I}$, its supremum $(Q, r)$ is given by
Proposition~\ref{prop:PX:Ycomplete}, and we have seen that $(Q, r)$
was then the (naive) supremum of ${(Q_i, r_i)}_{i \in I}$.  This shows
our continuity claim on $\pi_1$.  Since $\Box U$ is open not only in
the upper Vietoris topology, but also in the $\dQ$-Scott topology by
Theorem~\ref{thm:Q:Vietoris}, $\Box_{\mathcal P} U$ is open in the
$\dP$-Scott topology.

We reason similarly with the second projection, showing that it is
continuous with respect to the $\dP$-Scott and $\dH$-Scott topologies.
Using Theorem~\ref{thm:H:Vietoris}, $\Diamond_{\mathcal P} U$ is open
not only in the Vietoris topology, but also in the $\dP$-Scott
topology.  This shows that the Vietoris topology is coarser than the
$\dP$-Scott topology.

Conversely, a base of the $\dP$-Scott topology is given by the open
balls $B^{\dP}_{((\upc E, \dc E), 0), <\epsilon}$, where
$(\upc E, \dc E)$ ranges over the center points of the space of
quasi-lenses, as given in Theorem~\ref{thm:P:alg}.  That open ball is
the set of quasi-lenses $(Q, C)$ such that
$\dP ((\upc E, \dc E), (Q, C)) < \epsilon$, equivalently such that
$Q \in B^{\dQ}_{\upc E, <\epsilon}$ and
$C \in B^{\dH}_{\dc E, \epsilon}$.  Those are open balls centered at
center points (by Theorem~\ref{thm:Q:alg} and
Theorem~\ref{thm:H:alg}), hence are open in the $\dQ$-Scott and
$\dH$-Scott topologies, respectively.  By
Proposition~\ref{prop:Q:Vietoris} and
Proposition~\ref{prop:H:Vietoris}, they are open in the upper Vietoris
topology on $\Smyth X$ and in the lower Vietoris topology on
$\Hoare X$ respectively.  Then
$B^{\dP}_{((\upc E, \dc E), 0), <\epsilon}$ is the intersection of
$\pi_1^{-1} (B^{\dQ}_{\upc E, <\epsilon})$ and of
$\pi_2^{-1} (B^{\dH}_{\dc E, \epsilon})$, which are both open in the
Vietoris topology.  Indeed, $\pi_1$ and $\pi_2$ are continuous with
respect with the Vietoris topologies, since
$\pi_1^{-1} (\Box U) = \Box_{\mathcal P} U$ and
$\pi_1^{-1} (\Diamond U) = \Diamond_{\mathcal P} U$.  It follows that
$B^{\dP}_{((\upc E, \dc E), 0), <\epsilon}$ is open in the Vietoris
topology, whence the Vietoris topology is also finer than the
$\dP$-Scott topology.  \qed

\begin{thm}[$\dP$ quasi-metrizes the Vietoris topology]
  \label{thm:P:Vietoris}
  Let $X, d$ be a continuous Yoneda-complete quasi-metric space.  The
  $\dP$-Scott topology coincides with the Vietoris topology on the
  space of quasi-lenses on $X$.
\end{thm}
\proof This is as for Theorem~\ref{thm:V:weak=dScott},
Theorem~\ref{thm:H:Vietoris}, or Theorem~\ref{thm:Q:Vietoris}.  By
\cite[Theorem~7.9]{JGL:formalballs}, $X, d$ is the $1$-Lipschitz
continuous retract of an algebraic Yoneda-complete quasi-metric space
$Y, \partial$.  Call $s \colon X \to Y$ the section and
$r \colon Y \to X$ the retraction.  Using Corollary~\ref{corl:P:prev},
we confuse quasi-lenses with discrete normalized forks.  Then
$\Prev s$ and $\Prev r$ form a $1$-Lipschitz continuous
section-retraction pair by Lemma~\ref{lemma:P:functor}, and in particular
$\Prev s$ is an embedding of the space of quasi-lenses on $X$ into the
space of quasi-lenses on $Y$, with their $\dP$-Scott topologies.
However, $s$ and $r$ are also just continuous, by
Proposition~\ref{prop:cont}, so $\Prev s$ and $\Prev r$ also form a
section-retraction pair between the same spaces, this time with their
weak topologies (as spaces of previsions), by Fact~\ref{fact:Pf:weak},
that is, with their Vietoris topologies, by
Lemma~\ref{lemma:P:V=weak}.  By Proposition~\ref{prop:P:Vietoris}, the
two topologies on the space of quasi-lenses on $Y$ are the same.
Fact~\ref{fact:retract:two} then implies that the two topologies on
the space of quasi-lenses on $X$ are the same as well.  \qed

\subsection{Forks}
\label{sec:forks}

Just as one would expect by analogy with spaces of sublinear and
superlinear previsions, the space of (sub)normalized forks on a space
$X$ arises as a retract of the Plotkin powerdomain on its space of
(sub)normalized linear previsions.

We will equate continuous valuations with linear previsions, as usual,
and we will therefore write $\Val X$ for the space of linear
previsions on $X$, $\Val_{\leq 1} X$ for the space of subnormalized
linear previsions on $X$, $\Val_1 X$ for the space of normalized
linear previsions on $X$, both with the weak topology.  We will in
general use the notation $\Val_\bullet X$, where $\bullet$ can be
nothing, ``$\leq 1$'', or ``$1$''.

For every quasi-lens $(\mathcal Q, \mathcal C)$ on $\Val_\bullet X$,
let $r_{\ADN} (\mathcal Q, \mathcal C)$ denote the pair
$(r_{\DN} (\mathcal Q), r_{\AN} (\mathcal C))$.  Conversely, for every
(sub)normalized fork $(F^-, F^+)$ on $X$, let
$s_{\ADN}^\bullet (F^-, F^+)$ be
$(s_\DN^\bullet (F^-), s_\AN^\bullet (F^+))$.  Those are close cousins
of the eponymous maps of \cite[Definition~3.23]{JGL-mscs16}.

Following \cite[Definition~4.6]{Keimel:topcones2}, we say that
addition is \emph{almost open} on $\Lform X$ if and only if, for any
two open subsets $U$ and $V$ of $\Lform X$, the upward closure
$\upc (U + V)$ is open.  This happens notably when $X$ is core-compact
and core-coherent \cite[Lemma~3.24]{JGL-mscs16}; for a definition of
core-coherence, see Definition~5.2.18 of \cite{JGL-topology}; every
locally compact, coherent space is core-compact and core-coherent,
where coherence means that the intersection of two compact saturated
sets is compact \cite[Lemma~5.2.24]{JGL-topology}.

The following is a variant of Proposition~3.32 and Proposition~4.8 of
\cite{JGL-mscs16}, using quasi-lenses instead of lenses.  We say that
a quasi-lens $(Q, C)$ is \emph{convex} if and only if both $Q$ and $C$
are.

\begin{customprop}{\thesubsection.A}
  \label{prop:rsADN}
  Let $X$ be a space such that $\Lform X$ is locally convex and has an
  almost open addition map, and let $\bullet$ be nothing,
  ``$\leq 1$'', or ``$1$''.  Additionally, if $\bullet$ is ``$1$'', we
  assume that $X$ is compact.

  The map $r_\ADN$ forms a retraction of the space of quasi-lenses on
  $\Val_\bullet X$ with the Vietoris topology onto the space of forks
  (resp., of subnormalized forks if $\bullet$ is ``$\leq 1$'', of
  normalized forks if $\bullet$ is ``$1$'') with the weak topology,
  with associated section $s_\ADN^\bullet$.

  This retraction restricts to a homeomorphism between the subspace of
  convex quasi-lenses on $\Val_\bullet X$ and the space of forks
  (resp., of subnormalized forks if $\bullet$ is ``$\leq 1$'', of
  normalized forks if $\bullet$ is ``$1$'') with the weak topology.
\end{customprop}
\proof We already know that $r_\DN$ forms a retraction of
$\Smyth (\Val_\bullet X)$ onto the space of superlinear previsions on
$X$ (with the required (sub)normalization requirement, depending on
$\bullet$) with the weak topology, and that $s_\DN^\bullet$ is the
associated section.  We also know that $r_\AN$ forms a retraction of
$\Hoare (\Val_\bullet X)$ onto the space of sublinear previsions on
$X$ (again with the required (sub)normalization requirement) with the
weak topology, and that $s_\AN^\bullet$ is the associated section.

\emph{The map $r_\ADN$ takes its values in a space of forks.}
For every quasi-lens $(\mathcal Q, \mathcal C)$ on $\Val_\bullet X$,
we need to check that $r_\ADN (\mathcal Q, \mathcal C)$ is a fork, and
the only thing that remains to be checked is Walley's conditions.

For every $h \in \Lform X$, let
$h^\perp \colon \Val_\bullet X \to \creal$ map $G$ to $G (h)$.  This
is a lower semicontinuous map, since
${h^\perp}^{-1} (]a, +\infty]) = [h > a]$ for every $a \in \Rp$.
Also, for all $h, h' \in \Lform X$,
$(h+h')^\perp = h^\perp + {h'}^\perp$, because every $G$ in
$\Val_\bullet X$ is linear.  We note that, for every quasi-lens
$(\mathcal Q, \mathcal C)$ on $\Val_\bullet X$, for every
$h \in \Lform X$,
$r_\DN (\mathcal Q) (h) = \inf_{G \in \mathcal Q} G (h) = \inf_{G \in
  \mathcal Q} h^\perp (G) = F_{\mathcal Q} (h^\perp)$, and similarly
$r_\AN (\mathcal C) (h) = F^{\mathcal C} (h^\perp)$.  By
Proposition~\ref{prop:fork=lens}, $(F_{\mathcal Q}, F^{\mathcal C})$
is a (discrete) normalized fork, so for all $h, h' \in \Lform X$,
\begin{align*}
  r_\DN (\mathcal Q) (h+h')
  & = F_{\mathcal Q} ((h+h')^\perp) = F_{\mathcal Q} (h^\perp +
    {h'}^\perp) \\
  & \leq F_{\mathcal Q} (h^\perp) + F^{\mathcal C} ({h'}^\perp)
  & \text{(Walley's condition)} \\
  & = r_\DN (\mathcal Q) (h) + r_\AN (\mathcal C) (h'),
\end{align*}
and
\begin{align*}
  r_\DN (\mathcal Q) (h) + r_\AN (\mathcal C) (h')
  & = F_{\mathcal Q} (h^\perp) + F^{\mathcal C} ({h'}^\perp) \\
  & \leq F^{\mathcal C} (h^\perp + {h'}^\perp)
  & \text{(Walley's condition)} \\
  & = F^{\mathcal C} ((h+h')^\perp) = r_\AN (h+h').
\end{align*}

\emph{The map $r_\ADN$ is continuous.}  We recall that $r_\DN$ and
$r_\AN$ are both continuous, being part of a retraction-section pair.
Since the weak topology on spaces of forks is induced by the product
topology on the product of the spaces of superlinear and sublinear
previsions, it suffices to show that the maps
$(\mathcal Q, \mathcal C) \mapsto r_\DN (\mathcal Q)$ and
$(\mathcal Q, \mathcal C) \mapsto r_\AN (\mathcal C)$ are continuous
in order to establish that $r_\ADN$ is continuous.  In turn, this
follows from the fact that the projection
maps$(\mathcal Q, \mathcal C) \mapsto \mathcal Q$ and
$(\mathcal Q, \mathcal C) \mapsto \mathcal C$ are continuous, which is
clear since the inverse image of any basic open set $\Box U$ by the
first map is $\Box_{\mathcal P} U$ and the inverse image of any
subbasic open set $\Diamond U$ by the second map is
$\Diamond_{\mathcal P} U$.

\emph{The map $s_\ADN^\bullet$ takes its values in the space of
  quasi-lenses.}  Let $(F^-, F^+)$ be an arbitrary fork on $X$
(resp. subnormalized, or normalized, depending on $\bullet$), and let
$(\mathcal Q, \mathcal C) = s_\ADN^\bullet (F^-, F^+)$.  We already
know that $\mathcal Q$ is a non-empty compact saturated subset of
$\Val_\bullet X$, and that $\mathcal C$ is a non-empty closed subset
of $\Val_\bullet X$.

We use Lemma~3.28 of \cite{JGL-mscs16}, which says that, for every
$G' \in \Val_\bullet X$ such that $F^- \leq G'$, there is a
$G \in \Val_\bullet X$ such that $F^- \leq G \leq F^+$ and
$G \leq G'$.  (This requires no condition on the space $X$.)
By definition of $\mathcal Q$ and $\mathcal C$, this can be rephrased
as: for every $G' \in \mathcal Q$, there is an element $G \in \mathcal
Q \cap \mathcal C$ such that $G \leq G'$.  In other words, $\mathcal Q
\subseteq \upc (\mathcal Q \cap \mathcal C)$.

Lemma~3.29 of \cite{JGL-mscs16} states that, given that $\Lform X$ is
locally convex and has an almost open addition map, and that $X$ is
compact if $\bullet$ is ``$1$'', then for every
$G' \in \Val_\bullet X$ such that $G' \leq F^+$, there is a
$G \in \Val_\bullet X$ such that $F^- \leq G \leq F^+$ and
$G' \leq G$.  This means that for every $G' \in \mathcal C$, $G'$ is
in $\dc (\mathcal Q \cap \mathcal C)$.  Therefore
$\mathcal C \subseteq \dc (\mathcal Q \cap \mathcal C) \subseteq cl
(\mathcal Q \cap \mathcal C)$.  In particular, for every open
neighborhood $\mathcal U$ of $\mathcal Q$, $\mathcal C$ is included in
$cl (\mathcal Q \cap \mathcal U)$, so $(\mathcal Q, \mathcal C)$ is a
quasi-lens.

\emph{The map $s_\ADN^\bullet$ is continuous.}
This follows from the fact that $s_\DN^\bullet$ and $s_\AN^\bullet$ are continuous.

\emph{$r_\ADN \circ s_\ADN^\bullet$ is the identity map.}  This
follows since $r_\DN \circ s_\DN^\bullet$ and
$r_\AN \circ s_\AN^\bullet$ are both identity maps.

\emph{$s_\ADN^\bullet \circ r_\ADN$ is the identity map on the space
  of convex quasi-lenses.}
It suffices to observes that $s_\DN^\bullet \circ r_\DN$ is the
identity map on the space of convex non-empty compact saturated
subsets of $X$, and that $s_\AN^\bullet \circ r_\AN$ is the
identity map on the space of convex non-empty closed
subsets of $X$.  Those are Propositions~4.5 and~4.3 of
\cite{JGL-mscs16}, respectively; the second one requires that $\Lform$
be locally convex, while the first one requires nothing from $X$.
\qed

\begin{customrem}{\thesubsection.B}
  \label{rem:lens:qlens}
  Under the assumption that $\Lform X$ is locally convex and has an
  almost open addition map (and that $X$ is compact in the case where
  $\bullet$ is ``$1$''), then the composition of the homeomorphism of
  Proposition~\ref{prop:rsADN} with the homeomorphism of
  \cite[Proposition~4.8]{JGL-mscs16} yields a homeomorphism from the
  space of convex lenses to the space of convex quasi-lenses on
  $\Val_\bullet X$.
\end{customrem}

\begin{thm}[Continuity for forks]
  \label{thm:ADN:cont}
  Let $X, d$ be a continuous Yoneda-complete quasi-metric space,
  $a > 0$, and $\bullet$ be ``$\leq 1$'' or ``$1$''.  Let us also
  assume that $\Lform X$ has an almost open addition map, and that $X$
  is compact in the case where $\bullet$ is ``$1$''.

  The space of all subnormalized (if $\bullet$ is ``$\leq 1$'',
  normalized if $\bullet$ is ``$1$'') forks on $X$ with the $\dKRH^a$
  quasi-metric is continuous Yoneda-complete.

  It arises as a $1$-Lipschitz continuous retract of the Plotkin
  powerdomain over $\Val_\bullet X$ through $r_\ADN$ and
  $s_\ADN^\bullet$.  That retraction cuts down to an isometry between
  the space of (sub)normalized forks on $X$ and the space of convex
  quasi-lenses over $\Val_\bullet X$, with the $\mP {(\dKRH^a)}$
  quasi-metric.

  The supremum of a directed family of formal balls
  ${((F^-_i, F^+_i), r_i)}_{i \in I}$ is $((F^-, F^+), r)$ where
  $r = \inf_{i \in I} r_i$, $F^+$ is the naive supremum of
  ${(F^+_i, r_i)}_{i \in I}$, and $F^-$ is the naive supremum of
  ${(F^-_i, r_i)}_{i \in I}$.  In particular, $(F^+, r)$ is the
  supremum of ${(F^+_i, r_i)}_{i \in I}$ in the space of
  (sub)normalized sublinear previsions and $(F^-, r)$ is the supremum
  of ${(F^-_i, r_i)}_{i \in I}$ in the space of (sub)normalized
  superlinear previsions.
\end{thm}
\proof We recall that, since $X, d$ is continuous Yoneda-complete,
$\Lform X$ is automatically locally convex
(Proposition~\ref{prop:locconvex}).  We use Theorem~\ref{thm:AN:cont}
and Lemma~\ref{lemma:AN:cont}, resp.\ Theorem~\ref{thm:DN:cont}, which
state analogous results for $\Hoare X, \dKRH^a$ and
$\Smyth X, \dKRH^a$, without any further reference.

Given a directed family of formal balls
${((F^-_i, F^+_i), r_i)}_{i \in I}$, ${(F^-_i, r_i)}_{i \in I}$ is a
directed family of formal balls on the space of (sub)normalized
superlinear previsions on $X$ with the $\dKRH^a$ quasi-metric, and
${(F^+_i, r_i)}_{i \in I}$ is a directed family of formal balls on the
space of (sub)normalized sublinear previsions on $X$ with the
$\dKRH^a$ quasi-metric.  Let $(F^-, r)$ be the (naive) supremum of the
former and $(F^+, r)$ be the (naive) supremum of the latter.
Lemma~\ref{lemma:PX:naivesup} states that $(F^-, F^+)$ is a fork,
provided we check that $X$ is sober.  That is guaranteed by
\cite[Proposition~4.1]{JGL:formalballs}, which states that every
continuous Yoneda-complete quasi-metric space is sober.  One easily
checks that $(F^-, F^+)$ is the supremum of
${((F^-_i, F^+_i), r_i)}_{i \in I}$.

That characterization of directed suprema, together with the analogous
characterization of directed suprema of formal balls on the Plotkin
powerdomain (Proposition~\ref{prop:PX:Ycomplete}) and the fact that
$r_\AN$, $r_\DN$, $s_\AN$, $s_\DN$ are $1$-Lipschitz continuous, shows
that $r_\ADN$ and $s_\ADN$ are $1$-Lipschitz continuous.  Note that we
apply Proposition~\ref{prop:PX:Ycomplete} to the Plotkin powerdomain
over $\Val_\bullet X, \dKRH^a$, and that is legitimate since
$\Val_\bullet X, \dKRH^a$ is continuous Yoneda-complete, by
Theorem~\ref{thm:Vleq1:cont} or Theorem~\ref{thm:V1:cont}.

We deal with the case of $r_\ADN$ to give the idea of the proof.
Let ${((Q_i, C_i), r_i)}_{i \in I}$ be a directed family of formal
balls on the Plotkin powerdomain of $\Val_\bullet X$, with supremum
$((Q, C), r)$.  In particular, $r = \inf_{i \in I} r_i$, $(Q, r)$ is
the supremum of the directed family ${(Q_i, r_i)}_{i \in I}$ and $(C,
r)$ is the supremum of the directed family ${(C_i, r_i)}_{i \in I}$.
Since $r_\DN$ and $r_\AN$ are $1$-Lipschitz continuous,
$(r_\DN (Q), r)$ is the supremum of ${(r_\DN (Q_i), r_i)}_{i \in I}$ and
$(r_\AN (C), r)$ is the supremum of ${(r_\AN (C_i), r_i)}_{i \in I}$.
We have just seen that this implies that $(r_\ADN (Q, C), r) = ((r_\DN
(Q), r_\AN (C)), r)$ is the supremum of ${((r_\DN (Q_i), r_\AN (C_i)),
  r_i)}_{i \in I}$, that is, of ${(r_\ADN (Q_i, C_i), r_i)}_{i \in I}$.

We have
$r_\ADN (s_\ADN (F^-, F^+))= (r_\DN (s_\DN (F^-)), r_\AN (s_\AN
(F^+))) = (F^-, F^+)$, so $r_\ADN$ and $s_\ADN$ form a $1$-Lipschitz
continuous retraction.  That cuts down to an isometry when we restrict
ourselves to convex quasi-lenses, because $r_\DN$, $s_\DN$ and
$r_\AN$, $s_\AN$ form isometries when restricted to convex non-empty
compact saturated sets and convex non-empty closed sets respectively,
by the final part of Proposition~\ref{prop:rsADN}

Since $\Val_\bullet X, \dKRH^a$ is continuous Yoneda-complete,
Theorem~\ref{thm:P:cont} tells us that the Plotkin powerdomain over
$\Val_\bullet X$, with the $\mP{(\dKRH^a)}$ quasi-metric, is
continuous Yoneda-complete.  Every $1$-Lipschitz continuous retract of
a continuous Yoneda-complete quasi-metric space is continuous
Yoneda-complete \cite[Theorem~7.1]{JGL:formalballs}, therefore the
space of (sub)normalized forks over $X$ with the $\dKRH^a$
quasi-metric is continuous Yoneda-complete.  \qed

\begin{thm}[Algebraicity for forks]
  \label{thm:ADN:alg}
  Let $X, d$ be an algebraic Yoneda-complete quasi-metric space, with
  a strong basis $\mathcal B$.  Let $a > 0$, and $\bullet$ be
  ``$\leq 1$'' or ``$1$''.  Let us also assume that $\Lform X$ has an
  almost open addition map, and that $X$ is compact in the case where
  $\bullet$ is ``$1$''.

  The space of all subnormalized (resp., normalized) forks on $X$ with
  the $\dKRH^a$ quasi-metric is algebraic Yoneda-complete.

  All the forks of the form $(F^-, F^+)$ where
  $F^- (h) = \min_{i=1}^m \sum_{j=1}^{n_i} a_{ij} h (x_{ij})$ and
  $F^+ (h) = \max_{i=1}^m \sum_{j=1}^{n_i} a_{ij} h (x_{ij})$, with
  $m \geq 1$, $\sum_{j=1}^{n_i} a_{ij} \leq 1$ (resp., $=1$) for every
  $i$, and where each $x_{ij}$ is a center point, are center points,
  and they form a strong basis, even when each $x_{ij}$ is constrained
  to lie in $\mathcal B$.
\end{thm}
\proof We start with the second statement.  Let $(F_0^-, F_0^+)$ where
$F_0^- (h) = \min_{i=1}^m \sum_{j=1}^{n_i} a_{ij} h (x_{ij})$ and
$F_0^+ (h) = \max_{i=1}^m \sum_{j=1}^{n_i} a_{ij} h (x_{ij})$, with
$m \geq 1$, $\sum_{j=1}^{n_i} a_{ij} \leq 1$ (resp., $=1$) for every
$i$, and where each $x_{ij}$ is a center point.  In order to show that
$(F_0^-, F_0^+)$ is a center point, we consider the open ball
$B^{\dKRH^{a+}}_{((F_0^-, F_0^+), 0) <\epsilon}$, and we claim that it
is Scott-open.  For every directed family of formal balls
${((F^-_i, F^+_i), r_i)}_{i \in I}$, with supremum $((F^-, F^+), r)$
inside $B^{\dKRH^{a+}}_{((F_0^-, F_0^+), 0) <\epsilon}$, we have
$r = \inf_{i \in I} r_i$, $\dKRH^a (F_0^-, F^-) < \epsilon-r$, and
$\dKRH^a (F_0^+, F^+) < \epsilon-r$.  Also, by
Theorem~\ref{thm:ADN:cont}, $(F_0^-, r)$ is the supremum of the
directed family ${(F^-_i, r_i)}_{i \in I}$.  Since $F_0^-$ is a center
point, $B^{\dKRH^{a+}}_{(F_0^-, 0), < \epsilon}$ is Scott-open, and
contains $(F_0^-, r)$, hence contains some $(F^-_i, r_i)$; similarly,
$B^{\dKRH^{a+}}_{(F_0^+, r), < \epsilon-r}$ contains some
$(F^+, r_i)$, and we can take the same $i$, by directedness.  It
follows that $\dKRH^a (F_0^-, F^-_i) < \epsilon-r_i$ and
$\dKRH^a (F_0^+, F_i^+) < \epsilon-r_i$.  Hence
$((F_i^-, F_i^+), r_i)$ is in
$B^{\dKRH^{a+}}_{((F_0^-, F_0^+), 0) <\epsilon}$.

Finally, we show that the forks $(F_0^-, F_0^+)$ as above, where
$x_{ij} \in \mathcal B$, form a strong basis.  By
Theorem~\ref{thm:V:alg} and Theorem~\ref{thm:V1:alg}, the
(sub)normalized simple valuations $\sum_{i=1}^n a_i \delta_{x_i}$
where $x_i \in \mathcal B$ form a strong basis of $\Val_\bullet X$.
Hence the forks of the form $(\upc E, \dc E)$, where $E$ is a finite
set of such simple valuations, form a strong basis of the Plotkin
powerdomain over $\Val_\bullet X$, by Theorem~\ref{thm:P:alg}.  By
Lemma~\ref{lemma:retract:alg}, every (sub)normalized fork is a
$\dKRH^a$-limit of a Cauchy-weightable net of points of the form
$r_\ADN (\upc E, \dc E)$, and this is exactly what we need to
conclude.  \qed

\begin{lem}
  \label{lemma:ADN:functor}
  Let $X, d$ and $Y, \partial$ be two continuous Yoneda-complete
  quasi-metric spaces, and $f \colon X, d \mapsto Y, \partial$ be a
  $1$-Lipschitz continuous map.  Let $a > 0$.  Let us also assume that
  $\Lform X$ and $\Lform Y$ have almost open addition maps, and that
  $X$ and $Y$ are compact in the case where $\bullet$ is ``$1$''.  The
  map $(F^-, F^+) \mapsto (\Prev f (F^-), \Prev f (F^+))$ is
  $1$-Lipschitz continuous from the space of normalized, resp.\
  subnormalized forks on $X$ to the space of normalized, resp.\
  subnormalized forks on $Y$, with the $\dKRH^a$ and $\KRH\partial^a$
  quasi-metrics.
\end{lem}
\proof That $(\Prev f (F^-), \Prev f (F^+))$ satisfies Walley's
condition is immediate.  Therefore it is a fork, which is normalized
or subnormalized, depending on $(F^-, F^+)$.

By definition of $\dKRH^a$ as a maximum
(Definition~\ref{defn:KRH:fork}), and since $\Prev f$ is $1$-Lipschitz
by Lemma~\ref{lemma:Pf:lip}, the map
$(F^-, F^+) \mapsto (\Prev f (F^-), \Prev f (F^+))$ is $1$-Lipschitz.
For every directed family ${((F_i^-, F_i^+), r_i)}_{i \in I}$ of
formal balls on (sub)normalized forks, let $((F^-, F^+), r)$ be its
supremum, as given in the last part of Theorem~\ref{thm:ADN:cont}:
$(F^-, r)$ is the supremum of the directed family
${(F_i^-, r_i)}_{i \in I}$ in the space of formal balls over the space
of (sub)normalized superlinear previsions, and $(F^+, r)$ is the
supremum of ${(F_i^+, r_i)}_{i \in I}$.  By
Lemma~\ref{lemma:DN:functor}, $(\Prev f (F^-), r)$ is the supremum of
${(\Prev f (F_i^-), r_i)}_{i \in I}$, and by
Lemma~\ref{lemma:AN:functor}, $(\Prev f (F^+), r)$ is the supremum of
${(\Prev f (F_i^+), r_i)}_{i \in I}$, so
$((\Prev f (F^-), \Prev f (F^+)), r)$ is the supremum of
${((\Prev f (F_i^-), \Prev f (F_i^+)), r_i)}_{i \in I}$, by the last
part of Theorem~\ref{thm:ADN:cont} again.  \qed

With Theorem~\ref{thm:ADN:alg} and Theorem~\ref{thm:ADN:cont}, we obtain
the following.
\begin{cor}
  \label{cor:ADN:functor}
  There is a functor from the category of continuous (resp.,
  algebraic) Yoneda-complete quasi-metric spaces $X$ such that $\Lform
  X$ has an almost open addition map, and $1$-Lipschitz
  continuous maps, to the category of continuous (resp.,
  algebraic) Yoneda-complete quasi-metric spaces,
  which sends every object $X, d$ to the space of
  subnormalized forks on $X$ with the $\dKRH^a$-Scott
  quasi-metric ($a > 0$), and every $1$-Lipschitz continuous map $f$ to
  the map $(F^-, F^+) \mapsto (\Prev f (F^-), \Prev f (F^+))$.

  There is a functor from the category of compact continuous (resp.,
  algebraic) Yoneda-complete quasi-metric spaces $X$ such that
  $\Lform X$ has an almost open addition map, and $1$-Lipschitz
  continuous maps, to the category of compact continuous (resp.,
  algebraic) Yoneda-complete quasi-metric spaces, which sends every
  object $X, d$ to the space of normalized forks on $X$ with the
  $\dKRH^a$-Scott quasi-metric ($a > 0$), and every $1$-Lipschitz
  continuous map $f$ to the map
  $(F^-, F^+) \mapsto (\Prev f (F^-), \Prev f (F^+))$.  \qed
\end{cor}

\begin{lem}
  \label{lemma:Ff:weak}
  Let $f \colon X \to Y$ be a continuous map between topological
  spaces.  The map $(F^-, F^+) \mapsto (\Prev f (F^-), \Prev f (F^+))$
  is continuous from the space of (sub)normalized forks on $X$ to the
  space of (sub)normalized forks on $Y$, both spaces being equipped
  with the weak topology.
\end{lem}
\proof The inverse image of
$[k > a]_- = \{(H^-, H^+) \mid H^- (k) > a\}$ is
$\{(F^-, F^+) \mid \Prev f (F^-) (k) > a\} = \{(F^-, F^+) \mid F^- (k
\circ f) > a\} = [k \circ f > a]_-$, and similarly the inverse image
of $[k > a]_+$ is $[k \circ f >a]_+$.  \qed

As in the proof of Proposition~\ref{prop:P:Vietoris}, we write
$\Box_{\mathcal P} U$ for the set of lenses $(Q, C)$ such that
$Q \subseteq U$, reserving the notation $\Box U$ for the set of
non-empty compact saturated subsets $Q$ of $X$ such that
$Q \subseteq U$.  We use a similar convention with
$\Diamond_{\mathcal P} U$ and $\Diamond U$.

\begin{prop}[$\dKRH^a$ quasi-metrizes the weak topology, forks]
  \label{prop:ADN:weak}
  Let $X, d$ be an continuous Yoneda-complete quasi-metric space.  Let us
  also assume that $\Lform X$ has an almost open addition map, and
  that $X$ is compact in the case where $\bullet$ is ``$1$''.  The
  $\dKRH^a$-Scott topology coincides with the weak topology on the
  space of (sub)normalized forks on $X$.
\end{prop}
\proof Let $Y$ denote the space of (sub)normalized forks on $X$, and
$Z$ denote the space of quasi-lenses on $\Val_\bullet X$.  By
Theorem~\ref{thm:V:weak=dScott}, the $\dKRH^a$-Scott topology
coincides with the weak topology on $\Val_\bullet X$, hence we may
take either topology in order to define the quasi-lenses on
$\Val_\bullet X$, hence the elements of $Z$.  Also,
$\Val_\bullet X, \dKRH^a$ is continuous Yoneda-complete by
Theorem~\ref{thm:Vleq1:cont} (if $\bullet$ is ``$\leq 1$''), or by
Theorem~\ref{thm:V1:cont} (if $\bullet$ is ``$1$'').  This allows us
to apply Theorem~\ref{thm:P:Vietoris} and deduce that the
$\mP {(\dKRH^a)}$-Scott and Vietoris topologies coincide on $Z$.

The pair of maps $r_\ADN$, $s_\ADN^\bullet$ defines a retraction of
$Z$ (with either topology) onto $Y$, with its weak topology, by
Proposition~\ref{prop:rsADN}.  By Theorem~\ref{thm:ADN:cont}, it also
defines a $1$-Lipschitz continuous retraction of $Z, \mP {(\dKRH^a)}$
onto $Y, \dKRH^a$.  Since the two spaces are continuous
Yoneda-complete, hence standard, this defines a topological retraction
of $Z$ onto $Y$, with its $\dKRH^a$-Scott topology.  This implies that
$s_\ADN^\bullet$ defines a topological embedding of $Y$, either with
its weak topology or with its $\dKRH^a$-Scott topology, into $Z$,
hence that the two topologies on $Y$ coincide, by
Fact~\ref{fact:retract:two}.  \qed

\begin{rem}
  \label{rem:ADN}
  We recall that every locally compact, coherent space $X$ is such
  that addition is almost open on $\Lform X$.  We may summarize some
  of our findings on forks in that useful situation as follows.  Let
  $a > 0$, and $\bullet$ be ``$\leq 1$'' or ``$1$''.  Let $X, d$ be a
  continuous Yoneda-complete quasi-metric space that is locally
  compact and coherent in its $d$-Scott topology, and also compact if
  $\bullet$ is ``$1$''.  Let $Z$ be the space of subnormalized forks
  on $X$ if $\bullet$ is ``$\leq 1$'', of normalized forks if
  $\bullet$ is ``$1$''.  Then:
  \begin{itemize}
  \item (Theorem~\ref{thm:ADN:cont}) $Z, \dKRH^a$ is continuous
    Yoneda-complete, and arises as a $1$-Lipschitz continuous retract
    of the Plotkin powerdomain over $\Val_\bullet X$ through $r_\ADN$
    and $s_\ADN^\bullet$.  That retraction cuts down to an isometry
    between $Z, \dKRH^a$ and the space of convex quasi-lenses over
    $\Val_\bullet X$, with the $\mP {(\dKRH^a)}$ quasi-metric.
    Directed suprema of formal balls are computed through naive
    suprema.
  \item (Theorem~\ref{thm:ADN:alg}) If $X, d$ is additionally
    algebraic, with a strong basis $\mathcal B$, then $Z, \dKRH^a$ is
    algebraic Yoneda-complete.
  \item (Proposition~\ref{prop:ADN:weak}) The $\dKRH^a$-Scott topology
    coincides with the weak topology on $Z$.
  \end{itemize}
\end{rem}

\section{Open Questions}
\label{sec:open-questions}

\begin{enumerate}
\item Assume $X, d$ standard algebraic.  Is $\Val_1 X, \dKRH$
  algebraic?  This is the case if $X$ has a root
  (Theorem~\ref{thm:V1:alg:root}), in particular if $d$ is bounded.
  Close results are that $\Val_1 X, \dKRH^a$ is algebraic for every
  $a > 0$ (Theorem~\ref{thm:V1:alg}), and that
  $\Val_{\leq 1} X, \dKRH$ is algebraic (Theorem~\ref{thm:V:alg}).
  Beware of Kravchenko's counterexample: the $\dKRH$-Scott topology on
  $\Val_1 X$ is in general different from the weak topology, and the
  coincidence with the weak topology in the rooted and $\dKRH^a$ cases
  followed more or less directly from algebraicity.
\item The above theorems apply to spaces of normalized, or
  subnormalized valuations.  Are there analogous results for the space
  $\Val X$ of \emph{all} continuous valuations?  I doubt it, since, by
  analogy, the space of all measures on a Polish space is in general
  not metrizable.
\item The nice properties we obtain on $\Val_1 X$ (algebraicity,
  continuity, retrieving the weak topology) were obtained either for
  $\dKRH$ under a rootedness assumption, or for $\dKRH^a$.  This
  prompted us to study sublinear previsions, superlinear previsions,
  and forks only through the $\dKRH^a$ quasi-metric.  I am pretty sure
  we would obtain analogous results with the $\dKRH$ quasi-metric in
  rooted cases.
\item Conversely, I have not studied the $\dKRH^a$ quasi-metric on the
  Plotkin powerdomain of quasi-lenses.  This should present no difficulty.
\item In case $X, d$ is complete metric (not just quasi-metric),
  $\Val_1 X, \dKRH$ is a complete metric space
  (Theorem~\ref{thm:V1:complete}).  We do not require any form of
  separability, and that is quite probably due to the fact that we consider
  valuations instead of measures.
  However, this does not say anything about a possible basis: is every
  probability valuation the ($\dKRH$-)limit of a Cauchy(-weightable)
  net of simple probability valuations in that case?  Note that every
  probability valuation is the $\dKRH^a$-limit of a Cauchy-weightable
  net of simple probability valuations.

\item The coupling Theorem~\ref{thm:coupling} is a form of linear
  duality theorem, in the style of Kantorovich and Rubinshte\u\i n,
  for quasi-metric spaces.  One should pursue this further, since the
  coupling $\Gamma$ is not known to arise from a probability valuation
  on $X \times X$ yet.  That extra effort was done in
  \cite{Gou-fossacs08b} for symcompact quasi-metric spaces, in their
  open ball topology.  Can we relax the assumptions?
\end{enumerate}




\bibliographystyle{alpha}
\newcommand{\etalchar}[1]{$^{#1}$}



\end{document}